\documentclass[12pt]{amsart}

\usepackage{amsmath,amsthm,amssymb}
\usepackage{graphicx,epstopdf}
\usepackage[margin=1in]{geometry}
\usepackage{kbordermatrix}
\usepackage[all]{xy}
\usepackage{physics}
\usepackage{tikz}
\usepackage{mathtools}
\usepackage[normalem]{ulem}
\usepackage[hidelinks]{hyperref}

\newtheorem{theorem}{Theorem}[section]
\newtheorem{lemma}[theorem]{Lemma}
\newtheorem{proposition}[theorem]{Proposition}
\newtheorem{corollary}[theorem]{Corollary}

\theoremstyle{definition} 
\newtheorem{definition}[theorem]{Definition}
\newtheorem{example}[theorem]{Example}
\newtheorem{convention}[theorem]{Convention}
\newtheorem{procedure}[theorem]{Procedure}
\newtheorem{remark}[theorem]{Remark}

\DeclareMathOperator{\ADBimod}{ADBimod}
\DeclareMathOperator{\can}{can}
\DeclareMathOperator{\forget}{forget}
\DeclareMathOperator{\fs}{fs}
\DeclareMathOperator{\fu}{fu}
\DeclareMathOperator{\HFK}{HFK}
\DeclareMathOperator{\Hom}{Hom}
\DeclareMathOperator{\hs}{hs}
\DeclareMathOperator{\id}{id}
\DeclareMathOperator{\low}{lower}
\DeclareMathOperator{\midd}{mid}
\DeclareMathOperator{\mult}{mult}
\DeclareMathOperator{\oo}{oo}
\DeclareMathOperator{\opp}{opp}
\DeclareMathOperator{\OSz}{OSz}
\DeclareMathOperator{\ou}{ou}
\DeclareMathOperator{\Rep}{Rep}
\DeclareMathOperator{\st}{st}
\DeclareMathOperator{\ungr}{ungr}
\DeclareMathOperator{\uo}{uo}
\DeclareMathOperator{\upp}{upper}
\DeclareMathOperator{\uu}{uu}

\renewcommand{\u}{\mathrm{u}}
\renewcommand{\o}{\mathrm{o}}

\newcommand{\C}{\mathbb{C}}
\newcommand{\F}{\mathbb{F}}
\newcommand{\Q}{\mathbb{Q}} 
\newcommand{\R}{\mathbb{R}}
\newcommand{\Z}{\mathbb{Z}}

\newcommand{\One}{\mathbf{1}}
\newcommand{\Ib}{\mathbf{I}}

\newcommand{\A}{\mathcal{A}}
\newcommand{\B}{\mathcal{B}}
\newcommand{\E}{\mathcal{E}}
\newcommand{\Fc}{\mathcal{F}}

\newcommand{\Sc}{\mathcal{S}}
\newcommand{\U}{\mathcal{U}}
\newcommand{\Zc}{\mathcal{Z}}

\newcommand{\BrCob}{\mathbf{BrCob}}
\newcommand{\co}{\colon}
\newcommand{\down}{\downarrow}
\newcommand{\gl}{\mathfrak{gl}}
\newcommand{\SC}{\mathcal{SC}}
\newcommand{\Udot}{\dot{U}}
\newcommand{\up}{\uparrow}

\newcommand{\ootimes}{ 
  \mathbin{
    \mathchoice
      {\buildcircleotimes{\displaystyle}}
      {\buildcircleotimes{\textstyle}}
      {\buildcircleotimes{\scriptstyle}}
      {\buildcircleotimes{\scriptscriptstyle}}
  } 
}

\newcommand\buildcircleotimes[1]{%
  \begin{tikzpicture}[baseline=(X.base), inner sep=0, outer sep=0]
    \node[draw,circle] (X)  {$#1\otimes$};
  \end{tikzpicture}%
}

\title[Trivalent vertices and bordered HFK in the standard basis]{Trivalent vertices and bordered knot Floer homology in the standard basis}
\author[Andrew Manion]{Andrew Manion}
\address{Department of Mathematics, North Carolina State University, 2108 SAS Hall, Raleigh, NC 27695}
\email{ajmanion@ncsu.edu}

\begin{document}

\begin{abstract}

We define new algebras, local bimodules, and bimodule maps in the spirit of Ozsv{\'a}th--Szab{\'o}'s bordered knot Floer homology. We equip them with the structure of 2-representations of the categorified negative half $\U^-$ of $U_q(\gl(1|1))$, 1-morphisms of such, and 2-morphisms respectively, and show that they categorify representations of $U_q(\gl(1|1)^-)$ and maps between them. Unlike with Ozsv{\'a}th--Szab{\'o}'s algebras, the algebras considered here can be built from a higher tensor product operation recently introduced by Rouquier and the author. 

Our bimodules are all motivated by holomorphic disk counts in Heegaard diagrams; for positive and negative crossings, the bimodules can also be expressed as mapping cones involving a singular-crossing bimodule and the identity bimodule. In fact, they arise from an action of the monoidal category of Soergel bimodules via Rouquier complexes in the usual way, the first time (to the author's knowledge) such an expression has been obtained for braiding bimodules in Heegaard Floer homology. 

Furthermore, the singular crossing bimodule naturally factors into two bimodules for trivalent vertices; such bimodules have not appeared in previous bordered-Floer approaches to knot Floer homology. The action of the Soergel category comes from an action of categorified quantum $\gl(2)$ on the 2-representation 2-category of $\U^-$ in line with the ideas of skew Howe duality, where the trivalent vertex bimodules are associated to 1-morphisms $\mathcal{E},\mathcal{F}$ in categorified quantum $\gl(2)$. 
\end{abstract}

\maketitle

\tableofcontents

\section{Introduction}

Heegaard Floer homology, due to Ozsv{\'a}th--Szab{\'o} \cite{HFOrig}, is a powerful set of invariants for low-dimensional manifolds; it is of modern interest in topology, representation theory, and physics. Heegaard Floer homology encodes a symplectic approach to Seiberg--Witten theory, which is closely related to categorified Witten--Reshetikhin--Turaev invariants associated to the Lie superalgebra $\gl(1|1)$. Indeed, Heegaard Floer homology contains a homological knot invariant (knot Floer homology or $\HFK$, \cite{HFKOrig,RasmussenThesis}) that categorifies the quantum $\gl(1|1)$ knot invariant (the Alexander polynomial), analogously to how Khovanov homology categorifies the Jones polynomial although the definitions of $\HFK$ and Khovanov homology are quite different.

Ozsv{\'a}th--Szab{\'o}'s theory of bordered $\HFK$ \cite{OSzNew,OSzNewer,OSzHolo,OSzLinks}, based on the ideas of bordered Floer homology \cite{LOTBorderedOrig}, is a major advance in the efficient computation of $\HFK$ (see \cite{HFKCalc}) as well as its algebraic structure and relationship to $\gl(1|1)$ categorification \cite{ManionDecat,ManionKS,LaudaManion,Hypertoric1}. If $V$ is the vector representation of $U_q(\gl(1|1))$, Ozsv{\'a}th--Szab{\'o} define algebras categorifying arbitrary mixed tensor products of $V$ and its dual $V^*$. They also define bimodules categorifying the intertwining maps associated to arbitrary tangles, recovering $\HFK$ for closed knots. Basic idempotents of Ozsv{\'a}th--Szab{\'o}'s algebras naturally correspond to canonical or crystal basis elements of the underlying representations; in particular, they do not correspond to elements of the standard tensor-product basis. Ellis--Petkova--V{\'e}rtesi have a related theory with similar properties \cite{PetkovaVertesi,EPV}, called tangle Floer homology; Tian \cite{TianUT} and Sartori \cite{SartoriCat} have categorifications of the $U_q(\gl(1|1))$ representation $V^{\otimes n}$ but do not recover $\HFK$ from their constructions.\footnote{For bordered $\HFK$, Tian's theory, and Sartori's theory, the quantum group that acts on $V^{\otimes n}$ is not $U_q(\gl(1|1))$ itself but rather a modified version. In Ellis--Petkova--Vert{\'e}si's theory, $U_q(\gl(1|1))$ itself acts, but on a representation that is slightly different from $V^{\otimes n}$.}

In general, Heegaard Floer homology associates chain complexes to Heegaard diagrams. Bordered Floer homology decomposes such diagrams along a set of non-intersecting cuts; to each cut one assigns an algebra, and to the pieces one assigns bimodules, recovering the chain complex of the original diagram after tensoring together the bimodules for local pieces over the boundary algebras. When applied to Heegaard diagrams arising from knot projections, the Heegaard diagram cuts often arise from cuts in the knot projection as in Figure~\ref{fig:OneCut}.

\begin{figure}
\includegraphics[scale=0.8]{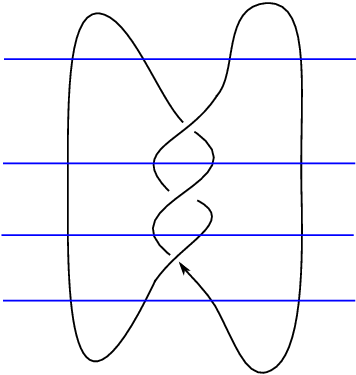}
\caption{Knot diagram cuts typically analyzed using bordered Floer homology.}
\label{fig:OneCut}
\end{figure}

Extending one level down, cornered Floer homology \cite{DM,DLM} decomposes a Heegaard diagram along two transverse cuts. In \cite{ManionRouquier}, Rapha{\"e}l Rouquier and the author reformulated and extended the cornered-Floer algebra gluing operation in terms of a tensor product construction $\ootimes$ for 2-representations of a monoidal dg category $\U$ introduced by Khovanov \cite{KhovOneHalf}. When applied to Heegaard diagrams arising from knot projections, one would like cornered Floer homology to recover the algebra for $n$ tangle endpoints (say positively oriented) as a tensor power $\ootimes^n$ of the algebra for a single positive endpoint, and one would like to recover the bordered-Floer bimodule for (e.g.) a crossing between strands $i,i+1$ out of $n$ strands from a simpler bimodule for a truly-local crossing with only two strands (see Figure~\ref{fig:TwoCuts}). 

\begin{figure}
\includegraphics[scale=0.8]{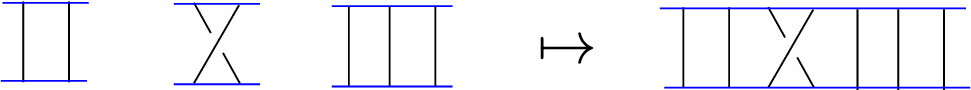}
\caption{Type of knot-diagram gluing where cornered Floer homology should be relevant.}
\label{fig:TwoCuts}
\end{figure}

The algebras appearing in Ozsv{\'a}th--Szab{\'o}'s bordered HFK cannot arise as tensor powers $\ootimes^n$; indeed, the basic idempotents of a tensor power $\ootimes^n$ would correspond to standard tensor-product basis elements of tensor-product representations, not canonical or crystal basis elements. For similar reasons, Petkova--V{\'e}rtesi's tangle Floer algebras \cite{PetkovaVertesi,EPV} cannot arise as tensor powers $\ootimes^n$; the algebras appearing in Tian \cite{TianUT} are closely related to tensor power algebras.

The above three categorifications of $V^{\otimes n}$ are closely related to the bordered-Floer ``strands algebras'' $\A(\Zc)$ for (generalized) arc diagrams\footnote{i.e. chord diagrams} $\Zc$ (Sartori's categorification is more algebraic but is closely related \cite{LaudaManion} to Ozsv{\'a}th--Szab{\'o}'s bordered $\HFK$). Specifically, the diagrams $\Zc$ in question appear as three of the four vertices of a square of diagrams; see Figure~\ref{fig:FourSquare}. The horizontal edges of the square are diagrammatic equivalences (sequences of arcslides) giving derived equivalences for the corresponding algebras $\A(\Zc)$. The vertical edges of the square are the operation $\mathcal{Z} \leftrightarrow \Zc_*$ from \cite{LOTMorphism} which should, by analogy with the results of \cite{LOTMorphism}, give Koszul dualities for the corresponding algebras $\A(\Zc)$.

\begin{figure}
\includegraphics[scale=0.7]{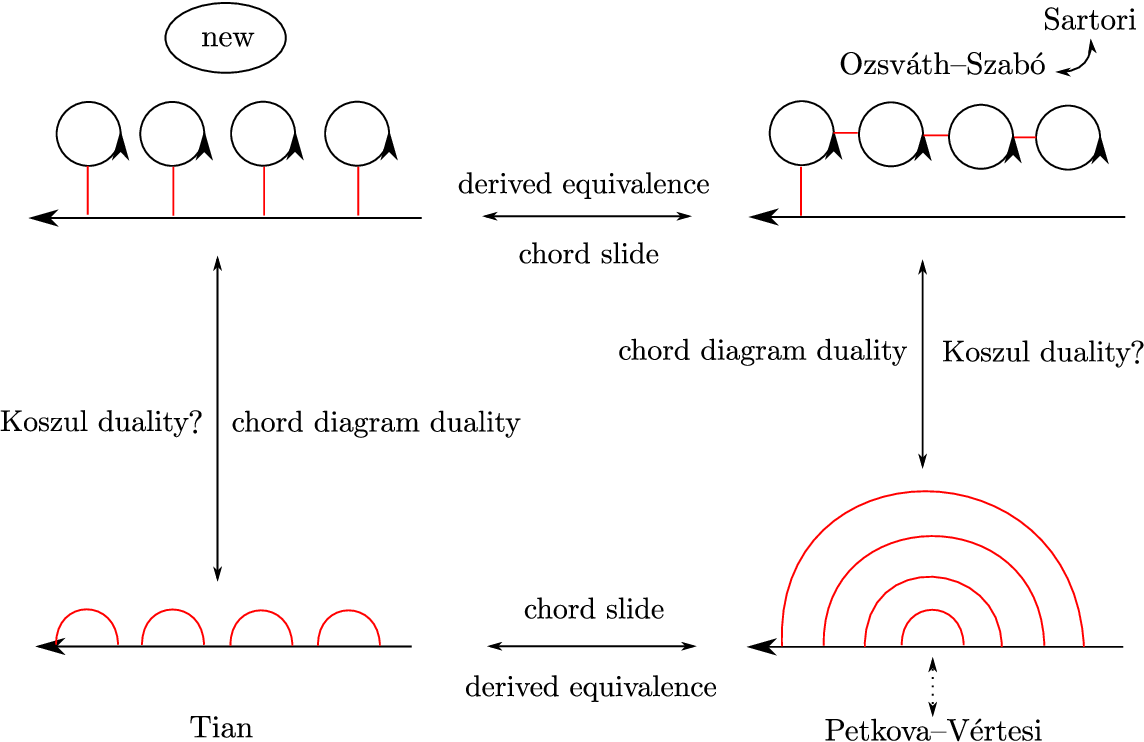}
\caption{Square of four chord diagrams (for $n=4$).}
\label{fig:FourSquare}
\end{figure}

The diagram $\Zc$ in the top left corner of Figure~\ref{fig:FourSquare} is the one remaining diagram in this square that has not yet been used to categorify $V^{\otimes n}$ or define bordered-Floer bimodules. Like the diagram related to Tian's theory (bottom left corner), the algebra $\A(\Zc)$ for the top-left diagram $\Zc$ arises as a tensor power $\ootimes^n$.

In this paper we initiate the study of categorifications of $\gl(1|1)$ representations and bordered theories for $\HFK$ based on the diagram $\Zc$ in the top-left corner of Figure~\ref{fig:FourSquare}. In order to work with bigraded algebras and bimodules, we define a differential bigraded 2-category $\U^-$ (based on $\U$) that categorifies the idempotented form $\Udot_q(\gl(1|1)^-)$ of the negative half of $U_q(\gl(1|1))$, as well as a bigraded version of the tensor product $\ootimes$ from \cite{ManionRouquier}.

\begin{theorem}[cf. Section~\ref{sec:BigradedTensor}, Section~\ref{sec:BigradedTensorDecat}]\label{thm:IntroCat}
The bigraded tensor product operation for 2-representations of $\U^-$ is well-defined and categorifies\footnote{See Section~\ref{sec:SplitGG} for the notion of categorification we use here.} the usual tensor product of representations of $\Udot_q(\gl(1|1)^-)$.
\end{theorem}

We then define bigraded algebras $\A_K$ for $K \geq 1$, equipped with 2-actions of $\U^-$, that categorify the representations $\wedge_q^K V$ of $\Udot_q(\gl(1|1)^-)$. The algebra $\A_1$ agrees with the strands algebra of the $n=1$ diagram in the top left corner of Figure~\ref{fig:FourSquare}; the other algebras $\A_K$ are new but closely related to $\A_1$. Taking bigraded tensor products, we get the following corollary.

\begin{corollary}[cf. Corollary~\ref{cor:GeneralKDecat}]
The algebra
\[
\A_{K_1,\ldots,K_n} \coloneqq \A_{K_1} \ootimes \cdots \ootimes \A_{K_n},
\]
with its 2-action of $\U^-$, categorifies the representation $\wedge_q^{K_1} V \otimes \cdots \otimes \wedge_q^{K_n} V$ with its action of $\Udot_q(\gl(1|1)^-)$.
\end{corollary}

While analogous categorification results often require detailed computations to show that the action of the quantum group on the decategorification is correct, here we get the result immediately from properties of the bigraded tensor product. The results of \cite{ManionRouquier} imply that $\A_{1,\ldots,1}$ agrees with the strands algebra of the general diagram in the top-left corner of Figure~\ref{fig:FourSquare}; the bigraded tensor product gives this algebra a bigrading.

We then proceed to define the basic ``truly local'' pieces appearing in gluings like the one shown in Figure~\ref{fig:TwoCuts}, beginning with bimodules for trivalent vertices. In Sections \ref{sec:EasyVertex} and \ref{sec:HardVertex} we describe bimodules ${^{\vee}}\Lambda$ and ${^{\vee}}Y$ associated to trivalent vertices; ${^{\vee}}\Lambda$ is a $(\A_2,\A_{1,1})$-bimodule and ${^{\vee}}Y$ is a $(\A_{1,1},\A_2)$-bimodule. Both ${^{\vee}}\Lambda$ and ${^{\vee}}Y$ are AD bimodules, a type of $A_{\infty}$ bimodule appearing in bordered Floer homology.

\begin{theorem}[cf. Sections~\ref{sec:EasyVertex}, \ref{sec:HardVertex}]
The trivalent vertex bimodules ${^{\vee}}\Lambda$ and ${^{\vee}}Y$ are well-defined AD bimodules and categorify $U_q(\gl(1|1))$-intertwining maps $V^{\otimes 2} \to \wedge_q^2 V$ and $\wedge_q^2 V \to V^{\otimes 2}$ respectively arising from skew Howe duality (see Appendix~\ref{sec:Uqgl11Review}).
\end{theorem}

\begin{corollary}
The singular crossing AD bimodule ${^{\vee}}X := {^{\vee}}Y \boxtimes {^{\vee}}\Lambda$ of Section~\ref{sec:SingularCrossing}, over the algebra $\A_{1,1}$, is well-defined and categorifies the $U_q(\gl(1|1))$-intertwining map $V^{\otimes 2} \to V^{\otimes 2}$ for a singular crossing.
\end{corollary}

Both trivalent vertex bimodules are directly motivated by domains for holomorphic curve counts in the two Heegaard diagrams on the left of Figure~\ref{fig:IntroHDs}; see Figures~\ref{fig:EasyVertexDomains},~\ref{fig:HardVertexDomains}. As the rest of Figure~\ref{fig:IntroHDs} illustrates, these diagrams glue to a diagram that agrees (after Heegaard moves) with a reasonable adaptation of the Ozsv{\'a}th--Stipsicz--Szab{\'o} diagram \cite{OSSz} for a singular crossing; see also Figure~\ref{fig:OSSzDiag}. The presence of $O_1$ and $O_2$ basepoints in one of the trivalent-vertex diagrams leads to an infinitely-generated bimodule; we simplify to a homotopy equivalent but finitely generated version and work primarily with this simplification. 

\begin{remark}
The compatibility requirements of the structures considered here are relatively restrictive on possible choices of convention; these requirements led us to formulate the above results in terms of AD bimodules rather than the more familiar DA bimodules. However, since we still prefer DA bimodules, we will do most computations in terms of DA bimodules, then take duals to get AD bimodule versions.
\end{remark}

\begin{figure}
\includegraphics[scale=0.8]{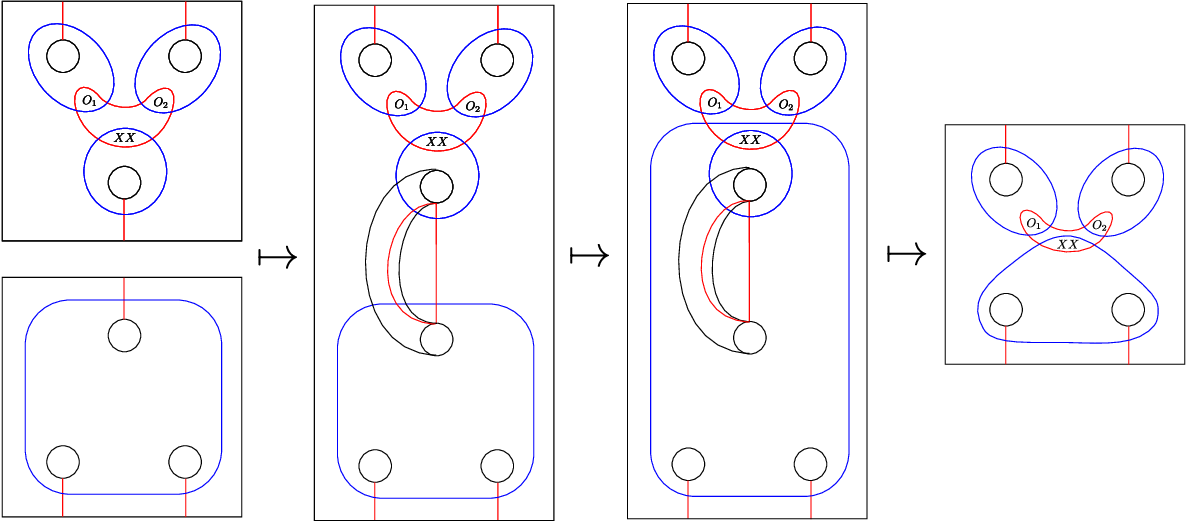}
\caption{Heegaard diagrams for trivalent vertices being glued to form Heegaard diagram for singular crossing (after a $\beta$ handleslide and a destabilization).}
\label{fig:IntroHDs}
\end{figure}

\begin{figure}
\includegraphics[scale=0.8]{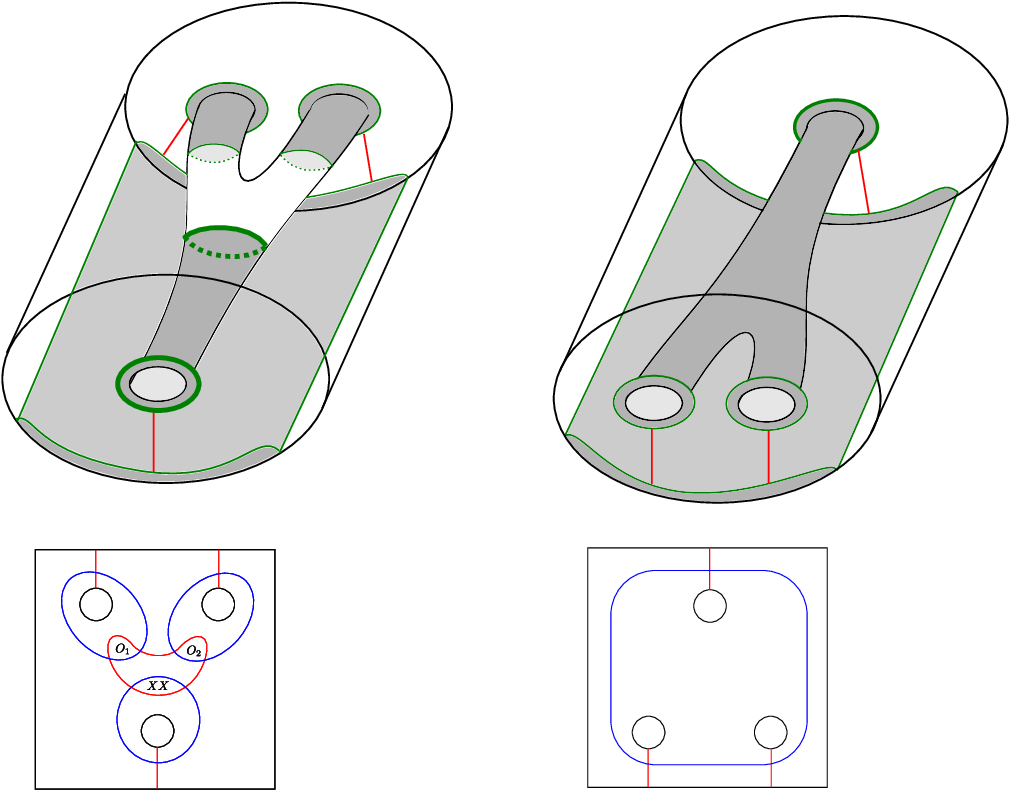}
\caption{Bordered sutured cobordisms represented by the Heegaard diagrams from Figure~\ref{fig:IntroHDs}.}
\label{fig:Cobordisms}
\end{figure}

\begin{remark}
The new algebra $\A_2$ categorifying $\wedge_q^2 V$, with its extra polynomial generator as compared to $\A_1$, is exactly what is needed to permit the recovery of the singular-crossing bimodule from the trivalent-vertex pieces. The algebras $\A_K$ are defined by analogy to $\A_2$.
\end{remark}

\begin{remark}\label{rem:WeightedSutures}
By the rules for assigning 3d cobordisms to Heegaard diagrams in bordered sutured Floer homology \cite{Zarev}, the trivalent-vertex Heegaard diagrams in Figure~\ref{fig:IntroHDs} should represent bordered sutured cobordisms like the ones shown in Figure~\ref{fig:Cobordisms} (see \cite[Figure 2(b)]{Zarev} for a figure drawn with similar visual conventions). Each 3d cobordism is the complement of (an open tubular neighborhood of) the corresponding trivalent vertex, viewed as a web in $D^2 \times [0,1]$, with a bordered sutured structure on the boundary.

While the literature does not seem to give a sutured interpretation for the doubled $X$ basepoint in the top-left Heegaard diagram of Figure~\ref{fig:IntroHDs}, it is reasonable to suppose that an $X$ or $O$ basepoint appearing $m$ times in some region of a Heegaard diagram should give rise to a ``suture with weight $m$,'' and that sutured and bordered sutured Floer homology can be generalized to accommodate such weighted sutures. In Figure~\ref{fig:Cobordisms}, we indicate weight-two sutures with green circles that are thicker than the ones for weight-one sutures.
\end{remark}

\begin{remark}
In contrast with the top-left Heegaard diagram of Figure~\ref{fig:IntroHDs}, the bottom-left Heegaard diagram of this figure does not by itself seem to give any indication that the algebra on the one-circle side should be $\A_2$ rather than $\A_1$. It is plausible that, if a generalized bordered sutured theory including higher-weight sutures and algebras $\A_K$ does exist, one would need to specify, along with the bottom-left Heegaard diagram of Figure~\ref{fig:IntroHDs}, the additional data that the top circle has weight 2 in order to get a bimodule involving $\A_2$ from the diagram. 
\end{remark}

We equip the trivalent vertex bimodules (and thus their compositions, like singular crossing bimodules) with the structure of 1-morphisms between 2-representations of $\U^-$, for a definition of 1-morphism adapted to the setting of AD and DA bimodules as discussed in Section~\ref{sec:2RepMorphisms}. 

Passing to 2-morphisms, we further upgrade these trivalent vertex bimodules to form the 1-morphism part of a functor from $\dot{\U}_q(\gl(2))$, the 2-category categorifying the idempotented form $\Udot_q(\gl(2))$ (see \cite{KLIII, Rouquier2KacMoody}), to the 2-representation bicategory $2\Rep(\U^-)$ of $\U^-$ as defined in Section~\ref{sec:2RepMorphisms}.\footnote{This appearance of $\gl(2)$ acting on representations of $\gl(1|1)$ is an instance of skew Howe duality; see Remark~\ref{rem:IntroSkewHowe} for further discussion.}
\begin{theorem}[cf. Theorem~\ref{thm:SkewHowe2Action}]\label{thm:IntroSkewHowe}
The 2-morphisms in $2\Rep(\U^-)$ associated to the generating dot, crossing, cap, and cup 2-morphisms in $\dot{\U}_q(\gl(2))$ are well-defined and satisfy the defining relations for 2-morphisms in $\dot{\U}_q(\gl(2))$, so they give a functor from $\dot{\U}_q(\gl(2))$ to $2\Rep(\U^-)$ (descending to the ``Schur quotient'' $\Sc(2,2)$ of $\dot{\U}_q(\gl(2))$ from Mackaay--Sto{\v s}i{\' c}--Vaz \cite{MSVSchur}).
\end{theorem}

\begin{remark}
In particular, our trivalent vertex bimodules are biadjoint to each other; this property follows from the relations in $\dot{\U}_q(\gl(2))$.
\end{remark}

We get the following corollary from \cite{MSVSchur}.

\begin{corollary}[cf. Corollary~\ref{cor:SoergelFunctor}]\label{cor:IntroSoergel}
We have a functor from the Soergel category $\SC'_1$ (viewed as a 2-category with 1 object; see Section~\ref{sec:Soergel} for notation) to $2\Rep(\U^-)$.
\end{corollary}

Further, by Elias--Krasner \cite{EliasKrasner} we also get a functor from the 2-strand braid cobordism category. Let ${^{\vee}}P'$ and ${^{\vee}}N'$ be the resulting AD bimodules for positive and negative crossings, obtained as mapping cones involving ${^{\vee}}X$ and the identity AD bimodule over $\A_{1,1}$, as in Section~\ref{sec:PosNegCrossings}.

\begin{corollary}[cf. Corollary~\ref{cor:BrCob}]\label{cor:IntroBraidCob}
The positive and negative crossing bimodules ${^{\vee}}P'$ and ${^{\vee}}N'$ extend to a functor from the 2-strand braid cobordism category $\BrCob(2)$ (viewed as a 2-category with 1 object) to a bicategory $\ADBimod$ of dg categories, certain AD bimodules, and homotopy classes of closed AD bimodule morphisms.
\end{corollary}

Due to algebraic subtleties in taking mapping cones on 2-morphisms, we do not automatically get a functor from $\BrCob(2)$ to $2\Rep(\U^-)$. However, at least at the level of 1-morphisms, we can lift from $\BrCob(2) \to \ADBimod$ to $\BrCob(2) \to 2\Rep(\U^-)$ as discussed in Remark~\ref{rem:UpgradingBrCobFunctor}. In particular, for truly local positive and negative crossings, we have 1-morphisms between 2-representations of $\U^-$, and for 2-strand braid cobordisms we have bimodule maps (presumably 2-morphisms) satisfying the movie moves. Functoriality for 4d cobordisms, like braid cobordisms, is an important \textit{a priori} part of Heegaard Floer theory, but to our knowledge braid cobordism maps have not yet been defined in any local, bordered-Floer-based approach to $\HFK$.

Although they are defined over different algebras, it is natural to ask how our positive and negative crossing bimodules relate to Ozsv{\'a}th--Szab{\'o}'s bordered $\HFK$ bimodules for positive and negative crossings. To address this question, we define explicit bimodules for the derived equivalence from the top edge of Figure~\ref{fig:FourSquare} in the $n=2$ case, based on certain ``change-of-basis'' Heegaard diagrams.

\begin{theorem}[cf. Section~\ref{sec:ChangeOfBasis}]\label{thm:IntroChangeOfBasis}
Our change-of-basis bimodules are well-defined, homotopy inverses to each other, and categorify the change-of-basis matrices between the tensor-product basis and the canonical basis (see Appendix~\ref{sec:NonstandardBases}) for $V^{\otimes 2}$.
\end{theorem}

Tensoring our nonsingular crossing bimodules on either side with change-of-basis bimodules, we can compare with Ozsv{\'a}th--Szab{\'o}'s nonsingular crossing bimodules.

\begin{theorem}[cf. Section~\ref{sec:OSzRelationship}]\label{thm:IntroOSz}
After categorified change of basis, our nonsingular crossing bimodules are homotopy equivalent to Ozsv{\'a}th--Szab{\'o}'s.
\end{theorem}

\begin{figure}
\includegraphics[scale=0.7]{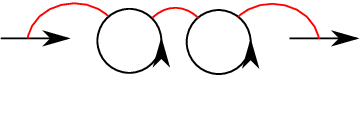}
\caption{Diagram $\Zc$ needed for Ozsv{\'a}th--Szab{\'o}'s minimal local crossing bimodules.}
\label{fig:OSzFullLocalZ}
\end{figure}

\begin{remark}
The Ozsv{\'a}th--Szab{\'o} bimodules appearing in Theorem~\ref{thm:IntroOSz} are simpler than the minimal bimodules Ozsv{\'a}th--Szab{\'o} need to encode all the relevant local holomorphic-curve data in their theory. We work with Ozsv{\'a}th--Szab{\'o} bimodules defined over an algebra $\A_{1,1}^{\can}$ (quasi-isomorphic to the strands algebra of the $n=2$ top-right corner of Figure~\ref{fig:FourSquare} by \cite{LP,MMW1,MMW2}; our notation indicates the relationship between this algebra and the canonical basis for $V^{\otimes 2}$). However, the minimal local bimodules in Ozsv{\'a}th--Szab{\'o}'s theory are defined over a larger algebra $\B(2)$ quasi-isomorphic to $\A(\Zc)$ for $\Zc$ the diagram shown in Figure~\ref{fig:OSzFullLocalZ}. While $\A_{1,1}^{\can}$ has four basic idempotents (corresponding to four basis elements of $V^{\otimes 2}$), the algebra $\B(2)$ has eight basic idempotents.

Even though simpler Ozsv{\'a}th--Szab{\'o} bimodules over the smaller algebra $\A_{1,1}^{\can}$ appear in Theorem~\ref{thm:IntroOSz}, we expect that the corresponding bimodules over $\A_{1,1}$ (with their 1-morphism structure) contain all the local data needed to build $n$-strand crossing bimodules as in Figure~\ref{fig:TwoCuts}. This is one advantage of the algebras $\A_{1,\ldots,1}$ and the tensor product construction; just as standard tensor-product basis elements are easier to glue together than canonical basis elements are, the extension process that builds $n$-strand bimodules out of our 2-strand bimodules should be simpler (in terms of numbers of idempotents, although perhaps not in terms of holomorphic geometry) and more algebraically structured than in Ozsv{\'a}th--Szab{\'o}'s theory.

\end{remark}

\begin{remark}
The Heegaard diagrams appearing in this paper can be readily drawn and composed in the plane (although as in Figure~\ref{fig:IntroHDs}, the compositions often produce tubes which can be removed after destabilization). Thus, one can think of the theory developed here as an instance of ``bordered knot Floer homology done using planar diagrams'' in a broad sense. However, ``the planar diagram'' also has a more specific meaning in knot Floer homology, referring to a specific diagram introduced in \cite{OSzCube,ManolescuCube}. The diagrams considered here do not look exactly like this specific planar diagram, although they are related; for a proposed approach to defining a bordered HFK theory based on the planar diagram of \cite{OSzCube,ManolescuCube} using the $n$-strand Ozsv{\'a}th--Szab{\'o} algebras $\A_{1,\ldots,1}^{\can}$, see \cite{ManionDiagrams}.
\end{remark}

\begin{remark}
Various bimodules below will be equipped with the structure of 1-morphisms of 2-representations. For bimodules arising from Heegaard diagrams, the 1-morphism structure should also arise from counts of holomorphic disks whose domains have multiplicity at corners of the Heegaard diagram as in \cite{DM,DLM}. We hope to return to this point in a future paper. It would also be desirable to have Heegaard diagram interpretations for the bimodule maps we define in Section~\ref{sec:SkewHowe2Action} below; perhaps these maps arise from Heegaard diagram representations of 4d cobordisms between web complements given by foam complements with certain sutured structure.
\end{remark}

\begin{remark}\label{rem:COB1Mor}
The change-of-basis bimodules in Section~\ref{sec:ChangeOfBasis} can be upgraded to 1-morphisms of 2-representations, and one can deduce 1-morphism structure for the Ozsv{\'a}th--Szab{\'o} positive and negative crossing bimodules considered in Section~\ref{sec:OSzRelationship}. Diagrammatically, though, one can see that the 1-morphism structure on our change-of-basis bimodules will not be enough to define change-of-basis bimodules between $\A_{1,\ldots,1}$ and $\A_{1,\ldots,1}^{\can}$ for an arbitrary number $n$ of strands; a more global definition will be required. For Ozsv{\'a}th--Szab{\'o}'s bimodules, the 1-morphism structure is most interesting for their full local bimodules over the algebra $\B(2)$, where it encodes a compatibility between the summands of their bimodules that morally underlies their extensions as in Figure~\ref{fig:TwoCuts}; this is another point to which we hope to return in a future paper.
\end{remark}

\begin{remark}
One would like to extend $\ootimes$ to a monoidal structure on $2\Rep(\U^-)$ (or a related 2-category) and then upgrade to a braided monoidal structure. Unlike with Ozsv{\'a}th--Szab{\'o}'s bimodules, the bimodules for positive and negative crossings defined here should be instances of this higher braiding.
\end{remark}

\begin{remark}\label{rem:IntroSkewHowe}
Our constructions give a categorification of the ``skew Howe'' representation $\wedge_q^2 (\C_q^{1|1} \otimes \C_q^2)$ with its commuting actions of $U_q(\gl(1|1)^-)$ and $U_q(\gl(2))$; see \cite{QS, TVW, LQR}. Both actions are categorified here, and the 1-morphism structure on our trivalent vertex bimodules encodes the commutativity. It would be desirable to extend to a categorification of $\wedge_q^K(\C_q^{1|1} \otimes \C_q^n)$ for arbitrary $n$ and $K$.
\end{remark}

\begin{remark}
To recover $\HFK$ from the constructions of this paper, one would first want to extend $\ootimes$ from objects to at least (DA or AD bimodule) 1-morphisms in $2\Rep(\U^-)$, allowing one to perform gluings as in Figure~\ref{fig:TwoCuts}. One would then want to extend to mixed orientations (both $V$ and $V^*$) and define bimodules for local maximum and minimum points. Invariance of tangle invariants under Reidemeister moves, and recovery of $\HFK$, would follow if one could also define change-of-basis bimodules between $\A_{1,\ldots,1}$ and $\A_{1,\ldots,1}^{\can}$ as mentioned in Remark~\ref{rem:COB1Mor} and establish a general relationship with Ozsv{\'a}th--Szab{\'o}'s theory. It would also be desirable to derive the Reidemeister invariance as a consequence of an action of the $n$-strand braid cobordism category coming from an action of $\dot{\U}_q(\gl(n))$, as part of a categorification of $\wedge_q^K(\C_q^{1|1} \otimes \C_q^n)$ for all $n,K$.
\end{remark}

\subsection*{Organization}

In Section~\ref{sec:BorderedAlg}, we review what we need of the algebra of bordered Floer homology. We also introduce a matrix-based notation to facilitate computations with the $A_{\infty}$ bimodules and morphisms that will appear frequently below. 

In Section~\ref{sec:2Reps}, we define a bigraded tensor product operation and discuss its decategorification, proving Theorem~\ref{thm:IntroCat}. We also define 1-morphisms and 2-morphisms of 2-representations of the dg category $\U^-$ introduced in this section.

In Section~\ref{sec:OurExamples}, we define the algebras $\A_K$ and study their tensor products, including their 2-representation structure. For $\A_{1,\ldots,1}$, we discuss how the 2-representation structure relates to strands pictures and to Heegaard diagrams.

In Section~\ref{sec:EasyVertex}, we define a 1-morphism of 2-representations for one type of trivalent vertex, and in Section~\ref{sec:HardVertex} we do the same for the other type. Compositions of these bimodules are analyzed in Section~\ref{sec:SimpleWebs}.

In Section~\ref{sec:SkewHowe2Action}, we extend to an action of 2-morphisms in $\dot{\U}_q(\gl(2))$, proving Theorem~\ref{thm:IntroSkewHowe}. In Section~\ref{sec:SoergelBraidCob} we deduce Corollaries~\ref{cor:IntroSoergel} and \ref{cor:IntroBraidCob}, while also extending our functor to the Soergel category $\SC'_1$.

In Section~\ref{sec:PosNegCrossings}, we simplify the resulting bimodules for positive and negative crossings and relate them to bimodules motivated by holomorphic curve counts in Heegaard diagrams. In Section~\ref{sec:ChangeOfBasis}, we define change-of-basis bimodules between our algebras and the analogous algebras in Ozsv{\'a}th--Szab{\'o}'s bordered $\HFK$, proving Theorem~\ref{thm:IntroChangeOfBasis}. In Section~\ref{sec:OSzRelationship} we use the change-of-basis bimodules to prove Theorem~\ref{thm:IntroOSz}.

Finally, in Appendix~\ref{sec:Uqgl11Review}, we review what we need of the representation theory of $U_q(\gl(1|1))$.

\subsection*{Acknowledgments}

The author would like to thank Aaron Lauda, Ciprian Manolescu, Aaron Mazel--Gee, Peter Ozsv{\'a}th, Ina Petkova, Rapha{\"e}l Rouquier, and Zolt{\'a}n Szab{\'o} for useful conversations. The author would especially like to thank the referee for a careful reading of the details of the paper. The author was partially supported by NSF grants DMS-1902092 and DMS-2151786 and Army Research Office W911NF2010075.

\section{Bordered algebra and matrix notation}\label{sec:BorderedAlg}

In this section we review relevant aspects of the algebra of bordered Floer homology and introduce a convenient matrix-based notation for DA bimodules. Let $k$ be a field of characteristic $2$; we will occasionally work over more general $k$, e.g. polynomial rings over fields of characteristic 2, but what we say here generalizes to this setting without issue. 

All categories, algebras, modules, etc. discussed below are assumed to be $k$-linear. We will also assume that these algebras and modules come equipped with a bigrading by $\Z^2$, consisting of a $\Z$-grading $\deg^q$ (the quantum grading) with respect to which algebra multiplication maps, module action maps, and differentials are degree $0$, as well as a $\Z$-grading $\deg^h$ (the homological grading) with respect to which differentials are degree $1$. For a bigraded vector space $W$, we write $q^i W [j]$ for $W$ in which all quantum degrees have been shifted up by $i$ and all homological degrees have been shifted down by $j$. 

\begin{remark}
The below definitions also make sense in the ungraded setting, where the degree shifts in the below formulas should be ignored, as well as in the more general setting of gradings by nonabelian groups $G$ as is common in bordered Floer homology (see \cite{LOTBimodules} for the definitions in the $G$-graded case).
\end{remark}

\subsection{DA bimodules}

If $A$ is a dg category, we write $I$ for the non-full subcategory of $A$ having the same objects as $A$ but having only identity morphisms. We think of $A$ as being a (possibly non-unital) dg algebra equipped with a collection of orthogonal idempotents given by identity morphisms on the objects of $A$. Correspondingly, we will call $I$ the \emph{idempotent ring} or \emph{ring of idempotents} of $A$, even though $I$ is a category. In all examples considered in this paper, $A$ (and thus $I$) will have finitely many objects.

By convention, we assume all dg categories $A$ come equipped with an augmentation functor $\epsilon\co A \to I$ restricting to the identity on $I \subset A$; we write $A_+$ for the kernel of $\epsilon$.

\begin{definition}\label{def:DABimod}
Let $A, A'$ be dg categories with rings of idempotents $I, I'$. A DA bimodule over $(A',A)$ is a pair $(M,\{\delta^1_i: i \geq 1\})$ where $M$ is a bigraded $(I',I)$-bimodule (i.e. a $k$-linear functor from $I' \otimes I^{\opp}$ to bigraded $k$-vector spaces) and, for $i \geq 1$,
\[
\delta^1_i\co M \otimes_{I} (A[1])^{\otimes(i-1)} \to A'[1] \otimes_{I'} M
\]
is an $(I',I)$-bilinear map of bidegree zero satisfying the DA bimodule relations
\begin{align*}
&\sum_{j=1}^i(\mu' \otimes \id) \circ (\id \otimes \delta^1_{i-j+1}) \circ (\delta^1_j \otimes \id) \\
&+ \sum_{j=1}^{i-1} \delta^1_i \circ (\id \otimes \partial_j) \\
&+ \sum_{j=1}^{i-2} \delta^1_{i-1} \circ (\id \otimes \mu_{j,j+1}) \\
&+ (\partial' \otimes \id) \circ \delta^1_i \\
&= 0,
\end{align*}
where $\mu', \partial'$ denote the multiplication and differential on $A'$, $\partial_j$ denotes the differential on the $j^{th}$ factor of $A^{\otimes(i-1)}$, and $\mu_{j,j+1}$ denotes the multiplication on the $j^{th}$ and $(j+1)^{st}$ factors of $A^{\otimes(i-1)}$.
\end{definition}

\begin{example}
The identity DA bimodule $\mathbb{I}_A$ over $A$ has $\mathbb{I}_A(S',S) = k$ if $S = S'$ and $\mathbb{I}_A(S',S) = 0$ otherwise, so $\mathbb{I}_A$ is the identity bimodule over the idempotent ring $I$. We set $\delta^1_2$ to be the identity map (endofunctor) on $A[1]$; we set $\delta^1_i = 0$ for $i \neq 2$. One can check that the DA bimodule relations are satisfied.
\end{example}

\begin{remark}
A DA bimodule $M$ has an underlying $A_{\infty}$ bimodule $A' \otimes_{I'} M$. As a left $A_{\infty}$ module, $A' \otimes_{I'} M$ has no higher actions; it is a left dg module (even projective as a non-differential module). However, the right action of $A$ on $A' \otimes_{I'} M$ may have higher $A_{\infty}$ terms. We have a natural identification of $A \otimes_{I} \mathbb{I}_A$ with the ordinary identity bimodule over $A$.
\end{remark}

Let $\delta^1 \coloneqq \sum_i \delta^1_i$, a map from $M \otimes_{I} T^*(A[1])$ to $A'[1] \otimes_{I'} M$ of bidegree zero. Define 
\[
\delta^j_i\co M \otimes_{I} (A[1])^{\otimes(i-1)} \to (A'[1])^{\otimes j} \otimes_{I'} M
\]
by
\[
\delta^j_i(-,a_1,\ldots,a_{i-1}) = \sum_{i-1 = (i_1 - 1) + \cdots + (i_j - 1)} (\id \otimes \delta^1_{i_j}(-,\ldots,a_{i-1})) \circ \cdots \circ \delta^1_{i_1}(-,a_1,\ldots,a_{i_1 - 1}),
\]
where there are $j$ factors in the composition.

\begin{definition}
We say that a DA bimodule $M$ is:
\begin{itemize}
\item left bounded if for each $x \in M$ and each $i \geq 1$, there exists $n$ such that $\delta^j_i(x,-,\ldots,-)$ vanishes on $(A[1])^{\otimes(i-1)}$ for all $j > n$;
\item finitely generated if $M$ is finite-dimensional over $k$;
\item strictly unital if, for all $x \in M(S',S)$, we have 
\begin{itemize}
\item $\delta^1_2(x,\id_{S}) = \id_{S'} \otimes x$ 
\item $\delta^1_i(x,\ldots,\id_{S''},\ldots) = 0$ for $i > 2$ and any $S''$.
\end{itemize}
\end{itemize} 
\end{definition}

\subsection{Matrix notation}

We specify a DA bimodule $M$ (finitely generated at first) by giving a ``primary matrix'' and a ``secondary matrix.'' Let $A$ and $A'$ be dg categories with finitely many objects; we start with a primary matrix whose columns are indexed by objects $S$ of $A$, whose rows are indexed by objects $S'$ of $A'$, and whose entries are sets (finite sets at first). If the primary matrix has block form, we will often give each block separately. From the primary matrix, we can define $M$ as an $(I',I)$-bimodule; we take the set in row $S'$, column $S$ as a basis for $M(S',S)$. 

We also assume we are given a secondary matrix with both rows and columns indexed by entries of the primary matrix. The entries of the secondary matrix should be sums (potentially infinite) of expressions like $a' \otimes (a_1, \ldots, a_{i-1})$ where $a' \in A'$ and $(a_1, \ldots, a_{i-1})$ are distinct sequences of elements of some chosen homogeneous basis for (the morphism spaces in) $A$ such that: 
\begin{itemize}
\item identity morphisms are basis elements;
\item all other basis elements are in $A_+$.
\end{itemize}
We require that no $a_j$ is an identity morphism in any such sequence $(a_1, \ldots, a_{i-1})$. By convention, we write $0$ for the empty sum, and when $i=1$ we write $a'$ as shorthand for $a' \otimes ()$. As with the primary matrix, we will often give secondary matrices block-by-block. There are many examples below; see, for instance, Section~\ref{sec:SingularCrossing}.

\begin{definition}
Given this data, if $(a_1,\ldots,a_{i-1})$ is a sequence of elements in the chosen basis for $A$ (possibly empty but none an identity morphism) and $x$ is a basis element of $M(S',S)$, we take $\delta^1_i(x \otimes a_1 \otimes \cdots \otimes a_{i-1})$ to be the sum, over all rows $y$ in column $x$ of the secondary matrix, of $a' \otimes y$ if $a' \otimes (a_1, \ldots, a_{i-1})$ is a term of the secondary matrix entry in row $y$ and column $x$ and zero otherwise. For any identity morphism $\id_{S}$ in $A$ and any basis element $x$ of $M(S',S)$, we set $\delta^1_2(x,\id_{S}) = \id_{S'} \otimes x$. We extend $\delta^1_i$ linearly over $k$. 
\end{definition}

A main advantage of this matrix-based notation is that it allows computations involving DA bimodules to be formulated using familiar linear-algebraic operations. For example, we describe a procedure for checking that a DA bimodule defined by primary and secondary matrices is well-defined.
\begin{procedure}\label{proc:DAWellDef}
To check that the DA bimodule relations hold for $M$ defined by primary and secondary matrices, the first step is to multiply the secondary matrix by itself. When multiplying (a term of) an entry $a' \otimes (a_1, \ldots, a_{i-1})$ in the right matrix factor with another (term of an) entry $b' \otimes (b_1, \ldots, b_{j-1})$ in the left matrix factor, the result is $a'b' \otimes (a_1, \ldots, a_{i-1}, b_1,\ldots,b_{j-1})$. 

Next, one considers a ``multiplication matrix'' derived from the secondary matrix as follows: whenever an algebra input $a_j$ in a term $a' \otimes (a_1,\ldots,a_{i-1})$ of the secondary matrix (row $y$, column $x$) is a nonzero term of the basis expansion of $b_1 \cdot b_2$ for two basis elements $b_1, b_2$ of $A$ (with coefficient $c \in k$), add a term $ca' \otimes (a_1,\ldots,a_{j-1},b_1,b_2,a_{j+1},\ldots,a_{i-1})$ to the multiplication matrix in row $y$ and column $x$. 

Finally, one considers two ``differential matrices;'' the first is obtained from the secondary matrix by replacing each term $a' \otimes (a_1,\ldots,a_{i-1})$ with $\partial'(a') \otimes (a_1,\ldots,a_{i-1})$ in each entry. The second is obtained as follows: whenever an algebra input $a_j$ in a term $a' \otimes (a_1,\ldots,a_{i-1})$ of the secondary matrix (row $y$, column $x$) is a nonzero term of the basis expansion of $\partial(b)$ for some basis element $b$ of $A$ (with coefficient $c \in k$), add a term $ca' \otimes (a_1,\ldots,a_{j-1},b,a_{j+1},\ldots,a_{i-1})$ to the second differential matrix in row $y$ and column $x$.

The DA bimodule relations amount to checking that the sum of the squared secondary matrix, the multiplication matrix, and the two differential matrices is zero.
\end{procedure}

\begin{remark}
A priori, one would want to add terms $\id_{S'} \otimes \id_{S}$ to the diagonal entries of the secondary matrix before performing the above operations; one can check that the above procedure (without $\id_{S'} \otimes \id_{S}$ terms) suffices to check the DA bimodule relations, and that DA bimodules arising from primary and secondary matrices are strictly unital.
\end{remark}

\begin{remark}
With appropriate care, one can apply the same ideas to infinitely-generated DA bimodules whose generators come in regular families. We will see several examples below; in the secondary matrix, even if there are infinitely many rows, one requires that only finitely many entries of each column for any given input sequence are nonzero.
\end{remark}

\begin{remark}
Suppose that $M$ is a DA bimodule over $(A',A)$ with $\delta^1_i = 0$ for $i > 2$. It follows that $A' \otimes_{I'} M$ is an ordinary dg bimodule, with differential given by $\delta^1_1$ and right action of $A$ given by $\delta^1_2$ (the left action of $A'$ comes from multiplication in $A'$). For such bimodules, it will often be convenient to omit the full secondary matrix in favor of matrices for the differential and for the right action of each multiplicative generator of $A$; this type of description can be considerably simpler. 

Assume for simplicity that $A$ and $A'$ have no differential (this will be true for nearly all the examples in this paper). To check that such a bimodule $M$ is well-defined, it suffices to check that:
\begin{itemize}
\item the matrix for the differential on $M$ squares to zero;
\item the matrices for the right action commute with the matrix for the differential;
\item whatever relations are satisfied by the multiplicative generators in the algebra are satisfied for the corresponding right-action matrices.
\end{itemize}
\end{remark}

\subsection{Morphisms}

\begin{definition}
Let $A, A'$ be dg categories with rings of idempotents $I, I'$, and let $M, N$ be DA bimodules over $(A',A)$. A DA bimodule morphism $f\co M \to N$ is a collection $\{f_i : i \geq 1\}$ where
\[
f_i\co M \otimes_{I} (A_+[1])^{\otimes(i-1)} \to A' \otimes_{I'} N
\]
are $(I',I)$-bilinear maps (if all have the same bidegree, then $f$ is said to be homogeneous of this bidegree). Such morphisms form a chain complex whose differential is defined by
\begin{align*}
d(f)_i &= \sum_{j=1}^i (\mu' \otimes \id) \circ (\id \otimes \delta^1_{i-j+1}) \circ f_j \\
&+ \sum_{j=1}^i (\mu' \otimes \id) \circ (\id \otimes f_{i-j+1}) \circ \delta^1_j \\
&+ \sum_{j=1}^{i-1} f_i \circ (\id \otimes \partial_j) \\
&+ \sum_{j=1}^{i-2} f_{i-1} \circ (\id \otimes \mu_{j,j+1}) \\
&+ (\partial' \otimes \id) \circ f_i.
\end{align*}
We say that $f$ is strict if $f_i = 0$ for $i > 1$.
\end{definition}

\begin{remark}
A morphism of DA bimodules $f\co M \to N$ gives a morphism of $A_{\infty}$ bimodules $A' \otimes_{I'} M \to A' \otimes_{I'} N$, compatibly with the differential on morphisms.
\end{remark}

If $M$ and $M'$ are defined by specifying primary and secondary matrices, then we can also specify a morphism $f$ by giving a matrix. The columns of the matrix for $f$ are indexed by entries $x$ of the primary matrix for $M$, and the rows are indexed by entries $y$ of the primary matrix for $M'$. The entry of the matrix for $f$ in row $y$ and column $x$ should be a sum of expressions like $a' \otimes (a_1, \ldots, a_{i-1})$ where $a' \in A'$ and $(a_1, \ldots, a_{i-1})$ are distinct sequences of elements of a chosen basis for $A$ as above. We require that no $a_j$ is an identity morphism in any such sequence (so all $a_j$ are in $A_+$), we write $0$ for the empty sum, and when $i=1$ we write $a'$ as shorthand for $a' \otimes ()$. 

\begin{definition}
Given this data, if $(a_1,\ldots,a_{i-1})$ is a sequence of elements in the chosen basis for $A$ and $x$ is a basis element of $M(S',S)$, we take $f_i(x \otimes a_1 \otimes \cdots \otimes a_{i-1})$ to be the sum, over all rows $y$ in column $x$ of the matrix for $f$, of $a' \otimes y$ if $a' \otimes (a_1, \ldots, a_{i-1})$ is a term of the matrix entry for $f$ in row $y$ and column $x$ and zero otherwise. We extend $\delta^1_i$ linearly over $k$.
\end{definition}

To compose DA bimodule morphisms $f$ and $g$ given by matrices, one multiplies the matrices the same way one does when computing the squared secondary matrix in Procedure~\ref{proc:DAWellDef}.

\begin{procedure}
If $f\co M \to N$ is a morphism of DA bimodules and $M, N, f$ are given in matrix notation, a matrix for $d(f)$ can be obtained as the sum of the following matrices:
\begin{itemize}
\item Two matrices from multiplying the matrix with $f$ and the secondary matrices for $M, N$ (in the order that makes sense),
\item One ``multiplication matrix'' and two ``differential matrices'' obtained from the matrix for $f$ as in Procedure~\ref{proc:DAWellDef}.
\end{itemize}

\end{procedure}

For a given $(A',A)$, the DA bimodules over $(A',A)$ and chain complexes of DA bimodule morphisms form a dg category; isomorphism and homotopy equivalence of DA bimodules are defined in terms of this category. A closed morphism of DA bimodules $f\co M \to N$ is called a quasi-isomorphism if its strict part $f_1$ induces an isomorphism on the homology of the underlying left dg modules $A' \otimes_{I'} M$ and $A' \otimes_{I'} N$.

\begin{remark}
In fact, every quasi-isomorphism of DA bimodules is a homotopy equivalence (see \cite[Corollary 2.4.4]{LOTBimodules}); this is a general feature of $A_{\infty}$ morphisms.
\end{remark}

\subsection{Box tensor products}\label{sec:BoxTensor}

If $M$ and $N$ are two DA bimodules over $(A'',A')$ and $(A',A)$ respectively, one can define a DA bimodule $M \boxtimes N$ as in \cite[Section 2.3.2]{LOTBimodules} (given suitable finiteness conditions). If both $M$ and $N$ are left bounded, then $M \boxtimes N$ is well-defined and left bounded by \cite[Proposition 2.3.10(1), Remark 2.3.12]{LOTBimodules}.

If $M$ and $N$ are given by primary and secondary matrices, then we can describe $M \boxtimes N$ similarly. 

\begin{procedure}\label{proc:BoxTensorBimods}
The primary matrix for $M \boxtimes N$ is obtained by multiplying the primary matrices for $M$ and $N$. Each entry of the primary matrices is a set; to multiply two entries, take the Cartesian product. To sum over all products of entries for each entry in the resulting matrix, take the disjoint union.

The secondary matrix for $M \boxtimes N$ is obtained as follows: let $(x,y)$ and $(x',y')$ be entries of the primary matrix for $M \boxtimes N$. For every choice of:  
\begin{itemize}
\item $i \geq 1$
\item basis elements $y = y_1, y_2, \ldots, y_i = y'$ of $N$
\item term $a'' \otimes (a'_1,\ldots,a'_{i-1})$ in row $x'$ and column $x$ of the secondary matrix of $M$
\item terms $b'_j \otimes (a_{j,1},\ldots,a_{j,i_j-1})$ of the secondary matrix of $N$ in row $y_{j+1}$ and column $y_j$ for $1 \leq j \leq i-1$, where $c_j a'_j$ is a nonzero term in the basis expansion of $b'_j$ (for some $c_j \in k$),
\end{itemize}
add the entry
\[
c_1 \cdots c_{i-1} a'' \otimes (a_{1,1},\ldots,a_{1,i_1-1}, \ldots, a_{i-1,1}, \ldots, a_{i-1,i_{i-1}-1})
\]
to row $(x',y')$ and column $(x,y)$ of the secondary matrix for $M \boxtimes N$.
\end{procedure}

By \cite[Lemma 2.3.14(2)]{LOTBimodules}, box tensor products of DA bimodules are associative up to canonical isomorphism. Also, there are canonical isomorphisms between $M$ and the box tensor product of $M$ with an identity DA bimodule on either side.

\subsection{Box tensor products of morphisms}\label{sec:BoxTensorOfMorphisms}

Let $M_1, M_2$ be DA bimodules over $(A'',A')$ and let $N_1, N_2$ be DA bimodules over $(A',A)$. Let $f\co M_1 \to M_2$ and $g\co N_1 \to N_2$ be morphisms. As discussed in \cite[Section 2.3.2]{LOTBimodules}, the box tensor product $f \boxtimes g$ is only defined up to homotopy. However, expressions like $f \boxtimes \id$ and $\id \boxtimes g$ are unambiguously defined and we can compute them in matrix notation.

\begin{procedure}\label{proc:BoxTensorMorphisms}
Suppose that $N_1 = N_2 = N$ and we have a morphism $f\co M_1 \to M_2$. The matrix for $f \boxtimes \id_N$ has columns indexed by primary matrix entries for $M_1 \boxtimes N$ and rows indexed by primary matrix entries for $M_2 \boxtimes N$. For each entry $a'' \otimes (a'_1,\ldots,a'_{i-1})$ of the matrix for $f$ (in column $x_1$ and row $x_2$), consider the $(i-1)^{st}$ power of the secondary matrix for $N$, defined by concatenating both input sequences and output terms when multiplying individual entries. If $(x_1,y_1)$ and $(x_2,y_2)$ are entries of the primary matrices of $M_1 \boxtimes N$ and $M_2 \boxtimes N$ respectively, then for each entry of the $(i-1)^{st}$ power matrix of $N$ in column $y_1$ and row $y_2$ (say with output sequence $(b'_1, \ldots, b'_{i-1})$ and input sequence $\vec{S}$), let $c_j \in k$ be the coefficient of $a'_j$ in the basis expansion of $b'_j$ for $1 \leq j \leq i-1$. Let $c = c_1 \cdots c_{k-1}$ and add $ca'' \otimes \vec{S}$ to the matrix for $f \boxtimes \id_N$ in column $(x_1,y_1)$ and row $(x_2,y_2)$.

Now suppose that $M_1 = M_2 = M$ and we have a morphism $g\co N_1 \to N_2$. The matrix for $\id_M \boxtimes g$ is defined as above, except the matrix for $f$ above is replaced by the secondary matrix for $M$ with the terms $\id_S \otimes \id_{S'}$ included. The $(i-1)^{st}$ power matrix above is replaced, for $i \geq 2$, by the sum of all ways of multiplying one instance of the matrix for $g$ with $i-2$ instances of the secondary matrices for $N_1$ or $N_2$ as appropriate, again concatenating both input sequences and output terms when multiplying. 
\end{procedure}

Note that $i=1$ terms in the secondary matrix for $M$ do not contribute to the matrix for $\id_M \boxtimes g$.

\subsection{Simplifying DA bimodules}\label{sec:PrelimSimplifying}

Suppose that a DA bimodule $M$ is given by primary and secondary matrices, and that some entry of the secondary matrix (say in column $x$ and row $y$) is an identity morphism of $A'$ (i.e. $\id_{S'} = \id_{S'} \otimes ()$). We can simplify $M$ to obtain a homotopy equivalent DA bimodule $\widetilde{M}$ having two fewer basis elements than $M$. We describe $\widetilde{M}$ in matrix notation as follows.

\begin{procedure}
The primary matrix of $\widetilde{M}$ is the primary matrix of $M$ with the elements $x$ and $y$ removed from their corresponding entries. The secondary matrix of $\widetilde{M}$ is obtained as follows:
\begin{itemize}
\item Start by deleting the row and column corresponding to $x$, as well as the row and column corresponding to $y$, from the secondary matrix of $M$. 
\item Now, for each term $a' \otimes (a_1,\ldots,a_{i-1})$ of an entry of the original secondary matrix of $M$ in row $y$ (and some column $x'$) other than the $\id_{S'}$ we are canceling, and each such term $b' \otimes (b_1,\ldots,b_{j-1})$ in column $x$ (and some row $y'$) other than the $\id_{S'}$ we are canceling, add the term $a'b' \otimes (a_1,\ldots,a_{i-1},b_1,\ldots,b_{j-1})$ to the secondary matrix of $\widetilde{M}$ in column $x'$ and row $y'$.

\end{itemize}
\end{procedure}

If we can cancel $x$ and $y$ from $M$ as above, we will call $(x,y)$ a canceling pair in $M$. If there are multiple disjoint canceling pairs in $M$, we can cancel them all at once, potentially picking up more general terms than the ones above; we can even do this when $M$ has an infinite set of disjoint canceling pairs, assuming all resulting matrix entries and sums are finite.

\subsection{Duals and AD bimodules}\label{sec:DualsAD}

One can modify Definition~\ref{def:DABimod} by interchanging left and right everywhere; one obtains the definition of an AD bimodule. We can specify AD bimodules $N$ over $(A',A)$ using matrix notation, just as for DA bimodules; the difference is that the entries of a secondary matrix specifying an AD bimodule are sums of terms of the form $(a'_1, \ldots, a'_{i-1}) \otimes a$. Right boundedness, finite generation, and strict unitality for AD bimodules are defined like left boundedness, finite generation, and strict unitality for DA bimodules. Morphisms, including isomorphisms, homotopy equivalences, and quasi-isomorphisms, are also defined similarly, and every quasi-isomorphism of AD bimodules is a homotopy equivalence.

The box tensor product of AD bimodules is also defined analogously to the DA bimodule case, and is associative up to canonical isomorphism. There are identity AD bimodules as in the DA case, and the box tensor product of an AD bimodule $N$ with an identity AD bimodule on either side is canonically isomorphic to $N$.

If $M$ is a finitely generated DA bimodule over $(A',A)$, the opposite AD bimodule to $M$ (over $(A,A')$) is defined in \cite[Definition 2.2.53]{LOTBimodules}. Similarly, a finitely generated AD bimodule $N$ over $(A',A)$ has an opposite DA bimodule over $(A,A')$. The opposite of $M$ corresponds to the bimodule $\Hom_{A'}(A' \otimes_{I'} M, A')$ with an induced left $A_{\infty}$ action of $A$, so we will write the opposite AD bimodule to $M$ as ${^{\vee}}M$ and call it the left dual of $M$. Similarly, the opposite of $N$ corresponds to $\Hom_{A}(N \otimes_{I} A, A)$; we will write the opposite of $N$ as $N^{\vee}$ and call the right dual of $N$. 

Duality preserves finite generation and strict unitality; it sends left bounded DA bimodules to right bounded AD bimodules and vice-versa. It reverses both $q$-degree and homological degree; we have ${^{\vee}}(q^i M [j]) = q^{-i} ({^{\vee}}M) [-j]$. The opposite of an identity DA bimodule is an identity AD bimodule and vice-versa.

Duality is compatible with box tensor products after reversing the order: we have natural identifications ${^{\vee}}(M_1 \boxtimes M_2) \cong {^{\vee}}M_2 \boxtimes {^{\vee}}M_1$ and $(N_1 \boxtimes N_2)^{\vee} \cong N_2^{\vee} \boxtimes N_1^{\vee}$. 

A morphism of finitely generated DA bimodules $f\co M \to M'$ gives a dual morphism ${^{\vee}}f\co {^{\vee}}M' \to {^{\vee}}M$, such that ${^{\vee}}(g \circ f) = {^{\vee}}f \circ {^{\vee}}g$ and ${^{\vee}}(d(f)) = d({^{\vee}}f)$. Similarly, we can dualize morphisms of finitely generated AD bimodules to get morphisms of DA bimodules.

\begin{remark}
Since duality reverses bidegrees on DA and AD bimodules while simultaneously reversing the direction of morphisms, the dual ${^{\vee}}f$ of a DA bimodule morphism $f$ has the same bidegree as $f$, and similarly for duals of AD bimodule morphisms.
\end{remark}

Concretely, duality acts on bimodules presented in matrix notation as follows.

\begin{procedure}
For a finitely generated DA bimodule $M$ given by a primary matrix and secondary matrix, form a primary matrix for ${^{\vee}}M$ by taking the transpose of the primary matrix for $M$. Form a secondary matrix for ${^{\vee}}M$ by
\begin{itemize}
\item taking the transpose of the secondary matrix for $M$;
\item changing all entries $a' \otimes (a_1, \ldots, a_{i-1})$ in this secondary matrix to $(a_1, \ldots, a_{i-1}) \otimes a'$.
\end{itemize}
\end{procedure}

\subsection{Split Grothendieck groups}\label{sec:SplitGG}

In this section we will require our dg categories $A$ to have finitely many objects for simplicity. While the Grothendieck groups we consider below can be defined over $\Z[q,q^{-1}]$, we will pass to $\C_q \coloneqq \C(q)$ since we will be relating these Grothendieck groups to representations of $U_q(\gl(1|1))$.

\begin{definition}
Let $A$ be a dg category with finitely many objects; recall that we assume $A$ is equipped with a functor $\epsilon\co A \to I$ restricting to the identity on $I \subset A$. Let $G_0(A)$ be the split Grothendieck group of the closure of $A$ under direct sums and degree shifts (both quantum and homological), tensored over $\Z[q,q^{-1}]$ with $\C_q$. In other words, $G_0(A)$ is the $\C_q$-vector space generated by isomorphism classes of objects $[S]$ of this closure of $A$ under the relations
\[
[S_1 \oplus S_2] = [S_1] + [S_2], \,\,\, [qS] = q[S], \,\,\, [S[1]] = -[S].
\]
It follows from the existence of $\epsilon$ that $G_0(A)$ has a basis given by the objects of $A$ ($\epsilon$ ensures that no two distinct objects of $A$ can be isomorphic).
\end{definition}

Our convention will be to use AD bimodules to induce maps on split Grothendieck groups as follows. Suppose that $A$ and $A'$ each have finitely many objects.

\begin{definition}
Let $N$ be a finitely generated AD bimodule over $(A',A)$. Define a $\C_q$-linear map $[N]\co G_0(A) \to G_0(A')$ by declaring that, for basis vectors $[S]$ of $G_0(A)$ and $[S']$ of $G_0(A')$ coming from objects $S$ of $A$ and $S'$ of $A'$, the matrix entry of $[N]$ in row $[S']$ and column $[S]$ is the $q$-graded Euler characteristic of the finite-dimensional bigraded $k$-module $N(S',S)$.
\end{definition}

Heuristically, with this convention one thinks of applying the functor $N \boxtimes -$ to simple $A$-modules and expanding the result in terms of simple $A'$-modules, but we will stick with the above elementary definition in this paper.

We can also decategorify DA bimodules; let $K_0(A)$ be the $\C_q$-vector space defined in the same way as $G_0(A)$, except that we call the generators $[P]$ instead of $[S]$.
\begin{definition}
Let $M$ be a finitely generated DA bimodule over $(A',A)$. Define a $\C_q$-linear map $[M]\co K_0(A) \to K_0(A')$ by declaring that, for basis vectors $[P]$ of $K_0(A)$ and $[P']$ of $K_0(A')$ coming from objects $P$ of $A$ and $P'$ of $A'$, the matrix entry of $[M]$ in row $[P']$ and column $[P]$ is the $q$-graded Euler characteristic of the finite-dimensional bigraded $k$-module $M(P',P)$.
\end{definition}

With $K_0$ one thinks of applying the functor $M \boxtimes -$ to indecomposable projective $A$-modules and expanding the result in terms of indecomposable projective $A'$-modules.

In terms of matrix notation, we have the following decategorification procedure.

\begin{procedure}
To decategorify a finitely generated AD or DA bimodule given in matrix notation, one discards the secondary matrix and replaces sets in the entries of the primary matrix with sums of $(-1)^j q^i$ for each element with $\deg^q = i$ and $\deg^h = j$.

If $M$ is a finitely generated DA bimodule given in matrix notation, then to decategorify ${^{\vee}}M$, one discards the secondary matrix, replaces sets in the entries of the primary matrix of $M$ with sums of $(-1)^j q^{-i}$ for each element with $\deg^q = i$ and $\deg^h = j$, and takes the transpose.
\end{procedure}

\section{2-representations and morphisms}\label{sec:2Reps}

\subsection{Preliminaries}

Let $\U$ denote the differential monoidal category of \cite[Section 4.1.1]{ManionRouquier}. We first define a variant $\U^-$ of $\U$. Write $\varepsilon_1 = (1,0)$ and $\varepsilon_2 = (0,1)$ for the standard basis vectors of $\Z^2$ and let $\alpha = \varepsilon_1 - \varepsilon_2 = (1,-1) \in \Z^2$.

\begin{definition}
Let $\U^-$ be the strict dg 2-category defined as follows:
\begin{itemize}
\item The objects are lattice points $\omega = \omega_1 \varepsilon_1 + \omega_2 \varepsilon_2 \in \Z^2$.

\item The 1-morphisms are generated under composition by $f_{\omega}\co \omega \to \omega - \alpha$ for all $\omega \in \Z^2$.

\item The 2-morphisms are generated under horizontal and vertical composition by endomorphisms $\tau_{\omega}$ of $f_{\omega - \alpha} f_{\omega}$ for $\omega \in \Z^2$, subject to
\[
\tau_{\omega}^2 = 0, \,\,\, f_{\omega - 2\alpha} \tau_{\omega} \circ \tau_{\omega - \alpha} f_{\omega} \circ f_{\omega - 2\alpha} \tau_{\omega} = \tau_{\omega - \alpha} f_{\omega} \circ f_{\omega - 2\alpha} \tau_{\omega} \circ \tau_{\omega - \alpha} f_{\omega}, \,\,\, d(\tau_{\omega}) = \id_{f_{\omega - \alpha} f_{\omega}}.
\]
We set $\deg^q(\tau_{\omega}) = 0$ and $\deg^h(\tau_{\omega}) = -1$.
\end{itemize}

\end{definition}

We let $H_n$ denote the dg 2-endomorphism algebra of the 1-morphism $f_{\omega-(n-1)\alpha} \circ \cdots \circ f_{\omega}$ of $\U^-$ for any $\omega \in \Z^2$ (the dg algebra $H_n$ does not depend on the choice of $\omega$). 

A bimodule 2-representation of $\U^-$ is a functor from $\U^-$ into (dg categories, dg bimodules, bimodule maps); it amounts to dg categories $A_{\omega}$, dg bimodules ${_{\omega - \alpha}}F_{\omega}$, and bimodule endomorphisms $\tau_{\omega}$ of ${_{\omega - 2\alpha}} F_{\omega - \alpha} F_{\omega}$ satisfying the above relations. By forgetting the bigrading and weight decomposition, a bimodule 2-representation of $\U^-$ gives a bimodule 2-representation of $\U$.

\begin{definition}
We say a bimodule 2-representation of $\U^-$ is right finite if each bimodule ${_{\omega - \alpha}}F_{\omega}$ is given by $N_{\omega} \otimes_{I_{\omega}} A_{\omega}$ for some AD bimodule $N_{\omega}$ over $(A_{\omega - \alpha}, A_{\omega})$ with vanishing higher action terms $\delta^1_i$ for $i > 2$; in this case, we write ${_{\omega}}F^{\vee}_{\omega - \alpha}$ for the right dual of ${_{\omega - \alpha}}F_{\omega}$ (equal to $A_{\omega} \otimes_{I_{\omega}} N_{\omega}^{\vee}$ where $N_{\omega}^{\vee}$ is defined as in Section~\ref{sec:DualsAD}).
\end{definition}

When working with right finite 2-representations, we will often take the AD bimodules representing the dg bimodules ${_{\omega - \alpha}} F_{\omega}$ to be part of the data.

\subsection{Bigraded tensor product}\label{sec:BigradedTensor}

In this section we make use of the tensor product $\ootimes$ from \cite[Section 5.3.4]{ManionRouquier}, which is an instance of the $\Delta'_{\lambda}$ construction of \cite[Sections 5.3.1--5.3.3]{ManionRouquier}. Specifically, it is an instance of the $\Delta_{\sigma}$ construction from \cite[Sections 5.3.2--5.3.3]{ManionRouquier}, and we will refer to the maps $u$ and $w$ from \cite[Section 5.3.3]{ManionRouquier}. 

If $(A_1, F_1, \tau_1)$ and $(A_2, F_2, \tau_2)$ are right finite bimodule 2-representations of $\U^-$, we can take the tensor product $\ootimes$ of the underlying bimodule 2-representations of $\U$ to get another bimodule 2-representation of $\U$. We will call this tensor product 2-representation $\left( A_1 \ootimes A_2, X, \tau \right)$. We now define a bimodule 2-representation $\left( (A_1 \ootimes A_2)_\omega, \, {_{\omega - \alpha}}F_{\omega}, \, \tau_{\omega} \right)$ of $\U^-$ lifting $\left( A_1 \ootimes A_2, X, \tau \right)$.
\begin{definition}
We first define a dg category $(A_1 \ootimes A_2)_{\omega}$ for each $\omega \in \Z^2$. Recall that the objects of $A_1 \ootimes A_2$ are the same as the objects of $A_1 \otimes A_2$, namely pairs $S = (S_1, S_2)$ where $S_i$ is an object of $A_i$. 
\begin{itemize}
\item For $\omega \in \Z^2$, we let $(A_1 \ootimes A_2)_{\omega}$ be the full differential subcategory of $A_1 \ootimes A_2$ on objects $(S_1,S_2)$ such that, for some $\omega', \omega'' \in \Z^2$, $S_1$ is an object of $(A_1)_{\omega'}$, $S_2$ is an object of $(A_2)_{\omega''}$, and we have $\omega = \omega' + \omega''$. Note that $A_1 \ootimes A_2 = \oplus_{\omega \in \Z^2} (A_1 \ootimes A_2)_{\omega}$ as differential categories.

\item For objects $S = (S_1,S_2)$ and $T = (T_1,T_2)$ of $(A_1 \otimes A_2)_{\omega}$ where $S_1$ is an object of $(A_1)_{\omega'}$, $S_2$ is an object of $(A_2)_{\omega''}$, and $\omega = \omega' + \omega''$, we lift $\Hom_{A_1 \ootimes A_2}(S,T)$ to a bigraded chain complex by setting
\[
\Hom_{(A_1 \ootimes A_2)_{\omega}}(S,T) = \bigoplus_{i \geq 0} \left({_{\omega'+i\alpha}}(\widetilde{F}_1^{\vee})_{\omega'}^i(T_1, S_1)\right) \otimes_{H_i} \Big({_{\omega''-i\alpha}}(F_2)^i_{\omega''} (T_2, S_2)\Big)
\]
where
\[
{_{\omega' + \alpha}}(\widetilde{F}_1^{\vee})_{\omega'} \coloneqq q^{\omega'_1 + \omega'_2} {_{\omega' + \alpha}}(F_1^{\vee})_{\omega'} [\omega'_2 - 1].
\]
\end{itemize}
By construction, composition in $(A_1 \ootimes A_2)_{\omega}$ has degree zero, so $(A_1 \ootimes A_2)_{\omega}$ is a dg category.
\end{definition}

Recall from \cite{ManionRouquier} that $X = X(T,S)$ is the mapping cone of the (ungraded) map
\begin{equation}\label{eq:MapU}
u(T,S)\co \Big((A_1 \otimes F_2) \otimes_{A_1 \otimes A_2} (A_1 \ootimes A_2)\Big)(T,S) \to \Big((F_1 \otimes A_2) \otimes_{A_1 \otimes A_2} (A_1 \ootimes A_2)\Big)(T,S)
\end{equation}
where $u$ is adjoint to the multiplication map as in \cite[Section 5.3.3]{ManionRouquier}; see also \cite[Section 5.3.4]{ManionRouquier}. Note that for an object $S$ of $(A_1 \ootimes A_2)_{\omega}$, the left differential $A_1 \ootimes A_2$-module $X(-,S)$ vanishes on $T$ unless $T$ is an object of $(A_1 \ootimes A_2)_{\omega - \alpha}$. Thus, $X = \oplus_{\omega \in \Z^2} ({_{\omega - \alpha}}F_{\omega})$ as differential bimodules over $A_1 \ootimes A_2$, where for $\omega \in \Z^2$, ${_{\omega - \alpha}}F_{\omega}$ is the differential bimodule over $((A_1 \ootimes A_2)_{\omega - \alpha}, (A_1 \ootimes A_2)_{\omega})$ defined by ${_{\omega - \alpha}}F_{\omega}(T,S) = X(T,S)$ for objects $S$ of $(A_1 \ootimes A_2)_{\omega}$ and $T$ of $(A_1 \ootimes A_2)_{\omega - \alpha}$.

\begin{definition}
For $\omega \in \Z^2$, we give the differential bimodule ${_{\omega - \alpha}}F_{\omega}$ a bigrading as follows. For objects $S = (S_1, S_2)$ of $(A_1 \ootimes A_2)_{\omega}$ and $T = (T_1, T_2)$ of $(A_1 \ootimes A_2)_{\omega - \alpha}$ such that $T_1$ is an object of $(A_1)_{\omega'}$, $T_2$ is an object of $(A_2)_{\omega'' - \alpha}$, and $\omega = \omega' + \omega''$, the map $u(T,S)$ of equation~\eqref{eq:MapU} has degree zero when viewed as a map
\begin{align*}
u(T,S)\co & \,\,\, q^{\omega'_1 + \omega'_2} \Big( ((A_1)_{\omega'} \otimes ({_{\omega'' - \alpha}}(F_2)_{\omega''})) \otimes_{(A_1)_{\omega'} \otimes (A_2)_{\omega''}} \left(A_1 \ootimes A_2\right)_{\omega}\Big) (T,S) [\omega'_2 - 1] \\
&\to \Big(({_{\omega'}}(F_1)_{\omega'+\alpha} \otimes (A_2)_{\omega'' - \alpha}) \otimes_{(A_1)_{\omega' + \alpha} \otimes (A_2)_{\omega'' - \alpha}} \left(A_1 \ootimes A_2\right)_{\omega} \Big)(T,S).
\end{align*}
Indeed, the multiplication map
\begin{align*}
\Big( ({_{\omega' + \alpha}}(\widetilde{F}_1^{\vee})_{\omega'}) &\otimes ({_{\omega'' - \alpha}}(F_2)_{\omega''})) \otimes_{(A_1)_{\omega'} \otimes (A_2)_{\omega''}} \left(A_1 \ootimes A_2\right)_{\omega}\Big) (T,S) \\
&\to \left(A_1 \ootimes A_2\right)_{\omega} (T,S)
\end{align*}
has degree zero by the definition of $A_1 \ootimes A_2$, and to get $u(T,S)$ we have expanded out the grading shifts in the definition of $\widetilde{F}_1^{\vee}$ and dualized.

We let ${_{\omega - \alpha}}F_{\omega}(T,S)$ be the mapping cone of the above degree-zero map $u(T,S)$; we get a dg bimodule ${_{\omega - \alpha}}F_{\omega}$ over $((A_1 \otimes A_2)_{\omega - \alpha}, (A_1 \ootimes A_2)_{\omega})$ such that, as a non-differential bimodule, ${_{\omega - \alpha}}F_{\omega}(T,S)$ is a direct sum of
\[
q^{\omega'_1 + \omega'_2} \Big( ((A_1)_{\omega'} \otimes ({_{\omega'' - \alpha}}(F_2)_{\omega''})) \otimes_{(A_1)_{\omega'} \otimes (A_2)_{\omega''}} \left(A_1 \ootimes A_2\right)_{\omega}\Big) (T,S) [\omega'_2]
\]
and
\[
\Big(({_{\omega'}}(F_1)_{\omega'+\alpha} \otimes (A_2)_{\omega'' - \alpha}) \otimes_{(A_1)_{\omega' + \alpha} \otimes (A_2)_{\omega'' - \alpha}} \left(A_1 \ootimes A_2\right)_{\omega} \Big)(T,S)
\]
\end{definition}

\begin{proposition}
The map defining the left action of $(A_1 \ootimes A_2)_{\omega - \alpha}$ on ${_{\omega - \alpha}}F_{\omega}$ (i.e. the map $w$ from \cite[Section 5.3.3]{ManionRouquier}; see also the $2 \times 2$ matrices in \cite[Remark 5.3.5]{ManionRouquier}) has degree zero.
\end{proposition}

\begin{proof}
By expanding out grading shifts in the definition of $\widetilde{F}_1^{\vee}$, and taking the grading shift $q^{\omega'_1 + \omega'_2}[\omega'_2]$ of the first summand of ${_{\omega - \alpha}}F_{\omega}(T,S)$ into account, one sees that the maps
\begin{align*}
&q^{2 \omega'_1 + 2\omega'_2} \Big( \Big( {_{\omega'+\alpha}}(F_1^{\vee})_{\omega'} \otimes (A_2)_{\omega''-2\alpha} \Big) \otimes_{(A_1)_{\omega'} \otimes (A_2)_{\omega''-2\alpha}} \Big( (A_1)_{\omega'} \otimes {_{\omega''-2\alpha}}(F_2)_{\omega''-\alpha} \Big) \\
& \quad \quad \otimes_{(A_1)_{\omega'} \otimes (A_2)_{\omega''-\alpha}} \Big( (A_1)_{\omega'} \otimes {_{\omega''-\alpha}}(F_2)_{\omega''} \Big) \\
& \quad \quad \otimes_{(A_1)_{\omega'} \otimes (A_2)_{\omega''}} (A_1 \ootimes A_2)_{\omega} \Big) (T,S) [2 \omega'_2 - 1] \\
&\xrightarrow{\id \otimes \tau_2 \otimes \id} q^{2 \omega'_1 + 2\omega'_2} \Big( \Big( {_{\omega'+\alpha}}(F_1^{\vee})_{\omega'} \otimes (A_2)_{\omega''-2\alpha} \Big) \otimes_{(A_1)_{\omega'} \otimes (A_2)_{\omega''-2\alpha}} \Big( (A_1)_{\omega'} \otimes {_{\omega''-2\alpha}}(F_2)_{\omega''-\alpha} \Big) \\
& \quad \quad \otimes_{(A_1)_{\omega'} \otimes (A_2)_{\omega''-\alpha}} \Big( (A_1)_{\omega'} \otimes {_{\omega''-\alpha}}(F_2)_{\omega''} \Big) \\
& \quad \quad \otimes_{(A_1)_{\omega'} \otimes (A_2)_{\omega''}} (A_1 \ootimes A_2)_{\omega} \Big) (T,S) [2 \omega'_2 - 2] \\
&\xrightarrow{\lambda \otimes \id} q^{2\omega'_1 + 2\omega'_2} \Big( \Big( (A_1)_{\omega'+\alpha} \otimes {_{\omega''-2\alpha}}(F_2)_{\omega''-\alpha} \Big) \otimes_{(A_1)_{\omega'+\alpha} \otimes (A_2)_{\omega''-\alpha}} \Big( {_{\omega' + \alpha}}(F_1^{\vee})_{\omega'} \otimes (A_2)_{\omega''-\alpha} \Big) \\
& \quad \quad \otimes_{(A_1)_{\omega'} \otimes (A_2)_{\omega''-\alpha}} \Big( (A_1)_{\omega'} \otimes {_{\omega''-\alpha}}(F_2)_{\omega''} \Big) \\
& \quad \quad \otimes_{(A_1)_{\omega'} \otimes (A_2)_{\omega''}} (A_1 \ootimes A_2)_{\omega} \Big) (T,S) [2 \omega'_2 - 2] \\
&\xrightarrow{\id \otimes \mult} q^{\omega'_1 + \omega'_2} \Big( \Big( (A_1)_{\omega'+\alpha} \otimes {_{\omega''-2\alpha}}(F_2)_{\omega''-\alpha} \Big) \otimes_{(A_1)_{\omega'+\alpha} \otimes (A_2)_{\omega''-\alpha}} (A_1 \ootimes A_2)_{\omega} \Big) (T,S) [\omega'_2 - 1] \\
\end{align*}
each have degree zero, where
\begin{align*}
\lambda\co &\Big( {_{\omega'+\alpha}}(F_1^{\vee})_{\omega'} \otimes (A_2)_{\omega''-2\alpha} \Big) \otimes_{(A_1)_{\omega'} \otimes (A_2)_{\omega''-2\alpha}} \Big( (A_1)_{\omega'} \otimes {_{\omega''-2\alpha}}(F_2)_{\omega''-\alpha} \Big) \\
& \to \Big( (A_1)_{\omega'+\alpha} \otimes {_{\omega''-2\alpha}}(F_2)_{\omega''-\alpha} \Big) \otimes_{(A_1)_{\omega'+\alpha} \otimes (A_2)_{\omega''-\alpha}} \Big( {_{\omega' + \alpha}}(F_1^{\vee})_{\omega'} \otimes (A_2)_{\omega''-\alpha} \Big)
\end{align*}
is defined in \cite[equation (5.3.1)]{ManionRouquier} and is equal to the usual swap isomorphism by the definitions in \cite[Section 5.3.4]{ManionRouquier}. Thus, the composite of these maps (the top-left entry of the $2 \times 2$ matrix defining the left action of $(A_1 \ootimes A_2)_{\omega}$ on ${_{\omega-\alpha}}F_{\omega}$) has degree zero. Similarly, the maps
\begin{align*}
&q^{\omega'_1 + \omega'_2} \Big( \Big( {_{\omega'}}(F_1^{\vee})_{\omega' - \alpha} \otimes (A_2)_{\omega''-\alpha} \Big) \otimes_{(A_1)_{\omega' - \alpha} \otimes (A_2)_{\omega''-\alpha}} \Big( (A_1)_{\omega' - \alpha} \otimes {_{\omega''-\alpha}}(F_2)_{\omega''} \Big) \\
& \quad \quad \otimes_{(A_1)_{\omega' - \alpha} \otimes (A_2)_{\omega''}} \Big( {_{\omega' - \alpha}}(F_1)_{\omega'} \otimes (A_2)_{\omega''} \Big) \\
& \quad \quad \otimes_{(A_1)_{\omega'} \otimes (A_2)_{\omega''}} (A_1 \ootimes A_2)_{\omega} \Big) (T,S) [\omega'_2] \\
&\xrightarrow{\id \otimes \sigma \otimes \id} q^{\omega'_1 + \omega'_2} \Big( \Big( {_{\omega'}}(F_1^{\vee})_{\omega' - \alpha} \otimes (A_2)_{\omega''-\alpha} \Big) \otimes_{(A_1)_{\omega' - \alpha} \otimes (A_2)_{\omega''-\alpha}} \Big( {_{\omega' - \alpha}}(F_1)_{\omega'} \otimes (A_2)_{\omega''-\alpha} \Big) \\
& \quad \quad \otimes_{(A_1)_{\omega'} \otimes (A_2)_{\omega''-\alpha}} \Big( (A_1)_{\omega'} \otimes {_{\omega''-\alpha}}(F_2)_{\omega''} \Big) \\
& \quad \quad \otimes_{(A_1)_{\omega'} \otimes (A_2)_{\omega''}} (A_1 \ootimes A_2)_{\omega} \Big) (T,S) [\omega'_2] \\
&\xrightarrow{\varepsilon \otimes \id} q^{\omega'_1 + \omega'_2} \Big( \Big( (A_1)_{\omega'} \otimes {_{\omega''-\alpha}}(F_2)_{\omega''} \Big) \otimes_{(A_1)_{\omega'} \otimes (A_2)_{\omega''}} (A_1 \ootimes A_2)_{\omega} \Big) (T,S) [\omega'_2] \\
\end{align*}
each have degree zero where $\sigma$ is the swap isomorphism as in \cite[Section 5.3.4]{ManionRouquier} and $\varepsilon$ is the counit of the adjunction $F_1^{\vee} \dashv F_1$. Finally, the maps
\begin{align*}
&q^{\omega'_1 + \omega'_2} \Big( \Big( {_{\omega'}}(F_1^{\vee})_{\omega' - \alpha} \otimes (A_2)_{\omega''-\alpha} \Big) \otimes_{(A_1)_{\omega' - \alpha} \otimes (A_2)_{\omega''-\alpha}} \Big( (A_1)_{\omega' - \alpha} \otimes {_{\omega''-\alpha}}(F_2)_{\omega''} \Big) \\
& \quad \quad \otimes_{(A_1)_{\omega' - \alpha} \otimes (A_2)_{\omega''}} \Big( {_{\omega' - \alpha}}(F_1)_{\omega'} \otimes (A_2)_{\omega''} \Big) \\
& \quad \quad \otimes_{(A_1)_{\omega'} \otimes (A_2)_{\omega''}} (A_1 \ootimes A_2)_{\omega} \Big) (T,S) [\omega'_2] \\
&\xrightarrow{\id \otimes \sigma \otimes \id} q^{\omega'_1 + \omega'_2} \Big( \Big( {_{\omega'}}(F_1^{\vee})_{\omega' - \alpha} \otimes (A_2)_{\omega''-\alpha} \Big) \otimes_{(A_1)_{\omega' - \alpha} \otimes (A_2)_{\omega''-\alpha}} \Big( {_{\omega' - \alpha}}(F_1)_{\omega'} \otimes (A_2)_{\omega''-\alpha} \Big) \\
& \quad \quad \otimes_{(A_1)_{\omega'} \otimes (A_2)_{\omega''-\alpha}} \Big( (A_1)_{\omega'} \otimes {_{\omega''-\alpha}}(F_2)_{\omega''} \Big) \\
& \quad \quad \otimes_{(A_1)_{\omega'} \otimes (A_2)_{\omega''}} (A_1 \ootimes A_2)_{\omega} \Big) (T,S) [\omega'_2] \\
&\xrightarrow{\rho \otimes \id} q^{\omega'_1 + \omega'_2} \Big( \Big( {_{\omega'}}(F_1)_{\omega' + \alpha} \otimes (A_2)_{\omega''-\alpha} \Big) \otimes_{(A_1)_{\omega' + \alpha} \otimes (A_2)_{\omega''-\alpha}} \Big( {_{\omega' + \alpha}}(F_1^{\vee})_{\omega'} \otimes (A_2)_{\omega''-\alpha} \Big) \\
& \quad \quad \otimes_{(A_1)_{\omega'} \otimes (A_2)_{\omega''-\alpha}} \Big( (A_1)_{\omega'} \otimes {_{\omega''-\alpha}}(F_2)_{\omega''} \Big) \\
& \quad \quad \otimes_{(A_1)_{\omega'} \otimes (A_2)_{\omega''}} (A_1 \ootimes A_2)_{\omega} \Big) (T,S) [\omega'_2 - 1] \\
&\xrightarrow{\id \otimes \mult} \Big( \Big( {_{\omega'}}(F_1)_{\omega' + \alpha} \otimes (A_2)_{\omega'' - \alpha} \Big) \otimes_{(A_1)_{\omega' + \alpha} \otimes (A_2)_{\omega'' - \alpha}} (A_1 \ootimes A_2)_{\omega} \Big) (T,S)
\end{align*}
each have degree zero, where $\rho$ is defined in \cite[equation (4.4.2)]{ManionRouquier} (note that $\rho$ has the same degree as $\tau$).
\end{proof}

It follows that ${_{\omega-\alpha}}F_{\omega}$ is a dg bimodule over $((A_1 \ootimes A_2)_{\omega-\alpha}, (A_1 \ootimes A_2)_{\omega})$ lifting ${_{\omega-\alpha}}F_{\omega}$ as a differential bimodule over $((A_1 \ootimes A_2)_{\omega-\alpha}, (A_1 \ootimes A_2)_{\omega})$.

\begin{proposition} 
The map defining the endomorphism $\tau$ of $X^2$ (see \cite[equation (5.3.4)]{ManionRouquier}) restricts to an endomorphism $\tau_{\omega}$ of the dg bimodule ${_{\omega - 2\alpha}}F^2_{\omega}$ with $\deg^q(\tau_{\omega}) = 0$ and $\deg^h(\tau_{\omega}) = -1$.
\end{proposition}

\begin{proof}
For the top-left and bottom-right entries of the $4 \times 4$ matrix defining $\tau$, this is clear. For the entry in row 2, column 3, the map
\begin{align*}
&q^{\omega'_1 + \omega'_2} \Big( \Big( {_{\omega'- \alpha}}(F_1)_{\omega'}  \otimes (A_2)_{\omega'' - \alpha} \Big) \otimes_{(A_1)_{\omega'} \otimes (A_2)_{\omega''-\alpha}} \Big( (A_1)_{\omega'} \otimes {_{\omega''- \alpha}}(F_2)_{\omega''} \Big) \\
& \quad \quad \otimes_{(A_1)_{\omega'} \otimes (A_2)_{\omega''}} (A_1 \ootimes A_2)_{\omega} \Big) (T,S)[\omega'_2] \\
&\xrightarrow{\sigma^{-1}} q^{\omega'_1 + \omega'_2} \Big( \Big( (A_1)_{\omega'-\alpha} \otimes {_{\omega''-\alpha}}(F_2)_{\omega''} \Big) \otimes_{(A_1)_{\omega'-\alpha} \otimes (A_2)_{\omega''}} \Big( {_{\omega'-\alpha}}(F_1)_{\omega'}  \otimes (A_2)_{\omega''} \Big)  \\
& \quad \quad \otimes_{(A_1)_{\omega'} \otimes (A_2)_{\omega''}} (A_1 \ootimes A_2)_{\omega} \Big) (T,S)[\omega'_2] 
\end{align*}
has bidegree zero, so
\begin{align*}
&q^{\omega'_1 + \omega'_2} \Big( \Big( {_{\omega'- \alpha}}(F_1)_{\omega'}  \otimes (A_2)_{\omega'' - \alpha} \Big) \otimes_{(A_1)_{\omega'} \otimes (A_2)_{\omega''-\alpha}} \Big( (A_1)_{\omega'} \otimes {_{\omega''- \alpha}}(F_2)_{\omega''} \Big) \\
& \quad \quad \otimes_{(A_1)_{\omega'} \otimes (A_2)_{\omega''}} (A_1 \ootimes A_2)_{\omega} \Big) (T,S)[\omega'_2] \\
&\xrightarrow{\sigma^{-1}} q^{\omega'_1 + \omega'_2} \Big( \Big( (A_1)_{\omega'-\alpha} \otimes {_{\omega''-\alpha}}(F_2)_{\omega''} \Big) \otimes_{(A_1)_{\omega'-\alpha} \otimes (A_2)_{\omega''}} \Big( {_{\omega'-\alpha}}(F_1)_{\omega'}  \otimes (A_2)_{\omega''} \Big)  \\
& \quad \quad \otimes_{(A_1)_{\omega'} \otimes (A_2)_{\omega''}} (A_1 \ootimes A_2)_{\omega} \Big) (T,S)[\omega'_2 + 1] 
\end{align*}
has $\deg^q = 0$ and $\deg^h = -1$. This degree is what we want because, when viewing $\sigma^{-1}$ as a matrix entry of the new map $\tau$ and taking grading shifts of the first summand of $F$ in the domain and codomain of $\tau$ into account, the domain shift is determined by $\omega' \in \mathbb{Z}^2$ while the codomain shift is determined by $\omega' - \alpha \in \mathbb{Z}^2$.
\end{proof}

\subsection{Decategorification}\label{sec:BigradedTensorDecat}

Below we will consider right finite bimodule 2-representations $(A,F,\tau)$ of $\U^-$ where the dg category $A$ has only finitely many objects and where each bimodule ${_{\omega - \alpha}}F_{\omega}$ comes with the data of an AD bimodule $N_{\omega}$ (without higher terms $\delta^1_i$ for $i > 2$) satisfying ${_{\omega - \alpha}}F_{\omega} = N_{\omega} \otimes_{I_{\omega}} A_{\omega}$. 

From Section~\ref{sec:SplitGG} we have a map $[{_{\omega - \alpha}}F_{\omega}]\co G_0(A_{\omega}) \to G_0(A_{\omega - \alpha})$; summing over all $\omega$, we get an endomorphism $[F]$ of $G_0(A)$. If $(A,F,\tau)$ is a right finite bimodule 2-representation of $\U^-$, then $(G_0(A), [F])$ is a representation of $U_q(\gl(1|1)^-)$ (see Appendix~\ref{sec:Uqgl11Review}) equipped with a decomposition into $\gl(1|1)$ weight spaces for $\omega \in \Z^2$. The weight space decomposition determines the structure of a super or $\Z_2$-graded $\C_q$-vector space on $G_0(A)$; for an object $S$ of $A_{\omega}$ for $\omega \in \Z^2$, we declare $[S]$ to be even if $\omega_2$ is even and odd if $\omega_2$ is odd. The weight space decomposition also determines an action on $G_0(A)$ of the Hopf subalgebra of $U_q(\gl(1|1))$ generated by $F$, $q^{H_1}$, and $q^{H_2}$, with $q^{H_i} [S] \coloneqq q^{\omega_i} [S]$ if $S$ is an object of $A_{\omega}$; as in Appendix~\ref{sec:Uqgl11Review}, we can equivalently view this data as an action of $\Udot_q(\gl(1|1)^-)$ on $G_0(A)$.

\begin{proposition}
If $(A_1, F_1, \tau_1)$ and $(A_2, F_2, \tau_2)$ are right finite bimodule 2-representations of $\U^-$, we have a natural identification
\[
G_0(A_1 \ootimes A_2) \cong G_0(A_1) \otimes G_0(A_2)
\]
as modules over $\Udot_q(\gl(1|1)^-)$

\end{proposition}

\begin{proof}
It is clear that $G_0(A_1 \ootimes A_2)$ and $G_0(A_1) \otimes G_0(A_2)$ agree as super $\C_q$-vector spaces and that the actions of $q^{H_1}$ and $q^{H_2}$ agree; we will show that $[F] = [F_1] \otimes 1 + q^{H_1 + H_2} \otimes [F_2]$ as endomorphisms of this super vector space, where $[F]$ is the map on $G_0$ induced by the dg bimodule over $A_1 \ootimes A_2$ defined by the tensor product construction.

First, write $F_1 = F'_1 \otimes_{I_1} A_1$ and $F_2 = F'_2 \otimes_{I_2} A_2$. We have
\begin{align*}
& (F_1 \otimes A_2)  \otimes_{A_1 \otimes A_2} (A_1 \ootimes A_2) \\
&= (F'_1 \otimes I_2) \otimes_{I_1 \otimes I_2} (A_1 \otimes A_2) \otimes_{A_1 \otimes A_2} (A_1 \ootimes A_2) \\
&= (F'_1 \otimes I_2) \otimes_{I_1 \otimes I_2} (A_1 \ootimes A_2)
\end{align*}
and
\begin{align*}
& (A_1 \otimes F_2) \otimes_{A_1 \otimes A_2} (A_1 \ootimes A_2) \\
&= (I_1 \otimes F_2) \otimes_{I_1 \otimes I_2} (A_1 \otimes A_2) \otimes_{A_1 \otimes A_2} (A_1 \ootimes A_2) \\
&= (I_1 \otimes F_2) \otimes_{I_1 \otimes I_2} (A_1 \ootimes A_2);
\end{align*}
note that $I_1 \otimes I_2$ is the non-full subcategory of $A_1 \ootimes A_2$ containing all objects but only identity morphisms. 

Recall that $F = \oplus_{\omega \in \Z^2} \left( {_{\omega - \alpha}}F_{\omega} \right)$; for $\omega \in \Z^2$ we have
\[
{_{\omega - \alpha}}F_{\omega} = {_{\omega - \alpha}}(F')_{\omega} \otimes_{I_1 \otimes I_2} (A_1 \ootimes A_2)
\]
as $(I_1 \otimes I_2, I_1 \otimes I_2)$-bimodules, where
\[
{_{\omega - \alpha}}(F')_{\omega} = \bigoplus_{\omega' + \omega'' = \omega} \Big(q^{\omega'_1 + \omega'_2} ((I_1)_{\omega'} \otimes {_{\omega''- \alpha}}(F'_2)_{\omega''}) [\omega'_2] \Big) \oplus \Big({_{\omega' - \alpha}}(F'_1)_{\omega'} \otimes (I_2)_{\omega''}\Big)
\]
(ignoring the differential). For objects $S = (S_1, S_2)$ and $T = (T_1, T_2)$ of $(I_1 \otimes I_2)_{\omega}$ with $S_1 \in (I_1)_{\omega'}$, $S_2 \in (I_2)_{\omega''}$, and $\omega = \omega' + \omega''$, we have
\begin{align*}
{_{\omega - \alpha}}(F')_{\omega}(T,S) &= q^{\omega'_1 + \omega'_2} \Big( (I_1)_{\omega'}(T_1, S_1) \otimes {_{\omega'' - \alpha}}(F'_2)_{\omega''} (T_2, S_2) \Big) [\omega'_2] \\
& \oplus \Big( {_{\omega' - \alpha}}(F'_1)_{\omega'}(T_1, S_1) \otimes (I_2)_{\omega''}(T_2,S_2) \Big).
\end{align*}
Taking $\chi_q$, the entry of $[F]$ in row $[T]$ and column $[S]$ equals
\[
(-1)^{\omega'_2} q^{\omega'_1 + \omega'_2} \delta_{T_1 = S_1} \chi_q( {_{\omega''-\alpha}}(F'_2)_{\omega''} (T_2, S_2) ) + \delta_{T_2 = S_2} \chi_q ( {_{\omega'-\alpha}}(F'_1)_{\omega'} (T_1, S_1) ).
\]
This expression is also the coefficient of $[T]$ in $([F_1] \otimes 1 + q^{H_1 + H_2} \otimes [F_2])([S])$; indeed, $q^{H_1 + H_2}$ acts on $[S_1]$ to give $q^{\omega'_1 + \omega'_2} [S_1]$, while commuting $[F_2]$ past $[S_1]$ picks up a factor of $(-1)^{\omega'_2}$ since $[F_2]$ is odd and $[S_1]$ has parity given by the sign of $\omega'_2$.
\end{proof}

\subsection{Morphisms}\label{sec:2RepMorphisms}

\subsubsection{Dg bimodule 1-morphisms of 2-representations}

Following \cite{ManionRouquier} and adding bigradings, dg bimodule 1-morphisms between bimodule 2-representations of $\U^-$ can be defined as follows.

\begin{definition}[Section 5.1.1 of \cite{ManionRouquier}]
Given bimodule 2-representations $(A,F,\tau)$ and $(A',F',\tau')$ of $\U^-$, a dg bimodule 1-morphism from $(A,F,\tau)$ to $(A',F',\tau')$ is, for each $\omega \in \Z^2$, a dg $(A'_{\omega},A_{\omega})$-bimodule $P_{\omega}$ equipped with a closed degree-zero isomorphism of $(A'_{\omega-\alpha},A_{\omega})$-bimodules
\[
\varphi_{\omega}\co P_{\omega - \alpha} \otimes_{A_{\omega - \alpha}} {_{\omega - \alpha}}F_{\omega} \xrightarrow{\cong} {_{\omega - \alpha}}F'_{\omega} \otimes_{A'_{\omega}} P_{\omega}
\]
such that
\begin{align*}
&(\tau'_{\omega} \otimes \id_{P_{\omega}}) \circ (\id_{{_{\omega - 2\alpha}}F'_{\omega - \alpha}} \otimes \varphi_{\omega}) \circ (\varphi_{\omega - \alpha} \otimes \id_{{_{\omega - \alpha}}F_{\omega}}) \\
&= (\id_{{_{\omega - 2\alpha}}F'_{\omega - \alpha}} \otimes \varphi_{\omega}) \circ (\varphi_{\omega - \alpha} \otimes \id_{{_{\omega - \alpha}}F_{\omega}}) \circ (\id_{P_{\omega - 2\alpha}} \otimes \tau_{\omega})
\end{align*}
as maps from ${_{\omega - 2\alpha}}(PF^2)_{\omega}$ to ${_{\omega - 2\alpha}}((F')^2 P)_{\omega}$.
\end{definition}
If $(A,F,\tau)$ and $(A',F',\tau')$ are right finite, we can equivalently define a 1-morphism from $(A,F,\tau)$ to $(A',F',\tau')$ using a closed isomorphism from ${_{\omega + \alpha}}(F')^{\vee}_{\omega} \otimes_{A'_{\omega}} P_{\omega}$ to $P_{\omega + \alpha} \otimes_{A_{\omega + \alpha}} {_{\omega + \alpha}}F^{\vee}_{\omega}$; we can also reverse the direction and define 1-morphisms in terms of
\[
\varphi_{\omega}\co P_{\omega + \alpha} \otimes_{A_{\omega + \alpha}} {_{\omega + \alpha}}F^{\vee}_{\omega} \xrightarrow{\cong} {_{\omega + \alpha}}(F')^{\vee}_{\omega} \otimes_{A'_{\omega}} P_{\omega}
\]
satisfying
\begin{align*}
&(\tau'_{\omega} \otimes \id_{P_{\omega}}) \circ (\id_{{_{\omega + 2\alpha}}(F')^{\vee}_{\omega + \alpha}} \otimes \varphi_{\omega}) \circ (\varphi_{\omega + \alpha} \otimes \id_{{_{\omega + \alpha}}F^{\vee}_{\omega}}) \\
&= (\id_{{_{\omega + 2\alpha}}(F')^{\vee}_{\omega + \alpha}} \otimes \varphi_{\omega}) \circ (\varphi_{\omega + \alpha} \otimes \id_{{_{\omega + \alpha}}F^{\vee}_{\omega}}) \circ (\id_{P_{\omega + 2\alpha}} \otimes \tau_{\omega}).
\end{align*}

\subsubsection{AD and DA bimodule 1-morphisms of 2-representations}

We relax some equalities and isomorphisms to homotopies and homotopy equivalences here, as required to accommodate the examples we will consider below. If $(A,F,\tau)$ is a bimodule 2-representation of $\U^-$, we say that $(A,F,\tau)$ is right bounded if $F = N \otimes_I A$ for some right bounded AD bimodule $N$ (not to be confused with right finiteness for $F$ which asks for finite generation of $N$).

\begin{remark}
Below we will identify right finite dg bimodules $F$ with their underlying finitely generated AD bimodules, rather than maintaining a notational distinction as above.
\end{remark}

\begin{definition}\label{def:AD1Mor}
Assume that $(A,F,\tau)$, $(A',F',\tau')$ are right finite, right bounded bimodule 2-representations of $\U^-$. An AD bimodule 1-morphism of 2-representations $(P,\varphi)$ from $(A,F,\tau)$ to $(A',F',\tau')$ consists of, for $\omega \in \Z^2$, finitely generated right bounded AD bimodules $P_{\omega}$ over $(A'_{\omega},A_{\omega})$ together with homotopy equivalences of AD bimodules
\[
\varphi_{\omega}\co P_{\omega - \alpha} \boxtimes_{A_{\omega - \alpha}} {_{\omega - \alpha}}F_{\omega} \xrightarrow{\cong} {_{\omega - \alpha}}F'_{\omega} \boxtimes_{A'_{\omega}} P_{\omega}
\]
satisfying
\begin{align*}
&(\tau'_{\omega} \boxtimes \id_{P_{\omega}}) \circ (\id_{{_{\omega - 2\alpha}}F'_{\omega - \alpha}} \boxtimes \varphi_{\omega}) \circ (\varphi_{\omega - \alpha} \boxtimes \id_{{_{\omega - \alpha}}F_{\omega}}) \\
&= (\id_{{_{\omega - 2\alpha}}F'_{\omega - \alpha}} \boxtimes \varphi_{\omega}) \circ (\varphi_{\omega - \alpha} \boxtimes \id_{{_{\omega - \alpha}}F_{\omega}}) \circ (\id_{P_{\omega - 2\alpha}} \boxtimes \tau_{\omega}).
\end{align*}
\end{definition}

We have a similar definition using DA bimodules.

\begin{definition}\label{def:DA1Mor}
Assume that $(A,F,\tau)$, $(A',F',\tau')$ are right finite, right bounded bimodule 2-representations of $\U^-$. A DA bimodule 1-morphism of 2-representations $(P,\varphi)$ from $(A,F,\tau)$ to $(A',F',\tau')$ consists of, for $\omega \in \Z^2$, finitely generated left bounded DA bimodules $P_{\omega}$ over $(A'_{\omega},A_{\omega})$ together with $A_{\infty}$ homotopy equivalences
\[
\varphi_{\omega}\co P_{\omega + \alpha} \boxtimes_{A_{\omega + \alpha}} {_{\omega + \alpha}}F^{\vee}_{\omega} \xrightarrow{\sim} {_{\omega + \alpha}}(F')^{\vee}_{\omega} \boxtimes_{A'_{\omega}} P_{\omega}
\]
satisfying
\begin{align*}
&(\tau'_{\omega} \boxtimes \id_{P_{\omega}}) \circ (\id_{{_{\omega + 2\alpha}}(F')^{\vee}_{\omega + \alpha}} \boxtimes \varphi_{\omega}) \circ (\varphi_{\omega + \alpha} \boxtimes \id_{{_{\omega + \alpha}}F^{\vee}_{\omega}}) \\
&= (\id_{{_{\omega + 2\alpha}}(F')^{\vee}_{\omega + \alpha}} \boxtimes \varphi_{\omega}) \circ (\varphi_{\omega + \alpha} \boxtimes \id_{{_{\omega + \alpha}}F^{\vee}_{\omega}}) \circ (\id_{P_{\omega + 2\alpha}} \boxtimes \tau_{\omega}).
\end{align*}
\end{definition}

\begin{remark}
We require finite generation of $P$ in Definition~\ref{def:AD1Mor} for duality $P \leftrightarrow P^{\vee}$ to be well-behaved, and similarly in Definition~\ref{def:DA1Mor}, but the above definitions make sense without these finiteness assumptions. Some DA and AD bimodules arising from Heegaard diagrams with basepoints in this paper are not finitely generated; however, when discussing 1-morphism structure, we will work primarily with homotopy-equivalent versions that are finitely generated.
\end{remark}

If $(P,\varphi)$ is a DA bimodule 1-morphism of 2-representations from $(A,F,\tau)$ to $(A',F',\tau')$, then for each $\omega \in \Z^2$ we have a finitely generated right bounded AD bimodule ${^{\vee}}P_{\omega}$. Dualizing the map $\varphi_{\omega}$, we get
\[
\varphi'_{\omega + \alpha}\co {^{\vee}}P_{\omega} \boxtimes_{A'_{\omega}} {_{\omega}}F'_{\omega + \alpha} \xrightarrow{\sim} {_{\omega}}F_{\omega + \alpha} \boxtimes_{A_{\omega + \alpha}} {^{\vee}}P_{\omega + \alpha}
\]
giving us an AD bimodule 1-morphism $({^{\vee}}P, \varphi')$ from $(A',F',\tau')$ to $(A,F,\tau)$. Similarly, we can dualize AD bimodule 1-morphisms to get DA bimodule 1-morphisms.

\begin{example}
Given $(A,F,\tau)$, we upgrade the identity AD and DA bimodules over $A$ to identity 1-morphisms of 2-representations by taking $\varphi_{\omega} = \id$ for all $\omega$.
\end{example}

Suppose that $(A,F,\tau) \xrightarrow{(P,\varphi)} (A',F',\tau') \xrightarrow{(P',\varphi')} (A'',F'',\tau'')$ are AD bimodule 1-morphisms. Define their composition to be $P' \boxtimes P$ equipped with the maps
\[
(\varphi_{\omega} \boxtimes \id_{P_{\omega}}) \circ (\id_{P'_{\omega - \alpha}} \boxtimes \varphi'_{\omega})
\]
for $\omega \in \Z^2$; one can check that the square for compatibility with $\tau$ commutes. Similarly, the composition of DA bimodule 1-morphisms $(A,F,\tau) \xrightarrow{(P,\varphi)} (A',F',\tau') \xrightarrow{(P',\varphi')} (A'',F'',\tau'')$ is defined to be $P' \boxtimes P$ equipped with the maps
\[
(\varphi_{\omega} \boxtimes \id_{P_{\omega}}) \circ (\id_{P'_{\omega + \alpha}} \boxtimes \varphi'_{\omega})
\]
for $\omega \in \Z^2$. The dual of a composition of DA bimodule 1-morphisms can be naturally identified with the composition of duals (as AD bimodule 1-morphisms) after reversing the order, and vice-versa.

\subsubsection{2-morphisms}

\begin{definition}
Assume that $(A,F,\tau)$, $(A',F',\tau')$ are right finite, right bounded bimodule 2-representations of $\U^-$; let $(P,\varphi), (P',\varphi')$ be AD bimodule 1-morphisms of 2-representations from $(A,F,\tau)$ to $(A',F',\tau')$. A 2-morphism from $P$ to $P'$ consists of, for $\omega \in \Z^2$, AD bimodule morphisms $f_{\omega}\co P_{\omega} \to P'_{\omega}$ such that
\[
(\id_{{_{\omega - \alpha}}F'_{\omega}} \boxtimes f_{\omega}) \circ \varphi_{\omega} \sim \varphi'_{\omega} \circ (f_{\omega - \alpha} \boxtimes \id_{{_{\omega - \alpha}}F_{\omega}}),
\]
where $\sim$ denotes homotopy of AD bimodule morphisms.
\end{definition}

Given $(A,F,\tau)$ and $(A',F',\tau')$, the AD bimodule 1-morphisms from $(A,F,\tau)$ to $(A',F',\tau')$ and the 2-morphisms between them form a dg category; homotopy of 2-morphisms is defined in terms of this dg category, as are homotopy equivalence and isomorphism of 1-morphisms. Compositions $\boxtimes$ of AD bimodule 1-morphisms are associative up to canonical isomorphism of 1-morphisms, and the composition of an AD bimodule 1-morphism $(P,\varphi)$ with the identity 1-morphism on either side is canonically isomorphic (as 1-morphisms) to $(P,\varphi)$.

\begin{definition}
The $k$-linear $\Z^2$-graded bicategory $2\Rep(\U^-)^{*,*}$ is defined as follows.
\begin{itemize}
\item Objects of $2\Rep(\U^-)^{*,*}$ are right finite, right bounded bimodule 2-representations $(A,F,\tau)$ of $\U^-$.
\item For two objects $(A,F,\tau)$ and $(A',F',\tau')$ of $2\Rep(\U^-)^{*,*}$, the $k$-linear $\Z^2$-graded category 
\[
\Hom_{2\Rep(\U^-)^{*,*}}((A,F,\tau),(A',F',\tau'))
\]
is the bigraded homotopy category $H^{*,*}$ of the dg category with
\begin{itemize}
\item objects: finitely generated right bounded AD bimodule 1-morphisms
\item morphism complexes of 2-morphisms between AD bimodule 1-morphisms.
\end{itemize}
\item Identity 1-morphisms are identity AD bimodules.
\item Composition functors are defined as above on objects and send $(f,g)$ to the homotopy class $[f \boxtimes g]$; by \cite[Lemma 2.3.3]{LOTBimodules}, this homotopy class is well-defined, and one can check that $[f \boxtimes g]$ is the homotopy class of a 2-morphism.
\item Unitors and associators are the canonical isomorphisms from the paragraph above this definition.
\end{itemize}
\end{definition}

The superscript $*,*$ is just a reminder that there is a bigrading. Well-definedness of $2\Rep(\U^-)^{*,*}$ follows from the results of \cite[Section 2.3.2]{LOTBimodules}, especially \cite[Corollary 2.3.5]{LOTBimodules}. 

All of the above has an analogue for DA bimodule 1-morphisms. In particular, we can define 2-morphisms of DA bimodules as follows.
\begin{definition}
Assume that $(A,F,\tau)$, $(A',F',\tau')$ are right finite, right bounded bimodule 2-representations of $\U^-$; let $(P,\varphi), (P',\varphi')$ be DA bimodule 1-morphisms of 2-representations from $(A,F,\tau)$ to $(A',F',\tau')$. A 2-morphism from $P$ to $P$ consists of, for $\omega \in \Z^2$, DA bimodule morphisms $f_{\omega}\co P_{\omega} \to P'_{\omega}$ such that
\[
(\id_{{_{\omega + \alpha}}(F')^{\vee}_{\omega}} \boxtimes f_{\omega}) \circ \varphi_{\omega} \sim \varphi'_{\omega} \circ (f_{\omega + \alpha} \boxtimes \id_{{_{\omega + \alpha}}F^{\vee}_{\omega}}).
\]

\end{definition}

Duality gives an equivalence (reversing the direction of both 1-morphisms and 2-morphisms while preserving bidegrees of 2-morphisms) between $2\Rep(\U^-)^{*,*}$ and the analogous bicategory defined using DA 1-morphisms rather than AD 1-morphisms.

\begin{remark}\label{rem:Strong2Mor}
A limitation of the above definitions is that one cannot naturally define the structure of a 1-morphism on the mapping cone of a 2-morphism. To enable the definition of mapping cones, we could (say in the DA bimodule setting) define a strong 2-morphism from $(P,\varphi)$ to $(P',\varphi')$ to be the data of:
\begin{itemize}
\item For $\omega \in \Z^2$, DA bimodule morphisms $f_{\omega}\co P_{\omega} \to P'_{\omega}$;
\item For $\omega \in \Z^2$, DA bimodule morphisms $h_{\omega}\co P_{\omega + \alpha} \boxtimes {_{\omega + \alpha}}F^{\vee}_{\omega} \to {_{\omega + \alpha}}(F')^{\vee}_{\omega} \boxtimes P'_{\omega}$ satisfying
\[
d(h) = (\id_{{_{\omega + \alpha}}(F')^{\vee}_{\omega}} \boxtimes f_{\omega}) \circ \varphi_{\omega} - \varphi'_{\omega} \circ (f_{\omega + \alpha} \boxtimes \id_{{_{\omega + \alpha}}F^{\vee}_{\omega}})
\]
and
\begin{align*}
&\left(\tau'_{\omega} \boxtimes \id_{P'_{\omega}}\right) \circ \bigg( \left( \id_{{_{\omega + 2\alpha}}(F')^{\vee}_{\omega + \alpha}} \boxtimes h_{\omega} \right) \circ \left(\varphi_{\omega + \alpha} \boxtimes \id_{{_{\omega + \alpha}}F^{\vee}_{\omega}} \right) \\
&\qquad + \left( \id_{{_{\omega + 2\alpha}}(F')^{\vee}_{\omega+\alpha}} \boxtimes \varphi'_{\omega} \right) \circ \left( h_{\omega + \alpha} \boxtimes \id_{{_{\omega + \alpha}}F^{\vee}_{\omega}} \right) \bigg) \\
&= \bigg( \left( \id_{{_{\omega + 2\alpha}}(F')^{\vee}_{\omega + \alpha}} \boxtimes h_{\omega} \right) \circ \left(\varphi_{\omega + \alpha} \boxtimes \id_{{_{\omega + \alpha}}F^{\vee}_{\omega}} \right) \\
&\qquad + \left( \id_{{_{\omega + 2\alpha}}(F')^{\vee}_{\omega+\alpha}} \boxtimes \varphi'_{\omega} \right) \circ \left( h_{\omega + \alpha} \boxtimes \id_{{_{\omega + \alpha}}F^{\vee}_{\omega}} \right) \bigg) \circ \left( \id_{P_{\omega + 2\alpha}} \boxtimes \tau_{\omega} \right).
\end{align*}
\end{itemize}
One can check that if $(f,h)$ is a strong 2-morphism from $(P,\varphi)$ to $(P',\varphi')$, then the matrix $\begin{bmatrix} \varphi & 0 \\ h & \varphi' \end{bmatrix}$ defines a valid 1-morphism structure for the mapping cone of $f$.
\end{remark}

\section{A family of 2-representations}\label{sec:OurExamples}

\subsection{Definitions}

\subsubsection{Fundamental examples}\label{sec:Fundamental2RepExamples}

\begin{definition}
For $K \geq 1$ we define an $\F_2$-linear right finite, right bounded bimodule 2-representation $(\A_K, F, \tau)$ of $\U^-$ by:
\begin{itemize}
\item $(\A_K)_{\varepsilon_1 + (K-1)\varepsilon_2} = \F_2[e_1,\ldots,e_{K-1}]$ where $\deg^q(e_i) = -2i$ and $\deg^h(e_i) = 2i$, with no differential;
\item $(\A_K)_{K \varepsilon_2} = \F_2[e_1, \ldots, e_K]$ where $\deg^q(e_i) = -2i$ and $\deg^h(e_i) = 2i$, with no differential;
\item ${_{K\varepsilon_2}}F_{\varepsilon_1 + (K-1)\varepsilon_2} = \F_2[e_1,\ldots,e_{K-1}]$ as a bimodule over $\F_2[e_1,\ldots,e_{K-1}]$, where the left action of $e_K$ is zero;
\item all other summands of $F$ are zero, and $\tau = 0$.
\end{itemize}
\end{definition}

We choose monomials in the variables $e_i$ as our preferred homogeneous basis for morphism spaces in $\A_K$; we define an augmentation functor $\epsilon$ from $\A_K$ to its idempotent ring by sending all $e_i$ variables to zero.

We will refer to the unique object of $(\A_K)_{\varepsilon_1 + (K-1)\varepsilon_2}$ (respectively $(\A_K)_{K\varepsilon_2}$) as $\u$ for ``unoccupied'' (respectively $\o$ for ``occupied.'') This language of occupancy is motivated by how the algebras relate to the strands algebra construction (see Section~\ref{sec:ExamplesStrandsInterp}) and, correspondingly, holomorphic disk counts in Heegaard diagrams (see Section~\ref{sec:ExamplesHDInterp}).

\begin{remark}\label{rem:EtaForAK}
For the 2-representation $(\A_K, F, \tau)$, the right dual $F^{\vee}$ of $F$ is given by
\[
{_{\varepsilon_1 + (K-1)\varepsilon_2}}F^{\vee}_{K\varepsilon_2} = \F_2[e_1,\ldots,e_{K-1}]
\]
with all other summands zero; the right action of $e_K$ is zero and the degrees of $e_i$ are the same as above. Let $\xi_{\o,\u}$ and $\xi_{\u,\o}$ denote the generators of $F$ and $F^{\vee}$ over $\F_2[e_1,\ldots,e_{K-1}]$ respectively. The unit $\eta$ of the adjunction $(F^{\vee} \otimes_{\A_K} -) \dashv (F \otimes_{\A_K} -)$ of functors between left dg module categories is induced from the dg bimodule map $\eta \colon \id_{\A_K} \to F \otimes_{\A_K} F^{\vee}$ sending $\id_{\o} \in \id_{\A_K}$ to $\xi_{\o,\u} \otimes \xi_{\u,\o} \in F \otimes_{\A_K} F^{\vee}$ and acting as the identity on $\F_2[e_1,\ldots,e_{K-1}]$; the map $\eta$ sends $\id_{\u} \in \id_{\A_K}$ to zero.
\end{remark}

\begin{example}
If $K = 1$, then $\A_K = \A_1$ is the right finite bimodule 2-representation of $\U^-$ with:
\begin{itemize}
\item $(\A_1)_{\varepsilon_1} = \F_2$;
\item $(\A_1)_{\varepsilon_2} = \F_2[U]$ where $\deg^q(U) = -2$ and $\deg^h(U) = 2$, with no differential;
\item ${_{\varepsilon_2}}F_{\varepsilon_1} = \F_2$, where the left action of $U$ is zero;
\item all other summands of $F$ are zero, and $\tau = 0$.
\end{itemize}
\end{example}

\begin{example}
If $K = 2$, then $\A_K = \A_2$ is the right finite bimodule 2-representation of $\U^-$ with:
\begin{itemize}
\item $(\A_2)_{\varepsilon_1 + \varepsilon_2} = \F_2[e_1]$ where $\deg^q(e_1) = -2$ and $\deg^h(e_1) = 2$;
\item $(\A_2)_{2\varepsilon_2} = \F_2[e_1, e_2]$ where $\deg^q(e_i) = -2i$ and $\deg^h(e_i) = 2i$, with no differential;
\item ${_{2\varepsilon_2}}F_{\varepsilon_1 + \varepsilon_2} = \F_2[e_1]$, where the left action of $e_2$ is zero;
\item all other summands of $F$ are zero, and $\tau = 0$.
\end{itemize}
\end{example}

We describe the DA bimodule $F^{\vee} = {_{\varepsilon_1 + (K-1)\varepsilon_2}}F^{\vee}_{K\varepsilon_2}$ over $\A_K$ in matrix notation.

\begin{example}
The primary matrix for the bimodule $F^{\vee}$ over $\A_K$ is
\[
\kbordermatrix{
& \u & \o \\
\u & & \xi_{\u,\o} \\
\o & &
};
\]
the one generator $\xi_{\u,\o}$ has $q$-degree and homological degree both equal to zero. The matrix for the right action of $e_i$ is 
\[
\kbordermatrix{
& \xi_{\u,\o} \\
\xi_{\u,\o} & e_i
}
\]
for $1 \leq i \leq K-1$; the matrix for the right action of $e_K$ is zero.
\end{example}

By construction, $(G_0(\A_K), [F])$ can be identified with $\wedge_q^k V$ where $V$ is the vector representation of $U_q(\gl(1|1))$ (restricted to a representation of the subalgebra generated by $F$, $q^{H_1}$, and $q^{H_2}$, or alternatively a representation of $\Udot_q(\gl(1|1)^-)$); see Appendix~\ref{sec:Uqgl11Review}. The basis elements $[S_{\u}]$ and $[S_{\o}]$ of $G_0(\A_K)$ get identified with the basis elements $\ket{0} \wedge (\ket{1})^{\wedge(k-1)}$ and $(\ket{1})^{\wedge k}$ of $\wedge_q^k V$ respectively.

\subsubsection{Tensor products}
Write
\[
\A_{K_1,\ldots,K_n} \coloneqq \A_{K_1} \ootimes \cdots \ootimes \A_{K_n}
\]
as 2-representations of $\U^-$.
\begin{corollary}\label{cor:GeneralKDecat}
We have an identification
\[
G_0\left(\A_{K_1, \ldots, K_n}\right) \cong \wedge_q^{K_1} V \otimes \cdots \otimes \wedge_q^{K_n} V
\]
as representations of $\Udot_q(\gl(1|1)^-)$.
\end{corollary}

In particular, $G_0 \left( \A_{1,\ldots,1} \right)$ is identified with $V^{\otimes n}$. We will often write objects of $\A_{K_1,\ldots,K_n}$ as sequences of letters $\u$ (``unoccupied'') and $\o$ (``occupied'').

\begin{example}
As a dg category, the 2-representation $\A_{K_1,K_2}$ of $\U^-$ can be described as follows:
\begin{itemize}
\item $(\A_{K_1,K_2})_{2\varepsilon_1 + (K_1+K_2-2)\varepsilon_2} = \F_2[e_{1,1}, \ldots, e_{1,K_1-1}, e_{2,1}, \ldots, e_{2,K_2-1}]$
 with one object $\uu$; we have
\[
\deg^q(e_{j,i}) = -2i, \qquad \deg^h(e_{j,i}) = 2i.
\]
\item $(\A_{K_1,K_2})_{\varepsilon_1 + (K_1+K_2-1)\varepsilon_2}$ has two objects $\ou$ and $\uo$ with morphisms generated by endomorphisms 
\[
e_{1,1}, \ldots, e_{1,K_1}, e_{2,1}, \ldots, e_{2,K_2-1}
\]
of $\ou$, endomorphisms 
\[
e_{1,1}, \ldots, e_{1,K_1-1}, e_{2,1}, \ldots, e_{2,K_2}
\]
of $\uo$, and a morphism $\lambda$ from $\ou$ to $\uo$, subject to the relations
\begin{itemize}
\item $e_{j,i} g = g e_{j,i}$ for $j=1,2$, $1 \leq i \leq K_j - 1$, and all morphisms $g$,
\item $e_{2,K_2} \lambda = 0$ and $\lambda e_{1,K_1} = 0$.
\end{itemize}
We have $\deg^q(e_{j,i}) = -2i$, $\deg^h(e_{j,i}) = 2i$, $\deg^q(\lambda) = K_1$, and $\deg^h(\lambda) = 1-K_1$.
\item $(\A_{K_1,K_2})_{(K_1+K_2)\varepsilon_2} = \F_2[e_{1,1}, \ldots, e_{1,K_1}, e_{2,1}, \ldots, e_{2,K_2}]$ with one object $\oo$; we have 
\[
\deg^q(e_{j,i}) = -2i, \qquad \deg^h(e_{j,i}) = 2i.
\]
\end{itemize}
\end{example}

Indeed, the morphism $\lambda$ is $\xi_{\u,\o} \otimes \xi_{\o,\u}$ in the summand
\[
\Big( {_{\varepsilon_1 + (K_1-1)\varepsilon_2}}(\widetilde{F}_1^{\vee})_{K_1 \varepsilon_2} \Big) \otimes_{H_1} \Big( {_{K_2 \varepsilon_2}}(F_2)_{\varepsilon_1 + (K_2-1)\varepsilon_2} \Big) \subset {_{\varepsilon_1 + (K_1 + K_2 - 1)\varepsilon_2}}(\A_{K_1} \ootimes \A_{K_2})_{\varepsilon_1 + (K_1 + K_2 - 1)\varepsilon_2};
\]
note that ${_{\varepsilon_1 + (K_1-1)\varepsilon_2}}(\widetilde{F}_1^{\vee})_{K_1\varepsilon_2} := q^{K_1} {_{\varepsilon_1 + (K_1-1)\varepsilon_2}}(F_1^{\vee})_{K_1\varepsilon_2} [K_1-1]$,
explaining the $q$- and $h$-degrees of $\lambda$.

Next, to compute the bimodule $F$ over $\A_{K_1,K_2}$ (and thereby compute $F^{\vee}$), we need to compute the map $u$ on the generator $\xi_{\oo,\ou} := \id_{\o} \otimes \xi_{\o,\u}$ of $\A_{K_1} \otimes F_2 \subset F$ ($u$ is zero on $\id_{\u} \otimes \xi_{\o,\u}$). This map $u$ applies the unit $\eta$ from Remark~\ref{rem:EtaForAK} to the left of $\id_{\o} \otimes \xi_{\o,\u}$ to get an element of $((F_1 \otimes \A_{K_2}) \otimes (F_1^{\vee} \otimes \A_{K_2})) \otimes (\A_{K_1} \otimes F_2)$, then reparenthesizes and reinterprets the result as an element of $(F_1 \otimes \A_{K_2}) \otimes (\A_{K_1} \ootimes \A_{K_2})$. In the first step we get $(\xi_{\o,\u} \otimes \id_{\o}) \otimes (\xi_{\u,\o} \otimes \id_{\o}) \otimes (\id_{\o} \otimes \xi_{\o,\u})$. The second and third tensor factors then combine to become $\xi_{\u,\o} \otimes \xi_{\o,\u} = \lambda$ in $\widetilde{F}_1^{\vee} \otimes_{H_1} F_2 \subset \A_{K_1} \ootimes \A_{K_2}$, and we can write $\xi_{\oo,\uo} := \xi_{\o,\u} \otimes \id_{\o}$ for the first tensor factor. Let $\xi_{\ou,\oo}$ and $\xi_{\uo,\oo}$ be the generators of $F^{\vee}$ dual to the generators $\xi_{\oo,\ou}$ and $\xi_{\oo,\uo}$ of $F$ respectively. It follows that in the differential on $F^{\vee}$, $d(\xi_{\uo,\oo})$ is given by $\lambda \cdot \xi_{\ou,\oo}$. One can also compute the rest of the structure of $F^{\vee}$, resulting in the matrices of the following example.

\begin{example}
The DA bimodule $F^{\vee}$ over $\A_{K_1,K_2}$ can be described in matrix notation as follows. The primary matrix for ${_{2\varepsilon_1 + (K_1+K_2-2)\varepsilon_2}}(F^{\vee})_{\varepsilon_1 + (K_1+K_2-1)\varepsilon_2}$ is
\[
\kbordermatrix{
& \ou & \uo \\
\uu & \xi_{\uu,\ou} & \xi_{\uu,\uo}
};
\]
we have $\deg^q(\xi_{\uu,\ou}) = \deg^h(\xi_{\uu,\ou}) = 0$ as well as $\deg^q(\xi_{\uu,\uo}) = -K_1$ and $\deg^h(\xi_{\uu,\uo}) = K_1 - 1$.

The matrix for the right action of $\lambda$ is
\[
\kbordermatrix{
& \xi_{\uu,\ou} & \xi_{\uu,\uo} \\
\xi_{\uu,\ou} & 0 & 1 \\
\xi_{\uu,\uo} & 0 & 0
}.
\]
For $j = 1,2$ and $1 \leq i \leq K_j - 1$, the right action by $e_{j,i}$ has matrix
\[
\kbordermatrix{
& \xi_{\uu,\ou} & \xi_{\uu,\uo} \\
\xi_{\uu,\ou} & e_{j,i} & 0 \\
\xi_{\uu,\uo} & 0 & e_{j,i}
}.
\]
The differential and the right action of $e_{1,K_1}$ and $e_{2,K_2}$ are zero.

The primary matrix for ${_{\varepsilon_2+(K_1+K_2-1)}}(F^{\vee})_{(K_1+K_2)\varepsilon_2}$ is
\[
\kbordermatrix{
& \oo \\
\ou & \xi_{\ou,\oo} \\
\uo & \xi_{\uo,\oo} \\
};
\]
we have $\deg^q(\xi_{\ou,\oo}) = -K_1$, $\deg^h(\xi_{\ou,\oo}) = K_1$, $\deg^q(\xi_{\uo,\oo}) = 0$, and $\deg^h(\xi_{\uo,\oo}) = 0$.

The differential has matrix
\[
\kbordermatrix{
& \xi_{\ou,\oo} & \xi_{\uo,\oo} \\
\xi_{\ou,\oo} & 0 & \lambda \\
\xi_{\uo,\oo} & 0 & 0
},
\]
the right action by $e_{1,K_1}$ has matrix
\[
\kbordermatrix{
& \xi_{\ou,\oo} & \xi_{\uo,\oo} \\
\xi_{\ou,\oo} & e_{1,K_1} & 0 \\
\xi_{\uo,\oo} & 0 & 0
},
\]
the right action by $e_{2,K_2}$ has matrix
\[
\kbordermatrix{
& \xi_{\ou,\oo} & \xi_{\uo,\oo} \\
\xi_{\ou,\oo} & 0 & 0 \\
\xi_{\uo,\oo} & 0 & e_{2,K_2}
},
\]
and the right action by $e_{j,i}$ for $j = 1,2$, $1 \leq i \leq K_j - 1$ has matrix
\[
\kbordermatrix{
& \xi_{\ou,\oo} & \xi_{\uo,\oo} \\
\xi_{\ou,\oo} & e_{j,i} & 0 \\
\xi_{\uo,\oo} & 0 & e_{j,i}
}.
\]
\end{example}

\begin{example}
The endomorphism $\tau$ of the bimodule ${_{2\varepsilon_1 + (K_1+K_2-2)\varepsilon_2}}(F^{\vee})^2_{(K_1+K_2)\varepsilon_2}$ has matrix
\[
\kbordermatrix{
&\xi_{\uu,\ou} \xi_{\ou,\oo} & \xi_{\uu,\uo} \xi_{\uo,\oo} \\
\xi_{\uu,\ou} \xi_{\ou,\oo} & 0 & 0 \\
\xi_{\uu,\uo} \xi_{\uo,\oo} & 1 & 0
}.
\]
\end{example}

\begin{figure}
\includegraphics[scale=0.4]{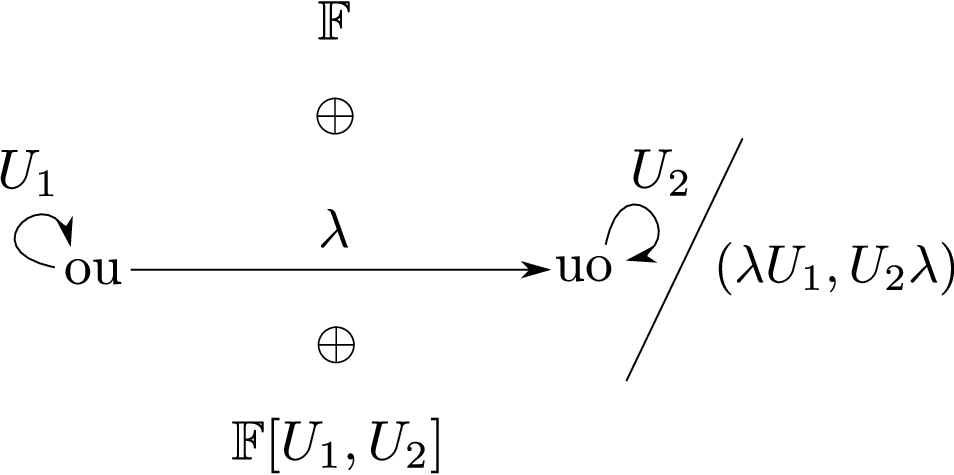}
\caption{The dg category $\A_{1,1}$.}
\label{fig:A11Quiver}
\end{figure}

For convenience, we specialize to the case $K_1 = K_2 = 1$ which will be especially important below. In this case we write $U_1$ rather than $e_{1,1}$ and $U_2$ rather than $e_{2,1}$.

\begin{example}
The dg category $\A_{1,1}$ is shown in Figure~\ref{fig:A11Quiver}; arrows point from source to target. The differential on $\A_{1,1}$ is zero; the gradings are given by
\begin{itemize}
\item $\deg^q(U_1) = -2$, $\deg^h(U_1) = 2$,
\item $\deg^q(\lambda) = 1$, $\deg^h(\lambda) = 0$, 
\item $\deg^q(U_2) = -2$, $\deg^h(U_2) = 2$.
\end{itemize}

The primary matrix for ${_{2\varepsilon_1}}(F^{\vee})_{\varepsilon_1 + \varepsilon_2}$ is
\[
\kbordermatrix{
& \ou & \uo \\
\uu & \xi_{\uu,\ou} & \xi_{\uu,\uo}
};
\]
the generators $\xi_{\uu,\ou}$ and $\xi_{\uu,\uo}$ have $q$-degree $0$ and $-1$ respectively, and both generators have homological degree $0$. The matrix for the right action of $\lambda$ is
\[
\kbordermatrix{
& \xi_{\uu,\ou} & \xi_{\uu,\uo} \\
\xi_{\uu,\ou} & 0 & 1 \\
\xi_{\uu,\uo} & 0 & 0
};
\]
the differential and the right action of $U_1$ and $U_2$ are zero (so that the above matrix is equivalently the secondary matrix for ${_{2\varepsilon_1}}(F^{\vee})_{\varepsilon_1 + \varepsilon_2}$).

The primary matrix for ${_{\varepsilon_1 + \varepsilon_2}}(F^{\vee})_{2\varepsilon_2}$ is
\[
\kbordermatrix{
& \oo \\
\ou & \xi_{\ou,\oo} \\
\uo & \xi_{\uo,\oo} \\
};
\]
the generator $\xi_{\ou,\oo}$ has $q$-degree $-1$ and homological degree $1$, while $\xi_{\uo,\oo}$ has $q$-degree $0$ and homological degree $0$. The differential has matrix
\[
\kbordermatrix{
& \xi_{\ou,\oo} & \xi_{\uo,\oo} \\
\xi_{\ou,\oo} & 0 & \lambda \\
\xi_{\uo,\oo} & 0 & 0
},
\]
the right action by $U_1$ has matrix
\[
\kbordermatrix{
& \xi_{\ou,\oo} & \xi_{\uo,\oo} \\
\xi_{\ou,\oo} & U_1 & 0 \\
\xi_{\uo,\oo} & 0 & 0
},
\]
and the right action by $U_2$ has matrix
\[
\kbordermatrix{
& \xi_{\ou,\oo} & \xi_{\uo,\oo} \\
\xi_{\ou,\oo} & 0 & 0 \\
\xi_{\uo,\oo} & 0 & U_2
}.
\]
Equivalently, the secondary matrix for ${_{\varepsilon_1 + \varepsilon_2}}(F^{\vee})_{2\varepsilon_2}$ is
\[
\kbordermatrix{
& \xi_{\ou,\oo} & \xi_{\uo,\oo} \\
\xi_{\ou,\oo} & U_1^{k+1} \otimes U_1^{k+1} & \lambda \\
\xi_{\uo,\oo} & 0 & U_2^{k+1} \otimes U_2^{k+1}
}
\]
(where $k$ runs over all nonnegative integers in each entry with a $k$ index).

The endomorphism $\tau$ of ${_{2\varepsilon_1}}(F^{\vee})^2_{2\varepsilon_2}$ has matrix
\[
\kbordermatrix{
&\xi_{\uu,\ou} \xi_{\ou,\oo} & \xi_{\uu,\uo} \xi_{\uo,\oo} \\
\xi_{\uu,\ou} \xi_{\ou,\oo} & 0 & 0 \\
\xi_{\uu,\uo} \xi_{\uo,\oo} & 1 & 0
}.
\]
\end{example}

\subsection{Strands interpretation}\label{sec:ExamplesStrandsInterp}

\subsubsection{Strands algebras and chord diagrams}

\begin{definition}[cf. Section 7.2.4 of \cite{ManionRouquier}, Definition 2.1.1 of \cite{Zarev}] A chord diagram\footnote{Terminology following e.g. \cite{ACPRS}.} $\Zc = (\Zc, \mathbf{a}, \mu)$ is a finite collection $\Zc$ of oriented circles and intervals, equipped with a two-to-one matching $\mu$ on a finite subset $\mathbf{a}$ of points in the interior of $\Zc$. 
\end{definition}
Zarev allows only intervals in $\Zc$, not circles. We will draw the intervals and circles in black while indicating the matching $\mu$ with red arcs between points of $\mathbf{a}$ (endpoints of the red arcs). Figures~\ref{fig:Z1} and \ref{fig:Zstn} show some examples of chord diagrams.

Differential (ungraded) strands categories for general chord diagrams $\Zc$ are defined in \cite{ManionRouquier}.
\begin{itemize}
\item Let $N = |\mathbf{a}|/2$.
\item Let $\mathbf{a}/\sim$ be the set of two-element equivalence classes of points of $\mathbf{a}$ determined by the matching.
\item Let $Z$ be the singular curve associated to $\Zc$ in \cite[Section 7.2.4]{ManionRouquier} (we can view $\mathbf{a}/\sim$ as a subset of $Z$).
\item Let $\Sc(Z)$ be the strands category of $Z$ defined in \cite[Definition 7.4.22]{ManionRouquier}.
\end{itemize}

\begin{definition}
Let $\A(\Zc) = \Sc_{\mathbf{a}/\sim}(Z)$, the full differential subcategory of $\Sc(Z)$ on objects contained in the finite set $\mathbf{a}/\sim$. We can write
\[
\A(\Zc) \coloneqq \bigoplus_{k=0}^N \A(\Zc,k)
\]
where $\A(\Zc,k)$ is the full subcategory of $\A(\Zc)$ on objects $S$ with $|S|=k$.
\end{definition}

Furthermore, \cite{ManionRouquier} defines (ungraded) bimodule 2-representations $(\A(\Zc), L_{\xi_+}, \tau)$ and $(\A(\Zc), R_{\xi_-}, \tau)$ of $\U$ for each chord diagram $\Zc$ equipped with a chosen interval component as follows. By \cite[Section 8.1.6]{ManionRouquier}, the chosen interval component of $\Zc$ gives injective morphisms of curves $\xi_+\co \R_{>0} \to Z$ and $\xi_-\co \R_{<0} \to Z$ which are outgoing and incoming for $Z$ respectively (as defined in \cite[Section 8.1.1]{ManionRouquier}). 

By \cite[Proposition 8.1.3]{ManionRouquier}, taking $M = \mathbf{a} / \sim$, we get a left finite 2-representation $(\A(\Zc), L_{\xi_+}, \tau)$ of $\U$. Its left dual $(\A(\Zc), {^{\vee}}L_{\xi_+}, \tau)$ is thus a right finite 2-representation of $\U$ which agrees with $(\A(\Zc), R_{\xi_-}, \tau)$ by \cite[Proposition 8.1.15]{ManionRouquier}. It follows from \cite[Proposition 8.1.10]{ManionRouquier} (resp. \cite[Section 8.1.5]{ManionRouquier}) that $L_{\xi_+}$ is left bounded (resp. $R_{\xi_-}$ is right bounded).

\subsubsection{Fundamental examples}

\begin{figure}
\includegraphics[scale=0.6]{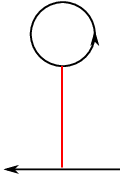}
\caption{The chord diagram $\Zc_1$.}
\label{fig:Z1}
\end{figure}

\begin{example}
The differential category $\A(\Zc_1)$ agrees with $\A_1$ (ignoring gradings), where $\Zc_1$ is shown in Figure~\ref{fig:Z1}; we have $\A(\Zc_1,0) = (\A_1)_{\varepsilon_1}$ and $\A(\Zc_1,1) = (\A_1)_{\varepsilon_2}$.  

\end{example}

Basis elements of $\A(\Zc_1)$ over $\F_2$ can be represented by strands pictures in the disjoint union of a rectangle and a cylinder; the element $U^i$ of $\A(\Zc_1,1) = \F_2[U]$ wraps around the cylinder $i$ times, while $1 \in \A(\Zc_1,1) = \F_2$ is a pair of dotted lines corresponding to the two matched points in the chord diagram $\Zc_1$ (see Figure~\ref{fig:LStrandsExamples}).

\begin{figure}
\includegraphics[scale=0.6]{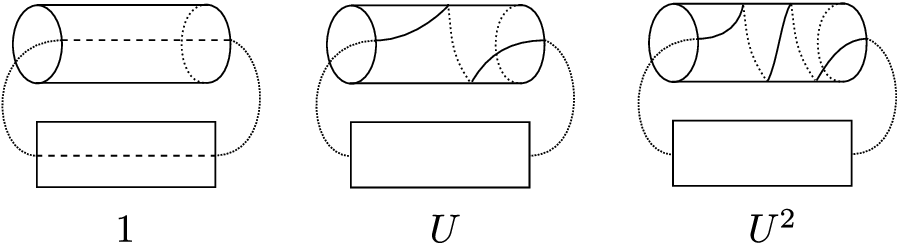}
\caption{Elements $1$, $U$, and $U^2$ of $\A(\Zc_1,1) = \F_2[U]$.}
\label{fig:LStrandsExamples}
\end{figure}

The bimodule $R_{\xi_-}$ over $\A(\Zc_1)$ can be viewed as a bimodule over $(\A(\Zc_1,1), \A(\Zc_1,0))$ with a single basis element over $\F_2$, shown on the left of Figure~\ref{fig:LStrandsBimodExamples}. The left action of $U \in \A(\Zc_1,1) = \F_2[U]$ on this basis element is zero. We can thus identify $R_{\xi_-}$ with the bimodule $F = {_{\varepsilon_2}}F_{\varepsilon_1}$ over $\A_1$ (ignoring the gradings on the latter bimodule).

Similarly, the only nonzero summand of the bimodule $L_{\xi_+}$ over $\A(\Zc_1)$ is a bimodule over $(\A(\Zc_1,0), \A(\Zc_1,1))$ with a single basis element over $\F_2$, shown on the right of Figure~\ref{fig:LStrandsBimodExamples}. The right action of $U \in \A(\Zc_1,1) = \F_2[U]$ on this basis element is zero. We can thus identify $L_{\xi_+}$ with the bimodule $F^{\vee} = {_{\varepsilon_1}}F^{\vee}_{\varepsilon_2}$ over $\A_1$ (ignoring gradings again).

\begin{figure}
\includegraphics[scale=0.6]{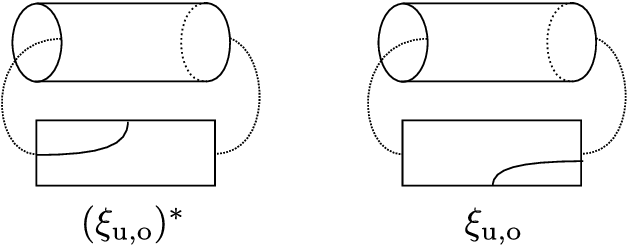}
\caption{The unique basis elements of the bimodules $F = R_{\xi_-}$ and $F^{\vee} = L_{\xi^+}$ over $\A_1$.}
\label{fig:LStrandsBimodExamples}
\end{figure}

\begin{corollary}
The 2-representation of $\U$ underlying the 2-representation $(\A_1,F,0)$ of $\U^-$ agrees with $(\A(\Zc_1), R_{\xi_-}, 0)$, and its dual agrees with $(\A(\Zc_1), L_{\xi_+}, 0)$.
\end{corollary}

\subsubsection{Tensor products}

\begin{figure}
\includegraphics[scale=0.6]{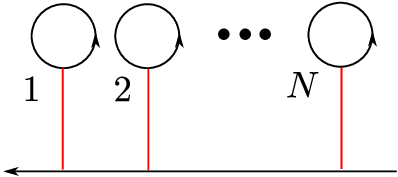}
\caption{The chord diagram $\Zc_{1,\ldots,1}^{\st}$.}
\label{fig:Zstn}
\end{figure}

The following proposition follows from \cite[Theorem 8.3.1]{ManionRouquier}.
\begin{proposition}
We have a canonical isomorphism
\[
\A_{1,\ldots,1} \cong \A(\Zc^{\st}_{1,\ldots,1})
\]
as 2-representations of $\U$, where $\Zc^{\st}_{1,\ldots,1}$ is the chord diagram shown in Figure~\ref{fig:Zstn}.
\end{proposition}

The dg category $\A(\Zc^{\st}_{1,\ldots,1})$ comes from a pointed strands category as in \cite[Definition 7.4.22]{ManionRouquier}, so it comes with a preferred choice of basis (basis elements correspond to strands pictures as in Figure~\ref{fig:LStrandsExamples}). We define an augmentation functor $\epsilon$ from $\A(\Zc^{\st}_{1,\ldots,1})$ to its idempotent ring by sending all non-identity basis elements to zero; on $\A(\Zc^{\st}_1) = \A_1$, this basis and augmentation agree with the ones defined in Section~\ref{sec:Fundamental2RepExamples}. 

\begin{figure}
\includegraphics[scale=0.6]{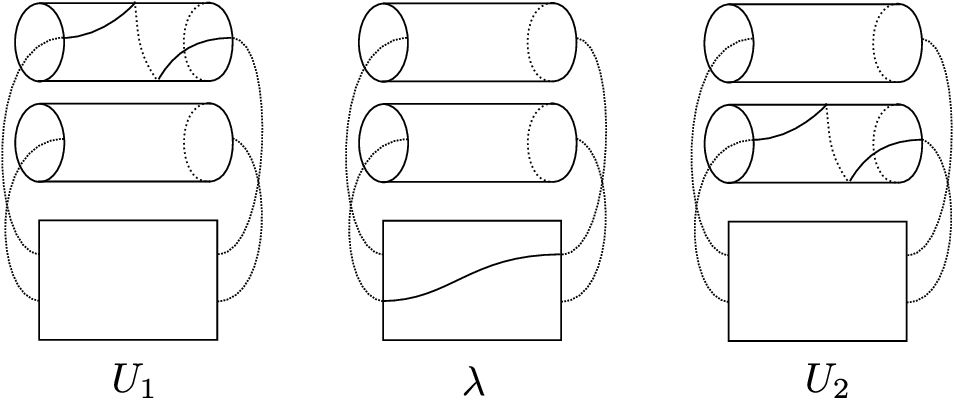}
\caption{Generators $U_1, \lambda$, and $U_2$ of $\A(\Zc_{1,1}^{\st},1)$.}
\label{fig:LLStrandsExamples}
\end{figure}

Figure~\ref{fig:LLStrandsExamples} shows the generators $U_1$, $\lambda$, and $U_2$ of $(\A_{1,1})_{\varepsilon_1 + \varepsilon_2} \cong \A(\Zc_{1,1}^{\st},1)$. Similarly, Figure~\ref{fig:LLStrandsBimodExamples} shows the two generators $\xi_{\uu,\uo}$ and $\xi_{\uu,\ou}$ of ${_{2\varepsilon_1}}F^{\vee}_{\varepsilon_1 + \varepsilon_2}$, and Figure~\ref{fig:LLMoreStrandsBimodExamples} shows the two generators $\xi_{\ou,\oo}$ and $\xi_{\uo,\oo}$ of ${_{\varepsilon_1 + \varepsilon_2}}F^{\vee}_{2\varepsilon_2}$.

\begin{figure}
\includegraphics[scale=0.6]{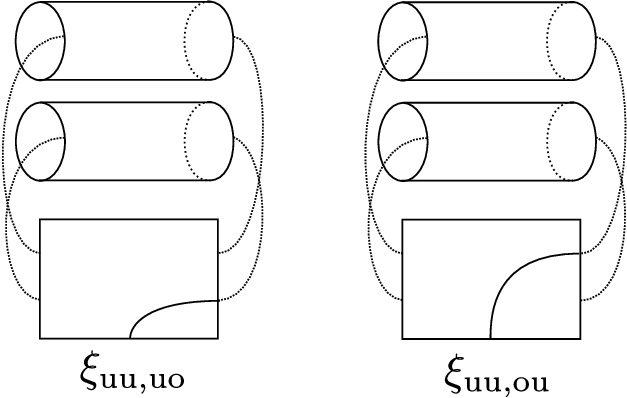}
\caption{Generators $\xi_{\uu,\uo}$ and $\xi_{\uu,\ou}$ of the bimodule ${_{2\varepsilon_1}}F^{\vee}_{\varepsilon_1 + \varepsilon_2}$ as a left module.}
\label{fig:LLStrandsBimodExamples}
\end{figure}

\begin{figure}
\includegraphics[scale=0.6]{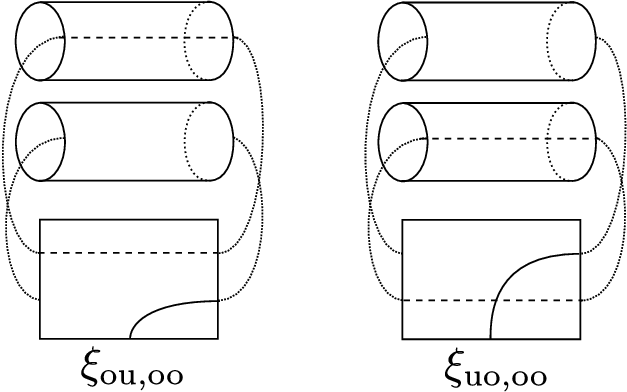}
\caption{Generators $\xi_{\ou,\oo}$ and $\xi_{\uo,\oo}$ of the bimodule ${_{\varepsilon_1+\varepsilon_2}}F^{\vee}_{2\varepsilon_2}$ as a left module.}
\label{fig:LLMoreStrandsBimodExamples}
\end{figure}

\subsection{Heegaard diagram interpretation}\label{sec:ExamplesHDInterp}

By ideas resembling those found in \cite{EPV,Hypertoric1}, the strands (or tensor-product) bimodule $F^{\vee}$ over $\A_{1,\ldots,1}$ should also have a Heegaard diagram interpretation (the same should be true for $F$ but for simplicity we will work with DA bimodules when possible). Heegaard Floer bimodules based on holomorphic curve counts have not yet been defined for any general family of Heegaard diagrams that includes the diagrams in question, so we will just use these diagrams as heuristic motivation for the above definitions; despite this, we will speak as if the holomorphic-curve bimodules are well-defined and the below computations are accurate. We require sets of intersection points in the Heegaard diagram of Figure~\ref{fig:EDual2Strands} to contain exactly one point on the blue curve, and we do not count holomorphic curves with nonzero multiplicity on the left and right vertical portions of the boundary of the diagram.

\begin{figure}
\includegraphics[scale=0.4]{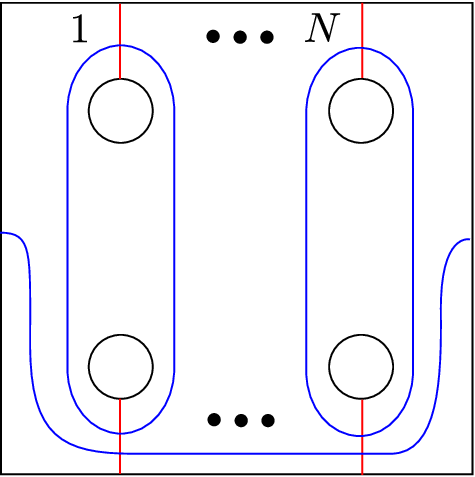}
\caption{Heegaard diagram for the bimodule $F^{\vee}$ over $\A_{1,\ldots,1}$.}
\label{fig:EDual2Strands}
\end{figure}

In particular, the DA bimodule $F^{\vee}$ over $\A_1$ comes from the Heegaard diagram in Figure~\ref{fig:EDual1Strand}. Figure~\ref{fig:EDual1StrandGen} shows the one generator $\xi_{\u,\o}$ of the bimodule as a set of intersection points in the Heegaard diagram from Figure~\ref{fig:EDual1Strand}. Figure~\ref{fig:EDual2StrandsGens} shows generators for the top summand ${_{2\varepsilon_1}}(F^{\vee})_{\varepsilon_1 + \varepsilon_2}$ of $F^{\vee}$ as sets of intersection points, together with the domain giving rise to the right action of $\lambda$. Figure~\ref{fig:EDual2StrandsGensLower} shows the generators for the bottom summand ${_{\varepsilon_1 + \varepsilon_2}}(F^{\vee})_{2\varepsilon_2}$ of $F^{\vee}$, and Figure~\ref{fig:EDual2StrandsDomainsLower} shows the domains for the secondary matrix of ${_{\varepsilon_1 + \varepsilon_2}}(F^{\vee})_{2\varepsilon_2}$ (the entries $U_i^{k+1} \otimes U_i^{k+1}$ correspond to domains of multiplicity $k+1$, which are shown with darker shading in Figure~\ref{fig:EDual2StrandsDomainsLower}). 

\begin{figure}
\includegraphics[scale=0.3]{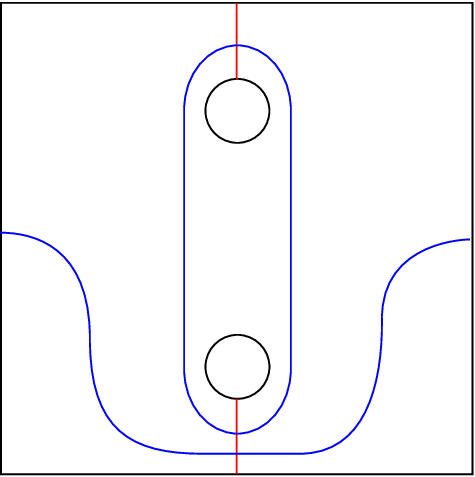}
\caption{Heegaard diagram for bimodule $F^{\vee}$ over $\A_1$.}
\label{fig:EDual1Strand}
\end{figure}

\begin{figure}
\includegraphics[scale=0.65]{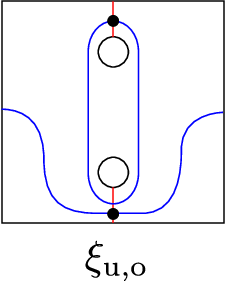}
\caption{Basis element $\xi_{\u,\o}$ of $F^{\vee}$ in terms of the Heegaard diagram.}
\label{fig:EDual1StrandGen}
\end{figure}

\begin{figure}
\includegraphics[scale=0.6]{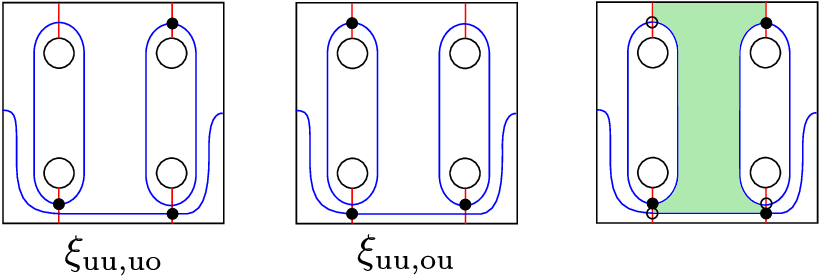}
\caption{Generators for ${_{2\varepsilon_1}}F^{\vee}_{\varepsilon_1 + \varepsilon_2}$ and domain for the right action of $\lambda$.}
\label{fig:EDual2StrandsGens}
\end{figure}

\begin{figure}
\includegraphics[scale=0.6]{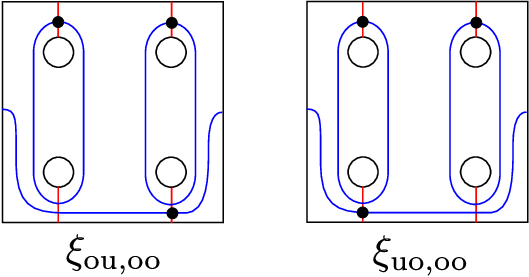}
\caption{Generators for ${_{\varepsilon_1 + \varepsilon_2}}F^{\vee}_{2\varepsilon_2}$.}
\label{fig:EDual2StrandsGensLower}
\end{figure}

\begin{figure}
\includegraphics[scale=0.6]{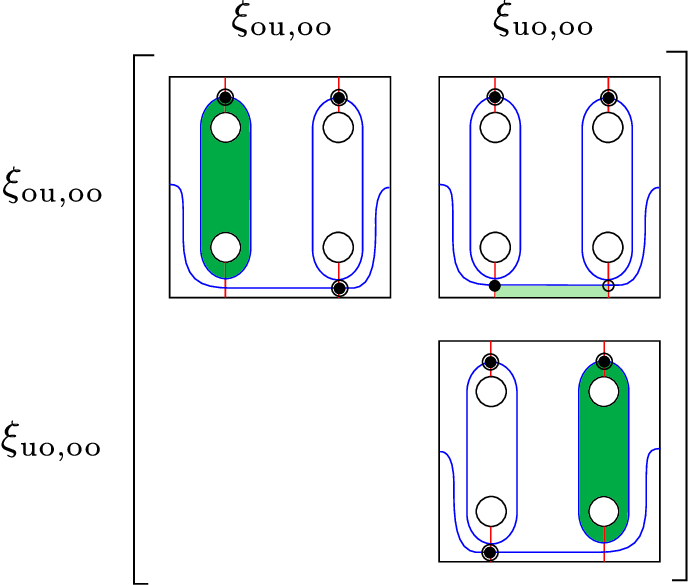}
\caption{Domains for ${_{\varepsilon_1 + \varepsilon_2}}F^{\vee}_{2\varepsilon_2}$.}
\label{fig:EDual2StrandsDomainsLower}
\end{figure}

\begin{convention}
We will always draw domains in the presence of a starting set of intersection points, drawn with solid dots, and an ending set of intersection points, drawn with open dots. Orientation conventions are fixed by requiring that in a simple bigon with positive multiplicity, a positively oriented path around the boundary starting and ending at the solid dot should traverse an $\alpha$ curve (red) before reaching the open dot and a $\beta$ curve afterward; this is the usual convention in Heegaard Floer homology.

\end{convention}

\begin{remark}
The Heegaard diagram of Figure~\ref{fig:EDual2Strands} should encode the right holomorphic curve counts for the bimodule $F^{\vee}$, but it does not seem to represent a 3-dimensional cobordism under the usual ways of building cobordisms from Heegaard diagrams in bordered Heegaard Floer homology. However, as shown in Figure~\ref{fig:ECobordisms}, there are at least two alternate Heegaard diagrams whose holomorphic curve counts (disallowing multiplicity on the left and right boundary) should give the same bimodule $F^{\vee}$. The diagram on the left does not differ from Figure~\ref{fig:EDual2Strands} in any domain that stays away from the left and right boundary. The diagram on the right is such that the bordered sutured cobordism on the left is obtained by a product decomposition (\cite[Definition 9.11]{Juhasz}) from the bordered sutured cobordism on the right. Diagrammatically, one can cut the diagram on the right along an interval stretching from the left boundary to the right boundary, disjoint from $\alpha$ and $\beta$ curves, to get the diagram on the left; the holomorphic geometry of curves away from the left and right boundary is unaffected.
\end{remark}

\begin{figure}
\includegraphics[scale=0.8]{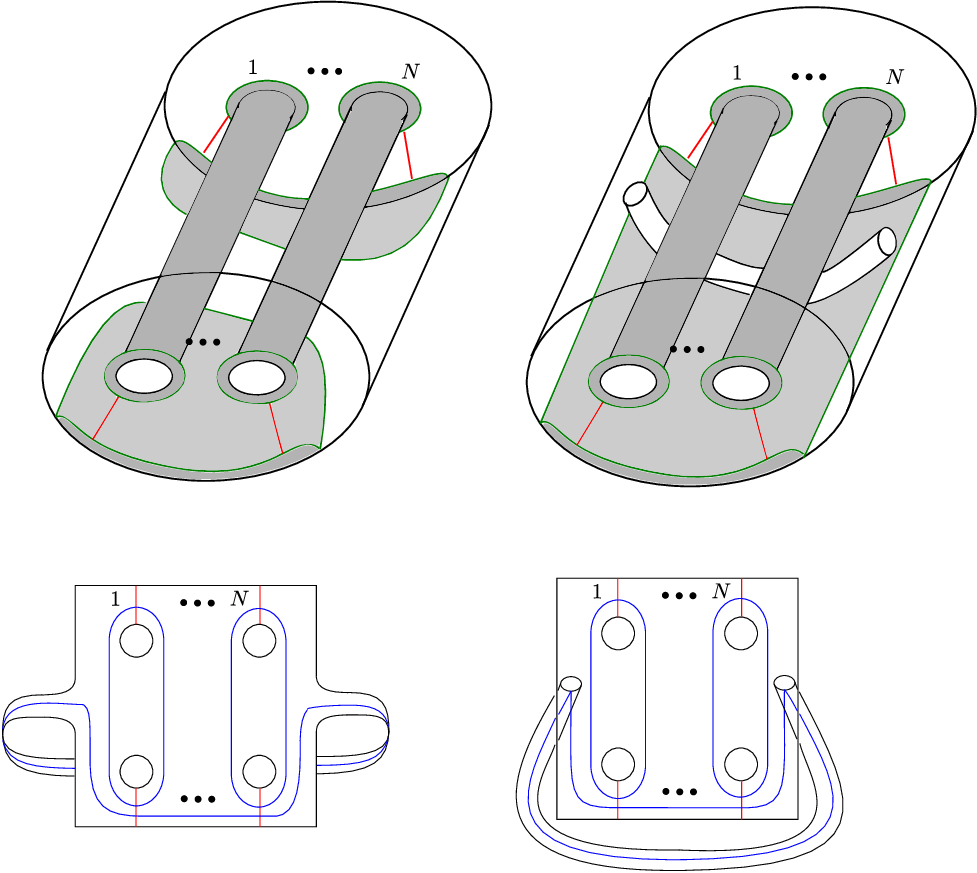}
\caption{Bordered sutured cobordisms represented by variants of the Heegaard diagram from Figure~\ref{fig:EDual2Strands}.}
\label{fig:ECobordisms}
\end{figure}

\section{Bimodule for the \texorpdfstring{$\lambda$}{lambda}-shaped trivalent vertex}\label{sec:EasyVertex}

\subsection{Definition of the bimodule}

\begin{figure}
\includegraphics[scale=0.4]{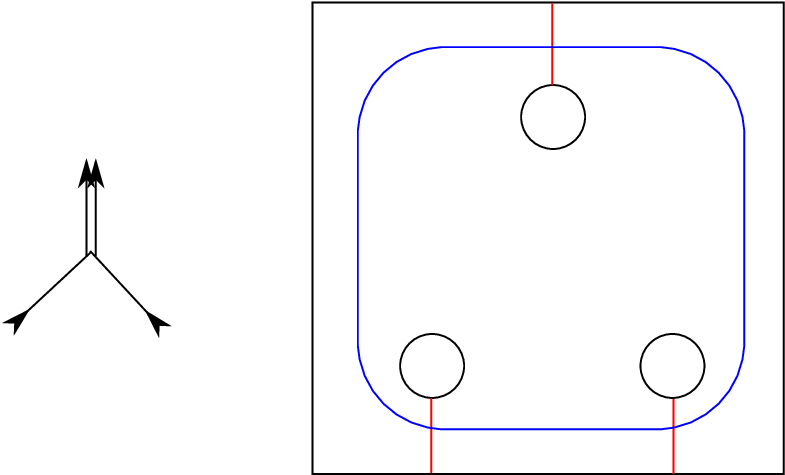}
\caption{Heegaard diagram for the ``$\lambda$-shaped'' trivalent vertex}
\label{fig:EasyVertex}
\end{figure}

In this section we will define a bimodule for one of the two types of trivalent vertex, based on the Heegaard diagram in Figure~\ref{fig:EasyVertex}.

At the decategorified level, one can view maps for trivalent vertices as arising from skew Howe duality. The morphism $1_{2,0} E_{\gl(2)} 1_{1,1}$ of the idempotented quantum group $\Udot_q(\gl(2))$ induces a $U_q(\gl(1|1))$-linear map $V^{\otimes 2} \to \wedge_q^2 V$; see Appendix~\ref{sec:AppendixSkewHoweMaps}. The morphism $1_{0,2}F_{\gl(2)}1_{1,1}$ of $\Udot(\gl(2))$ induces the same map $V^{\otimes 2} \to \wedge_q^2 V$. This map respects the $\gl(1|1)$ weight spaces of its domain and codomain; we will refer to it simply as $1_{2,0} E_{\gl(2)} 1_{1,1}$ (or $1_{0,2}F_{\gl(2)}1_{1,1}$). 

The nonzero $\gl(1|1)$ weight spaces of $\wedge_q^2 V$ are $\varepsilon_1 + \varepsilon_2$ and $2\varepsilon_2$, while the nonzero weight spaces of $V^{\otimes 2}$ are $2\varepsilon_1$, $\varepsilon_1 + \varepsilon_2$, and $2\varepsilon_2$. We will refer to these weight spaces as the ``upper,'' ``middle,'' and ``lower'' weight spaces respectively. The codomain $\wedge_q^2 V$ of the map $1_{2,0} E_{\gl(2)} 1_{1,1}$ is zero in the upper weight space, so $1_{2,0} E_{\gl(2)} 1_{1,1}$ is the direct sum of maps on the middle and lower weight spaces. 

We will define DA bimodules $\Lambda_{\midd}$ and $\Lambda_{\low}$ whose left duals categorify $1_{2,0} E_{\gl(2)} 1_{1,1}$ acting on these middle and lower weight spaces. The bimodules are obtained by heuristically counting holomorphic disks using the Heegaard diagram shown in Figure~\ref{fig:EasyVertex}; conjecturally, these disk counts are accurate once the analytic theory has been defined.

Neither of these bimodules is $A_{\infty}$; the bimodule in the middle weight space has a differential, while in the lower weight space it is an ordinary bimodule. We will set $\Lambda \coloneqq \Lambda_{\midd} \oplus \Lambda_{\low}$. We define $\Lambda_{\upp}$ to be zero; more generally, we set $\Lambda_{a\varepsilon_1 + b\varepsilon_2} = 0$ unless $(a,b)$ equals $(1,1)$ or $(0,2)$.

\begin{definition}
The dg bimodule $\Lambda_{\midd}$ over $((\A_{1,1})_{\varepsilon_1+\varepsilon_2}, (\A_2)_{\varepsilon_1+\varepsilon_2})$ has primary matrix
\[
\kbordermatrix{ 
& \u \\
\ou & X \\
\uo & Y
}.
\]
We set $\deg^q(X) = -1$, $\deg^h(X) = 1$, $\deg^q(Y) = 0$, and $\deg^h(Y) = 0$. The differential is given by the matrix
\[
\kbordermatrix{
& X & Y \\
X & 0 & \lambda \\
Y & 0 & 0
},
\]
and the right action of $e_1$ is given by the matrix
\[
\kbordermatrix{
& X & Y \\
X & U_1 & 0 \\
Y & 0 & U_2
}.
\]
Since the above matrices commute and are compatible with the gradings, this dg bimodule is well-defined.
\end{definition}

\begin{definition}
The (ordinary) bimodule $\Lambda_{\low}$ over $((\A_{1,1})_{2\varepsilon_2}, (\A_2)_{2\varepsilon_2})$ has primary matrix
\[
\kbordermatrix{
& \o \\
\oo & Z
}.
\]
We set $\deg^q(Z) = \deg^h(Z) = 0$.

The matrix for right action of $e_1$ is
\[
\kbordermatrix{
& Z \\
Z & U_1 + U_2
}
\]
and the matrix for right action of $e_2$ is
\[
\kbordermatrix{
& Z \\
Z & U_1 U_2
}.
\]
The matrices for the actions of $e_1$ and $e_2$ commute and are compatible with the gradings, so the bimodule is well-defined.
\end{definition}

\begin{figure}
\includegraphics[scale=0.6]{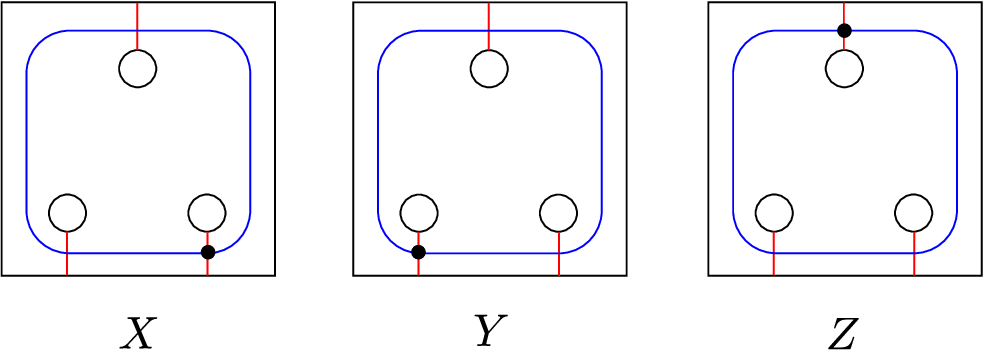}
\caption{Generators of $\Lambda$ in terms of intersection points.}
\label{fig:EasyVertexGens}
\end{figure}

\begin{figure}
\includegraphics[scale=0.6]{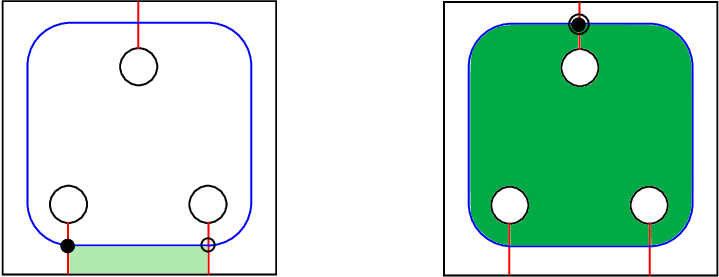}
\caption{Domains giving rise to the differential and right action of $e_2^k$ on $\Lambda$.}
\label{fig:EasyVertexDomains}
\end{figure}

Figure~\ref{fig:EasyVertexGens} shows the correspondence between generators of $\Lambda$ and sets of intersection points in the Heegaard diagram of Figure~\ref{fig:EasyVertex}. Figure~\ref{fig:EasyVertexDomains} shows the domains in the Heegaard diagram corresponding to the differential on $\Lambda_{\midd}$ and the right action of $e_2$ on $\Lambda_{\low}$. The given formula for the right action of $e_1$ is necessary for the theory to work, but it does not appear to come from the Heegaard diagram, at least in currently-known versions of Heegaard Floer homology.

\subsection{Decategorification of \texorpdfstring{$\Lambda$}{Lambda}}

\begin{proposition}
The DA bimodule $\Lambda$ categorifies the map from $K_0(\A_2)$ to $K_0(\A_{1,1})$ with matrix
\[
\kbordermatrix{
& [P_{\u}] & [P_{\o}] & \\
{[P_{\uu}]} & 0 & 0  \\
{[P_{\ou}]} & -q^{-1} & 0 \\
{[P_{\uo}]} & 1 & 0 \\
{[P_{\oo}]} & 0 & 1 \\
}.
\]
Equivalently, the AD bimodule ${^{\vee}}\Lambda$ categorifies the map from $G_0(\A_{1,1})$ to $G_0(\A_2)$ with matrix
\[
\kbordermatrix{
& {[S_{\uu}]} & {[S_{\ou}]} & {[S_{\uo}]} & {[S_{\oo}]} \\
{[S_{\u}]} & 0 & -q & 1 & 0 \\
{[S_{\o}]} & 0 & 0 & 0 & 1
}.
\]
This latter map can be identified with $1_{2,0} E_{\gl(2)} 1_{1,1}$ (or equivalently $1_{0,2} F_{\gl(2)} 1_{1,1}$) mapping from $V^{\otimes 2}$ to $\wedge_q^2 V$ as in Appendix~\ref{sec:AppendixSkewHoweMaps}.
\end{proposition}

\subsection{2-representation morphism structure}\label{sec:Lambda2RepMorphism}

We equip $\Lambda = \Lambda_{\midd} \oplus \Lambda_{\low}$ with the structure of a DA bimodule $1$-morphism of $2$-representations of $\U^-$. We need a homotopy equivalence
\[
\varphi\co \Lambda_{\midd} \boxtimes F^{\vee} \to F^{\vee} \boxtimes \Lambda_{\low}
\]
as well as a homotopy equivalence
\[
\varphi\co 0 \to F^{\vee} \boxtimes \Lambda_{\midd}
\]
(i.e. we need $F^{\vee} \boxtimes \Lambda_{\midd}$ to be contractible).

The following three propositions follow from Procedure~\ref{proc:BoxTensorBimods}.

\begin{proposition}\label{prop:EDualAfterMiddleEasy}

The dg bimodule $F^{\vee} \boxtimes \Lambda_{\midd}$ has primary matrix
\[
\kbordermatrix{
& \u \\
\uu & \xi_{\uu,\ou} X \quad \xi_{\uu,\uo} Y
}
\]
and differential
\[
\kbordermatrix{
& \xi_{\uu,\ou} X & \xi_{\uu,\uo} Y \\
\xi_{\uu,\ou} X & 0 & 1 \\
\xi_{\uu,\uo} Y & 0 & 0 
}.
\]
The right action of $e_1$ on $F^{\vee} \boxtimes \Lambda_{\midd}$ is zero.
\end{proposition}

\begin{proposition}\label{prop:EDualAfterLowerEasy}
The dg bimodule $F^{\vee} \boxtimes \Lambda_{\low}$ has primary matrix
\[
\kbordermatrix{
& \o \\
\ou & \xi_{\ou,\oo} Z \\
\uo & \xi_{\uo,\oo} Z
}
\]
and differential
\[
\kbordermatrix{
& \xi_{\ou,\oo} Z & \xi_{\uo,\oo} Z \\
\xi_{\ou,\oo} Z & 0 & \lambda \\
\xi_{\uo,\oo} Z & 0 & 0
}.
\]
The matrix for right action of $e_1$ is
\[
\kbordermatrix{
& \xi_{\ou,\oo} Z & \xi_{\uo,\oo} Z \\
\xi_{\ou,\oo} Z & U_1 & 0 \\
\xi_{\uo,\oo} Z & 0 & U_2
};
\]
the matrix for right action of $e_2$ is zero.

\end{proposition}

\begin{proposition}\label{prop:EDualBeforeMiddleEasy}
The dg bimodule $\Lambda_{\midd} \boxtimes F^{\vee}$ has primary matrix
\[
\kbordermatrix{
& \o \\
\ou & X \xi_{\u,\o} \\
\uo & Y \xi_{\u,\o}
}
\]
and differential
\[
\kbordermatrix{
& X \xi_{\u,\o} & Y \xi_{\u,\o} \\
X \xi_{\u,\o} & 0 & \lambda \\
Y \xi_{\u,\o} & 0 & 0
}.
\]
The matrix for the right action of $e_1$ is
\[
\kbordermatrix{
& X \xi_{\u,\o} & Y \xi_{\u,\o} \\
X \xi_{\u,\o} & U_1 & 0 \\
Y \xi_{\u,\o} & 0 & U_2
};
\]
the matrix for the right action of $e_2$ is zero.
\end{proposition}

\begin{definition}\label{def:YEasy2RepAlpha}
We will write $\Lambda$ for the 1-morphism between 2-representations of $\U^-$ given by $(\Lambda, \varphi)$, where 
\[
\varphi\co \Lambda \boxtimes F^{\vee} \to F^{\vee} \boxtimes \Lambda
\]
is zero as a map from $0$ to $F^{\vee} \boxtimes \Lambda_{\midd}$ and is given by the matrix
\[
\kbordermatrix{
& X \xi_{\u,\o} & Y \xi_{\u,\o} \\
\xi_{\ou,\oo} Z & 1 & 0 \\
\xi_{\uo,\oo} Z & 0 & 1
}
\]
as a map from $\Lambda_{\midd} \boxtimes F^{\vee}$ to $F^{\vee} \boxtimes \Lambda_{\low}$. By Propositions~\ref{prop:EDualAfterLowerEasy} and \ref{prop:EDualBeforeMiddleEasy}, $\varphi$ is a closed morphism of dg bimodules. Note that $\varphi$ is not an isomorphism, but it is a homotopy equivalence by Proposition~\ref{prop:EDualAfterMiddleEasy}.

\end{definition}

Because the square
\[
\xymatrix{ 
0 = \Lambda_{\upp} \boxtimes (F^{\vee})^{\boxtimes 2} \ar[rr]^{\varphi \boxtimes \id_{F^{\vee}}} \ar[d]_{\id_{\Lambda_{\upp}} \boxtimes \tau} && F^{\vee} \boxtimes \Lambda_{\midd} \boxtimes F^{\vee} \ar[rr]^{\id_{F^{\vee}} \boxtimes \varphi} && (F^{\vee})^{\boxtimes 2} \boxtimes \Lambda_{\low} \ar[d]^{\tau \boxtimes \id_{\Lambda_{\low}}} \\
0 = \Lambda_{\upp} \boxtimes (F^{\vee})^{\boxtimes 2} \ar[rr]^{\varphi \boxtimes \id_{F^{\vee}}} && F^{\vee} \boxtimes \Lambda_{\midd} \boxtimes F^{\vee} \ar[rr]^{\id_{F^{\vee}} \boxtimes \varphi} && (F^{\vee})^{\boxtimes 2} \boxtimes \Lambda_{\low}
}
\]
automatically commutes ($\Lambda_{\upp} = 0$), $(\Lambda,\varphi)$ is a valid $1$-morphism.

\section{Bimodule for the Y-shaped trivalent vertex}\label{sec:HardVertex}

Now we define a bimodule for the other type of trivalent vertex, based on the Heegaard diagram in Figure~\ref{fig:HardVertex} and categorifying the map $1_{1,1} F_{\gl(2)} 1_{2,0} = 1_{1,1} E_{\gl(2)} 1_{0,2}$ from $\wedge_q^2 V$ to $V^{\otimes 2}$ arising from skew Howe duality (see Appendix~\ref{sec:AppendixSkewHoweMaps}). As in Section~\ref{sec:EasyVertex}, the bimodule is defined to be zero in the upper weight space $2\varepsilon_1$ (see the beginning of Section~\ref{sec:EasyVertex}); we describe the middle ($\varepsilon_1 + \varepsilon_2$) and bottom ($2\varepsilon_2$) weight spaces below.
\begin{figure}
\includegraphics[scale=0.4]{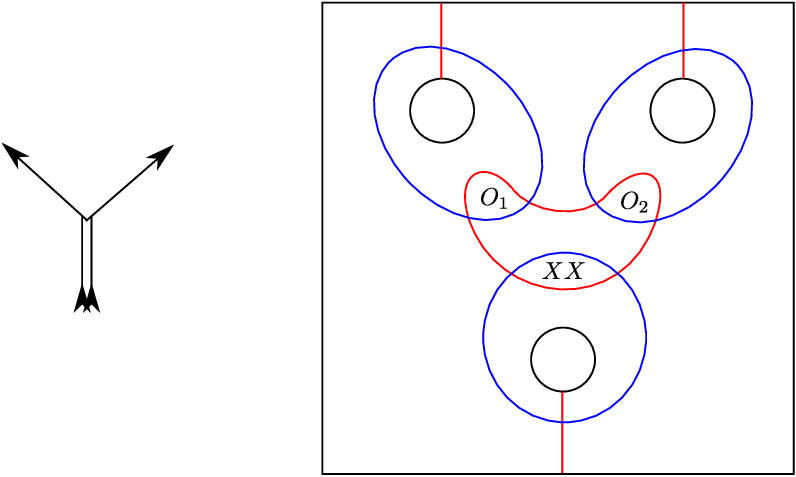}
\caption{Heegaard diagram for the ``Y-shaped'' trivalent vertex}
\label{fig:HardVertex}
\end{figure}

\subsection{Middle weight space, fully unsimplified version}

We first define a subsidiary DA bimodule $Y'_{\midd}$, after introducing a useful convention for secondary matrices.

\begin{convention}
Any unspecified indices appearing in exponents of secondary matrix entries, for example the indices $k$ and $l$ below, are assumed to range over all nonnegative values. For example, the top-left entry of the secondary matrix in Definition~\ref{def:YPrimeHard} is the infinite sum $w_1 \otimes U_1 + w_1^2 \otimes U_1^2 + \cdots$. We start the indexing at $k+1$ rather than $k$ because we do not include $\id \otimes \id$ terms in secondary matrices. Sometimes, to simplify notation, we will index a secondary matrix entry in such a way that $\id \otimes \id$ is a term, but such $\id \otimes \id$ terms should be implicitly omitted.
\end{convention}

\begin{definition}\label{def:YPrimeHard} The DA bimodule $Y'_{\midd}$ over 
\[
((\A_2)_{\varepsilon_1 + \varepsilon_2} \otimes \F_2[w_1,w_2], (\A_{1,1})_{\varepsilon_1 + \varepsilon_2} \otimes \F_2[w_1, w_2])
\]
has primary matrix
\[
\kbordermatrix{ 
& \ou & \uo \\ 
\u & A \quad A' & B \quad B'
}
\]
and secondary matrix
\[
\kbordermatrix{
& A & A' & B & B' \\
A & w_1^{k+1} \otimes U_1^{k+1} & 0 & w_2^k \otimes (U_2^{k+1},\lambda) & w_1^l w_2^k \otimes (U_2^{k+1}, \lambda, U_1^{l+1}) \\
A' & w_2 & w_1^{k+1} \otimes U_1^{k+1} & 1 \otimes \lambda & w_1^k \otimes (\lambda, U_1^{k+1}) \\
B & 0 & 0 & w_2^{k+1} \otimes U_2^{k+1} & 0 \\
B' & 0 & 0 & w_1 & w_2^{k+1} \otimes U_2^{k+1}
}.
\]
All generators are also $w_1$- and $w_2$-equivariant in the sense that all entries of the above matrix representing $\delta^1_i$ terms for $i > 1$ should be multiplied by $w_1^p w_2^q$ on both sides (in each entry, for higher terms, with the output exponents given by the sum of the input exponents on both $w_1$ and $w_2$). Equivalently, we could treat $\F_2[w_1,w_2]$ rather than $\F_2$ as the ground ring, in which case we do not require these entries.

We define gradings as follows:
\begin{itemize}
\item $\deg^q(A) = -2$, $\deg^h(A) = 2$
\item $\deg^q(A') = 0$, $\deg^h(A') = 1$
\item $\deg^q(B) = -1$, $\deg^h(B) = 1$
\item $\deg^q(B') = 1$, $\deg^h(B') = 0$
\item $\deg^q(w_1) = -2$, $\deg^h(w_1) = 2$
\item $\deg^q(w_2) = -2$ $\deg^h(w_2) = 2$.
\end{itemize}

\end{definition}

\begin{figure}
\includegraphics[scale=0.6]{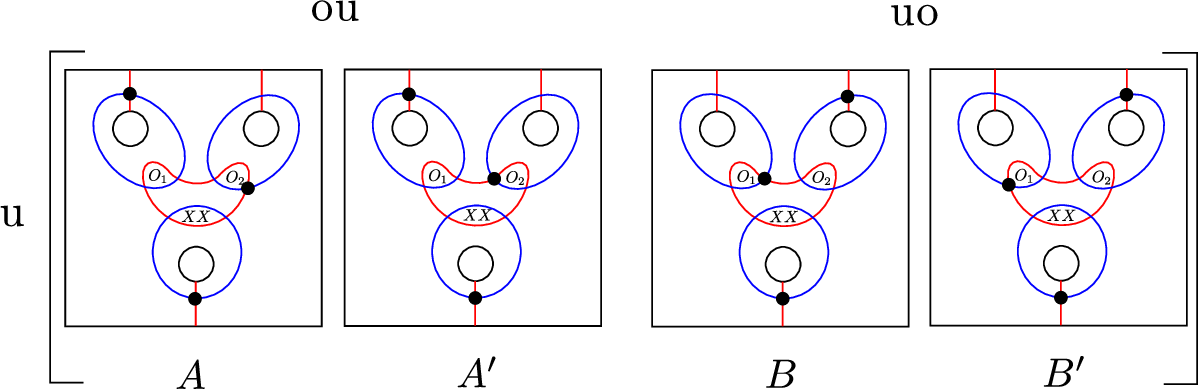}
\caption{Generators of $Y'_{\midd}$ in terms of intersection points}
\label{fig:HardVertexGens}
\end{figure}

The correspondence between sets of intersection points in the Heegaard diagram of Figure~\ref{fig:HardVertex} and generators of $Y'_{\midd}$ (primary matrix entries) is shown in Figure~\ref{fig:HardVertexGens}. Not all sets of intersection points in the diagram appear in Figure~\ref{fig:HardVertexGens}; the rest contribute to the summand $Y'_{\low}$ of $Y'$ in the lower weight space and appear in Figure~\ref{fig:HardVertexGensLower}. The domains giving rise to secondary matrix entries are shown in Figure~\ref{fig:HardVertexDomains}. 

\begin{figure}
\includegraphics[scale=0.6]{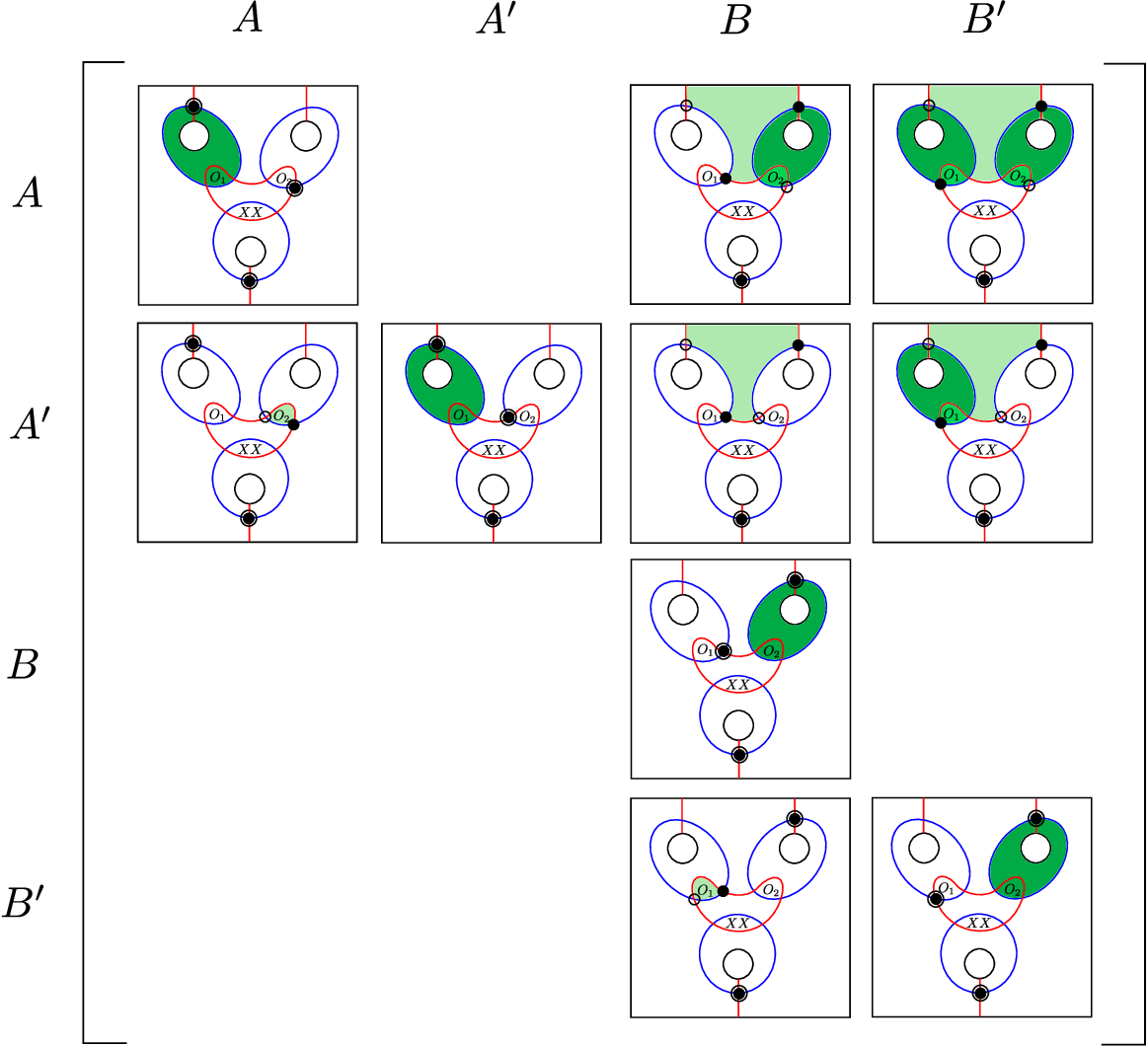}
\caption{Domains giving rise to the secondary matrix for $Y'_{\midd}$.}
\label{fig:HardVertexDomains}
\end{figure}

\begin{proposition}\label{prop:HardFullyUnsimplifiedMiddleWellDefined}
The DA bimodule $Y'_{\midd}$ is well-defined.
\end{proposition}

\begin{proof}
Treating $\F_2[w_1,w_2]$ as the ground ring, the squared secondary matrix is
\[
\resizebox{\textwidth}{!}{
\kbordermatrix{
& A & A' & B & B' \\
A & w_1^{k+1} \otimes U_1^{k+1} & 0 & w_2^k \otimes (U_2^{k+1},\lambda) & w_1^l w_2^k \otimes (U_2^{k+1}, \lambda, U_1^{l+1}) \\
A' & w_2 & w_1^{k+1} \otimes U_1^{k+1} & 1 \otimes \lambda & w_1^k \otimes (\lambda, U_1^{k+1}) \\
B & 0 & 0 & w_2^{k+1} \otimes U_2^{k+1} & 0 \\
B' & 0 & 0 & w_1 & w_2^{k+1} \otimes U_2^{k+1}
}
\kbordermatrix{
& A & A' & B & B' \\
A & w_1^{k+1} \otimes U_1^{k+1} & 0 & w_2^k \otimes (U_2^{k+1},\lambda) & w_1^l w_2^k \otimes (U_2^{k+1}, \lambda, U_1^{l+1}) \\
A' & w_2 & w_1^{k+1} \otimes U_1^{k+1} & 1 \otimes \lambda & w_1^k \otimes (\lambda, U_1^{k+1}) \\
B & 0 & 0 & w_2^{k+1} \otimes U_2^{k+1} & 0 \\
B' & 0 & 0 & w_1 & w_2^{k+1} \otimes U_2^{k+1}
}
}
\]
which equals
\[
\resizebox{\textwidth}{!}{
\kbordermatrix{
& A & A' & B & B' \\
A & w_1^{k+p+2} \otimes (U_1^{k+1}, U_1^{p+1}) & 0 & w_2^{k+p+1} \otimes (U_2^{k+1},U_2^{p+1},\lambda) &  \begin{matrix} w_1^{k+p+1} w_2^q \otimes (U_2^{q+1}, \lambda, U_1^{k+1}, U_1^{p+1}) \\ + w_1^k w_2^{l+p+1} \otimes (U_2^{l+1}, U_2^{p+1}, \lambda, U_1^{k+1}) \end{matrix} \\
A' & 0 & w_1^{k+p+2} \otimes (U_1^{k+1}, U_1^{p+1}) & 0 & w_1^{k+p+1} \otimes (\lambda, U_1^{k+1}, U_1^{p+1}) \\
B & 0 & 0 & w_2^{k+p+2} \otimes (U_2^{k+1},U_2^{p+1}) & 0 \\
B' & 0 & 0 & 0 & w_2^{k+p+2} \otimes (U_2^{k+1},U_2^{p+1})
}.
}
\]
This matrix is also the matrix of multiplication terms for $Y'_{\midd}$, so $Y'_{\midd}$ satisfies the DA bimodule structure relations.
\end{proof}

From $Y'_{\midd}$, we can get an infinitely generated DA bimodule over $((\A_2)_{\varepsilon_1 + \varepsilon_2}, (\A_{1,1})_{\varepsilon_1 + \varepsilon_2})$ by restricting both the left and right actions on $Y'_{\midd}$ via the inclusions of $\A_2$ and $\A_{1,1}$ into $\A_2 \otimes \F_2[w_1,w_2]$ and $\A_{1,1} \otimes \F_2[w_1, w_2]$. The primary matrix of the result is
\[
\kbordermatrix{ 
& \ou & \uo \\ 
\u & w_1^k w_2^l A \quad w_1^k w_2^l A' & w_1^k w_2^l B \quad w_1^k w_2^l B'
}
\]
(letting $k$ and $l$ range over all nonnegative integers). The secondary matrix is
\[
\resizebox{\textwidth}{!}{
\kbordermatrix{
& w_1^k w_2^l A & w_1^k w_2^l A' & w_1^k w_2^l B & w_1^k w_2^l B' \\
A & w_1^{k+p+1} w_2^l \otimes U_1^{p+1} & 0 &  w_1^k w_2^{l+p} \otimes (U_2^{p+1},\lambda) & w_1^{k+q} w_2^{l+p} \otimes (U_2^{p+1}, \lambda, U_1^{q+1}) \\
A' & w_1^k w_2^{l+1} & w_1^{k+p+1} w_2^l \otimes U_1^{p+1} & w_1^k w_2^l \otimes \lambda & w_1^{k+p} w_2^l \otimes (\lambda, U_1^{p+1}) \\
B & 0 & 0 & w_1^k w_2^{l+p+1} \otimes U_2^{p+1} &  0 \\
B' & 0 & 0 & w_1^{k+1} w_2^l & w_1^k w_2^{l+p+1} \otimes U_2^{p+1}
};
}
\]
multiplication by $w_1$ or $w_2$ on the ``output side'' of a secondary matrix entry should not be viewed as multiplication by an algebra element, but rather as specifying in which row of the (infinite) secondary matrix the entry belongs.

Motivated by the two-term complexes appearing in \cite[proof of Theorem 4.1]{OSzCube} (see also e.g. \cite[Theorem 2.3]{ManolescuCube}), the DA bimodule $Y_{\fu,\midd}$ is defined to be 
\[
\xymatrix{
q^{-2} Y'_{\midd}[-1] \ar@/^1.5pc/[rr]^{\Xi} & \oplus & Y'_{\midd}
}
\]
where $\Xi$ is the DA bimodule endomorphism of $Y'_{\midd}$ given by
\[
\kbordermatrix{
& A & A' & B &  B' \\
A & w_1 + w_2 + e_1 & 0 & 0 & 0 \\
A' & 0 & w_1 + w_2 + e_1 & 0 & 0 \\
B & 0 & 0 & w_1 + w_2 + e_1 & 0 \\
B' & 0 & 0 & 0 & w_1 + w_2 + e_1
}.
\]

\begin{proposition}
The DA bimodule endomorphism $\Xi$ is closed.
\end{proposition}

\begin{proof}

The product of the matrix for the endomorphism with the secondary matrix for $Y'_{\midd}$, in either order, equals the secondary matrix for $Y'_{\midd}$ with the output of each entry multiplied by $w_1 + w_2 + e_1$.

\end{proof}

As with $Y'_{\midd}$, we can view $Y_{\fu,\midd}$ as an infinitely generated DA bimodule over 
\[
((\A_2)_{\varepsilon_1+\varepsilon_2}, (\A_{1,1})_{\varepsilon_1 + \varepsilon_2}).
\]
We will write the generators of the two summands of $Y_{\fu,\midd}$ as 
\[
\{A_1, A'_1, B_1, B'_1\}, \quad \{A_2, A'_2, B_2, B'_2\}
\]
respectively, where $\Xi$ maps the ``$2$'' summand to the ``$1$'' summand.

\begin{proposition}\label{prop:HardFullyUnsimplifiedMiddleHomology}
As a left dg module over $(\A_2)_{\varepsilon_1 + \varepsilon_2} = \F_2[e_1]$, the homology of $Y_{\fu,\midd}$ is a free module of rank two generated by $A'_1 = w_1^0 w_2^0 A'_1$ and $B'_1$.
\end{proposition}

\begin{proof}
As a non-differential left module over $\F_2[e_1]$, $Y_{\fu,\midd}$ is free with one generator for each pair of a monomial $w_1^k w_2^l$ ($k,l \geq 0$) and an element of $\{A_i,A'_i,B_i,B'_i : i = 1,2\}$. The kernel of the differential on $Y_{\fu,\midd}$ consists of the summands corresponding to $w_1^k w_2^l A'_1$ and $w_1^k w_2^l B'_1$. In the homology of $Y_{\fu,\midd}$, we have $w_2 = w_1 + e_1$ (equivalently, $w_1 = w_2 + e_1$), so the homology is generated by elements $w_1^k A'_1$ and $w_2^l B'$. Furthermore, $w_1^{k+1} A'_1 = e_1 w_1^k A'_1$ and $w_2^{l+1} B'_1 = e_1 w_2^l B'_1$ modulo the image of the differential, so the homology is generated by the two elements $A'_1$ and $B'_1$. The image of the differential intersected with the $\F_2[e_1]$-submodule generated by $A'_1$ and $B'_1$ is trivial, proving the proposition.
\end{proof}

\subsection{Middle weight space, half-simplified version}

We now simplify $Y_{\fu,\midd}$, resulting in a half-simplified version $Y_{\hs,\midd}$. The equivariance over $\F_2[w_1,w_2]$ will go away, and we will treat $Y_{\hs,\midd}$ as an infinitely generated DA bimodule over $((\A_2)_{\varepsilon_1 + \varepsilon_2}, (\A_{1,1})_{\varepsilon_1 + \varepsilon_2})$. We also streamline notation by relabeling $A'_1, B'_1$ as $A',B'$.

\begin{definition}
The DA bimodule $Y_{\hs,\midd}$ has primary matrix
\[
\kbordermatrix{ 
& \ou & \uo \\ 
\u & w_1^k A \quad w_1^k A' & w_2^l B \quad w_2^l B'
}
\]
(where $k$ ranges over nonnegative integers) and secondary matrix
\[
\resizebox{\textwidth}{!}{
\kbordermatrix{
& w_1^k A & w_1^k A' & w_2^l B & w_2^l B' \\
A & w_1^{k+p+1} \otimes U_1^{p+1} & 0 & (w_1 + e_1)^{l+p} \otimes (U_2^{p+1},\lambda) & w_1^q (w_1 + e_1)^{l+p} \otimes (U_2^{p+1}, \lambda, U_1^{q+1}) \\
A' & w_1^k(w_1 + e_1) & w_1^{k+p+1} \otimes U_1^{p+1} & (w_1 + e_1)^l \otimes \lambda & w_1^p (w_1 + e_1)^l \otimes (\lambda, U_1^{p+1}) \\
B & 0 & 0 & w_2^{l+p+1} \otimes U_2^{p+1} & 0 \\
B' & 0 & 0 & w_2^l(w_2 + e_1) & w_2^{l+p+1} \otimes U_2^{p+1}
}.
}
\]
The $q$- and $h$-degrees of the generators are the same as they were in $Y_{\fu,\midd}$.
\end{definition}

One can check that $Y_{\hs,\midd}$ can be obtained from $Y_{\fu,\midd}$ by the simplification procedure from Section~\ref{sec:PrelimSimplifying}, so $Y_{\hs,\midd}$ is well-defined and homotopy equivalent to $Y_{\fu,\midd}$. We get the following corollary from Proposition~\ref{prop:HardFullyUnsimplifiedMiddleHomology}.

\begin{corollary}\label{prop:HardHalfSimplifiedMiddleHomology}
As a left dg module over $\F_2[e_1]$, the homology of $Y_{\hs,\midd}$ is a free module of rank two generated by $A' = w_1^0 A'$ and $B' = w_2^0 B'$.
\end{corollary}

\subsection{Middle weight space, fully simplified version}\label{sec:HardVertexMiddleWSSimpl}

We now simplify the bimodule $Y_{\hs,\midd}$ in the middle weight space even further, giving a finitely generated DA bimodule $Y_{\fs,\midd}$ over $((\A_2)_{\varepsilon_1 + \varepsilon_2}, (\A_{1,1})_{\varepsilon_1 + \varepsilon_2})$. The primary and secondary matrices of the result $Y_{\fs,\midd}$ are described below. One can check that these matrices arise from the simplification procedure of Section~\ref{sec:PrelimSimplifying}; note that we are canceling infinitely many disjoint pairs $(w_1^k A, w_1^{k+1} A')$ and $(w_2^l B, w_2^{l+1} B')$.

\begin{definition}
The DA bimodule $Y_{\fs,\midd}$ has primary matrix
\[
\kbordermatrix{ 
& \ou & \uo \\ 
\u & A' & B'
}
\]
and secondary matrix
\[
\kbordermatrix{
& A' & B' \\
A' & e_1^{k+1} \otimes U_1^{k+1} & \begin{matrix} e_1^k \otimes (\lambda, U_1^{k+1}) \\+ e_1^k \otimes (U_2^{k+1},\lambda)\end{matrix} \\ 
B' & 0 & e_1^{k+1} \otimes U_2^{k+1}
}.
\]
As before, we have $\deg^q(A') = 0$, $\deg^h(A') = 1$, $\deg^q(B') = 1$, and $\deg^h(B') = 0$.
\end{definition}

\subsection{Lower weight space}
Next we define the bimodule $Y'_{\low}$ in the lower weight space. Unlike with $Y'_{\midd}$ in the middle weight space, $Y'_{\low}$ will be an ordinary dg bimodule with no higher $A_{\infty}$ actions.

\begin{definition}
The DA bimodule $Y'_{\low}$ over 
\[
((\A_2)_{2\varepsilon_2} \otimes \F_2[w_1,w_2], (\A_{1,1})_{2\varepsilon_2} \otimes \F_2[w_1,w_2])
\]
has primary matrix
\[
\kbordermatrix{
& \oo \\
\o & C \quad C'
}.
\]
We set
\begin{itemize}
\item $\deg^q(C) = -3$, $\deg^h(C) = 3$,
\item $\deg^q(C') = 1$, $\deg^h(C') = 0$.
\end{itemize}
The secondary matrix is
\[
\kbordermatrix{
& C & C' \\
C & w_1^k w_2^l \otimes U_1^k U_2^l & 0 \\ 
C' & w_1 w_2 + e_2 & w_1^k w_2^l \otimes U_1^k U_2^l
}
\]
(recall that we implicitly exclude $\id \otimes \id$ if it appears in a secondary matrix entry). Equivalently, the differential has matrix
\[
\kbordermatrix{
& C & C' \\
C & 0 & 0 \\ 
C' & w_1 w_2 + e_2 & 0
}
\]
and right multiplication by $U_i$ has matrix
\[
\kbordermatrix{
& C & C' \\
C & w_i & 0 \\ 
C' & 0 & w_i
}
\]
for $i \in \{1,2\}$. The matrix entries maps are assumed to be $w_1$ and $w_2$-equivariant as in Definition~\ref{def:YPrimeHard}; equivalently, right multiplication by $w_i$ has the same matrix as right multiplication by $U_i$.
\end{definition}

\begin{proposition}
The DA bimodule $Y'_{\low}$ is well-defined.
\end{proposition}

\begin{proof}
The above matrices for the differential and the action of algebra generators commute with each other. One can check that these matrices together amount to the same structure as the given secondary matrix.
\end{proof}

\begin{figure}
\includegraphics[scale=0.6]{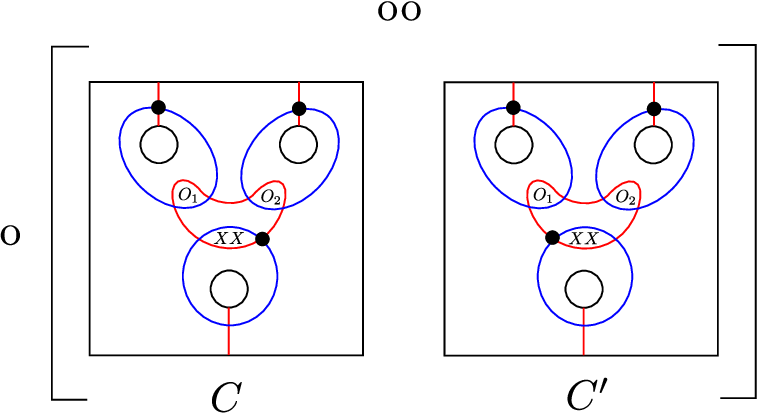}
\caption{Generators of $Y'_{\low}$ in terms of intersection points.}
\label{fig:HardVertexGensLower}
\end{figure}

\begin{figure}
\includegraphics[scale=0.6]{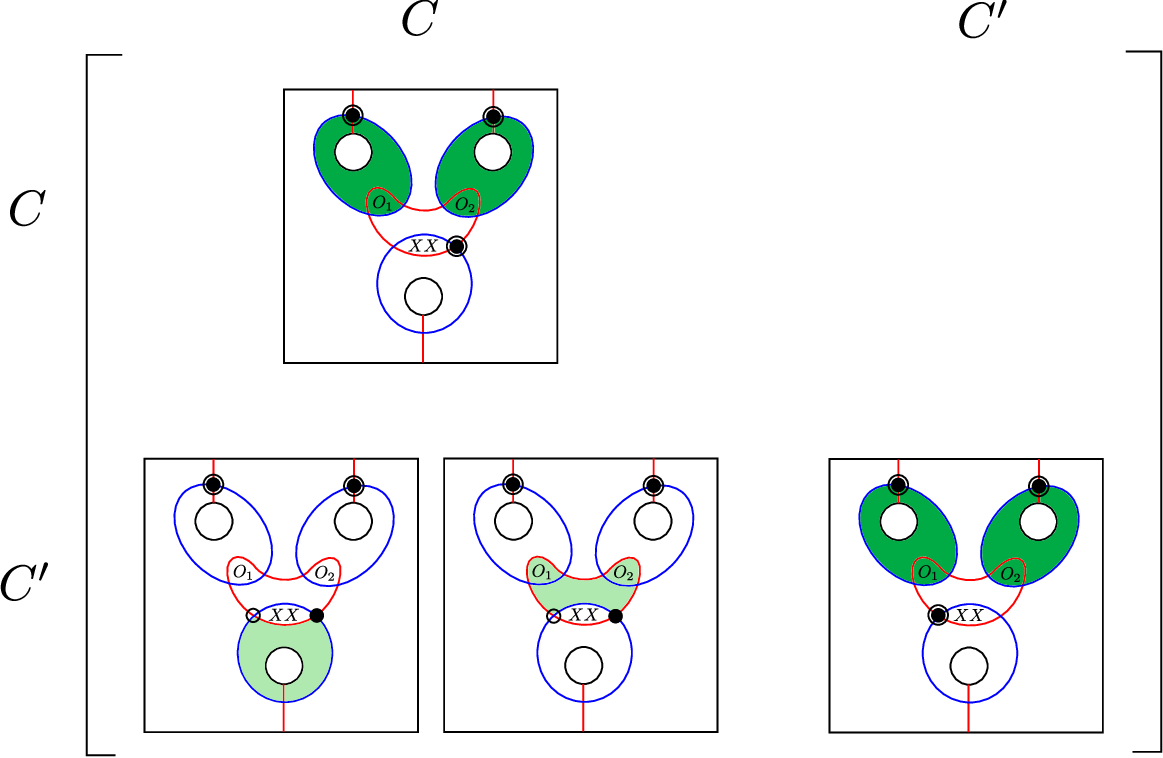}
\caption{Domains for the secondary matrix of $Y'_{\low}$.}
\label{fig:HardVertexDomainsLower}
\end{figure}

Figure~\ref{fig:HardVertexGensLower} shows the generators $C$ and $C'$ as intersection points in the Heegaard diagram of Figure~\ref{fig:HardVertex}; Figure~\ref{fig:HardVertexDomainsLower} shows the domains giving rise to the secondary matrix.

The bimodule $Y_{\fu,\low}$ is defined to be
\[
\xymatrix{
q^{-2} Y'_{\low}[-1] \ar@/^1.5pc/[rr]^{\Xi} & \oplus & Y'_{\low}
}
\] 
where $\Xi$ is the endomorphism of $Y'_{\low}$ with matrix
\[
\kbordermatrix{
& C & C' \\
C & w_1 + w_2 + e_1 & 0 \\ 
C' & 0 & w_1 + w_2 + e_1
}
\]

Thus, the secondary matrix of $Y_{\fu,\low}$ can be written as
\[
\kbordermatrix{
& C_2 & C'_2 & C_1 & C'_1 \\
C_2 & w_1^k w_2^l \otimes U_1^k U_2^l & 0 & 0 & 0 \\
C'_2 & w_1 w_2 + e_2 & w_1^k w_2^l \otimes U_1^k U_2^l & 0 & 0 \\
C_1 & w_1 + w_2 + e_1 & 0 & w_1^k w_2^l \otimes U_1^k U_2^l & 0 \\
C'_1 & 0 & w_1 + w_2 + e_1 &  w_1 w_2 + e_2 & w_1^k w_2^l \otimes U_1^k U_2^l \\
}
\]

Equivalently, the differential has matrix
\[
\kbordermatrix{
& C_2 & C'_2 & C_1 & C'_1 \\
C_2 & 0 & 0 & 0 & 0 \\
C'_2 & w_1 w_2 + e_2 & 0 & 0 & 0 \\
C_1 & w_1 + w_2 + e_1 & 0 & 0 & 0 \\
C'_1 & 0 & w_1 + w_2 + e_1 & w_1 w_2 + e_2 & 0 \\
}
\]
and right multiplication by $U_i$ or $w_i$ has matrix
\[
\kbordermatrix{
& C_2 & C'_2 & C_1 & C'_1 \\
C_2 & w_i & 0 & 0 & 0 \\
C'_2 & 0 & w_i & 0 & 0 \\
C_1 & 0 & 0 & w_i & 0 \\
C'_1 & 0 & 0 & 0 & w_i \\
}
\]
for $i \in \{1,2\}$. As with $Y'_{\midd}$ and $Y_{\fu,\midd}$, we can view $Y'_{\low}$ and $Y_{\fu,\low}$ as infinitely generated DA bimodules over $((\A_2)_{2\varepsilon_2}, (\A_{1,1})_{2\varepsilon_2})$. We will take this perspective and write the generators of $Y_{\fu,\low}$ as
\[
\{C_1, C'_1, C_2, C'_2\}
\]
times monomials in $w_1$ and $w_2$.

\begin{proposition}
As a left dg module over $(\A_2)_{2\varepsilon_2} = \F_2[e_1,e_2]$, $H_*(Y_{\fu,\low})$ is free of rank two. One can take $\{w_1 C'_1, C'_1\}$ to generate the homology.
\end{proposition}

\begin{proof}
The kernel of the differential is spanned by the elements $w_1^k w_2^l C'_1$ and the image of $d$. In homology, we have the relation $w_2 C'_1 = (w_1 + e_1)C'_1$, so the homology is generated by elements $w_1^k C'_1$. We also have the relation $(w_1^2 + e_1 w_1)C'_1 = w_1 w_2 C'_1 = e_2 C'_1$, so the homology is generated by the two elements $C'_1$ and $w_1 C'_1$. The submodule generated by these two elements has zero intersection with the image of the differential, so we can identify the homology with this submodule.
\end{proof}

The left submodule of $Y_{\fu,\low}$ generated by $\{w_1 C'_1, C'_1 \}$ is not closed under right multiplication. However, the left submodule generated by the remaining generators is closed under right multiplication, and we can view $H_*(Y_{\fu,\low})$ as the quotient of $Y_{\fu,\low}$ by this sub-bimodule. It follows that $Y_{\fu,\low}$ is a formal dg bimodule, since the quotient map is a quasi-isomorphism to its homology.

\begin{proposition}\label{prop:ActionsOnHomologyLower}
The right action of $U_1$ on $H_*(Y_{\fu,\low})$ is given by
\[
\kbordermatrix{
& w_1 C'_1 & C'_1 \\
w_1 C'_1  & e_1 & 1 \\
C'_1 & e_2 & 0
},
\]
and the right action of $U_2$ is given by
\[
\kbordermatrix{
& w_1 C'_1 & C'_1 \\
w_1 C'_1 & 0 & 1 \\
C'_1 & e_2 & e_1
}.
\]
\end{proposition}

We can use the above matrices to define the fully-simplified bimodule $Y_{\fs,\low}$; we will not need a half-simplified version. As above, we streamline notation by relabeling $C'_1$ as $C'$.
\begin{definition}
The (ordinary) bimodule $Y_{\fs,\low}$ over $((\A_2)_{2\varepsilon_2}, (\A_{1,1})_{2\varepsilon_2})$ has primary matrix
\[
\kbordermatrix{
& \oo \\
\o & w_1 C' \quad C'
}
\]
and right actions of $U_1$ and $U_2$ given by the matrices in Proposition~\ref{prop:ActionsOnHomologyLower}. We have 
\[
\deg^q(w_1 C') = -1, \quad \deg^h(w_1 C') = 2, \quad \deg^q(C') = 1, \quad \deg^h(C') = 0.
\]
\end{definition}

\subsection{Decategorification of \texorpdfstring{$Y$}{Y}}

\begin{proposition}

The AD bimodule ${^{\vee}}Y$ categorifies the map from $G_0(\A_2)$ to $G_0(\A_{1,1})$ with matrix
\[
\kbordermatrix{
& {[S_{\u}]} & {[S_{\o}]} \\
{[S_{\uu}]} & 0 & 0 \\
{[S_{\ou}]} & -1 & 0 \\
{[S_{\uo}]} & q^{-1} & 0 \\
{[S_{\oo}]} & 0 & q + q^{-1}
}.
\]
This map can be identified with $1_{1,1} F_{\gl(2)} 1_{2,0}$ (or equivalently with $1_{1,1} E_{\gl(2)} 1_{0,2}$) mapping from $\wedge_q^2 V$ to $V^{\otimes 2}$.
\end{proposition}

\subsection{2-representation morphism structure}

We will give
\[
Y_{\fs} \coloneqq Y_{\fs,\midd} \oplus Y_{\fs,\low}
\]
the structure of a $1$-morphism of $2$-representations of $\U^-$ by defining an isomorphism
\[
\varphi \co Y_{\fs} \boxtimes F^{\vee} \to F^{\vee} \boxtimes Y_{\fs}.
\]
The two bimodules in question are only nonzero as bimodules over $((\A_2)_{\varepsilon_1 + \varepsilon_2}, (\A_{1,1})_{2\varepsilon_2})$, so we need to give an isomorphism
\[
\varphi \co Y_{\fs,\midd} \boxtimes F^{\vee} \to F^{\vee} \boxtimes Y_{\fs,\low}.
\]

Using the matrix-based formulas for the box tensor product of DA bimodules in Section~\ref{sec:BoxTensor}, we can describe both $Y_{\fs,\midd} \boxtimes F^{\vee}$ and $F^{\vee} \boxtimes Y_{\fs,\low}$.

\begin{proposition}

The (a priori DA) bimodule $Y_{\fs,\midd} \boxtimes F^{\vee}$ has primary matrix
\[
\kbordermatrix{ 
& \oo \\ 
\u & A' \xi_{\ou,\oo} \quad B' \xi_{\uo,\oo}
}
\]
and secondary matrix
\[
\kbordermatrix{
& A' \xi_{\ou,\oo} & B' \xi_{\uo,\oo} \\
A' \xi_{\ou,\oo} & e_1^{k+1} \otimes U_1^{k+1} & \begin{matrix} e_1^k \otimes U_1^{k+1} \\+ e_1^k \otimes U_2^{k+1} \end{matrix} \\ 
B' \xi_{\uo,\oo} & 0 & e_1^{k+1} \otimes U_2^{k+1}
}.
\]
\end{proposition}

This bimodule is no longer $A_{\infty}$, and it has no differential. The right action of $U_1$ is given by the matrix
\[
\kbordermatrix{
& A' \xi_{\ou,\oo} & B' \xi_{\uo,\oo} \\
A' \xi_{\ou,\oo} & e_1 & 1 \\ 
B' \xi_{\uo,\oo} & 0 & 0
},
\]
and the right action of $U_2$ is given by the matrix
\[
\kbordermatrix{
& A' \xi_{\ou,\oo} & B' \xi_{\uo,\oo} \\
A' \xi_{\ou,\oo} & 0 & 1 \\ 
B' \xi_{\uo,\oo} & 0 & e_1
}.
\]

\begin{proposition}
The DA (or ordinary) bimodule $F^{\vee} \boxtimes Y_{\fs,\low}$ has primary matrix
\[
\kbordermatrix{
& \oo \\
\u & \xi_{\u,\o} (w_1 C') \quad \xi_{\u,\o} C'
}.
\]
The differential is zero; the matrix for the right action of $U_1$ is
\[
\kbordermatrix{
& \xi_{\u,\o} w_1 C' & \xi_{\u,\o} C' \\
\xi_{\u,\o} w_1 C' & e_1 & 1 \\
\xi_{\u,\o} C' & 0 & 0
},
\]
and the matrix for the right action of $U_2$ is
\[
\kbordermatrix{
& \xi_{\u,\o} w_1 C' & \xi_{\u,\o} C' \\
\xi_{\u,\o} w_1 C' & 0 & 1 \\
\xi_{\u,\o} C' & 0 & e_1
}.
\]
\end{proposition}

\begin{definition}\label{def:YHardFS2RepAlpha}
We will write $Y_{\fs}$ for the 1-morphism between 2-representations of $\U^-$ given by $(Y_{\fs}, \varphi)$, where 
\[
\varphi\co Y_{\fs} \boxtimes F^{\vee} \to F^{\vee} \boxtimes Y_{\fs}
\]
is given by the matrix
\[
\kbordermatrix{
& A' \xi_{\ou,\oo} & B' \xi_{\uo,\oo} \\
\xi_{\u,\o} w_1 C' & 1 & 0 \\ 
\xi_{\u,\o} C' & 0 & 1
}
\]
as a map from $Y_{\fs,\midd} \boxtimes F^{\vee}$ to $F^{\vee} \boxtimes Y_{\fs,\low}$ ($\varphi$ is zero as a map between the other summands).
\end{definition}

As in Section~\ref{sec:Lambda2RepMorphism}, the square
\[
\xymatrix{ 
0 = Y_{\upp,\fs} \boxtimes (F^{\vee})^{\boxtimes 2} \ar[rr]^{\varphi \boxtimes \id_{F^{\vee}}} \ar[d]_{\id_{Y_{\upp,\fs}} \boxtimes \tau} && F^{\vee} \boxtimes Y_{\midd,\fs} \boxtimes F^{\vee} \ar[rr]^{\id_{F^{\vee}} \boxtimes \varphi} && (F^{\vee})^{\boxtimes 2} \boxtimes Y_{\low,\fs} \ar[d]^{\tau \boxtimes \id_{Y_{\low,\fs}}} \\
0 = Y_{\upp,\fs} \boxtimes (F^{\vee})^{\boxtimes 2} \ar[rr]^{\varphi \boxtimes \id_{F^{\vee}}} && F^{\vee} \boxtimes Y_{\midd,\fs} \boxtimes F^{\vee} \ar[rr]^{\id_{F^{\vee}} \boxtimes \varphi} && (F^{\vee})^{\boxtimes 2} \boxtimes Y_{\low,\fs}
}
\]
automatically commutes ($Y_{\upp,\fs} = 0$), so $(Y_{\fs},\varphi)$ is a valid $1$-morphism.

\section{Bimodules for compositions of two trivalent vertices}\label{sec:SimpleWebs}

\subsection{The relevant Heegaard diagrams}

Figure~\ref{fig:SingularCrossing} shows two webs, a ``singular crossing'' composed of two trivalent vertices and the other possible composition of these vertices, together with two Heegaard diagrams. These diagrams are obtained from the two possible ways of gluing the diagram of Figure~\ref{fig:HardVertex} to the diagram of Figure~\ref{fig:EasyVertex} vertically, after some handleslides and destabilizations. In this section, we take tensor products of the bimodules for the hard and easy trivalent vertices in both directions to get bimodules for the singular crossing and other composite web.

\begin{figure}
\includegraphics[scale=0.4]{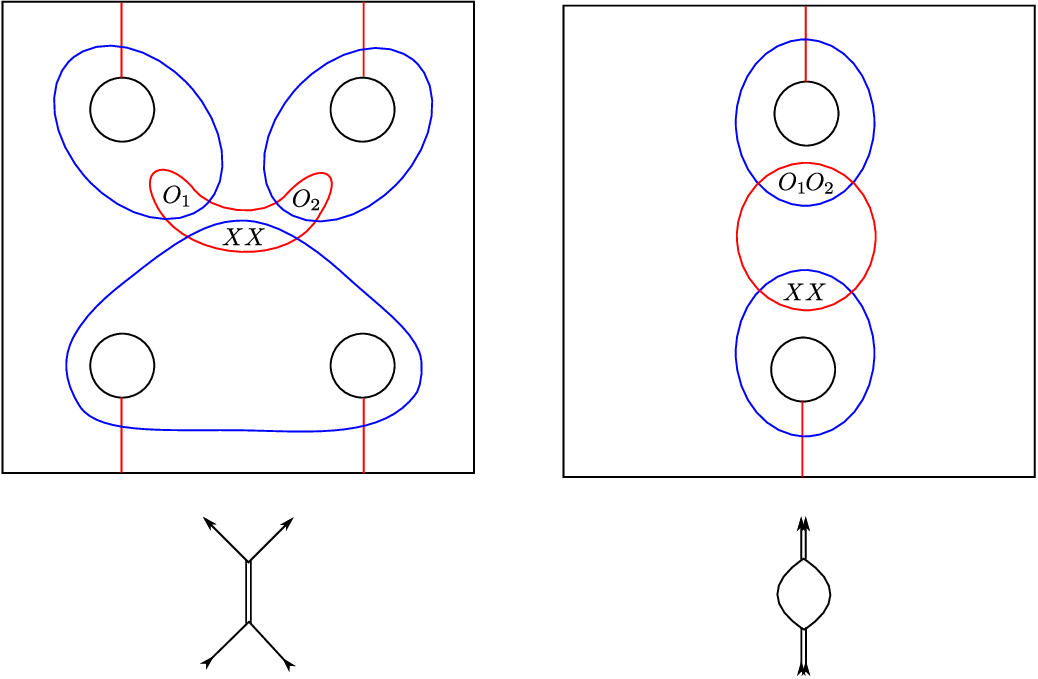}
\caption{Heegaard diagrams for compositions of two trivalent vertices}
\label{fig:SingularCrossing}
\end{figure}

\begin{remark}
Like the diagram of Figure~\ref{fig:HardVertex}, the diagrams in Figure~\ref{fig:SingularCrossing} have regions with multiple $X$ or $O$ basepoints, which are typically not allowed for the multi-pointed diagrams used in knot and link Floer homology (see e.g. \cite{HFLOrig}). The diagram for a singular crossing, however, is more-or-less standard; Figure~\ref{fig:OSSzDiag} shows a local version of the Heegaard diagram for a singular crossing defined in \cite{OSSz} (see also \cite{ManionDiagrams, ManionSingular}). It is reasonable to suppose that the two diagrams of Figure~\ref{fig:SingularCrossing} represent the complements of their corresponding webs in $D^2 \times I$ with certain sutured structures on the boundary (see Figure~\ref{fig:Cobordisms} and Remark~\ref{rem:WeightedSutures} in the introduction), although it would be good to have a general theory of such generalized diagrams and the sutured 3-manifolds or cobordisms they represent.
\end{remark}

\begin{remark}
Bimodules constructed literally from the Heegaard diagrams in Figure~\ref{fig:SingularCrossing} would be more like the fully unsimplified bimodules considered in Section~\ref{sec:HardVertex}, due to the presence of $O$ basepoints in the diagram. We will focus on the fully simplified bimodules instead.
\end{remark}

\begin{figure}
\includegraphics[scale=0.6]{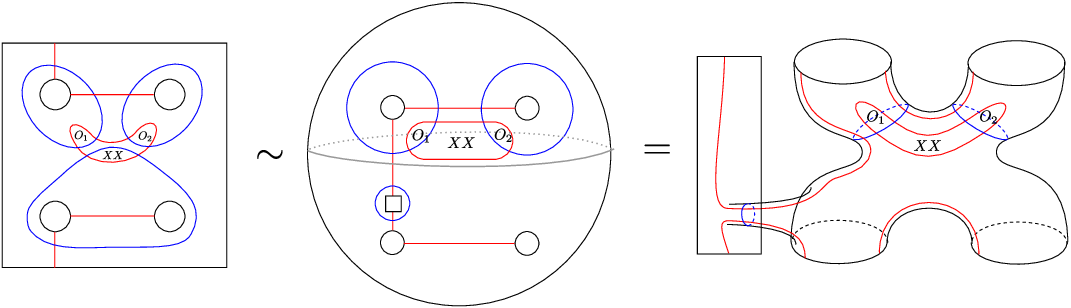}
\caption{Ozsv{\'a}th--Stipsicz--Szab{\'o} diagram for a singular crossing (the leftmost two figures are related by handleslides of $\beta$ circles).}
\label{fig:OSSzDiag}
\end{figure}

\subsection{Singular crossing}\label{sec:SingularCrossing}

\begin{definition}
Let 
\[
X_{\midd} \coloneqq \Lambda_{\midd} \boxtimes Y_{\fs,\midd}
\]
and
\[
X_{\low} \coloneqq \Lambda_{\low} \boxtimes Y_{\fs,\low}.
\]
\end{definition}

\begin{proposition}
The DA bimodule $X_{\midd}$ has primary matrix 
\[
\kbordermatrix{
& \ou & \uo \\
\ou & XA' & XB'\\
\uo & YA' & YB'
}
\]
and secondary matrix
\[
\kbordermatrix{
& XA' & YA' & XB' & YB' \\
XA' & U_1^{k+1} \otimes U_1^{k+1} & \lambda & \begin{matrix} U_1^k \otimes (\lambda, U_1^{k+1}) \\+ U_1^k \otimes (U_2^{k+1}, \lambda)\end{matrix} & 0 \\
YA' & 0 & U_2^{k+1} \otimes U_1^{k+1} & 0 & \begin{matrix} U_2^k \otimes (U_2^{k+1}, \lambda) \\+ U_2^k \otimes (\lambda, U_1^{k+1})\end{matrix} \\
XB' & 0 & 0 & U_1^{k+1} \otimes U_2^{k+1} & \lambda \\
YB' & 0 & 0 & 0 & U_2^{k+1} \otimes U_2^{k+1}
}.
\]
We have
\begin{itemize}
\item $\deg^q(XA') = -1$, $\deg^h(XA') = 2$,
\item $\deg^q(YA') = 0$, $\deg^h(YA') = 1$,
\item $\deg^q(XB') = 0$, $\deg^h(XB') = 1$,
\item $\deg^q(YB') = 1$, $\deg^h(YB') = 0$.
\end{itemize}
\end{proposition}

\begin{proposition}
The (ordinary) bimodule $X_{\low}$ has primary matrix
\[
\kbordermatrix{
& \oo \\
\oo & Z(w_1 C') \quad ZC'  \\
}.
\]
The right action of $U_1$ is given by the matrix
\[
\kbordermatrix{
& Z(w_1 C') & ZC' \\
Z(w_1 C') & U_1 + U_2 & 1 \\
ZC' & U_1 U_2 & 0
}
\]
and the right action of $U_2$ is given by the matrix\[
\kbordermatrix{
& Z(w_1 C') & ZC' \\
Z(w_1 C')  & 0 & 1 \\
ZC' & U_1 U_2 & U_1 + U_2
}.
\]
We have
\begin{itemize}
\item $\deg^q(Z (w_1 C')) = -1$, $\deg^h(Z (w_1 C')) = 2$,
\item $\deg^q(ZC') = 1$, $\deg^h(ZC') = 0$.
\end{itemize}
\end{proposition}

Both propositions can be proven by applying the definitions of Section~\ref{sec:BoxTensor}.

\begin{corollary}
The DA bimodule $X$ categorifies the map $K_0(\A_{1,1}) \to K_0(\A_{1,1})$ with matrix
\[
\kbordermatrix{
& [P_{\uu}] & [P_{\ou}] & [P_{\uo}] & [P_{\oo}] \\
{[P_{\uu}]} & 0 & 0 & 0 & 0 \\
{[P_{\ou}]} & 0 & q^{-1} & -1 & 0 \\
{[P_{\uo}]} & 0 & -1 & q & 0 \\
{[P_{\oo}]} & 0 & 0 & 0 & q + q^{-1}\\
}.
\]
Equivalently, the left dual ${^{\vee}}X$ of $X$ categorifies the map $G_0(\A_{1,1}) \to G_0(\A_{1,1})$ with matrix
\[
\kbordermatrix{
& {[S_{\uu}]} & {[S_{\ou}]} & {[S_{\uo}]} & {[S_{\oo}]} \\
{[S_{\uu}]} & 0 & 0 & 0 & 0 \\
{[S_{\ou}]} & 0 & q & -1 & 0 \\
{[S_{\uo}]} & 0 & -1 & q^{-1} & 0 \\
{[S_{\oo}]} & 0 & 0 & 0 & q + q^{-1} \\
}.
\]
Since the decategorification of a box tensor product is the product of decategorifications (and box tensor products are compatible with opposite bimodules), this endomorphism agrees with $1_{1,1} F_{\gl(2)} 1_{2,0} E_{\gl(2)} 1_{1,1}$ (or equivalently $1_{1,1} E_{\gl(2)} 1_{0,2} F_{\gl(2)} 1_{1,1}$) acting on $V ^{\otimes 2}$, a relation that can also be checked directly (see Appendix~\ref{sec:SingularNonsingular}).
\end{corollary}

We now consider the 1-morphism structure on $X$.
\begin{definition}
We write $X = (X,\varphi)$ for the $1$-morphism of $2$-representations from $\A_{1,1}$ to itself given by the composition of the $1$-morphisms 
\[
Y_{\fs}\co \A_{1,1} \to \A_2
\]
and
\[
\Lambda \co \A_2 \to \A_{1,1}.
\]
\end{definition}

The map $\varphi$ is the composition of $\id_{\Lambda} \boxtimes \varphi_{Y}$ and $\varphi_{\Lambda} \boxtimes \id_{Y_{\fs}}$ from Definitions~\ref{def:YHardFS2RepAlpha} and \ref{def:YEasy2RepAlpha} under the natural identification of $(\Lambda \boxtimes F^{\vee}) \boxtimes Y_{\fs}$ with $\Lambda \boxtimes (F^{\vee} \boxtimes Y_{\fs})$. The map $\id_{\Lambda} \boxtimes \varphi_{Y}$ has matrix
\[
\kbordermatrix{
& XA' \xi_{\ou,\oo} & YA' \xi_{\ou,\oo} & XB' \xi_{\uo,\oo} & YB' \xi_{\uo,\oo} \\
X \xi_{\u,\o} (w_1 C') & 1 & 0 & 0 & 0\\ 
Y \xi_{\u,\o} (w_1 C') & 0 & 1 & 0 & 0 \\
X \xi_{\u,\o} C' & 0 & 0 & 1 & 0 \\
Y \xi_{\u,\o} C' & 0 & 0 & 0 & 1
}
\]
and the map $\varphi_{\Lambda} \boxtimes \id_{Y_{\fs}}$ has matrix
\[
\kbordermatrix{
& X\xi_{\u,\o} (w_1 C') & Y\xi_{\u,\o} (w_1 C') & X \xi_{\u,\o} C' & Y \xi_{\u,\o} C' \\
\xi_{\ou,\oo} Z(w_1 C') & 1 & 0 & 0 & 0\\ 
\xi_{\uo,\oo} Z(w_1 C') & 0 & 1 & 0 & 0 \\
\xi_{\ou,\oo} ZC' & 0 & 0 & 1 & 0 \\
\xi_{\uo,\oo} ZC' & 0 & 0 & 0 & 1
}.
\]
Thus, the map $\varphi$ for $X$ is given explicitly by the matrix
\[
\kbordermatrix{
& XA' \xi_{\ou,\oo} & YA' \xi_{\ou,\oo} & XB' \xi_{\uo,\oo} & YB' \xi_{\uo,\oo} \\
\xi_{\ou,\oo} Z(w_1 C') & 1 & 0 & 0 & 0\\  
\xi_{\uo,\oo} Z(w_1 C') & 0 & 1 & 0 & 0 \\
\xi_{\ou,\oo} ZC' & 0 & 0 & 1 & 0 \\
\xi_{\uo,\oo} ZC' & 0 & 0 & 0 & 1
}.
\]

It follows from well-definedness of composition of $1$-morphisms between $2$-representations that the square
\[
\xymatrix{ 
X_{\upp} \boxtimes (F^{\vee})^{\boxtimes 2} \ar[rr]^{\varphi \boxtimes \id_{F^{\vee}}} \ar[d]_{\id_{X_{\upp}} \boxtimes \tau} && F^{\vee} \boxtimes X_{\midd} \boxtimes F^{\vee} \ar[rr]^{\id_{F^{\vee}} \boxtimes \varphi} && (F^{\vee})^{\boxtimes 2} \boxtimes X_{\low} \ar[d]^{\tau \boxtimes \id_{X_{\low}}} \\
X_{\upp} \boxtimes (F^{\vee})^{\boxtimes 2} \ar[rr]^{\varphi \boxtimes \id_{F^{\vee}}} && F^{\vee} \boxtimes X_{\midd} \boxtimes F^{\vee} \ar[rr]^{\id_{F^{\vee}} \boxtimes \varphi} && (F^{\vee})^{\boxtimes 2} \boxtimes X_{\low}
}
\]
commutes; alternatively, commutativity is also immediate because $X_{\upp} = 0$.

\subsection{The other composition}\label{sec:Bubble}

\begin{definition}
Let 
\[
\Phi_{\midd} \coloneqq Y_{\fs,\midd} \boxtimes \Lambda_{\midd}
\]
and 
\[
\Phi_{\low} \coloneqq Y_{\fs,\low} \boxtimes \Lambda_{\low}.
\]
\end{definition}

\begin{proposition}\label{prop:BubbleIsTwiceIdMiddle}
The (ordinary) bimodule $\Phi_{\midd}$ has primary matrix
\[
\kbordermatrix{ & \u \\
\u & A'X \quad B'Y
}
\]
and right action of $e_1$ given by
\[
\kbordermatrix{ & A'X & B'Y \\
A'X & e_1 & 0 \\
B'Y & 0 & e_1
}.
\]
We have $\deg^q(A'X) = -1$, $\deg^h(A'X) = 2$, $\deg^q(B'Y) = 1$, and $\deg^h(B'Y) = 0$.
\end{proposition}

\begin{proposition}\label{prop:BubbleIsTwiceIdLower}
The (ordinary) bimodule $\Phi_{\low}$ has primary matrix
\[
\kbordermatrix{ & \o \\
\o & (w_1 C') Z \quad C' Z
},
\]
right action of $e_1$ given by
\[
\kbordermatrix{ 
& (w_1 C') Z & C' Z \\
(w_1 C') Z  & e_1 & 0 \\
C' Z & 0 & e_1
},
\]
and right action of $e_2$ given by 
\[
\kbordermatrix{ 
& (w_1 C') Z & C' Z \\
(w_1 C') Z  & e_2 & 0 \\
C' Z & 0 & e_2
}.
\]
We have $\deg^q((w_1 C') Z) = -1$, $\deg^h((w_1 C')Z) = 2$, $\deg^q(C' Z) = 1$, and $\deg^h(C' Z) = 0$, .
\end{proposition}

\begin{proof}
These propositions follow from Section~\ref{sec:BoxTensor}. For example, for the right action of $e_2$ on $\Phi_{\low}$, note that
\[
\kbordermatrix{ 
& (w_1 C') Z & C' Z \\
(w_1 C') Z & e_1 & 1 \\
C' Z & e_2 & 0
}
\kbordermatrix{ 
& (w_1 C') Z & C' Z \\
(w_1 C') Z & 0 & 1 \\
C' Z & e_2 & e_1
}
\]
equals
\[
\kbordermatrix{ 
& (w_1 C') Z & C' Z \\
(w_1 C') Z & e_2 & 0 \\
C' Z & 0 & e_2
}.
\]
\end{proof}

\begin{remark}
Proposition \ref{prop:BubbleIsTwiceIdMiddle} implies that $\Phi_{\midd} \cong (q + q^{-1}h^2)\id_{(\A_2)_{\varepsilon_1 + \varepsilon_2}}$, where $h^2$ indicates a shift upward by two in the homological grading. Similarly, Proposition \ref{prop:BubbleIsTwiceIdLower} implies that $\Phi_{\low} \cong (q + q^{-1}h^2)\id_{(\A_2)_{2\varepsilon_2}}$. 
\end{remark}

\begin{corollary}
The DA bimodule $\Phi$ categorifies the map from $K_0(\A_2)$ to $K_0(\A_2)$ with matrix
\[
\kbordermatrix{
& [P_{\u}] & [P_{\o}] \\
{[P_{\u}]} & q + q^{-1} & 0 \\
{[P_{\o}]} & 0 & q + q^{-1} \\
};
\]
equivalently, ${^{\vee}}\Phi$ categorifies the map from $G_0(\A_2)$ to $G_0(\A_2)$ with matrix
\[
\kbordermatrix{
& {[S_{\u}]} & {[S_{\o}]} \\
{[S_{\u}]} & q + q^{-1} & 0 \\
{[S_{\o}]} & 0 & q + q^{-1} \\
};
\]
Since the decategorification of a box tensor product is the product of decategorifications, this endomorphism is equal to $1_{2,0} E_{\gl(2)} 1_{1,1} F_{\gl(2)} 1_{2,0}$ (or to $1_{0,2} F_{\gl(2)} 1_{1,1} E_{\gl(2)} 1_{0,2}$) acting on $\wedge_q^2 V$, a relation that can also be checked directly.
\end{corollary}

\begin{definition}
We write $\Phi = (\Phi,\varphi)$ for the $1$-morphism of $2$-representations from $\A_2$ to itself given by the composition of the $1$-morphisms 
\[
\Lambda \co \A_2 \to \A_{1,1}
\]
and
\[
Y_{\fs}\co \A_{1,1} \to \A_2.
\] 
\end{definition}

The map $\varphi$ is the composition of $\id_{Y_{\fs}} \boxtimes \varphi_{\Lambda}$ and $\varphi_{Y} \boxtimes \id_{\Lambda}$ from Definitions~\ref{def:YEasy2RepAlpha} and \ref{def:YHardFS2RepAlpha} under the natural identification of $(Y_{\fs} \boxtimes F^{\vee}) \boxtimes \Lambda$ with $Y_{\fs} \boxtimes (F^{\vee} \boxtimes \Lambda)$. The map $\id_{Y_{\fs}} \boxtimes \varphi_{\Lambda}$ has matrix
\[
\kbordermatrix{
& A'X\xi_{\u,\o} & B'Y \xi_{\u,\o} \\
A'\xi_{\ou,\oo} Z & 1 & 0 \\
B'\xi_{\uo,\oo} Z & 0 & 1
}
\]
and the map $\varphi_{Y} \boxtimes \id_{\Lambda}$ has matrix
\[
\kbordermatrix{
& A' \xi_{\ou,\oo} Z & B' \xi_{\uo,\oo} Z \\
\xi_{\u,\o} (w_1 C')Z & 1 & 0 \\
\xi_{\u,\o} C' Z & 0 & 1
}.
\]
Thus, the map $\varphi$ for $\Phi$ is given explicitly by the matrix
\[
\kbordermatrix{
& A'X\xi_{\u,\o} & B'Y\xi_{\u,\o} \\
\xi_{\u,\o} (w_1 C')Z  & 1 & 0 \\
\xi_{\u,\o} C'Z & 0 & 1
}.
\]

\begin{remark}
The above computation implies that $\Phi \cong (q + q^{-1} h^2) \id_{\A_2}$ as 2-representations of $\U^-$.
\end{remark}

\section{A skew Howe 2-action}\label{sec:SkewHowe2Action}

\subsection{A categorified quantum group}\label{sec:CatQGDef}

We recall a graded 2-category $\Sc(2,2)^*$ defined in Mackaay--Sto{\v{s}}i\'c--Vaz \cite{MSVSchur} as a quotient of the categorified quantum group $\dot{\U}_q(\gl(2))^*$ (the details of which are reviewed in \cite{MSVSchur}). While $\Q$ coefficients are assumed in \cite{MSVSchur}, the definitions of $\dot{\U}_q(\gl(2))^*$ and $\Sc(2,2)^*$ make sense over $\Z$; we will work with their reductions over $\F_2$.

\begin{remark}
The $*$ notation indicates that the morphism spaces consist of arbitrary-degree morphisms, not necessarily degree zero. In contrast with \cite{MSVSchur}, we will not take the objects of $\dot{\U}_q(\gl(2))^*$ or $\Sc(2,2)^*$ to be closed under grading shifts, although if one preferred, one could include them without issue.
\end{remark}

\begin{definition}
Let the graded $\F_2$-linear 2-category $\Sc(2,2)^*$ be the quotient of $\dot{\U}_q(\gl(2))^*$ (with no grading shifts on objects as remarked above) by the ideal generated by the identity 2-morphism on $\One_{\lambda}$ for all $\lambda \in \Z^2$ except for $\lambda \in \{(2,0), (1,1), (0,2)\}$.
\end{definition}

\begin{remark}
The $\gl(2)$ weights $\lambda$ appearing here should not be confused with $\gl(1|1)$ weights $\omega$.
\end{remark}

We will write $\Sc(2,2)^{\ungr}$ for $\Sc(2,2)$ with its grading ignored; we will be most interested in a bigraded lift $\Sc(2,2)^{*,*}$ of $\Sc(2,2)^{\ungr}$. Concretely, we can take $\Sc(2,2)^{*,*}$ to have objects $\One_{2,0}$, $\One_{1,1}$, and $\One_{0,2}$. A generating set of 1-morphisms is given by
\begin{itemize}
\item $\One_{1,1} \xleftarrow{\One_{1,1} \E \One_{0,2}} \One_{0,2}$,
\item $\One_{2,0} \xleftarrow{\One_{2,0} \E \One_{1,1}} \One_{1,1}$,
\item $\One_{1,1} \xleftarrow{\One_{1,1} \Fc \One_{2,0}} \One_{2,0}$,
\item $\One_{0,2} \xleftarrow{\One_{0,2} \Fc \One_{1,1}} \One_{1,1}$.
\end{itemize}

We will draw 1-morphisms as sequences of vertical strings with the regions between them labeled by $\gl(2)$ weights $\lambda \in \Z^2$; $\E$ morphisms will point upwards and $\Fc$ morphisms will point downwards.

The 2-morphisms in $\Sc(2,2)^{*,*}$ are generated by the following string diagrams, read from bottom to top:
\begin{itemize}
\item \includegraphics[scale=0.4]{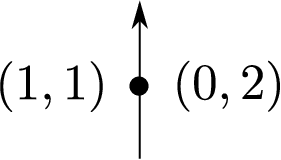} with $\deg^q = -2$ and $\deg^h = 2$
\item \includegraphics[scale=0.4]{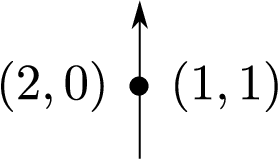} with $\deg^q = -2$ and $\deg^h = 2$
\item \includegraphics[scale=0.4]{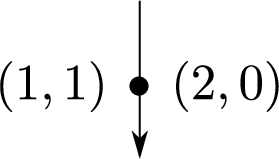} with $\deg^q = -2$ and $\deg^h = 2$
\item \includegraphics[scale=0.4]{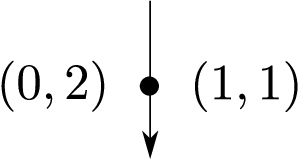} with $\deg^q = -2$ and $\deg^h = 2$
\item \includegraphics[scale=0.4]{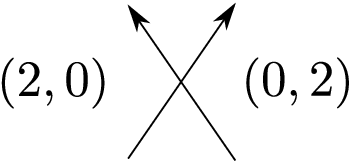} with $\deg^q = 2$ and $\deg^h = -2$
\item \includegraphics[scale=0.4]{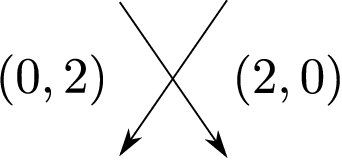} with $\deg^q = 2$ and $\deg^h = -2$
\item \includegraphics[scale=0.4]{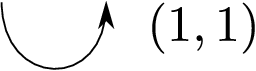} with $\deg^q = -1$ and $\deg^h = 0$
\item \includegraphics[scale=0.4]{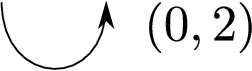} with $\deg^q = 1$ and $\deg^h = -2$
\item \includegraphics[scale=0.4]{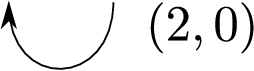} with $\deg^q = 1$ and $\deg^h = -2$
\item \includegraphics[scale=0.4]{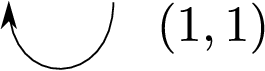} with $\deg^q = -1$ and $\deg^h = 0$
\item \includegraphics[scale=0.4]{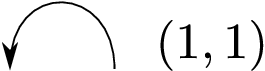} with $\deg^q = -1$ and $\deg^h = 2$
\item \includegraphics[scale=0.4]{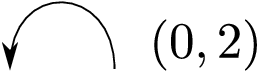} with $\deg^q = 1$ and $\deg^h = 0$
\item \includegraphics[scale=0.4]{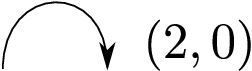} with $\deg^q = 1$ and $\deg^h = 0$
\item \includegraphics[scale=0.4]{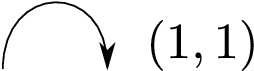} with $\deg^q = -1$ and $\deg^h = 2$
\end{itemize}

The relations imposed on the 2-morphisms are:
\begin{enumerate}
\item\label{it:Biadjointness1} \includegraphics[scale=0.4]{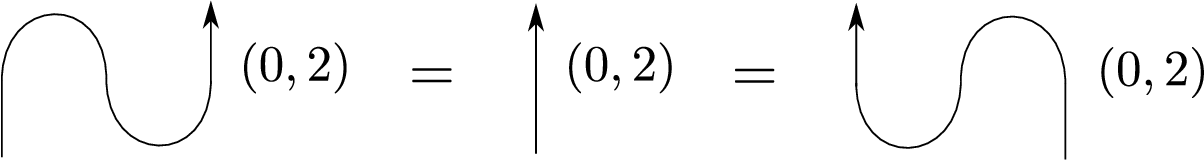}
\item\label{it:Biadjointness2} \includegraphics[scale=0.4]{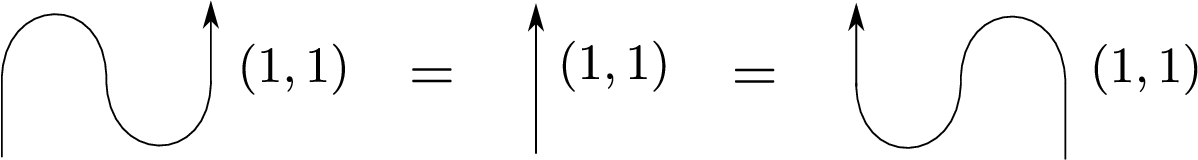}
\item\label{it:Biadjointness3} \includegraphics[scale=0.4]{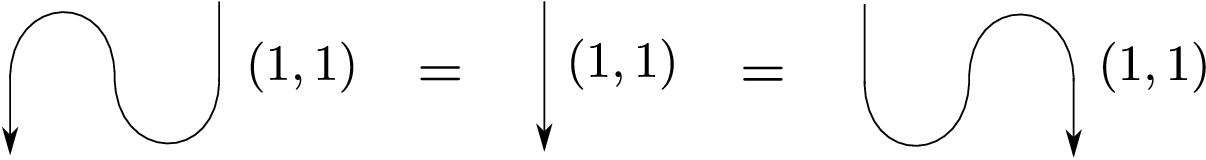}
\item\label{it:Biadjointness4} \includegraphics[scale=0.4]{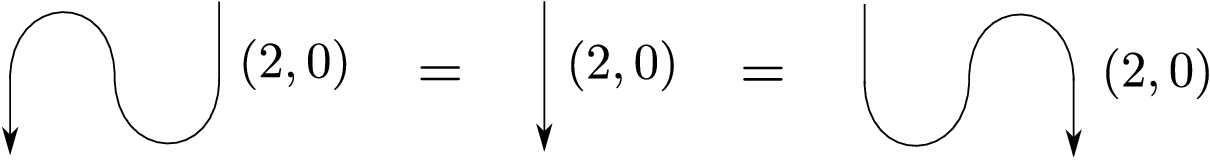}
\item\label{it:DotCyclic1} \includegraphics[scale=0.4]{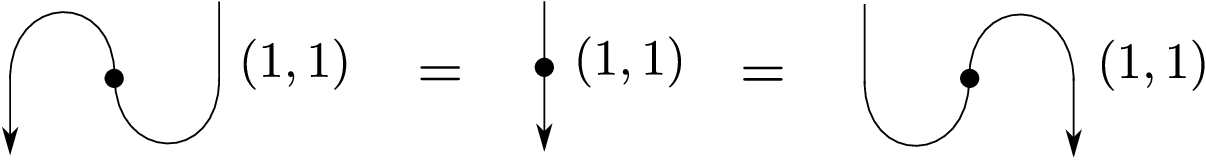}
\item\label{it:DotCyclic2} \includegraphics[scale=0.4]{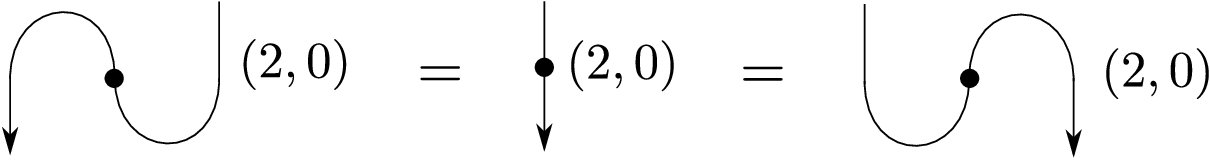}
\item\label{it:CrossingCyclic} \includegraphics[scale=0.64]{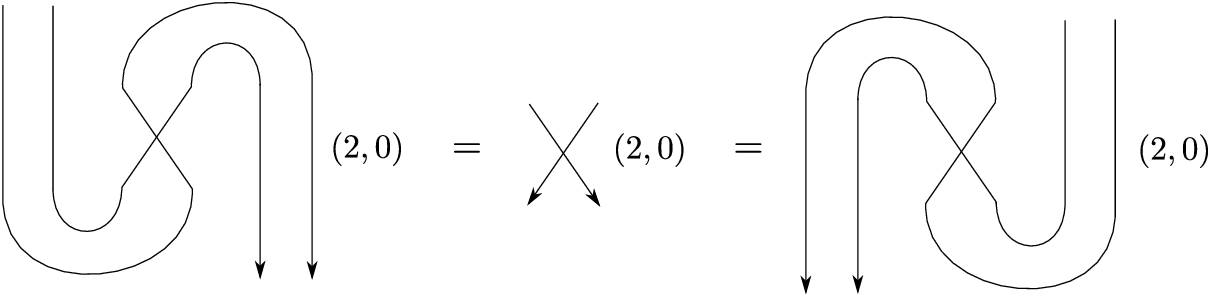}
\item\label{it:NegativeBubble1} \includegraphics[scale=0.4]{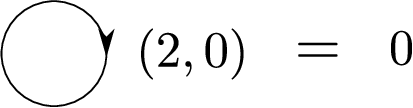}
\item\label{it:NegativeBubble2} \includegraphics[scale=0.4]{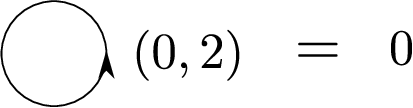}
\item\label{it:Deg0Bubble1} \includegraphics[scale=0.4]{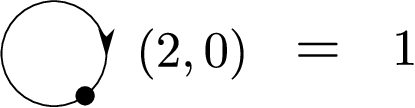}
\item\label{it:Deg0Bubble2} \includegraphics[scale=0.4]{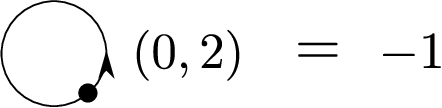}
\item\label{it:ExtendedSl2_1} \includegraphics[scale=0.4]{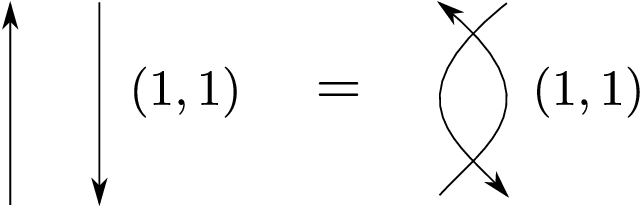}
\item\label{it:ExtendedSl2_2} \includegraphics[scale=0.4]{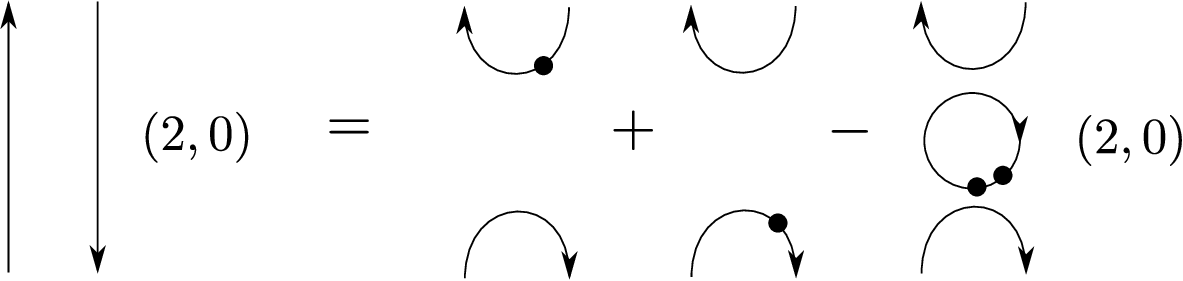}
\item\label{it:ExtendedSl2_3} \includegraphics[scale=0.4]{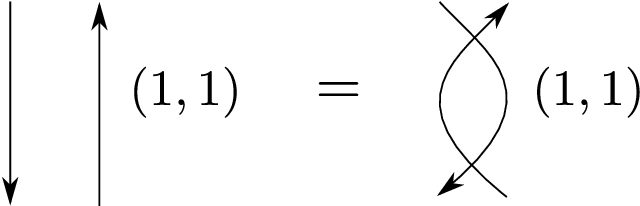}
\item\label{it:ExtendedSl2_4} \includegraphics[scale=0.4]{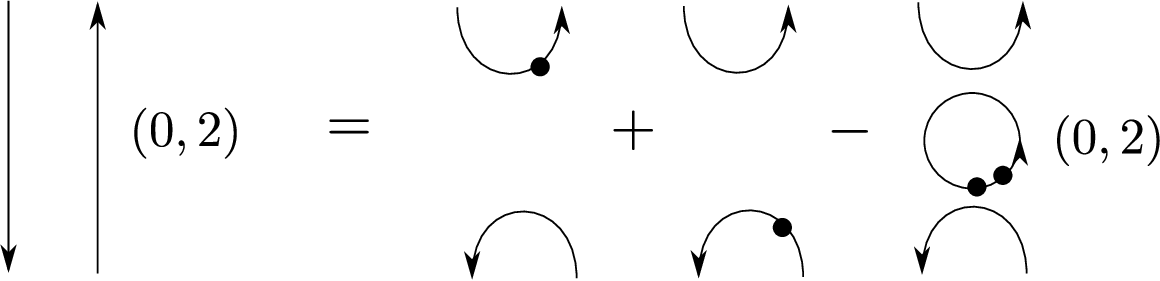}
\item\label{it:NilHecke1} \includegraphics[scale=0.4]{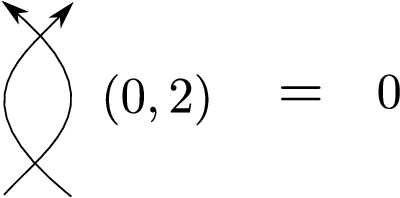}
\item\label{it:NilHecke2} \includegraphics[scale=0.4]{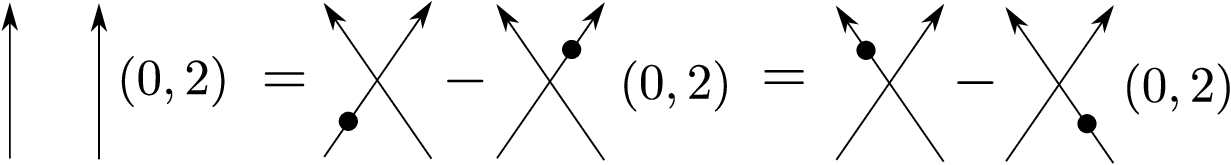}
\end{enumerate}
(we include signs in various places but they are unnecessary over $\F_2$). In items \eqref{it:ExtendedSl2_1} and \eqref{it:ExtendedSl2_3}, the sideways crossings are defined by

\begin{center}
\includegraphics[scale=0.4]{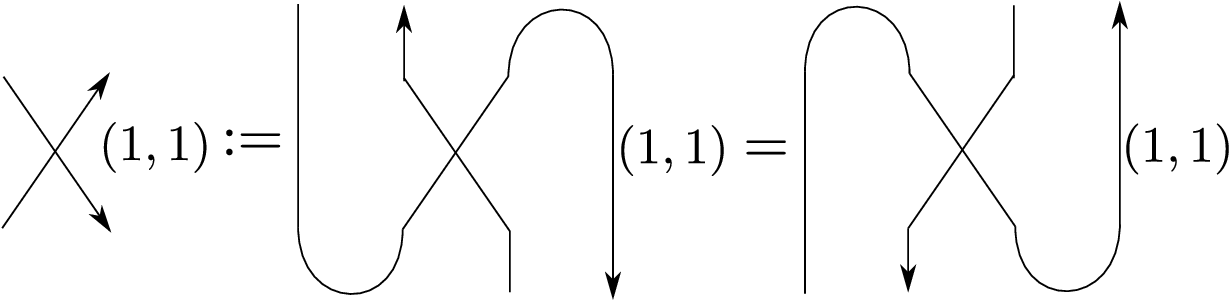} which has $\deg^q = 0$ and $\deg^h = 0$
\end{center}

\noindent and

\begin{center}
\includegraphics[scale=0.4]{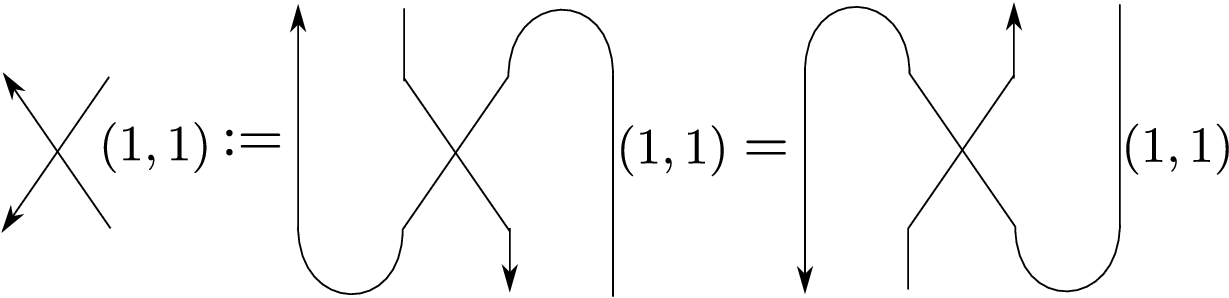} which has $\deg^q = 0$ and $\deg^h = 0$.
\end{center}

We will define a functor from $\Sc(2,2)^{*,*}$ into $2\Rep(\U^-)^{*,*}$.
On objects, it sends
\begin{itemize}
\item $\One_{2,0} \mapsto \A_2$
\item $\One_{1,1} \mapsto \A_{1,1}$
\item $\One_{0,2} \mapsto \A_2$
\end{itemize}
where we view $\A_2$ and $\A_{1,1}$ as 2-representations $(A,F,\tau)$ of $\U^-$.

\subsection{Bimodules for 1-morphisms}

We send
\begin{itemize}
\item $\One_{1,1} \E \One_{0,2} \mapsto {^{\vee}}Y$,
\item $\One_{2,0} \E \One_{1,1} \mapsto {^{\vee}}\Lambda$,
\item $\One_{1,1} \Fc \One_{2,0} \mapsto {^{\vee}}Y$,
\item $\One_{0,2} \Fc \One_{1,1} \mapsto {^{\vee}}\Lambda$,
\end{itemize}
where we view ${^{\vee}}Y$ and ${^{\vee}}\Lambda$ as AD bimodule 1-morphisms of 2-representations of $\U^-$.

\subsection{Bimodule maps for 2-morphisms}\label{sec:BimodMapsFor2Mors}

\subsubsection{Dots}

\begin{itemize}
\item For the 2-morphism \includegraphics[scale=0.4]{dot1.eps} in $\Sc(2,2)^{*,*}$, we define an endomorphism of ${^{\vee}}Y$ as the dual of the endomorphism $\delta^{\up}_{0,2}$ of $Y$ with matrix
\[
\kbordermatrix{
& A' & B' \\
A' & e_1 & 1 \otimes \lambda \\
B' & 0 & 0
}
\]
on the summand $Y_{\midd}$ and matrix
\[
\kbordermatrix{
& w_1 C' & C' \\
w_1 C' & e_1 & 1\\
C' & e_2 & 0
}
\]
on the summand $Y_{\low}$. 

\item For the 2-morphism \includegraphics[scale=0.4]{dot2.eps} in $\Sc(2,2)^{*,*}$, we define an endomorphism of ${^{\vee}}\Lambda$ as the dual of the endomorphism $\delta^{\up}_{1,1}$ of $\Lambda$ with matrix
\[
\kbordermatrix{
& X & Y \\
X & 0 & 0 \\
Y & 0 & U_2
}
\]
on the summand $\Lambda_{\midd}$ and matrix
\[
\kbordermatrix{
& Z \\
Z & U_2 \\
}
\]
on the summand $\Lambda_{\low}$. 

\item For the 2-morphism \includegraphics[scale=0.4]{dot3.eps} in $\Sc(2,2)^{*,*}$, we define an endomorphism of ${^{\vee}}Y$ as the dual of the endomorphism $\delta^{\down}_{2,0}$ of $Y$ with matrix
\[
\kbordermatrix{
& A' & B' \\
A' & 0 & 1 \otimes \lambda \\
B' & 0 & e_1
}
\]
on the summand $Y_{\midd}$ and
\[
\kbordermatrix{
& w_1 C' & C' \\
w_1 C' & 0 & 1 \\
C' & e_2 & e_1
}
\]
on the summand $Y_{\low}$.

\item For the 2-morphism \includegraphics[scale=0.4]{dot4.eps} in $\Sc(2,2)^{*,*}$, we define an endomorphism of ${^{\vee}}\Lambda$ as the dual of the endomorphism $\delta^{\down}_{1,1}$ of $\Lambda$ with matrix
\[
\kbordermatrix{
& X & Y \\
X & U_1 & 0 \\
Y & 0 & 0
}
\]
on the summand $\Lambda_{\midd}$ and matrix
\[
\kbordermatrix{
& Z \\
Z & U_1 \\
}
\]
on the summand $\Lambda_{\low}$.
\end{itemize}

\begin{proposition}
The four maps defined above are 2-morphisms.
\end{proposition}

\begin{proof}
\begin{itemize}
\item To see that $\delta^{\up}_{0,2}$ is a 2-morphism, we want the square
\[
\xymatrix{
YF^{\vee} \ar[d]_-{\delta^{\up}_{0,2} \boxtimes \id_{F^{\vee}}} \ar[rr]^-{\varphi_Y} && F^{\vee} Y \ar[d]^-{\id_{F^{\vee}} \boxtimes \delta^{\up}_{0,2}} \\
YF^{\vee} \ar[rr]_-{\varphi_Y} && F^{\vee} Y
}
\]
to commute up to homotopy. In the only nonzero case, the top and bottom edges of the square are
\[
\kbordermatrix{
& A' \xi_{\ou,\oo} & B' \xi_{\uo,\oo} \\
\xi_{\u,\o} w_1 C' & 1 & 0 \\ 
\xi_{\u,\o} C' & 0 & 1
};
\]
the left edge is
\[
\kbordermatrix{
& A' \xi_{\ou,\oo} & B' \xi_{\uo,\oo} \\
A' \xi_{\ou,\oo} & e_1 & 1 \\ 
B' \xi_{\uo,\oo} & 0 & 0
}
\]
and the right edge is
\[
\kbordermatrix{
& \xi_{\u,\o} w_1 C' & \xi_{\u,\o} C' \\
\xi_{\u,\o} w_1 C' & e_1 & 1 \\ 
\xi_{\u,\o} C' & 0 & 0
}.
\]
Thus, the square commutes.

\item To see that $\delta^{\up}_{1,1}$ is a 2-morphism, we want the square
\[
\xymatrix{
\Lambda F^{\vee} \ar[d]_-{\delta^{\up}_{1,1} \boxtimes \id_{F^{\vee}}} \ar[rr]^-{\varphi_{\Lambda}} && F^{\vee} \Lambda \ar[d]^-{\id_{F^{\vee}} \boxtimes \delta^{\up}_{1,1}} \\
\Lambda F^{\vee} \ar[rr]_-{\varphi_{\Lambda}} && F^{\vee} \Lambda
}
\]
to commute up to homotopy. Recall that the top and bottom edges of the square are
\[
\kbordermatrix{
& X \xi_{\o,\u} & Y \xi_{\o,\u} \\
\xi_{\ou,\oo} Z & 1 & 0 \\
\xi_{\uo,\oo} Z & 0 & 1
};
\]
the left edge is
\[
\kbordermatrix{
& X \xi_{\o,\u} & Y \xi_{\o,\u} \\
X \xi_{\o,\u} & 0 & 0 \\
Y \xi_{\o,\u} & 0 & U_2
}
\]
and the right edge is
\[
\kbordermatrix{
& \xi_{\ou,\oo} Z & \xi_{\uo,\oo} Z \\
\xi_{\ou,\oo} Z & 0 & 0 \\
\xi_{\uo,\oo} Z & 0 & U_2
}.
\]
Thus, the square commutes.

\item For $\delta^{\down}_{2,0}$, one can check similarly that the square
\[
\xymatrix{
YF^{\vee} \ar[d]_-{\delta^{\down}_{2,0} \boxtimes \id_{F^{\vee}}} \ar[rr]^-{\varphi_Y} && F^{\vee} Y \ar[d]^-{\id_{F^{\vee}}  \boxtimes \delta^{\down}_{2,0}} \\
YF^{\vee} \ar[rr]_-{\varphi_Y} && F^{\vee} Y
}
\]
commutes, so $\delta^{\down}_{2,0}$ is a 2-morphism.

\item For $\delta^{\down}_{1,1}$, one can check similarly that the square
\[
\xymatrix{
\Lambda F^{\vee} \ar[d]_-{\delta^{\down}_{1,1} \boxtimes \id_{F^{\vee}}} \ar[rr]^-{\varphi_{\Lambda}} && F^{\vee} \Lambda \ar[d]^-{\id_{F^{\vee}} \boxtimes \delta^{\down}_{1,1}} \\
\Lambda F^{\vee} \ar[rr]_-{\varphi_{\Lambda}} && F^{\vee} \Lambda
}
\]
commutes, so $\delta^{\down}_{1,1}$ is a 2-morphism.
\end{itemize}
\end{proof}

\begin{remark}
The map $\delta^{\up}_{0,2}$ can be thought of as ``right multiplication by $U_1$'' on $Y$; more precisely, to get the above matrices, one sums all ways of inserting $U_1$ as an algebra input in some slot for the right $A_{\infty}$ action on $Y$. The other maps defined above have similar interpretations:
\begin{itemize}
\item $\delta^{\up}_{1,1}$ comes from left multiplication by $U_2$.
\item $\delta^{\down}_{2,0}$ comes from right multiplication by $U_2$.
\item $\delta^{\down}_{1,1}$ comes from left multiplication by $U_1$.
\end{itemize}
\end{remark}

\subsubsection{Crossings}

For the 2-morphism \includegraphics[scale=0.4]{crossing1.eps} in $\Sc(2,2)^{*,*}$, we define an endomorphism of ${^{\vee}}\Phi$ as the dual of the endomorphism $\chi$ of $\Phi$ with matrix
\[
\kbordermatrix{
& A'X & B'Y \\
A'X & 0 & 0 \\
B'Y & 1 & 0
}
\]
in the middle summand and
\[
\kbordermatrix{
& w_1 C' Z & C' Z \\
w_1 C' Z & 0 & 0  \\
C' Z & 1 & 0
}
\]
in the lower summand; it is clear that $\chi$ is a 2-morphism.

To the 2-morphism \includegraphics[scale=0.4]{crossing2.eps} in $\Sc(2,2)^{*,*}$, we assign the same endomorphism of ${^{\vee}}\Phi$ as for \includegraphics[scale=0.4]{crossing1.eps} (namely the dual of $\chi$).

\subsubsection{Cups and caps}

\begin{itemize}
\item For the 2-morphism \includegraphics[scale=0.4]{cup1.eps} in $\Sc(2,2)^{*,*}$, we define a morphism from $\id$ to ${^{\vee}}X$ as the dual of the morphism
\[
\varepsilon'\co X \to \id
\]
with matrix 
\[
\kbordermatrix{
& XA' & YA' & XB' & YB' \\
\Ib_{\ou} & U_1 & 0 & 1 \otimes \lambda & 0 \\
\Ib_{\uo} & 0 & 0 & 0 & 1
}
\]
in the upper summand and 
\[
\kbordermatrix{
& Z(w_1 C') & ZC' \\
\Ib_{\oo} & U_1 & 1
}
\] 
in the lower summand. 

\item For the 2-morphism \includegraphics[scale=0.4]{cup2.eps} in $\Sc(2,2)^{*,*}$, we define a morphism from $\id$ to ${^{\vee}}\Phi$ as the dual of the morphism
\[
\varepsilon\co \Phi \to \id
\]
with matrix 
\[
\kbordermatrix{
& A'X & B'Y \\
\Ib_{\u} & 1 & 0
}
\]
in the upper summand and 
\[
\kbordermatrix{
& w_1 C' Z & C'Z \\
\Ib_{\o} & 1 & 0
}
\]
in the lower summand.

\item The 2-morphism \includegraphics[scale=0.4]{cup3.eps} gets assigned the dual of $\varepsilon\co \Phi \to \id$.
\item The 2-morphism \includegraphics[scale=0.4]{cup4.eps} gets assigned the dual of $\varepsilon'\co X \to \id$.

\item For the 2-morphism \includegraphics[scale=0.4]{cap1.eps} in $\Sc(2,2)^{*,*}$, we define a morphism from $X^{\vee}$ to $\id$ as the dual of the morphism
\[
\eta\co \id \to X
\]
with matrix
\[
\kbordermatrix{
& \Ib_{\ou} & \Ib_{\uo} \\ 
XA' & 1 & 0  \\
YA' & 0 & 1 \otimes \lambda \\
XB' & 0 & 0 \\
YB' & 0 & U_2
}
\]
in the upper summand and 
\[
\kbordermatrix{
& \Ib_{\oo} \\
Z(w_1 C') & 1 \\
ZC' & U_2
}
\] 
in the lower summand.

\item For the 2-morphism \includegraphics[scale=0.4]{cap2.eps} in $\Sc(2,2)^{*,*}$, we define a morphism from ${^{\vee}}\Phi$ to $\id$ as the dual of the morphism
\[
\eta'\co \id \to \Phi
\]
with matrix 
\[
\kbordermatrix{
& \Ib_{\u} \\
A'X & 0 \\
B'Y & 1
}
\] 
in the upper summand and 
\[
\kbordermatrix{
& \Ib_{\o} \\
(w_1 C')Z & 0 \\
C' Z & 1
}
\]
in the lower summand.

\item The 2-morphism \includegraphics[scale=0.4]{cap3.eps} gets assigned the dual of $\eta'\co \id \to \Phi$.
\item The 2-morphism \includegraphics[scale=0.4]{cap4.eps} gets assigned the dual of $\eta\co \id \to X$.
\end{itemize}

We first check that $\eta$ is a closed morphism of DA bimodules.

\begin{proposition}
The map $\eta$ is a closed morphism of DA bimodules in the upper summand.
\end{proposition}

\begin{proof}
Since the only $A_{\infty}$ input to $\eta$ is $\lambda$ which cannot be factored, we need to check that the matrix for $\eta$ intertwines the secondary matrices for the DA bimodules $X_{\midd}$ and $(\mathbb{I}_{\A_{1,1}})_{\midd}$. The product

\[
\kbordermatrix{
& XA' & YA' & XB' & YB' \\
XA' & U_1^{k+1} \otimes U_1^{k+1} & \lambda & \begin{matrix} U_1^k \otimes (\lambda, U_1^{k+1}) \\+ U_1^k \otimes (U_2^{k+1}, \lambda)\end{matrix} & 0 \\
YA' & 0 & U_2^{k+1} \otimes U_1^{k+1} & 0 & \begin{matrix} U_2^k \otimes (U_2^{k+1}, \lambda) \\+ U_2^k \otimes (\lambda, U_1^{k+1})\end{matrix} \\
XB' & 0 & 0 & U_1^{k+1} \otimes U_2^{k+1} & \lambda \\
YB' & 0 & 0 & 0 & U_2^{k+1} \otimes U_2^{k+1}
}
\cdot \kbordermatrix{
& \Ib_{\ou} & \Ib_{\uo} \\ 
XA' & 1 & 0  \\
YA' & 0 & 1 \otimes \lambda \\
XB' & 0 & 0 \\
YB' & 0 & U_2
}
\]
equals
\[
\kbordermatrix{
& \Ib_{\ou} & \Ib_{\uo} \\ 
XA' & U_1^{k+1} \otimes U_1^{k+1} & \lambda \otimes \lambda \\
YA' & 0 & U_2^{k+1} \otimes (U_2^{k+1}, \lambda) \\
XB' & 0 & 0 \\
YB' & 0 & U_2^{k+2} \otimes U_2^{k+1}
}.
\]
The product
\[
\kbordermatrix{
& \Ib_{\ou} & \Ib_{\uo} \\ 
XA' & 1 & 0  \\
YA' & 0 & 1 \otimes \lambda \\
XB' & 0 & 0 \\
YB' & 0 & U_2
}
\cdot \kbordermatrix{
& \Ib_{\ou} & \Ib_{\uo} \\
\Ib_{\ou} & U_1^{k+1} \otimes U_1^{k+1} & \lambda \otimes \lambda \\
\Ib_{\uo} & 0 & U_2^{k+1} \otimes U_2^{k+1}
}
\]
gives the same result.
\end{proof}

\begin{proposition}
The map $\eta$ is a valid homomorphism of (ordinary) bimodules, and thus a closed morphism of DA bimodules, in the lower summand.
\end{proposition}

\begin{proof}
To see that $\eta$ respects the right action of $U_1$, note that
\[
\kbordermatrix{
& Z(w_1 C') & ZC' \\
Z(w_1 C') & U_1 + U_2 & 1 \\
ZC' & U_1 U_2 & 0
}
\kbordermatrix{
& \Ib_{\oo} \\
Z(w_1 C') & 1 \\
ZC' & U_2
}
\]
equals
\[
\kbordermatrix{
& \Ib_{\oo} \\
Z(w_1 C') & U_1 \\
ZC' & U_1 U_2
},
\]
which also equals
\[
\kbordermatrix{
& \Ib_{\oo} \\
Z(w_1 C') & 1 \\
ZC' & U_2
}
\kbordermatrix{
& \Ib_{\oo} \\
\Ib_{\oo} & U_1 
}.
\]
Thus, $\eta$ respects the right action of $U_1$.

To see that $\eta$ respects the right action of $U_2$, note that
\[
\kbordermatrix{
& Z(w_1 C') & ZC' \\
Z(w_1 C') & 0 & 1 \\
ZC' & U_1 U_2 & U_1 + U_2
}
\kbordermatrix{
& \Ib_{\oo} \\
Z(w_1 C') & 1 \\
ZC' & U_2
}
\]
equals
\[
\kbordermatrix{
& \Ib_{\oo} \\
Z(w_1 C') & U_2 \\
ZC' & U_2^2
},
\]
which also equals
\[
\kbordermatrix{
& \Ib_{\oo} \\
Z(w_1 C') & 1 \\
ZC' & U_2
}
\kbordermatrix{
& \Ib_{\oo} \\
\Ib_{\oo} & U_2 
}
\]
Thus, $\eta$ is a bimodule homomorphism.
\end{proof}

The analogous result for $\varepsilon$ is clear.

\begin{remark}\label{rem:UsingSymmetry}
The matrices for $\eta'$ and $\varepsilon'$ are obtained from the matrices for $\varepsilon$ and $\eta$ respectively by:
\begin{enumerate}
\item\label{it:Symm1} transposing the matrices ``along their anti-diagonals,'' while preserving (not reversing) the ordering of basis elements;
\item\label{it:Symm2} exchanging $U_1$ and $U_2$ (as well as $w_1$ and $w_2$) everywhere they appear;
\item\label{it:Symm3} reversing the order of all higher sequences of $A_{\infty}$ inputs
\end{enumerate}
(the last item does not arise for the maps under consideration, and neither do $w_i$ elements). Closedness of these morphisms follows from closedness of $\varepsilon$ and $\eta$ because the secondary matrices for each summand of $\Lambda$ and $Y_{\fs}$ are symmetric under the simultaneous application of the above items. 

Indeed, $\Lambda$ and $Y$ are based on Heegaard diagrams that are symmetric under the composition of the following operations: reflect the diagram left-to-right in the plane, reverse the orientation on the Heegaard surface, and interchange $w_1$ and $w_2$ if the diagram has them. The algebraic symmetry in the secondary matrices corresponds to this diagrammatic symmetry. More specifically, the diagrammatic symmetry is preserved in the secondary matrices since we chose an ordered basis for the row and column indices such that left-to-right reflection in the plane on the row and column indices (viewed as sets of intersection points in the Heegaard diagram) corresponds to order reversal of the ordered basis.
\end{remark}

We now check that the maps for cups and caps are 2-morphisms.

\begin{proposition}\label{prop:FirstEtaIs2Mor}
The map $\varepsilon$ is a $2$-morphism.
\end{proposition}

\begin{proof}
For $\varepsilon$ to be a 2-morphism, the square
\[
\xymatrix{
\Phi_{\midd} \boxtimes F^{\vee}  \ar[d]_{\varphi} \ar[rr]^{\varepsilon \boxtimes \id_{F^{\vee}}} && F^{\vee} \ar[d]^{\id} \\
F^{\vee} \boxtimes \Phi_{\low} \ar[rr]_{\id_{F^{\vee}} \boxtimes \varepsilon} && F^{\vee}
}
\]
should commute, at least up to homotopy. 

The top edge of the square for $\varepsilon$ has matrix
\[
\kbordermatrix{
& A'X \xi_{\u,\o} & B'Y\xi_{\u,\o} \\
\xi_{\u,\o} & 1 & 0 \\
},
\]
and the bottom edge of this square has matrix
\[
\kbordermatrix{
& \xi_{\u,\o} (w_1 C')Z & \xi_{\u,\o} C' Z \\
\xi_{\u,\o} & 1 & 0
}.
\]
As computed in Section~\ref{sec:Bubble}, the left edge of the square has matrix
\[
\kbordermatrix{
& A'X \xi_{\u,\o} & B'Y\xi_{\u,\o} \\
\xi_{\u,\o} (w_1 C')Z & 1 & 0 \\
\xi_{\u,\o} C' Z & 0 & 1
}.
\]
Thus, the square for $\varepsilon$ commutes. 

\end{proof}

For $\eta$ to be a $2$-morphism, the squares
\[
\xymatrix{
F^{\vee} \ar[d]_{\id} \ar[rr]^-{\eta \boxtimes \id_{F^{\vee}}} && X_{\midd} \boxtimes F^{\vee} \ar[d]^{\varphi} \\
F^{\vee} \ar[rr]_-{\id_{F^{\vee}} \boxtimes \eta} && F^{\vee} \boxtimes X_{\low} 
}
\]
and
\[
\xymatrix{
F^{\vee} \ar[d]_{\id} \ar[rr]^-{\eta \boxtimes \id_{F^{\vee}}} && X_{\upp} \boxtimes F^{\vee} = 0 \ar[d]^{\varphi} \\
F^{\vee} \ar[rr]_-{\id_{F^{\vee}} \boxtimes \eta} && F^{\vee} \boxtimes X_{\midd}  
}
\]
should commute, at least up to homotopy.

\begin{proposition}\label{prop:FirstEpsIs2MorSquare1}
The first square for $\eta$ commutes.
\end{proposition}

\begin{proof}
The matrix for the top edge of the nontrivial square for $\eta$ is
\[
\kbordermatrix{
& \xi_{\ou,\oo} & \xi_{\uo,\oo} \\
XA'\xi_{\ou,\oo} & 1 & 0 \\
YA'\xi_{\ou,\oo} & 0 & 1 \\
XB'\xi_{\uo,\oo} & 0 & 0 \\
YB'\xi_{\uo,\oo} & 0 & U_2
},
\]
and the matrix for the bottom edge of the square is
\[
\kbordermatrix{
& \xi_{\ou,\oo} & \xi_{\uo,\oo} \\
\xi_{\ou,\oo} Z(w_1 C') & 1 & 0 \\
\xi_{\uo,\oo} Z(w_1 C') & 0 & 1 \\
\xi_{\ou,\oo} ZC' & 0 & 0 \\
\xi_{\uo,\oo} ZC' & 0 & U_2
}.
\]

As computed in Section~\ref{sec:SingularCrossing}, the matrix for the right edge of the square is 
\[
\kbordermatrix{
& XA' \xi_{\ou,\oo} & YA' \xi_{\ou,\oo} & XB' \xi_{\uo,\oo} & YB' \xi_{\uo,\oo} \\
\xi_{\ou,\oo} Z(w_1 C') & 1 & 0 & 0 & 0\\  
\xi_{\uo,\oo} Z(w_1 C') & 0 & 1 & 0 & 0 \\
\xi_{\ou,\oo} ZC' & 0 & 0 & 1 & 0 \\
\xi_{\uo,\oo} ZC' & 0 & 0 & 0 & 1
},
\]
so the first square for $\eta$ commutes.

\end{proof}

\begin{proposition}\label{prop:FirstEpsIs2MorSquare2}
The second square for $\eta$ commutes up to homotopy.
\end{proposition}

\begin{proof}
The map $\id_{F^{\vee}} \boxtimes \eta$ on the bottom edge has matrix
\[
\kbordermatrix{
& \xi_{\uu,\ou} & \xi_{\uu,\uo} \\
\xi_{\uu,\ou} XA' & 1 & 0 \\
\xi_{\uu,\uo} YA' & 0 & 1 \otimes \lambda \\
\xi_{\uu,\ou} XB' & 0 & 0 \\
\xi_{\uu,\uo} YB' & 0 & 0
}.
\]

The secondary matrix for the bimodule $F^{\vee} \boxtimes X_{\midd}$ is
\[
\kbordermatrix{
& \xi_{\uu,\ou} XA' & \xi_{\uu,\uo} YA' & \xi_{\uu,\ou} XB' & \xi_{\uu,\uo} YB' \\
\xi_{\uu,\ou} XA' & 0 & 1 & \begin{matrix} 1 \otimes (\lambda, U_1) \\+ 1 \otimes (U_2, \lambda)\end{matrix} & 0 \\
\xi_{\uu,\uo} YA' & 0 & 0 & 0 & \begin{matrix} U_2^k \otimes (U_2^{k+1}, \lambda) \\+ U_2^k \otimes (\lambda, U_1^{k+1})\end{matrix} \\
\xi_{\uu,\ou} XB' & 0 & 0 & 0 & 1 \\
\xi_{\uu,\uo} YB' & 0 & 0 & 0 & 0
};
\]
recall that the secondary matrix for the relevant summand of $F^{\vee}$ is
\[
\kbordermatrix{
& \xi_{\uu,\ou} & \xi_{\uu,\uo} \\
\xi_{\uu,\ou} & 0 & 1 \otimes \lambda \\
\xi_{\uu,\uo} & 0 & 0
}.
\]

If we let $h$ have matrix
\[
\kbordermatrix{
& \xi_{\uu,\ou} & \xi_{\uu,\uo} \\
\xi_{\uu,\ou} XA' & 0 & 0 \\
\xi_{\uu,\uo} YA' & 1 & 0 \\
\xi_{\uu,\ou} XB' & 0 & 0 \\
\xi_{\uu,\uo} YB' & 0 & 0
},
\]
then $h$ is a null-homotopy of $\id_{F^{\vee}} \boxtimes \eta$.
\end{proof}

\begin{corollary}\label{cor:FirstEpsIs2Mor}
The map $\eta$ is a $2$-morphism.
\end{corollary}

\begin{remark}\label{rem:EtaStrong}
One can check that the homotopy $h$ in the proof of Proposition~\ref{prop:FirstEpsIs2MorSquare2} makes $\eta$ into a strong 2-morphism as in Remark~\ref{rem:Strong2Mor}.
\end{remark}

\begin{proposition}
The map $\eta'$ is a $2$-morphism.
\end{proposition}

\begin{proof}
For $\eta'$ to be a $2$-morphism, the square
\[
\xymatrix{
F^{\vee} \ar[d]_{\id} \ar[rr]^{\eta' \boxtimes \id_{F^{\vee}}} && \Phi_{\midd} \boxtimes F^{\vee}  \ar[d]^{\varphi} \\
F^{\vee} \ar[rr]_{\id_{F^{\vee}} \boxtimes \eta'} && F^{\vee} \boxtimes \Phi_{\low} 
}
\]
should commute, at least up to homotopy. Commutativity follows from the proof of Proposition~\ref{prop:FirstEtaIs2Mor} by applying the modifications of items \eqref{it:Symm1}--\eqref{it:Symm3} in Remark~\ref{rem:UsingSymmetry} to all matrices in the proof.

\end{proof}

\begin{proposition}
The map $\varepsilon'$ is a $2$-morphism.
\end{proposition}

\begin{proof}
For $\varepsilon'$ to be a 2-morphism, the square
\[
\xymatrix{
X_{\midd} \boxtimes F^{\vee} \ar[d]_{\varphi} \ar[rr]^-{\varepsilon' \boxtimes \id_{F^{\vee}}} && F^{\vee} \ar[d]^{\id} \\
F^{\vee} \boxtimes X_{\low} \ar[rr]_-{\id_{F^{\vee}} \boxtimes \varepsilon'} && F^{\vee}
}
\]
should commute, at least up to homotopy; the square
\[
\xymatrix{
0 = X_{\upp} \boxtimes F^{\vee} \ar[d]_{\varphi} \ar[rr]^-{\varepsilon' \boxtimes \id_{F^{\vee}}} && F^{\vee} \ar[d]^{\id} \\
F^{\vee} \boxtimes X_{\midd} \ar[rr]_-{\id_{F^{\vee}} \boxtimes \varepsilon'} && F^{\vee} 
}
\]
automatically commutes. Commutativity of the first square follows from the proof of Proposition~\ref{prop:FirstEpsIs2MorSquare1} by applying the modifications of items \eqref{it:Symm1}--\eqref{it:Symm3} in Remark~\ref{rem:UsingSymmetry} to all matrices in the proof.
\end{proof}

\subsection{Checking the relations}

We begin by gathering some preliminary computations which can be obtained using Procedure~\ref{proc:BoxTensorMorphisms}.

\begin{proposition}\label{prop:HelperMorphisms}
The following string diagram 2-morphisms get assigned algebraic 2-morphisms as follows.
\begin{enumerate}
\item\label{it:IdEta} The 2-morphism \includegraphics[scale=0.4]{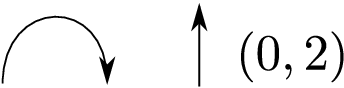} gets assigned the dual of the morphism from $Y$ to $Y\Lambda Y$ with matrix
\[
\kbordermatrix{
& A' & B' \\
A'XA' & 1 & 0 \\
B'YA' & 0 & 1 \otimes \lambda \\
A'XB' & 0 & 1 \\
B'YB' & 0 & e_1
}
\]
in the middle summand and
\[
\kbordermatrix{
& w_1 C' & C' \\
(w_1 C')Z(w_1 C') & 1 & 0 \\
C' Z (w_1 C') & 0 & 1 \\
(w_1 C')Z C' & 0 & 1 \\
C' Z C' & e_2 & e_1
}
\]
in the lower summand.

\item\label{it:EpsId} The 2-morphism \includegraphics[scale=0.4]{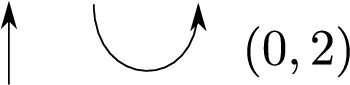} gets assigned the dual of the morphism from $Y \Lambda Y$ to $Y$ with matrix
\[
\kbordermatrix{
& A'XA' & B'YA' & A'XB' & B'YB' \\
A' & 1 & 0 & 0 & 0 \\
B' & 0 & 0 & 1 & 0
}
\]
in the middle summand and
\[
\kbordermatrix{
& (w_1 C')Z(w_1 C') & C' Z (w_1 C') & (w_1 C')Z C' & C' Z C' \\
w_1 C' & 1 & 0 & 0 & 0 \\
C' & 0 & 0 & 1 & 0
}
\]
in the lower summand.

\item\label{it:Eta'Id} The 2-morphism \includegraphics[scale=0.4]{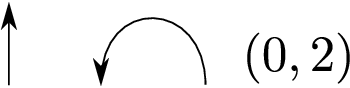} gets assigned the dual of the morphism from $Y$ to $Y\Lambda Y$ with matrix
\[
\kbordermatrix{
& A' & B' \\
A'XA' & 0 & 0 \\
B'YA' & 1 & 0 \\
A'XB' & 0 & 0 \\
B'YB' & 0 & 1
}
\]
in the middle summand and
\[
\kbordermatrix{
& w_1 C' & C' \\
(w_1 C')Z(w_1 C') & 0 & 0 \\
C' Z (w_1 C') & 1 & 0 \\
(w_1 C')Z C' & 0 & 0 \\
C' Z C' & 0 & 1
}
\]
in the lower summand.

\item\label{it:IdEps'} The 2-morphism \includegraphics[scale=0.4]{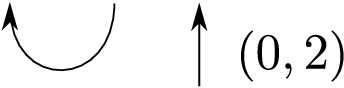} gets assigned the dual of the morphism from $Y \Lambda Y$ to $Y$ with matrix
\[
\kbordermatrix{
& A'XA' & B'YA' & A'XB' & B'YB' \\
A' & e_1 & 1 & 1 \otimes \lambda & 0 \\
B' & 0 & 0 & 0 & 1
}
\]
in the middle summand and
\[
\kbordermatrix{
& (w_1 C')Z(w_1 C') & C' Z (w_1 C') & (w_1 C')Z C' & C' Z C' \\
w_1 C' & e_1 & 1 & 1 & 0\\
C' & e_2 & 0 & 0 & 1
}
\]
in the lower summand.

\item\label{it:IdEta'} The 2-morphism \includegraphics[scale=0.4]{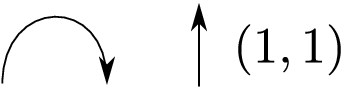} gets assigned the dual of the morphism from $\Lambda$ to $\Lambda Y \Lambda$ with matrix
\[
\kbordermatrix{
& X & Y \\
X A' X & 0 & 0 \\
Y A' X & 0 & 0 \\
X B' Y & 1 & 0 \\
Y B' Y & 0 & 1
}
\]
in the middle summand and
\[
\kbordermatrix{
& Z \\
Z(w_1 C')Z & 0 \\
ZC' Z & 1
}
\]
in the lower summand.

\item\label{it:Eps'Id} The 2-morphism \includegraphics[scale=0.4]{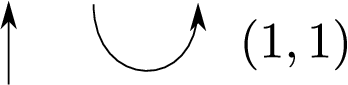} gets assigned the dual of the morphism from $\Lambda Y \Lambda$ to $\Lambda$ with matrix
\[
\kbordermatrix{
& X A' X & Y A' X & X B' Y & Y B' Y \\
X & U_1 & 0 & 1 & 0 \\
Y & 0 & 0 & 0 & 1
}
\]
in the middle summand and
\[
\kbordermatrix{
& Z(w_1 C')Z & ZC' Z \\
Z & U_1 & 1
}
\]
in the lower summand.

\item\label{it:EtaId} The 2-morphism \includegraphics[scale=0.4]{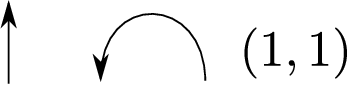} gets assigned the dual of the morphism from $\Lambda$ to $\Lambda Y \Lambda$ with matrix
\[
\kbordermatrix{
& X & Y \\
X A' X & 1 & 0 \\
Y A' X & 0 & 1 \\
X B' Y & 0 & 0 \\
Y B' Y & 0 & U_2
}
\]
in the middle summand and
\[
\kbordermatrix{
& Z \\
Z(w_1 C')Z & 1 \\
ZC' Z & U_2
}
\]
in the lower summand.

\item\label{it:IdEps} The 2-morphism \includegraphics[scale=0.4]{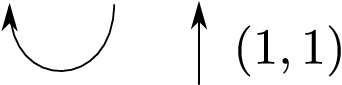} gets assigned the dual of the morphism from $\Lambda Y \Lambda$ to $\Lambda$ with matrix
\[
\kbordermatrix{
& X A' X & Y A' X & X B' Y & Y B' Y \\
X & 1 & 0 & 0 & 0 \\
Y & 0 & 1 & 0 & 0
}
\]
in the middle summand and
\[
\kbordermatrix{
& Z(w_1 C')Z & ZC' Z \\
Z & 1 & 0
}
\]
in the lower summand.

\item\label{it:IdEta'Again} The 2-morphism \includegraphics[scale=0.4]{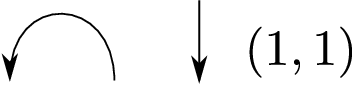} gets assigned the same morphism as in item \eqref{it:IdEta'}.

\item\label{it:Eps'IdAgain} The 2-morphism \includegraphics[scale=0.4]{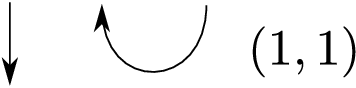} gets assigned the same morphism as in item \eqref{it:Eps'Id}.

\item\label{it:EtaIdAgain} The 2-morphism \includegraphics[scale=0.4]{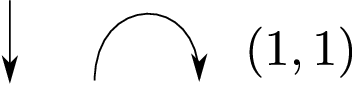} gets assigned the same morphism as in item \eqref{it:EtaId}.

\item\label{it:IdEpsAgain} The 2-morphism \includegraphics[scale=0.4]{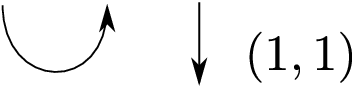} gets assigned the same morphism as in item \eqref{it:IdEps}.

\item\label{it:IdEtaAgain} The 2-morphism \includegraphics[scale=0.4]{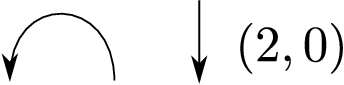} gets assigned the same morphism as in item \eqref{it:IdEta}.

\item\label{it:EpsIdAgain} The 2-morphism \includegraphics[scale=0.4]{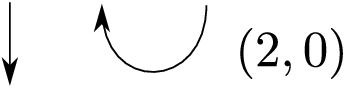} gets assigned the same morphism as in item \eqref{it:EpsId}.

\item\label{it:Eta'IdAgain} The 2-morphism \includegraphics[scale=0.4]{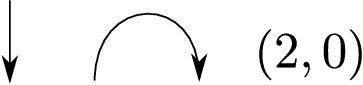} gets assigned the same morphism as in item \eqref{it:Eta'Id}.

\item\label{it:IdEps'Again} The 2-morphism \includegraphics[scale=0.4]{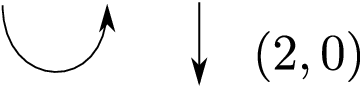} gets assigned the same morphism as in item \eqref{it:IdEps'}.

\item\label{it:UpDotId} The 2-morphism \includegraphics[scale=0.4]{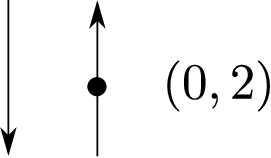} gets assigned the dual of the morphism from $Y \Lambda$ to $Y \Lambda$ with matrix
\[
\kbordermatrix{
& A'X & B'Y \\
A'X & e_1 & 1 \\
B'Y & 0 & 0
}
\]
in the middle summand and
\[
\kbordermatrix{
& (w_1 C') Z & C' Z \\
(w_1 C') Z & e_1 & 1 \\
C' Z & e_2 & 0
}
\]
in the lower summand.

\item\label{it:Extra} The 2-morphism \includegraphics[scale=0.4]{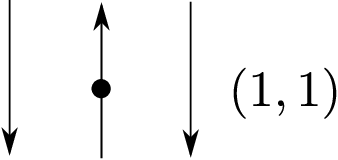} gets assigned the dual of the morphism from $\Lambda Y \Lambda$ to $\Lambda Y \Lambda$ with matrix
\[
\kbordermatrix{
& XA'X & YA'X & XB'Y & YB'Y \\
XA'X & U_1 & 0 & 1 & 0 \\
YA'X & 0 & U_2 & 0 & 1 \\
XB'Y & 0 & 0 & 0 & 0 \\
YB'Y & 0 & 0 & 0 & 0
}
\]
in the middle summand and
\[
\kbordermatrix{
& Z(w_1 C')Z & ZC' Z \\
Z(w_1 C')Z & U_1 + U_2 & 1 \\
ZC' Z & U_1 U_2 & 0
}
\]
in the lower summand.

\item\label{it:IdUpDot} The 2-morphism \includegraphics[scale=0.4]{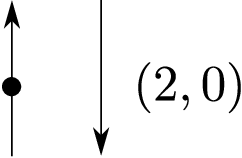} gets assigned the dual of the morphism from $Y\Lambda$ to $Y\Lambda$ with matrix
\[
\kbordermatrix{
& A'X & B'Y \\
A'X & 0 & 1 \\
B'Y & 0 & e_1
}
\]
in the middle summand and
\[
\kbordermatrix{
& (w_1 C')Z & C' Z \\
(w_1 C')Z & 0 & 1 \\
C' Z & e_2 & e_1
}
\]
in the lower summand.

\item\label{it:IdUpDotId} The 2-morphism \includegraphics[scale=0.4]{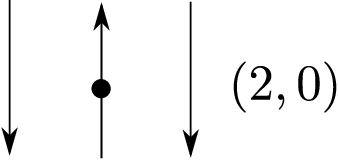} gets assigned the dual of the morphism from $Y \Lambda Y$ to $Y \Lambda Y$ with matrix
\[
\kbordermatrix{
& A'XA' & B'YA' & A'XB' & B'YB' \\
A'XA' & 0 & 1 & 0 & 0 \\
B'YA' & 0 & e_1 & 0 & 0 \\
A'XB' & 0 & 0 & 0 & 1 \\
B'YB' & 0 & 0 & 0 & e_1
}
\]
in the middle summand and
\[
\kbordermatrix{
& (w_1 C')Z(w_1 C') & C' Z (w_1 C') & (w_1 C')Z C' & C' Z C' \\
(w_1 C')Z(w_1 C') & 0 & 1 & 0 & 0 \\
C' Z (w_1 C') & e_2 & e_1 & 0 & 0 \\
(w_1 C')Z C' & 0 & 0 & 0 & 1 \\
C' Z C' & 0 & 0 & e_2 & e_1
}
\]
in the lower summand.

\item\label{it:DownDownCap} The 2-morphism \includegraphics[scale=0.4]{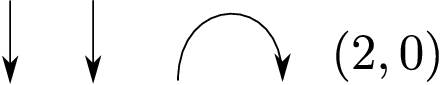} gets assigned the dual of the morphism from $Y \Lambda \to Y \Lambda Y \Lambda$ with matrix
\[
\kbordermatrix{
& A'X & B'Y \\
A'XA'X & 0 & 0 \\
B'YA'X & 1 & 0 \\
A'XB'Y & 0 & 0 \\
B'YB'Y & 0 & 1
}
\]
in the middle summand and 
\[
\kbordermatrix{
& (w_1 C')Z & C' Z \\
(w_1 C') Z (w_1 C') Z & 0 & 0 \\
C' Z (w_1 C') Z & 1 & 0 \\
(w_1 C') Z C' Z & 0 & 0 \\
C' Z C' Z & 0 & 1
}
\]
in the lower summand.

\item The 2-morphism \includegraphics[scale=0.4]{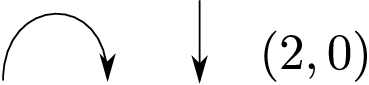} gets assigned the same morphism as in item \eqref{it:IdEta}.

\item\label{it:IdEtaId} The 2-morphism \includegraphics[scale=0.4]{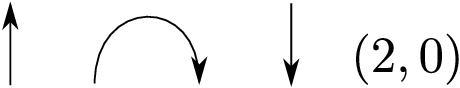} gets assigned the dual of the morphism from $Y \Lambda$ to $Y \Lambda Y \Lambda$ with matrix
\[
\kbordermatrix{
& A'X & B'Y \\
A'XA'X & 1 & 0 \\
B'YA'X & 0 & 1 \\
A'XB'Y & 0 & 1 \\
B'YB'Y & 0 & e_1
}
\]
in the middle summand and 
\[
\kbordermatrix{
& (w_1 C')Z & C' Z \\
(w_1 C') Z (w_1 C') Z & 1 & 0 \\
C' Z (w_1 C') Z & 0 & 1 \\
(w_1 C') Z C' Z & 0 & 1 \\
C' Z C' Z & e_2 & e_1
}
\]
in the lower summand.

\item\label{it:ComplicatedCap} The 2-morphism \includegraphics[scale=0.4]{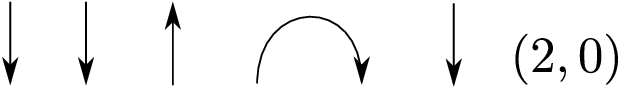} gets assigned the dual of the morphism from $Y \Lambda Y \Lambda$ to $Y \Lambda Y \Lambda Y \Lambda$ with matrix
\[
\kbordermatrix{
& A'XA'X & B'YA'X & A'XB'Y & B'YB'Y \\
A'XA'XA'X & 1 & 0 & 0 & 0 \\
B'YA'XA'X & 0 & 1 & 0 & 0 \\
A'XB'YA'X & 0 & 1 & 0 & 0 \\
B'YB'YA'X & 0 & e_1 & 0 & 0 \\
A'XA'XB'Y & 0 & 0 & 1 & 0 \\
B'YA'XB'Y & 0 & 0 & 0 & 1 \\
A'XB'YB'Y & 0 & 0 & 0 & 1 \\
B'YB'YB'Y & 0 & 0 & 0 & e_1
}
\]
in the middle summand and
\[
\kbordermatrix{
& (w_1 C') Z (w_1 C') Z & C' Z (w_1 C') Z & (w_1 C') Z C' Z & C' Z C' Z\\
(w_1 C') Z (w_1 C') Z (w_1 C') Z & 1 & 0 & 0 & 0 \\
C' Z (w_1 C') Z (w_1 C') Z & 0 & 1 & 0 & 0 \\
(w_1 C') Z C' Z (w_1 C') Z & 0 & 1 & 0 & 0 \\
C' Z C' Z (w_1 C') Z & e_2 & e_1 & 0 & 0 \\
(w_1 C') Z (w_1 C') Z C' Z & 0 & 0 & 1 & 0 \\
C' Z (w_1 C') Z C' Z & 0 & 0 & 0 & 1\\
(w_1 C') Z C' Z C' Z & 0 & 0 & 0 & 1 \\
C' Z C' Z C' Z & 0 & 0 & e_2 & e_1
}
\]
in the lower summand.

\item The 2-morphism \includegraphics[scale=0.4]{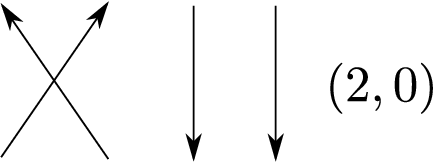} gets assigned the dual of the morphism from $Y \Lambda Y \Lambda$ to $Y \Lambda Y \Lambda$ with matrix
\[
\kbordermatrix{
& A'XA'X & B'YA'X & A'XB'Y & B'YB'Y \\
A'XA'X & 0 & 0 & 0 & 0 \\
B'YA'X & 0 & 0 & 0 & 0 \\
A'XB'Y & 1 & 0 & 0 & 0 \\
B'YB'Y & 0 & 1 & 0 & 0
}
\]
in the middle summand and
\[
\kbordermatrix{
& (w_1 C') Z (w_1 C') Z & C' Z (w_1 C') Z & (w_1 C') Z C' Z & C' Z C' Z\\
(w_1 C') Z (w_1 C') Z & 0 & 0 & 0 & 0 \\
C' Z (w_1 C') Z & 0 & 0 & 0 & 0 \\
(w_1 C') Z C' Z & 1 & 0 & 0 & 0 \\
C' Z C' Z & 0 & 1 & 0 & 0
}
\]
in the lower summand.

\item\label{it:BigCrossing} The 2-morphism \includegraphics[scale=0.4]{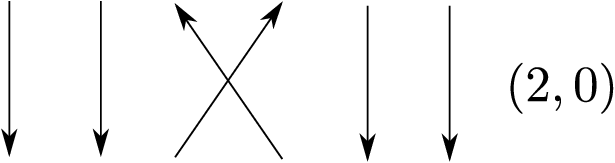} gets assigned the dual of the morphism from $Y \Lambda Y \Lambda Y \Lambda$ to $Y \Lambda Y \Lambda Y \Lambda$ with matrix
\[
\resizebox{\textwidth}{!}{
\kbordermatrix{
& A'XA'XA'X & B'YA'XA'X & A'XB'YA'X & B'YB'YA'X & A'XA'XB'Y & B'YA'XB'Y & A'XB'YB'Y & B'YB'YB'Y \\
A'XA'XA'X & 0 & 0 & 0 & 0 & 0 & 0 & 0 & 0 \\
B'YA'XA'X & 0 & 0 & 0 & 0 & 0 & 0 & 0 & 0 \\
A'XB'YA'X & 1 & 0 & 0 & 0 & 0 & 0 & 0 & 0 \\
B'YB'YA'X & 0 & 1 & 0 & 0 & 0 & 0 & 0 & 0 \\
A'XA'XB'Y & 0 & 0 & 0 & 0 & 0 & 0 & 0 & 0 \\
B'YA'XB'Y & 0 & 0 & 0 & 0 & 0 & 0 & 0 & 0 \\
A'XB'YB'Y & 0 & 0 & 0 & 0 & 1 & 0 & 0 & 0 \\
B'YB'YB'Y & 0 & 0 & 0 & 0 & 0 & 1 & 0 & 0
}
}
\]
in the middle summand and
\[
\resizebox{\textwidth}{!}{
\kbordermatrix{
& (w_1 C') Z (w_1 C') Z (w_1 C') Z & C' Z (w_1 C') Z (w_1 C') Z & (w_1 C') Z C' Z (w_1 C') Z & C' Z C' Z (w_1 C') Z & (w_1 C') Z (w_1 C') Z C' Z & C' Z (w_1 C') Z C' Z & (w_1 C') Z C' Z C' Z & C' Z C' Z C' Z \\
(w_1 C') Z (w_1 C') Z (w_1 C') Z & 0 & 0 & 0 & 0 & 0 & 0 & 0 & 0 \\
C' Z (w_1 C') Z (w_1 C') Z & 0 & 0 & 0 & 0 & 0 & 0 & 0 & 0 \\
(w_1 C') Z C' Z (w_1 C') Z & 1 & 0 & 0 & 0 & 0 & 0 & 0 & 0 \\
C' Z C' Z (w_1 C') Z & 0 & 1 & 0 & 0 & 0 & 0 & 0 & 0 \\
(w_1 C') Z (w_1 C') Z C' Z & 0 & 0 & 0 & 0 & 0 & 0 & 0 & 0 \\
C' Z (w_1 C') Z C' Z & 0 & 0 & 0 & 0 & 0 & 0 & 0 & 0 \\
(w_1 C') Z C' Z C' Z & 0 & 0 & 0 & 0 & 1 & 0 & 0 & 0 \\
C' Z C' Z C' Z & 0 & 0 & 0 & 0 & 0 & 1 & 0 & 0
}
}
\]
in the lower summand.

\item The 2-morphism \includegraphics[scale=0.4]{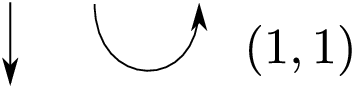} gets assigned the same morphism as in item \eqref{it:Eps'Id}.

\item\label{it:IdEps'Id} The 2-morphism \includegraphics[scale=0.4]{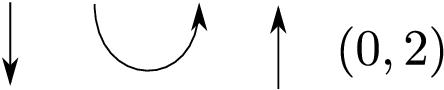} gets assigned the dual of the morphism from $Y \Lambda Y \Lambda$ to $Y \Lambda$ with matrix
\[
\kbordermatrix{
& A'XA'X & B'YA'X & A'XB'Y & B'YB'Y \\
A'X & e_1 & 1 & 1 & 0\\
B'Y & 0 & 0 & 0 & 1
}
\]
in the middle summand and
\[
\kbordermatrix{
& (w_1 C') Z (w_1 C') Z & C' Z (w_1 C') Z & (w_1 C') Z C' Z & C' Z C' Z\\
(w_1 C') Z & e_1 & 1 & 1 & 0 \\
C' Z & e_2 & 0 & 0 & 1
}
\]
in the lower summand. 

\item\label{it:ComplicatedCup} The 2-morphism \includegraphics[scale=0.4]{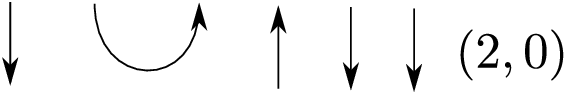} gets assigned the dual of the morphism from $Y \Lambda Y \Lambda Y \Lambda$ to $Y \Lambda Y \Lambda$ with matrix
\[
\resizebox{\textwidth}{!}{
\kbordermatrix{
& A'XA'XA'X & B'YA'XA'X & A'XB'YA'X & B'YB'YA'X & A'XA'XB'Y & B'YA'XB'Y & A'XB'YB'Y & B'YB'YB'Y \\
A'XA'X & e_1 & 0 & 1 & 0 & 1 & 0 & 0 & 0 \\
B'YA'X & 0 & e_1 & 0 & 1 & 0 & 1 & 0 & 0 \\
A'XB'Y & 0 & 0 & 0 & 0 & 0 & 0 & 1 & 0 \\
B'YB'Y & 0 & 0 & 0 & 0 & 0 & 0 & 0 & 1
}
}
\]
in the middle summand and
\[
\resizebox{\textwidth}{!}{
\kbordermatrix{
& (w_1 C') Z (w_1 C') Z (w_1 C') Z & C' Z (w_1 C') Z (w_1 C') Z & (w_1 C') Z C' Z (w_1 C') Z & C' Z C' Z (w_1 C') Z & (w_1 C') Z (w_1 C') Z C' Z & C' Z (w_1 C') Z C' Z & (w_1 C') Z C' Z C' Z & C' Z C' Z C' Z \\
(w_1 C') Z (w_1 C') Z & e_1 & 0 & 1 & 0 & 1 & 0 & 0 & 0 \\
C' Z (w_1 C') Z & 0 & e_1 & 0 & 1 & 0 & 1 & 0 & 0 \\
(w_1 C') Z C' Z & e_2 & 0 & 0 & 0 & 0 & 0 & 1 & 0 \\
C' Z C' Z & 0 & e_2 & 0 & 0 & 0 & 0 & 0 & 1
}
}
\]
in the lower summand.

\item\label{it:CupDownDown} The 2-morphism \includegraphics[scale=0.4]{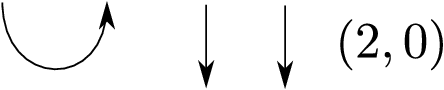} gets assigned the dual of the morphism from $Y \Lambda Y \Lambda$ to $Y \Lambda$ with matrix
\[
\kbordermatrix{
& A'XA'X & B'YA'X & A'XB'Y & B'YB'Y \\
A'X & 1 & 0 & 0 & 0 \\
B'Y & 0 & 1 & 0 & 0
}
\]
in the middle summand and
\[
\kbordermatrix{
& (w_1 C') Z (w_1 C') Z & C' Z (w_1 C') Z & (w_1 C') Z C' Z & C' Z C' Z\\
(w_1 C') Z & 1 & 0 & 0 & 0 \\
C' Z & 0 & 1 & 0 & 0
}
\]
in the lower summand. 

\item\label{it:CapDownDown} The 2-morphism \includegraphics[scale=0.4]{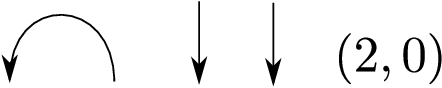} gets assigned the dual of the morphism from $Y \Lambda$ to $Y \Lambda Y \Lambda$ with matrix
\[
\kbordermatrix{
& A'X & B'Y \\
A'XA'X & 0 & 0 \\
B'YA'X & 0 & 0 \\
A'XB'Y & 1 & 0 \\
B'YB'Y & 0 & 1
}
\]
in the middle summand and 
\[
\kbordermatrix{
& (w_1 C')Z & C' Z \\
(w_1 C') Z (w_1 C') Z & 0 & 0 \\
C' Z (w_1 C') Z & 0 & 0 \\
(w_1 C') Z C' Z & 1 & 0 \\
C' Z C' Z & 0 & 1
}
\] 
in the lower summand.

\item The 2-morphism \includegraphics[scale=0.4]{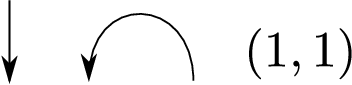} gets assigned the same morphism as in item \eqref{it:EtaId}.

\item The 2-morphism \includegraphics[scale=0.4]{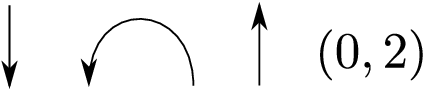} gets assigned the same morphism as in item \eqref{it:IdEtaId}.

\item\label{it:ComplicatedCap2} The 2-morphism \includegraphics[scale=0.4]{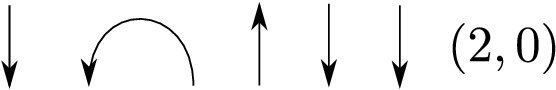} gets assigned the dual of the morphism from $Y \Lambda Y \Lambda$ to $Y \Lambda Y \Lambda Y \Lambda$ with matrix
\[
\kbordermatrix{
& A'XA'X & B'YA'X & A'XB'Y & B'YB'Y \\
A'XA'XA'X & 1 & 0 & 0 & 0 \\
B'YA'XA'X & 0 & 1 & 0 & 0 \\
A'XB'YA'X & 0 & 0 & 1 & 0 \\
B'YB'YA'X & 0 & 0 & 0 & 1 \\
A'XA'XB'Y & 0 & 0 & 1 & 0 \\
B'YA'XB'Y & 0 & 0 & 0 & 1 \\
A'XB'YB'Y & 0 & 0 & e_1 & 0 \\
B'YB'YB'Y & 0 & 0 & 0 & e_1
}
\]
in the middle summand and
\[
\kbordermatrix{
& (w_1 C') Z (w_1 C') Z & C' Z (w_1 C') Z & (w_1 C') Z C' Z & C' Z C' Z\\
(w_1 C') Z (w_1 C') Z (w_1 C') Z & 1 & 0 & 0 & 0 \\
C' Z (w_1 C') Z (w_1 C') Z & 0 & 1 & 0 & 0 \\
(w_1 C') Z C' Z (w_1 C') Z & 0 & 0 & 1 & 0 \\
C' Z C' Z (w_1 C') Z & 0 & 0 & 0 & 1 \\
(w_1 C') Z (w_1 C') Z C' Z & 0 & 0 & 1 & 0 \\
C' Z (w_1 C') Z C' Z & 0 & 0 & 0 & 1 \\
(w_1 C') Z C' Z C' Z & e_2 & 0 & e_1 & 0 \\
C' Z C' Z C' Z & 0 & e_2 & 0 & e_1
}
\]
in the lower summand.

\item The 2-morphism \includegraphics[scale=0.4]{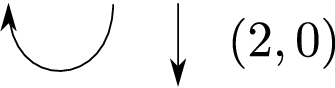} gets assigned the same morphism as in item \eqref{it:IdEps'}.

\item The 2-morphism \includegraphics[scale=0.4]{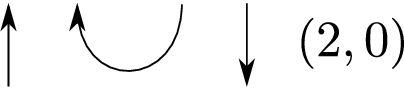} gets assigned the same morphism as in item \eqref{it:IdEps'Id}.

\item\label{it:ComplicatedCup2} The 2-morphism \includegraphics[scale=0.4]{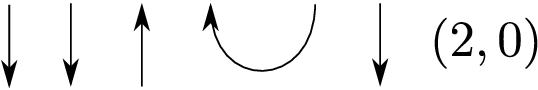} gets assigned the dual of the morphism from $Y \Lambda Y \Lambda Y \Lambda$ to $Y \Lambda Y \Lambda$ with matrix
\[
\resizebox{\textwidth}{!}{
\kbordermatrix{
& A'XA'XA'X & B'YA'XA'X & A'XB'YA'X & B'YB'YA'X & A'XA'XB'Y & B'YA'XB'Y & A'XB'YB'Y & B'YB'YB'Y \\
A'XA'X & e_1 & 1 & 1 & 0 & 0 & 0 & 0 & 0 \\
B'YA'X & 0 & 0 & 0 & 1 & 0 & 0 & 0 & 0 \\
A'XB'Y & 0 & 0 & 0 & 0 & e_1 & 1 & 1 & 0 \\
B'YB'Y & 0 & 0 & 0 & 0 & 0 & 0 & 0 & 1
}
}
\]
in the middle summand and
\[
\resizebox{\textwidth}{!}{
\kbordermatrix{
& (w_1 C') Z (w_1 C') Z (w_1 C') Z & C' Z (w_1 C') Z (w_1 C') Z & (w_1 C') Z C' Z (w_1 C') Z & C' Z C' Z (w_1 C') Z & (w_1 C') Z (w_1 C') Z C' Z & C' Z (w_1 C') Z C' Z & (w_1 C') Z C' Z C' Z & C' Z C' Z C' Z \\
(w_1 C') Z (w_1 C') Z & e_1 & 1 & 1 & 0 & 0 & 0 & 0 & 0 \\
C' Z (w_1 C') Z & e_2 & 0 & 0 & 1 & 0 & 0 & 0 & 0 \\
(w_1 C') Z C' Z & 0 & 0 & 0 & 0 & e_1 & 1 & 1 & 0 \\
C' Z C' Z & 0 & 0 & 0 & 0 & e_2 & 0 & 0 & 1
}
}
\]
in the lower summand.

\item\label{it:DownDownCup} The 2-morphism \includegraphics[scale=0.4]{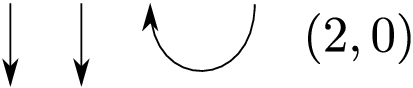} gets assigned the dual of the morphism from $Y \Lambda Y \Lambda$ to $Y \Lambda$ with matrix
\[
\kbordermatrix{
& A'XA'X & B'YA'X & A'XB'Y & B'YB'Y \\
A'X & 1 & 0 & 0 & 0 \\
B'Y & 0 & 0 & 1 & 0
}
\]
in the middle summand and
\[
\kbordermatrix{
& (w_1 C') Z (w_1 C') Z & C' Z (w_1 C') Z & (w_1 C') Z C' Z & C' Z C' Z\\
(w_1 C') Z & 1 & 0 & 0 & 0 \\
C' Z & 0 & 0 & 1 & 0
}
\]
in the lower summand. 

\item\label{it:EtaIdId} The 2-morphism \includegraphics[scale=0.4]{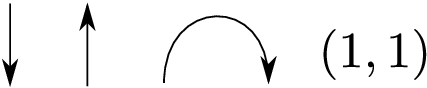} gets assigned the dual of the morphism from $\Lambda Y$ to $\Lambda Y \Lambda Y$ with matrix
\[
\kbordermatrix{
& XA' & YA' & XB' & YB' \\
XA'XA' & 1 & 0 & 0 & 0 \\
YA'XA' & 0 & 1 & 0 & 0 \\
XB'YA' & 0 & 0 & 0 & 0 \\
YB'YA' & 0 & U_2 & 0 & 0 \\
XA'XB' & 0 & 0 & 1 & 0 \\
YA'XB' & 0 & 0 & 0 & 1 \\
XB'YB' & 0 & 0 & 0 & 0 \\
YB'YB' & 0 & 0 & 0 & U_2
}
\]
in the middle summand and
\[
\kbordermatrix{
& Z (w_1 C') & Z C' \\
Z (w_1 C') Z (w_1 C') & 1 & 0 \\
Z C' Z (w_1 C') & U_2 & 0 \\
Z (w_1 C') Z C' & 0 & 1 \\
Z C' Z C' & 0 & U_2
}
\]
in the lower summand.

\item\label{it:IdCrossId} The 2-morphism \includegraphics[scale=0.4]{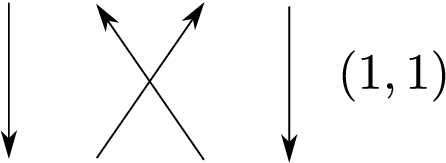} gets assigned the dual of the morphism from $\Lambda Y \Lambda Y$ to $\Lambda Y \Lambda Y$ with matrix
\[
\kbordermatrix{
& XA'XA' & YA'XA' & XB'YA' & YB'YA' & XA'XB' & YA'XB' & XB'YB' & YB'YB' \\
XA'XA' & 0 & 0 & 0 & 0 & 0 & 0 & 0 & 0 \\
YA'XA' & 0 & 0 & 0 & 0 & 0 & 0 & 0 & 0 \\
XB'YA' & 1 & 0 & 0 & 0 & 0 & 0 & 0 & 0 \\
YB'YA' & 0 & 1 & 0 & 0 & 0 & 0 & 0 & 0 \\
XA'XB' & 0 & 0 & 0 & 0 & 0 & 0 & 0 & 0 \\
YA'XB' & 0 & 0 & 0 & 0 & 0 & 0 & 0 & 0 \\
XB'YB' & 0 & 0 & 0 & 0 & 1 & 0 & 0 & 0 \\
YB'YB' & 0 & 0 & 0 & 0 & 0 & 1 & 0 & 0
}
\]
in the middle summand and
\[
\kbordermatrix{
& Z (w_1 C') Z (w_1 C') & Z C' Z (w_1 C') & Z (w_1 C') Z C' & Z C' Z C' \\
Z (w_1 C') Z (w_1 C') & 0 & 0 & 0 & 0 \\
Z C' Z (w_1 C') & 1 & 0 & 0 & 0 \\
Z (w_1 C') Z C' & 0 & 0 & 0 & 0 \\
Z C' Z C' & 0 & 0 & 1 & 0
}
\]
in the lower summand.

\item\label{it:IdIdEps'} The 2-morphism \includegraphics[scale=0.4]{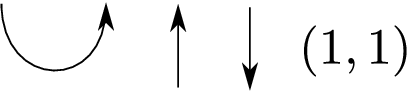} gets assigned the dual of the morphism from $\Lambda Y \Lambda Y$ to $\Lambda Y$ with matrix
\[
\kbordermatrix{
& XA'XA' & YA'XA' & XB'YA' & YB'YA' & XA'XB' & YA'XB' & XB'YB' & YB'YB' \\
XA' & U_1 & 0 & 1 & 0 & 1 \otimes \lambda & 0 & 0 & 0 \\
YA' & 0 & U_2 & 0 & 1 & 0 & 1 \otimes \lambda & 0 & 0 \\
XB' & 0 & 0 & 0 & 0 & 0 & 0 & 1 & 0 \\
YB' & 0 & 0 & 0 & 0 & 0 & 0 & 0 & 1
}
\]
in the middle summand and
\[
\kbordermatrix{
& Z (w_1 C') Z (w_1 C') & Z C' Z (w_1 C') & Z (w_1 C') Z C' & Z C' Z C' \\
Z (w_1 C') & U_1 + U_2 & 1 & 1 & 0 \\
Z C' & U_1 U_2 & 0 & 0 & 1
}
\]
in the lower summand.

\item\label{it:IdIdEta} The 2-morphism \includegraphics[scale=0.4]{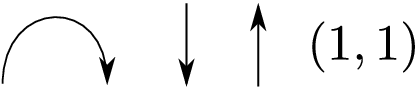} gets assigned the dual of the morphism from $\Lambda Y$ to $\Lambda Y \Lambda Y$ with matrix
\[
\kbordermatrix{
& XA' & YA' & XB' & YB' \\
XA'XA' & 1 & 0 & 0 & 0 \\
YA'XA' & 0 & 1 & 0 & 0 \\
XB'YA' & 0 & 0 & 1 \otimes \lambda & 0 \\
YB'YA' & 0 & 0 & 0 & 1 \otimes \lambda \\
XA'XB' & 0 & 0 & 1 & 0 \\
YA'XB' & 0 & 0 & 0 & 1 \\
XB'YB' & 0 & 0 & U_1 & 0 \\
YB'YB' & 0 & 0 & 0 & U_2
}
\]
in the middle summand and
\[
\kbordermatrix{
& Z (w_1 C') & Z C' \\
Z (w_1 C') Z (w_1 C') & 1 & 0 \\
Z C' Z (w_1 C') & 0 & 1 \\
Z (w_1 C') Z C' & 0 & 1 \\
Z C' Z C' & U_1 U_2 & U_1 + U_2
}
\]
in the lower summand.

\item\label{it:IdDownCrossId} The 2-morphism \includegraphics[scale=0.4]{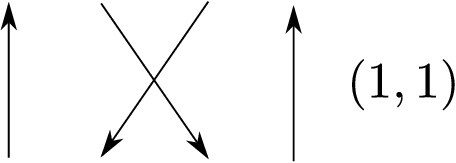} gets assigned the same morphism as in item \eqref{it:IdCrossId}.

\item\label{it:Eps'IdId} The 2-morphism \includegraphics[scale=0.4]{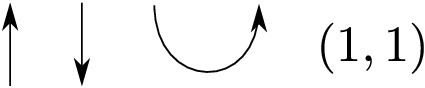} gets assigned the dual of the morphism from $\Lambda Y \Lambda Y$ to $\Lambda Y$ with matrix
\[
\kbordermatrix{
& XA'XA' & YA'XA' & XB'YA' & YB'YA' & XA'XB' & YA'XB' & XB'YB' & YB'YB' \\
XA' & U_1 & 0 & 1 & 0 & 0 & 0 & 0 & 0 \\
YA' & 0 & 0 & 0 & 1 & 0 & 0 & 0 & 0 \\
XB' & 0 & 0 & 0 & 0 & U_1 & 0 & 1 & 0 \\
YB' & 0 & 0 & 0 & 0 & 0 & 0 & 0 & 1
}
\]
in the middle summand and
\[
\kbordermatrix{
& Z (w_1 C') Z (w_1 C') & Z C' Z (w_1 C') & Z (w_1 C') Z C' & Z C' Z C' \\
Z (w_1 C') & U_1 & 1 & 0 & 0 \\
Z C' & 0 & 0 & U_1 & 1
}
\]
in the lower summand.

\item\label{it:DotForNilHecke1} The 2-morphism \includegraphics[scale=0.4]{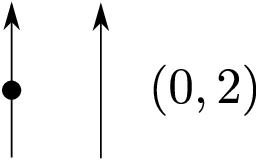} gets assigned the same morphism as in item \eqref{it:IdUpDot}.
 
\item\label{it:DotForNilHecke2} The 2-morphism \includegraphics[scale=0.4]{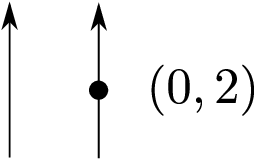} gets assigned the same morphism as in item \eqref{it:UpDotId}.

\end{enumerate}
\end{proposition}

\begin{lemma}
The maps for both types of sideways crossings \includegraphics[scale=0.4]{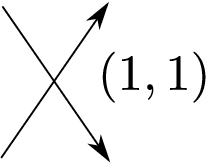} and \includegraphics[scale=0.4]{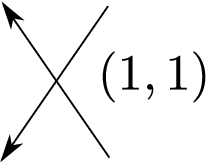} are the identity endomorphism of ${^{\vee}}X$.
\end{lemma}

\begin{proof}
The map for \includegraphics[scale=0.4]{sideways_simple1.eps} is \eqref{it:IdIdEps'}\eqref{it:IdCrossId}\eqref{it:EtaIdId}, which equals
\[
\resizebox{\textwidth}{!}{
\kbordermatrix{
& XA'XA' & YA'XA' & XB'YA' & YB'YA' & XA'XB' & YA'XB' & XB'YB' & YB'YB' \\
XA' & U_1 & 0 & 1 & 0 & 1 \otimes \lambda & 0 & 0 & 0 \\
YA' & 0 & U_2 & 0 & 1 & 0 & 1 \otimes \lambda & 0 & 0 \\
XB' & 0 & 0 & 0 & 0 & 0 & 0 & 1 & 0 \\
YB' & 0 & 0 & 0 & 0 & 0 & 0 & 0 & 1
}
\kbordermatrix{
& XA'XA' & YA'XA' & XB'YA' & YB'YA' & XA'XB' & YA'XB' & XB'YB' & YB'YB' \\
XA'XA' & 0 & 0 & 0 & 0 & 0 & 0 & 0 & 0 \\
YA'XA' & 0 & 0 & 0 & 0 & 0 & 0 & 0 & 0 \\
XB'YA' & 1 & 0 & 0 & 0 & 0 & 0 & 0 & 0 \\
YB'YA' & 0 & 1 & 0 & 0 & 0 & 0 & 0 & 0 \\
XA'XB' & 0 & 0 & 0 & 0 & 0 & 0 & 0 & 0 \\
YA'XB' & 0 & 0 & 0 & 0 & 0 & 0 & 0 & 0 \\
XB'YB' & 0 & 0 & 0 & 0 & 1 & 0 & 0 & 0 \\
YB'YB' & 0 & 0 & 0 & 0 & 0 & 1 & 0 & 0
}
\kbordermatrix{
& XA' & YA' & XB' & YB' \\
XA'XA' & 1 & 0 & 0 & 0 \\
YA'XA' & 0 & 1 & 0 & 0 \\
XB'YA' & 0 & 0 & 0 & 0 \\
YB'YA' & 0 & U_2 & 0 & 0 \\
XA'XB' & 0 & 0 & 1 & 0 \\
YA'XB' & 0 & 0 & 0 & 1 \\
XB'YB' & 0 & 0 & 0 & 0 \\
YB'YB' & 0 & 0 & 0 & U_2
}
}
\]
in the middle summand and
\[
\resizebox{\textwidth}{!}{
\kbordermatrix{
& Z (w_1 C') Z (w_1 C') & Z C' Z (w_1 C') & Z (w_1 C') Z C' & Z C' Z C' \\
Z (w_1 C') & U_1 + U_2 & 1 & 1 & 0 \\
Z C' & U_1 U_2 & 0 & 0 & 1
}
\kbordermatrix{
& Z (w_1 C') Z (w_1 C') & Z C' Z (w_1 C') & Z (w_1 C') Z C' & Z C' Z C' \\
Z (w_1 C') Z (w_1 C') & 0 & 0 & 0 & 0 \\
Z C' Z (w_1 C') & 1 & 0 & 0 & 0 \\
Z (w_1 C') Z C' & 0 & 0 & 0 & 0 \\
Z C' Z C' & 0 & 0 & 1 & 0
}
\kbordermatrix{
& Z (w_1 C') & Z C' \\
Z (w_1 C') Z (w_1 C') & 1 & 0 \\
Z C' Z (w_1 C') & U_2 & 0 \\
Z (w_1 C') Z C' & 0 & 1 \\
Z C' Z C' & 0 & U_2
}
}
\]
in the lower summand.

The map for \includegraphics[scale=0.4]{sideways_simple2.eps} is \eqref{it:Eps'IdId}\eqref{it:IdDownCrossId}\eqref{it:IdIdEta} or equivalently \eqref{it:Eps'IdId}\eqref{it:IdCrossId}\eqref{it:IdIdEta}, which equals
\[
\resizebox{\textwidth}{!}{
\kbordermatrix{
& XA'XA' & YA'XA' & XB'YA' & YB'YA' & XA'XB' & YA'XB' & XB'YB' & YB'YB' \\
XA' & U_1 & 0 & 1 & 0 & 0 & 0 & 0 & 0 \\
YA' & 0 & 0 & 0 & 1 & 0 & 0 & 0 & 0 \\
XB' & 0 & 0 & 0 & 0 & U_1 & 0 & 1 & 0 \\
YB' & 0 & 0 & 0 & 0 & 0 & 0 & 0 & 1
}
\kbordermatrix{
& XA'XA' & YA'XA' & XB'YA' & YB'YA' & XA'XB' & YA'XB' & XB'YB' & YB'YB' \\
XA'XA' & 0 & 0 & 0 & 0 & 0 & 0 & 0 & 0 \\
YA'XA' & 0 & 0 & 0 & 0 & 0 & 0 & 0 & 0 \\
XB'YA' & 1 & 0 & 0 & 0 & 0 & 0 & 0 & 0 \\
YB'YA' & 0 & 1 & 0 & 0 & 0 & 0 & 0 & 0 \\
XA'XB' & 0 & 0 & 0 & 0 & 0 & 0 & 0 & 0 \\
YA'XB' & 0 & 0 & 0 & 0 & 0 & 0 & 0 & 0 \\
XB'YB' & 0 & 0 & 0 & 0 & 1 & 0 & 0 & 0 \\
YB'YB' & 0 & 0 & 0 & 0 & 0 & 1 & 0 & 0
}
\kbordermatrix{
& XA' & YA' & XB' & YB' \\
XA'XA' & 1 & 0 & 0 & 0 \\
YA'XA' & 0 & 1 & 0 & 0 \\
XB'YA' & 0 & 0 & 1 \otimes \lambda & 0 \\
YB'YA' & 0 & 0 & 0 & 1 \otimes \lambda \\
XA'XB' & 0 & 0 & 1 & 0 \\
YA'XB' & 0 & 0 & 0 & 1 \\
XB'YB' & 0 & 0 & U_1 & 0 \\
YB'YB' & 0 & 0 & 0 & U_2
}
}
\]
in the middle summand and
\[
\resizebox{\textwidth}{!}{
\kbordermatrix{
& Z (w_1 C') Z (w_1 C') & Z C' Z (w_1 C') & Z (w_1 C') Z C' & Z C' Z C' \\
Z (w_1 C') & U_1 & 1 & 0 & 0 \\
Z C' & 0 & 0 & U_1 & 1
}
\kbordermatrix{
& Z (w_1 C') Z (w_1 C') & Z C' Z (w_1 C') & Z (w_1 C') Z C' & Z C' Z C' \\
Z (w_1 C') Z (w_1 C') & 0 & 0 & 0 & 0 \\
Z C' Z (w_1 C') & 1 & 0 & 0 & 0 \\
Z (w_1 C') Z C' & 0 & 0 & 0 & 0 \\
Z C' Z C' & 0 & 0 & 1 & 0
}
\kbordermatrix{
& Z (w_1 C') & Z C' \\
Z (w_1 C') Z (w_1 C') & 1 & 0 \\
Z C' Z (w_1 C') & 0 & 1 \\
Z (w_1 C') Z C' & 0 & 1 \\
Z C' Z C' & U_1 U_2 & U_1 + U_2
}
}
\]
in the lower summand.
\end{proof}

\begin{theorem}\label{thm:SkewHowe2Action}
The relations in $\Sc(2,2)^{*,*}$ hold for the maps defined above, so we have a functor of bicategories from $\Sc(2,2)^{*,*}$ to $2\Rep(\U^-)^{*,*}$.
\end{theorem}

\begin{proof}
We need to check the relations \eqref{it:Biadjointness1}--\eqref{it:NilHecke2} of Section~\ref{sec:CatQGDef}.
\begin{itemize}
\item For relation \eqref{it:Biadjointness1}, we want \eqref{it:EpsId}\eqref{it:IdEta} $= \id =$ \eqref{it:IdEps'}\eqref{it:Eta'Id} in the labeling of Proposition~\ref{prop:HelperMorphisms}. Indeed, the matrix products
\[
\kbordermatrix{
& A'XA' & B'YA' & A'XB' & B'YB' \\
A' & 1 & 0 & 0 & 0 \\
B' & 0 & 0 & 1 & 0
}
\kbordermatrix{
& A' & B' \\
A'XA' & 1 & 0 \\
B'YA' & 0 & 1 \otimes \lambda \\
A'XB' & 0 & 1 \\
B'YB' & 0 & e_1
}
\]
and
\[
\kbordermatrix{
& (w_1 C')Z(w_1 C') & C' Z (w_1 C') & (w_1 C')Z C' & C' Z C' \\
w_1 C' & 1 & 0 & 0 & 0 \\
C' & 0 & 0 & 1 & 0
}
\kbordermatrix{
& w_1 C' & C' \\
(w_1 C')Z(w_1 C') & 1 & 0 \\
C' Z (w_1 C') & 0 & 1 \\
(w_1 C')Z C' & 0 & 1 \\
C' Z C' & e_2 & e_1
}
\]
are both the identity, so \eqref{it:EpsId}\eqref{it:IdEta} $= \id$, and the products
\[
\kbordermatrix{
& A'XA' & B'YA' & A'XB' & B'YB' \\
A' & e_1 & 1 & 1 \otimes \lambda & 0 \\
B' & 0 & 0 & 0 & 1
}
\kbordermatrix{
& A' & B' \\
A'XA' & 0 & 0 \\
B'YA' & 1 & 0 \\
A'XB' & 0 & 0 \\
B'YB' & 0 & 1
}
\]
and
\[
\kbordermatrix{
& (w_1 C')Z(w_1 C') & C' Z (w_1 C') & (w_1 C')Z C' & C' Z C' \\
w_1 C' & e_1 & 1 & 1 & 0\\
C' & e_2 & 0 & 0 & 1
}
\kbordermatrix{
& w_1 C' & C' \\
(w_1 C')Z(w_1 C') & 0 & 0 \\
C' Z (w_1 C') & 1 & 0 \\
(w_1 C')Z C' & 0 & 0 \\
C' Z C' & 0 & 1
}
\]
are both the identity, so $\id =$ \eqref{it:IdEps'}\eqref{it:Eta'Id}.

\item For relation \eqref{it:Biadjointness2}, we want \eqref{it:Eps'Id}\eqref{it:IdEta'} $=\id=$ \eqref{it:IdEps}\eqref{it:EtaId}. Indeed, the matrix products 
\[
\kbordermatrix{
& X A' X & Y A' X & X B' Y & Y B' Y \\
X & U_1 & 0 & 1 & 0 \\
Y & 0 & 0 & 0 & 1
}
\kbordermatrix{
& X & Y \\
X A' X & 0 & 0 \\
Y A' X & 0 & 0 \\
X B' Y & 1 & 0 \\
Y B' Y & 0 & 1
}
\]
and
\[
\kbordermatrix{
& Z(w_1 C')Z & ZC' Z \\
Z & U_1 & 1
}
\kbordermatrix{
& Z \\
Z(w_1 C')Z & 0 \\
ZC' Z & 1
}
\]
are both the identity, so \eqref{it:Eps'Id}\eqref{it:IdEta'} $=\id$, and the matrix products
\[
\kbordermatrix{
& X A' X & Y A' X & X B' Y & Y B' Y \\
X & 1 & 0 & 0 & 0 \\
Y & 0 & 1 & 0 & 0
}
\kbordermatrix{
& X & Y \\
X A' X & 1 & 0 \\
Y A' X & 0 & 1 \\
X B' Y & 0 & 0 \\
Y B' Y & 0 & U_2
}
\]
and
\[
\kbordermatrix{
& Z(w_1 C')Z & ZC' Z \\
Z & 1 & 0
}
\kbordermatrix{
& Z \\
Z(w_1 C')Z & 1 \\
ZC' Z & U_2
}
\]
are both the identity, so $\id=$ \eqref{it:IdEps}\eqref{it:EtaId}.

\item For relation \eqref{it:Biadjointness3}, we want \eqref{it:Eps'IdAgain}\eqref{it:IdEta'Again} $=\id=$ \eqref{it:IdEpsAgain}\eqref{it:EtaIdAgain}; this amounts to \eqref{it:Eps'Id}\eqref{it:IdEta'} $=\id=$ \eqref{it:IdEps}\eqref{it:EtaId} which was proved above.

\item For relation \eqref{it:Biadjointness4}, we want \eqref{it:EpsIdAgain}\eqref{it:IdEtaAgain} $=\id=$ \eqref{it:IdEps'Again}\eqref{it:Eta'IdAgain}; this amounts to \eqref{it:EpsId}\eqref{it:IdEta} $=\id=$ \eqref{it:IdEps'}\eqref{it:Eta'Id} which was proved above.

\item For relation \eqref{it:DotCyclic1}, we want \eqref{it:Eps'IdAgain}\eqref{it:Extra}\eqref{it:IdEta'Again} $=\delta^{\down}_{1,1}=$ \eqref{it:IdEpsAgain}\eqref{it:Extra}\eqref{it:EtaIdAgain}; this amounts to \eqref{it:Eps'Id}\eqref{it:Extra}\eqref{it:IdEta'} $=\delta^{\down}_{1,1}=$ \eqref{it:IdEps}\eqref{it:Extra}\eqref{it:EtaId}. Indeed, the matrix products
\[
\kbordermatrix{
& X A' X & Y A' X & X B' Y & Y B' Y \\
X & U_1 & 0 & 1 & 0 \\
Y & 0 & 0 & 0 & 1
}
\kbordermatrix{
& X A' X & Y A' X & X B' Y & Y B' Y \\
X A' X & U_1 & 0 & 1 & 0 \\
Y A' X & 0 & U_2 & 0 & 1 \\
X B' Y & 0 & 0 & 0 & 0 \\
Y B' Y & 0 & 0 & 0 & 0
}
\kbordermatrix{
& X & Y \\
X A' X & 0 & 0 \\
Y A' X & 0 & 0 \\
X B' Y & 1 & 0 \\
Y B' Y & 0 & 1
}
\]
and
\[
\kbordermatrix{
& Z(w_1 C')Z & ZC' Z \\
Z & U_1 & 1
}
\kbordermatrix{
& Z(w_1 C')Z & ZC' Z \\
Z(w_1 C')Z & U_1 + U_2 & 1 \\
ZC' Z & U_1 U_2 & 0
}
\kbordermatrix{
& Z \\
Z(w_1 C')Z & 0 \\
ZC' Z & 1
}
\]
are equal to the two summands of $\delta^{\down}_{1,1}$, while the matrix products
\[
\kbordermatrix{
& X A' X & Y A' X & X B' Y & Y B' Y \\
X & 1 & 0 & 0 & 0 \\
Y & 0 & 1 & 0 & 0
}
\kbordermatrix{
& X A' X & Y A' X & X B' Y & Y B' Y \\
X A' X & U_1 & 0 & 1 & 0 \\
Y A' X & 0 & U_2 & 0 & 1 \\
X B' Y & 0 & 0 & 0 & 0 \\
Y B' Y & 0 & 0 & 0 & 0
}
\kbordermatrix{
& X & Y \\
X A' X & 1 & 0 \\
Y A' X & 0 & 1 \\
X B' Y & 0 & 0 \\
Y B' Y & 0 & U_2
}
\]
and
\[
\kbordermatrix{
& Z(w_1 C')Z & ZC' Z \\
Z & 1 & 0
}
\kbordermatrix{
& Z(w_1 C')Z & ZC' Z \\
Z(w_1 C')Z & U_1 + U_2 & 1 \\
ZC' Z & U_1 U_2 & 0
}
\kbordermatrix{
& Z \\
Z(w_1 C')Z & 1 \\
ZC' Z & U_2
}
\]
are also equal to the two summands of $\delta^{\down}_{1,1}$.

\item For relation \eqref{it:DotCyclic2}, we want \eqref{it:EpsIdAgain}\eqref{it:IdUpDotId}\eqref{it:IdEtaAgain} $=\delta^{\down}_{2,0}=$ \eqref{it:IdEps'Again}\eqref{it:IdUpDotId}\eqref{it:Eta'IdAgain}, or equivalently \eqref{it:EpsId}\eqref{it:IdUpDotId}\eqref{it:IdEta} $=\delta^{\down}_{2,0}=$ \eqref{it:IdEps'}\eqref{it:IdUpDotId}\eqref{it:Eta'Id}. Indeed, the matrix products
\[
\kbordermatrix{
& A'XA' & B'YA' & A'XB' & B'YB' \\
A' & 1 & 0 & 0 & 0 \\
B' & 0 & 0 & 1 & 0
}
\kbordermatrix{
& A'XA' & B'YA' & A'XB' & B'YB' \\
A'XA' & 0 & 1 & 0 & 0 \\
B'YA' & 0 & e_1 & 0 & 0 \\
A'XB' & 0 & 0 & 0 & 1 \\
B'YB' & 0 & 0 & 0 & e_1
}
\kbordermatrix{
& A' & B' \\
A'XA' & 1 & 0 \\
B'YA' & 0 & 1 \otimes \lambda \\
A'XB' & 0 & 1 \\
B'YB' & 0 & e_1
}
\]
and
\[
\resizebox{\textwidth}{!}{
\kbordermatrix{
& (w_1 C')Z(w_1 C') & C' Z (w_1 C') & (w_1 C')Z C' & C' Z C' \\
w_1 C' & 1 & 0 & 0 & 0 \\
C' & 0 & 0 & 1 & 0
}
\kbordermatrix{
& (w_1 C')Z(w_1 C') & C' Z (w_1 C') & (w_1 C')Z C' & C' Z C' \\
(w_1 C')Z(w_1 C') & 0 & 1 & 0 & 0 \\
C' Z (w_1 C') & e_2 & e_1 & 0 & 0 \\
(w_1 C')Z C' & 0 & 0 & 0 & 1 \\
C' Z C' & 0 & 0 & e_2 & e_1
}
\kbordermatrix{
& w_1 C' & C' \\
(w_1 C')Z(w_1 C') & 1 & 0 \\
C' Z (w_1 C') & 0 & 1 \\
(w_1 C')Z C' & 0 & 1 \\
C' Z C' & e_2 & e_1
}
}
\]
are equal to the two summands of $\delta^{\down}_{2,0}$, while the matrix products
\[
\kbordermatrix{
& A'XA' & B'YA' & A'XB' & B'YB' \\
A' & e_1 & 1 & 1 \otimes \lambda & 0 \\
B' & 0 & 0 & 0 & 1
}
\kbordermatrix{
& A'XA' & B'YA' & A'XB' & B'YB' \\
A'XA' & 0 & 1 & 0 & 0 \\
B'YA' & 0 & e_1 & 0 & 0 \\
A'XB' & 0 & 0 & 0 & 1 \\
B'YB' & 0 & 0 & 0 & e_1
}
\kbordermatrix{
& A' & B' \\
A'XA' & 0 & 0 \\
B'YA' & 1 & 0 \\
A'XB' & 0 & 0 \\
B'YB' & 0 & 1
}
\]
and
\[
\resizebox{\textwidth}{!}{
\kbordermatrix{
& (w_1 C')Z(w_1 C') & C' Z (w_1 C') & (w_1 C')Z C' & C' Z C' \\
w_1 C' & e_1 & 1 & 1 & 0\\
C' & e_2 & 0 & 0 & 1
}
\kbordermatrix{
& (w_1 C')Z(w_1 C') & C' Z (w_1 C') & (w_1 C')Z C' & C' Z C' \\
(w_1 C')Z(w_1 C') & 0 & 1 & 0 & 0 \\
C' Z (w_1 C') & e_2 & e_1 & 0 & 0 \\
(w_1 C')Z C' & 0 & 0 & 0 & 1 \\
C' Z C' & 0 & 0 & e_2 & e_1
}
\kbordermatrix{
& w_1 C' & C' \\
(w_1 C')Z(w_1 C') & 0 & 0 \\
C' Z (w_1 C') & 1 & 0 \\
(w_1 C')Z C' & 0 & 0 \\
C' Z C' & 0 & 1
}
}
\]
are also equal to the two summands of $\delta^{\down}_{2,0}$.

\item For relation \eqref{it:CrossingCyclic}, we want \eqref{it:CupDownDown}\eqref{it:ComplicatedCup}\eqref{it:BigCrossing}\eqref{it:ComplicatedCap}\eqref{it:DownDownCap} $=\chi=$ \eqref{it:DownDownCup}\eqref{it:ComplicatedCup2}\eqref{it:BigCrossing}\eqref{it:ComplicatedCap2}\eqref{it:CapDownDown}. Indeed, the matrix products
\[
\resizebox{\textwidth}{!}{
\kbordermatrix{
& A'XA'X & B'YA'X & A'XB'Y & B'YB'Y \\
A'X & 1 & 0 & 0 & 0 \\
B'Y & 0 & 1 & 0 & 0
}
\kbordermatrix{
& A'XA'XA'X & B'YA'XA'X & A'XB'YA'X & B'YB'YA'X & A'XA'XB'Y & B'YA'XB'Y & A'XB'YB'Y & B'YB'YB'Y \\
A'XA'X & e_1 & 0 & 1 & 0 & 1 & 0 & 0 & 0 \\
B'YA'X & 0 & e_1 & 0 & 1 & 0 & 1 & 0 & 0 \\
A'XB'Y & 0 & 0 & 0 & 0 & 0 & 0 & 1 & 0 \\
B'YB'Y & 0 & 0 & 0 & 0 & 0 & 0 & 0 & 1
}
\kbordermatrix{
& A'XA'XA'X & B'YA'XA'X & A'XB'YA'X & B'YB'YA'X & A'XA'XB'Y & B'YA'XB'Y & A'XB'YB'Y & B'YB'YB'Y \\
A'XA'XA'X & 0 & 0 & 0 & 0 & 0 & 0 & 0 & 0 \\
B'YA'XA'X & 0 & 0 & 0 & 0 & 0 & 0 & 0 & 0 \\
A'XB'YA'X & 1 & 0 & 0 & 0 & 0 & 0 & 0 & 0 \\
B'YB'YA'X & 0 & 1 & 0 & 0 & 0 & 0 & 0 & 0 \\
A'XA'XB'Y & 0 & 0 & 0 & 0 & 0 & 0 & 0 & 0 \\
B'YA'XB'Y & 0 & 0 & 0 & 0 & 0 & 0 & 0 & 0 \\
A'XB'YB'Y & 0 & 0 & 0 & 0 & 1 & 0 & 0 & 0 \\
B'YB'YB'Y & 0 & 0 & 0 & 0 & 0 & 1 & 0 & 0
}
\kbordermatrix{
& A'XA'X & B'YA'X & A'XB'Y & B'YB'Y \\
A'XA'XA'X & 1 & 0 & 0 & 0 \\
B'YA'XA'X & 0 & 1 & 0 & 0 \\
A'XB'YA'X & 0 & 1 & 0 & 0 \\
B'YB'YA'X & 0 & e_1 & 0 & 0 \\
A'XA'XB'Y & 0 & 0 & 1 & 0 \\
B'YA'XB'Y & 0 & 0 & 0 & 1 \\
A'XB'YB'Y & 0 & 0 & 0 & 1 \\
B'YB'YB'Y & 0 & 0 & 0 & e_1
}
\kbordermatrix{
& A'X & B'Y \\
A'XA'X & 0 & 0 \\
B'YA'X & 1 & 0 \\
A'XB'Y & 0 & 0 \\
B'YB'Y & 0 & 1
}
}
\]
and
\[
\resizebox{\textwidth}{!}{
\kbordermatrix{
& (w_1 C') Z (w_1 C') Z & C' Z (w_1 C') Z & (w_1 C') Z C' Z & C' Z C' Z\\
(w_1 C') Z & 1 & 0 & 0 & 0 \\
C' Z & 0 & 1 & 0 & 0
}
\kbordermatrix{
& (w_1 C') Z (w_1 C') Z (w_1 C') Z & C' Z (w_1 C') Z (w_1 C') Z & (w_1 C') Z C' Z (w_1 C') Z & C' Z C' Z (w_1 C') Z & (w_1 C') Z (w_1 C') Z C' Z & C' Z (w_1 C') Z C' Z & (w_1 C') Z C' Z C' Z & C' Z C' Z C' Z \\
(w_1 C') Z (w_1 C') Z & e_1 & 0 & 1 & 0 & 1 & 0 & 0 & 0 \\
C' Z (w_1 C') Z & 0 & e_1 & 0 & 1 & 0 & 1 & 0 & 0 \\
(w_1 C') Z C' Z & e_2 & 0 & 0 & 0 & 0 & 0 & 1 & 0 \\
C' Z C' Z & 0 & e_2 & 0 & 0 & 0 & 0 & 0 & 1
}
\kbordermatrix{
& (w_1 C') Z (w_1 C') Z (w_1 C') Z & C' Z (w_1 C') Z (w_1 C') Z & (w_1 C') Z C' Z (w_1 C') Z & C' Z C' Z (w_1 C') Z & (w_1 C') Z (w_1 C') Z C' Z & C' Z (w_1 C') Z C' Z & (w_1 C') Z C' Z C' Z & C' Z C' Z C' Z \\
(w_1 C') Z (w_1 C') Z (w_1 C') Z & 0 & 0 & 0 & 0 & 0 & 0 & 0 & 0 \\
C' Z (w_1 C') Z (w_1 C') Z & 0 & 0 & 0 & 0 & 0 & 0 & 0 & 0 \\
(w_1 C') Z C' Z (w_1 C') Z & 1 & 0 & 0 & 0 & 0 & 0 & 0 & 0 \\
C' Z C' Z (w_1 C') Z & 0 & 1 & 0 & 0 & 0 & 0 & 0 & 0 \\
(w_1 C') Z (w_1 C') Z C' Z & 0 & 0 & 0 & 0 & 0 & 0 & 0 & 0 \\
C' Z (w_1 C') Z C' Z & 0 & 0 & 0 & 0 & 0 & 0 & 0 & 0 \\
(w_1 C') Z C' Z C' Z & 0 & 0 & 0 & 0 & 1 & 0 & 0 & 0 \\
C' Z C' Z C' Z & 0 & 0 & 0 & 0 & 0 & 1 & 0 & 0
}
\kbordermatrix{
& (w_1 C') Z (w_1 C') Z & C' Z (w_1 C') Z & (w_1 C') Z C' Z & C' Z C' Z\\
(w_1 C') Z (w_1 C') Z (w_1 C') Z & 1 & 0 & 0 & 0 \\
C' Z (w_1 C') Z (w_1 C') Z & 0 & 1 & 0 & 0 \\
(w_1 C') Z C' Z (w_1 C') Z & 0 & 1 & 0 & 0 \\
C' Z C' Z (w_1 C') Z & e_2 & e_1 & 0 & 0 \\
(w_1 C') Z (w_1 C') Z C' Z & 0 & 0 & 1 & 0 \\
C' Z (w_1 C') Z C' Z & 0 & 0 & 0 & 1\\
(w_1 C') Z C' Z C' Z & 0 & 0 & 0 & 1 \\
C' Z C' Z C' Z & 0 & 0 & e_2 & e_1
}
\kbordermatrix{
& (w_1 C')Z & C' Z \\
(w_1 C') Z (w_1 C') Z & 0 & 0 \\
C' Z (w_1 C') Z & 1 & 0 \\
(w_1 C') Z C' Z & 0 & 0 \\
C' Z C' Z & 0 & 1
}
}
\]
equal the two summands of $\chi$, as do the matrix products
\[
\resizebox{\textwidth}{!}{
\kbordermatrix{
& A'XA'X & B'YA'X & A'XB'Y & B'YB'Y \\
A'X & 1 & 0 & 0 & 0 \\
B'Y & 0 & 0 & 1 & 0
}
\kbordermatrix{
& A'XA'XA'X & B'YA'XA'X & A'XB'YA'X & B'YB'YA'X & A'XA'XB'Y & B'YA'XB'Y & A'XB'YB'Y & B'YB'YB'Y \\
A'XA'X & e_1 & 1 & 1 & 0 & 0 & 0 & 0 & 0 \\
B'YA'X & 0 & 0 & 0 & 1 & 0 & 0 & 0 & 0 \\
A'XB'Y & 0 & 0 & 0 & 0 & e_1 & 1 & 1 & 0 \\
B'YB'Y & 0 & 0 & 0 & 0 & 0 & 0 & 0 & 1
}
\kbordermatrix{
& A'XA'XA'X & B'YA'XA'X & A'XB'YA'X & B'YB'YA'X & A'XA'XB'Y & B'YA'XB'Y & A'XB'YB'Y & B'YB'YB'Y \\
A'XA'XA'X & 0 & 0 & 0 & 0 & 0 & 0 & 0 & 0 \\
B'YA'XA'X & 0 & 0 & 0 & 0 & 0 & 0 & 0 & 0 \\
A'XB'YA'X & 1 & 0 & 0 & 0 & 0 & 0 & 0 & 0 \\
B'YB'YA'X & 0 & 1 & 0 & 0 & 0 & 0 & 0 & 0 \\
A'XA'XB'Y & 0 & 0 & 0 & 0 & 0 & 0 & 0 & 0 \\
B'YA'XB'Y & 0 & 0 & 0 & 0 & 0 & 0 & 0 & 0 \\
A'XB'YB'Y & 0 & 0 & 0 & 0 & 1 & 0 & 0 & 0 \\
B'YB'YB'Y & 0 & 0 & 0 & 0 & 0 & 1 & 0 & 0
}
\kbordermatrix{
& A'XA'X & B'YA'X & A'XB'Y & B'YB'Y \\
A'XA'XA'X & 1 & 0 & 0 & 0 \\
B'YA'XA'X & 0 & 1 & 0 & 0 \\
A'XB'YA'X & 0 & 0 & 1 & 0 \\
B'YB'YA'X & 0 & 0 & 0 & 1 \\
A'XA'XB'Y & 0 & 0 & 1 & 0 \\
B'YA'XB'Y & 0 & 0 & 0 & 1 \\
A'XB'YB'Y & 0 & 0 & e_1 & 0 \\
B'YB'YB'Y & 0 & 0 & 0 & e_1
}
\kbordermatrix{
& A'X & B'Y \\
A'XA'X & 0 & 0 \\
B'YA'X & 0 & 0 \\
A'XB'Y & 1 & 0 \\
B'YB'Y & 0 & 1
}
}
\]
and
\[
\resizebox{\textwidth}{!}{
\kbordermatrix{
& (w_1 C') Z (w_1 C') Z & C' Z (w_1 C') Z & (w_1 C') Z C' Z & C' Z C' Z\\
(w_1 C') Z & 1 & 0 & 0 & 0 \\
C' Z & 0 & 0 & 1 & 0
}
\kbordermatrix{
& (w_1 C') Z (w_1 C') Z (w_1 C') Z & C' Z (w_1 C') Z (w_1 C') Z & (w_1 C') Z C' Z (w_1 C') Z & C' Z C' Z (w_1 C') Z & (w_1 C') Z (w_1 C') Z C' Z & C' Z (w_1 C') Z C' Z & (w_1 C') Z C' Z C' Z & C' Z C' Z C' Z \\
(w_1 C') Z (w_1 C') Z & e_1 & 1 & 1 & 0 & 0 & 0 & 0 & 0 \\
C' Z (w_1 C') Z & e_2 & 0 & 0 & 1 & 0 & 0 & 0 & 0 \\
(w_1 C') Z C' Z & 0 & 0 & 0 & 0 & e_1 & 1 & 1 & 0 \\
C' Z C' Z & 0 & 0 & 0 & 0 & e_2 & 0 & 0 & 1
}
\kbordermatrix{
& (w_1 C') Z (w_1 C') Z (w_1 C') Z & C' Z (w_1 C') Z (w_1 C') Z & (w_1 C') Z C' Z (w_1 C') Z & C' Z C' Z (w_1 C') Z & (w_1 C') Z (w_1 C') Z C' Z & C' Z (w_1 C') Z C' Z & (w_1 C') Z C' Z C' Z & C' Z C' Z C' Z \\
(w_1 C') Z (w_1 C') Z (w_1 C') Z & 0 & 0 & 0 & 0 & 0 & 0 & 0 & 0 \\
C' Z (w_1 C') Z (w_1 C') Z & 0 & 0 & 0 & 0 & 0 & 0 & 0 & 0 \\
(w_1 C') Z C' Z (w_1 C') Z & 1 & 0 & 0 & 0 & 0 & 0 & 0 & 0 \\
C' Z C' Z (w_1 C') Z & 0 & 1 & 0 & 0 & 0 & 0 & 0 & 0 \\
(w_1 C') Z (w_1 C') Z C' Z & 0 & 0 & 0 & 0 & 0 & 0 & 0 & 0 \\
C' Z (w_1 C') Z C' Z & 0 & 0 & 0 & 0 & 0 & 0 & 0 & 0 \\
(w_1 C') Z C' Z C' Z & 0 & 0 & 0 & 0 & 1 & 0 & 0 & 0 \\
C' Z C' Z C' Z & 0 & 0 & 0 & 0 & 0 & 1 & 0 & 0
}
\kbordermatrix{
& (w_1 C') Z (w_1 C') Z & C' Z (w_1 C') Z & (w_1 C') Z C' Z & C' Z C' Z\\
(w_1 C') Z (w_1 C') Z (w_1 C') Z & 1 & 0 & 0 & 0 \\
C' Z (w_1 C') Z (w_1 C') Z & 0 & 1 & 0 & 0 \\
(w_1 C') Z C' Z (w_1 C') Z & 0 & 0 & 1 & 0 \\
C' Z C' Z (w_1 C') Z & 0 & 0 & 0 & 1 \\
(w_1 C') Z (w_1 C') Z C' Z & 0 & 0 & 1 & 0 \\
C' Z (w_1 C') Z C' Z & 0 & 0 & 0 & 1 \\
(w_1 C') Z C' Z C' Z & e_2 & 0 & e_1 & 0 \\
C' Z C' Z C' Z & 0 & e_2 & 0 & e_1
}
\kbordermatrix{
& (w_1 C')Z & C' Z \\
(w_1 C') Z (w_1 C') Z & 0 & 0 \\
C' Z (w_1 C') Z & 0 & 0 \\
(w_1 C') Z C' Z & 1 & 0 \\
C' Z C' Z & 0 & 1
}.
}
\]

\item For relation \eqref{it:NegativeBubble1}, we want $\varepsilon \eta' = 0$; indeed, the matrix products
\[
\kbordermatrix{
& A'X & B'Y \\
\Ib_{\u} & 1 & 0
}
\kbordermatrix{
& \Ib_{\u} \\
A'X & 0 \\
B'Y & 1
}
\]
and
\[
\kbordermatrix{
& w_1 C' Z & C' Z \\
\Ib_{\o} & 1 & 0
}
\kbordermatrix{
& \Ib_{\o} \\
w_1 C' Z & 0 \\
C' Z & 1
}
\]
are both zero. 

\item For relation \eqref{it:NegativeBubble2}, we also want $\varepsilon \eta' = 0$; this was just shown.

\item For relation \eqref{it:Deg0Bubble1}, we want $\varepsilon$\eqref{it:IdUpDot}$\eta' = \id$; indeed, the matrix products
\[
\kbordermatrix{
& A'X & B'Y \\
\Ib_{\u} & 1 & 0
}
\kbordermatrix{
& A'X & B'Y \\
A'X & 0 & 1 \\
B'Y & 0 & e_1
}
\kbordermatrix{
& \Ib_{\u} \\
A'X & 0 \\
B'Y & 1
}
\]
and
\[
\kbordermatrix{
& w_1 C' Z & C' Z \\
\Ib_{\o} & 1 & 0
}
\kbordermatrix{
& (w_1 C')Z & C' Z \\
(w_1 C')Z & 0 & 1 \\
C' Z & e_2 & e_1
}
\kbordermatrix{
& \Ib_{\o} \\
w_1 C' Z & 0 \\
C' Z & 1
}
\]
are both the identity.

\item For relation \eqref{it:Deg0Bubble2}, we want $\varepsilon$\eqref{it:UpDotId}$\eta' = \id$; indeed, the matrix products
\[
\kbordermatrix{
& A'X & B'Y \\
\Ib_{\u} & 1 & 0
}
\kbordermatrix{
& A'X & B'Y \\
A'X & e_1 & 1 \\
B'Y & 0 & 0
}
\kbordermatrix{
& \Ib_{\u} \\
A'X & 0 \\
B'Y & 1
}
\]
and
\[
\kbordermatrix{
& w_1 C' Z & C' Z \\
\Ib_{\o} & 1 & 0
}
\kbordermatrix{
& (w_1 C') Z & C' Z \\
(w_1 C') Z & e_1 & 1 \\
C' Z & e_2 & 0
}
\kbordermatrix{
& \Ib_{\o} \\
w_1 C' Z & 0 \\
C' Z & 1
}
\]
are both the identity.

\item Relation \eqref{it:ExtendedSl2_1} holds because both sideways crossings give identity morphisms.

\item For relation \eqref{it:ExtendedSl2_2}, the twice-dotted bubble \includegraphics[scale=0.4]{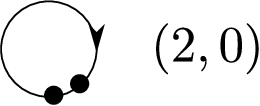} gets assigned $\varepsilon$\eqref{it:IdUpDot}\eqref{it:IdUpDot}$\eta'$ which has matrix
\[
\kbordermatrix{
& A'X & B'Y \\
\Ib_{\u} & 1 & 0
}
\kbordermatrix{
& A'X & B'Y \\
A'X & 0 & 1 \\
B'Y & 0 & e_1
}
\kbordermatrix{
& A'X & B'Y \\
A'X & 0 & 1 \\
B'Y & 0 & e_1
}
\kbordermatrix{
& \Ib_{\u} \\
A'X & 0 \\
B'Y & 1
}
\]
in the middle summand and
\[
\kbordermatrix{
& w_1 C' Z & C' Z \\
\Ib_{\o} & 1 & 0
}
\kbordermatrix{
& (w_1 C')Z & C' Z \\
(w_1 C')Z & 0 & 1 \\
C' Z & e_2 & e_1
}
\kbordermatrix{
& (w_1 C')Z & C' Z \\
(w_1 C')Z & 0 & 1 \\
C' Z & e_2 & e_1
}
\kbordermatrix{
& \Ib_{\o} \\
w_1 C' Z & 0 \\
C' Z & 1
}
\]
in the lower summand, amounting to multiplication by $e_1$ in both cases. Thus, we want to show that $\id = \eta' \varepsilon$\eqref{it:IdUpDot} $+$ \eqref{it:IdUpDot}$\eta' \varepsilon + e_1 \eta' \varepsilon$. Indeed, $\eta' \varepsilon$\eqref{it:IdUpDot} has matrix
\[
\kbordermatrix{
& \Ib_{\u} \\
A'X & 0 \\
B'Y & 1
}
\kbordermatrix{
& A'X & B'Y \\
\Ib_{\u} & 1 & 0
}
\kbordermatrix{
& A'X & B'Y \\
A'X & 0 & 1 \\
B'Y & 0 & e_1
}
\]
in the middle summand and
\[
\kbordermatrix{
& \Ib_{\o} \\
w_1 C' Z & 0 \\
C' Z & 1
}
\kbordermatrix{
& w_1 C' Z & C' Z \\
\Ib_{\o} & 1 & 0
}
\kbordermatrix{
& (w_1 C')Z & C' Z \\
(w_1 C')Z & 0 & 1 \\
C' Z & e_2 & e_1
}
\]
in the lower summand, \eqref{it:IdUpDot}$\eta' \varepsilon$ has matrix
\[
\kbordermatrix{
& A'X & B'Y \\
A'X & 0 & 1 \\
B'Y & 0 & e_1
}
\kbordermatrix{
& \Ib_{\u} \\
A'X & 0 \\
B'Y & 1
}
\kbordermatrix{
& A'X & B'Y \\
\Ib_{\u} & 1 & 0
}
\]
in the middle summand and
\[
\kbordermatrix{
& (w_1 C')Z & C' Z \\
(w_1 C')Z & 0 & 1 \\
C' Z & e_2 & e_1
}
\kbordermatrix{
& \Ib_{\o} \\
w_1 C' Z & 0 \\
C' Z & 1
}
\kbordermatrix{
& w_1 C' Z & C' Z \\
\Ib_{\o} & 1 & 0
}
\]
in the lower summand, and $e_1 \eta' \varepsilon$ has matrix
\[
e_1 \kbordermatrix{
& \Ib_{\u} \\
A'X & 0 \\
B'Y & 1
}
\kbordermatrix{
& A'X & B'Y \\
\Ib_{\u} & 1 & 0
}
\]
in the middle summand and
\[
e_1 \kbordermatrix{
& \Ib_{\o} \\
w_1 C' Z & 0 \\
C' Z & 1
}
\kbordermatrix{
& w_1 C' Z & C' Z \\
\Ib_{\o} & 1 & 0
}
\]
in the lower summand. The matrices in each summand add up to the identity.

\item Relation \eqref{it:ExtendedSl2_3} holds because both sideways crossings give identity morphisms.

\item For relation \eqref{it:ExtendedSl2_4}, the twice-dotted bubble \includegraphics[scale=0.4]{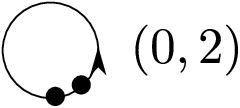} gets assigned $\varepsilon$\eqref{it:UpDotId}\eqref{it:UpDotId}$\eta'$ which has matrix
\[
\kbordermatrix{
& A'X & B'Y \\
\Ib_{\u} & 1 & 0
}
\kbordermatrix{
& A'X & B'Y \\
A'X & e_1 & 1 \\
B'Y & 0 & 0
}
\kbordermatrix{
& A'X & B'Y \\
A'X & e_1 & 1 \\
B'Y & 0 & 0
}
\kbordermatrix{
& \Ib_{\u} \\
A'X & 0 \\
B'Y & 1
}
\]
in the middle summand and
\[
\kbordermatrix{
& w_1 C' Z & C' Z \\
\Ib_{\o} & 1 & 0
}
\kbordermatrix{
& (w_1 C') Z & C' Z \\
(w_1 C') Z & e_1 & 1 \\
C' Z & e_2 & 0
}
\kbordermatrix{
& (w_1 C') Z & C' Z \\
(w_1 C') Z & e_1 & 1 \\
C' Z & e_2 & 0
}
\kbordermatrix{
& \Ib_{\o} \\
w_1 C' Z & 0 \\
C' Z & 1
}
\]
in the lower summand, amounting to multiplication by $e_1$ in both cases. Thus, we want to show that $\id = \eta' \varepsilon$\eqref{it:UpDotId} $+$ \eqref{it:UpDotId}$\eta' \varepsilon + e_1 \eta' \varepsilon$. Indeed, $\eta' \varepsilon$\eqref{it:UpDotId} has matrix
\[
\kbordermatrix{
& \Ib_{\u} \\
A'X & 0 \\
B'Y & 1
}
\kbordermatrix{
& A'X & B'Y \\
\Ib_{\u} & 1 & 0
}
\kbordermatrix{
& A'X & B'Y \\
A'X & e_1 & 1 \\
B'Y & 0 & 0
}
\]
in the middle summand and
\[
\kbordermatrix{
& \Ib_{\o} \\
w_1 C' Z & 0 \\
C' Z & 1
}
\kbordermatrix{
& w_1 C' Z & C' Z \\
\Ib_{\o} & 1 & 0
}
\kbordermatrix{
& (w_1 C')Z & C' Z \\
(w_1 C')Z & e_1 & 1 \\
C' Z & e_2 & 0
}
\]
in the lower summand, \eqref{it:UpDotId}$\eta' \varepsilon$ has matrix
\[
\kbordermatrix{
& A'X & B'Y \\
A'X & e_1 & 1 \\
B'Y & 0 & 0
}
\kbordermatrix{
& \Ib_{\u} \\
A'X & 0 \\
B'Y & 1
}
\kbordermatrix{
& A'X & B'Y \\
\Ib_{\u} & 1 & 0
}
\]
in the middle summand and
\[
\kbordermatrix{
& (w_1 C')Z & C' Z \\
(w_1 C')Z & e_1 & 1 \\
C' Z & e_2 & 0
}
\kbordermatrix{
& \Ib_{\o} \\
w_1 C' Z & 0 \\
C' Z & 1
}
\kbordermatrix{
& w_1 C' Z & C' Z \\
\Ib_{\o} & 1 & 0
}
\]
in the lower summand, and $e_1 \eta' \varepsilon$ has the same matrices as given above. The matrices in each summand add up to the identity.

\item For relation \eqref{it:NilHecke1}, the matrices
\[
\kbordermatrix{
& A'X & B'Y \\
A'X & 0 & 0  \\
B'Y & 1 & 0
}
\kbordermatrix{
& A'X & B'Y \\
A'X & 0 & 0  \\
B'Y & 1 & 0
}
\]
and
\[
\kbordermatrix{
& w_1 C' Z & C' Z \\
w_1 C' Z & 0 & 0  \\
C' Z & 1 & 0
}
\kbordermatrix{
& w_1 C' Z & C' Z \\
w_1 C' Z & 0 & 0  \\
C' Z & 1 & 0
}
\]
are both zero.

\item For relation \eqref{it:NilHecke2}, we want to show that $\id =$ \eqref{it:DotForNilHecke1}$\chi + \chi$\eqref{it:DotForNilHecke2} $=$ \eqref{it:DotForNilHecke2}$\chi + \chi$\eqref{it:DotForNilHecke1}. Equivalently, we want to show that $\id =$ \eqref{it:IdUpDot}$\chi + \chi$\eqref{it:UpDotId} $=$ \eqref{it:UpDotId}$\chi + \chi$\eqref{it:IdUpDot}. Indeed, \eqref{it:IdUpDot}$\chi$ has matrix
\[
\kbordermatrix{
& A'X & B'Y \\
A'X & 0 & 1 \\
B'Y & 0 & e_1
}
\kbordermatrix{
& A'X & B'Y \\
A'X & 0 & 0  \\
B'Y & 1 & 0
}
\]
in the middle summand and
\[
\kbordermatrix{
& (w_1 C')Z & C' Z \\
(w_1 C')Z & 0 & 1 \\
C' Z & e_2 & e_1
}
\kbordermatrix{
& w_1 C' Z & C' Z \\
w_1 C' Z & 0 & 0  \\
C' Z & 1 & 0
}
\]
in the lower summand, while $\chi$\eqref{it:UpDotId} has matrix
\[
\kbordermatrix{
& A'X & B'Y \\
A'X & 0 & 0  \\
B'Y & 1 & 0
}
\kbordermatrix{
& A'X & B'Y \\
A'X & e_1 & 1 \\
B'Y & 0 & 0
}
\]
in the middle summand and
\[
\kbordermatrix{
& w_1 C' Z & C' Z \\
w_1 C' Z & 0 & 0  \\
C' Z & 1 & 0
}
\kbordermatrix{
& (w_1 C') Z & C' Z \\
(w_1 C') Z & e_1 & 1 \\
C' Z & e_2 & 0
}
\]
in the lower summand. Similarly, \eqref{it:UpDotId}$\chi$ has matrix
\[
\kbordermatrix{
& A'X & B'Y \\
A'X & e_1 & 1 \\
B'Y & 0 & 0
}
\kbordermatrix{
& A'X & B'Y \\
A'X & 0 & 0  \\
B'Y & 1 & 0
}
\]
in the middle summand and
\[
\kbordermatrix{
& (w_1 C') Z & C' Z \\
(w_1 C') Z & e_1 & 1 \\
C' Z & e_2 & 0
}
\kbordermatrix{
& w_1 C' Z & C' Z \\
w_1 C' Z & 0 & 0  \\
C' Z & 1 & 0
}
\]
in the lower summand, while $\chi$\eqref{it:IdUpDot} has matrix
\[
\kbordermatrix{
& A'X & B'Y \\
A'X & 0 & 0  \\
B'Y & 1 & 0
}
\kbordermatrix{
& A'X & B'Y \\
A'X & 0 & 1 \\
B'Y & 0 & e_1
}
\]
in the middle summand and
\[
\kbordermatrix{
& w_1 C' Z & C' Z \\
w_1 C' Z & 0 & 0  \\
C' Z & 1 & 0
}
\kbordermatrix{
& (w_1 C')Z & C' Z \\
(w_1 C')Z & 0 & 1 \\
C' Z & e_2 & e_1
}
\]
in the lower summand.

\end{itemize}
\end{proof}

\section{The Soergel category and braid cobordisms}\label{sec:SoergelBraidCob}

\subsection{The Soergel category}\label{sec:Soergel}

Following Elias--Khovanov \cite{EliasKhovanov}, let $\SC_1^{\ungr}$ be the monoidal category denoted $\SC_1(\{1\})$ in Elias--Krasner \cite{EliasKrasner} (equivalently, $\SC_1(2)$ in Mackaay--Sto{\v{s}}i{\'c}--Vaz \cite{MSVSchur}) with its grading ignored. The morphisms in $\SC_1^{\ungr}$ are generated by pictures

\begin{center}
\noindent \includegraphics[scale=0.4]{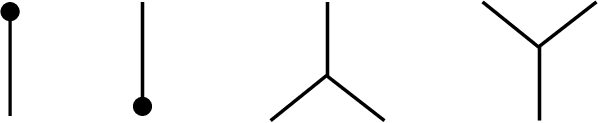}
\end{center}

\noindent modulo certain relations that we will not need to consider explicitly. Composition comes from vertical stacking of pictures, and the monoidal structure comes from horizontal stacking.

\begin{remark}
As Elias--Krasner remark in \cite[Section 5.2]{EliasKrasner}, $\SC_1^{\ungr}$ and its graded versions make sense over $\Z$; we will work with their reductions over $\F_2$. 
\end{remark}

Similarly, let $(\SC'_1)^{\ungr}$ be the monoidal category denoted $\SC'_1(\{1\})$ in \cite{EliasKrasner} (equivalently, $\SC'_1(2)$ in \cite{MSVSchur}) with its grading ignored. Compared with $\SC_1^{\ungr}$, the monoidal category $(\SC'_1)^{\ungr}$ has the same objects and additional generating morphisms $U_1$ and $U_2$ from the monoidal unit object to itself (represented graphically by boxes labeled $1$ or $2$). The additional relations beyond those of $\SC_1^{\ungr}$ are:
\begin{enumerate}
\item\label{it:SCPrimeRel1} \includegraphics[scale=0.4]{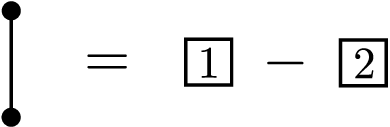}
\item\label{it:SCPrimeRel2} \includegraphics[scale=0.4]{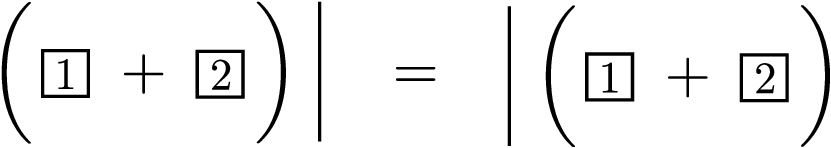}
\item\label{it:SCPrimeRel3} \includegraphics[scale=0.4]{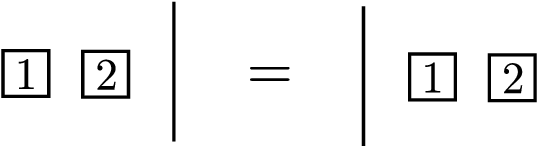}
\end{enumerate}

\begin{remark}
In characteristic two, the categorification relationship between Hecke algebras and these Soergel categories does not necessarily hold; also, the relationship between $\SC'_1$ and $\SC_1$ is more complicated than in characteristic zero. See \cite[Section 5.2]{EliasKrasner} for a more detailed discussion.
\end{remark}

We define a bigraded version $(\SC'_1)^{*,*}$ of $(\SC'_1)^{\ungr}$ by setting
\begin{itemize}
\item for \includegraphics[scale=0.4]{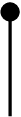}: $\deg^q = -1$, $\deg^h = 2$; \quad for \includegraphics[scale=0.4]{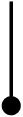}: $\deg^q = -1$, $\deg^h = 0$
\item for \includegraphics[scale=0.4]{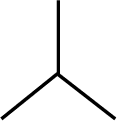}: $\deg^q = 1$, $\deg^h = 0$; \quad for \includegraphics[scale=0.4]{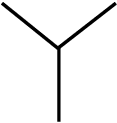}: $\deg^q = 1$, $\deg^h = -2$
\item for \includegraphics[scale=0.4]{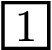} and \includegraphics[scale=0.4]{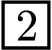}: $\deg^q = -2$ and $\deg^h = 2$.
\end{itemize}
We get a bigraded version $\SC_1^{*,*}$ of $\SC_1^{\ungr}$. The single grading on $\SC_1$ from \cite{EliasKhovanov,EliasKrasner,MSVSchur} is the negative of our $q$-grading.

By \cite[Lemma 6.5]{MSVSchur} we have, in this case and ignoring gradings, a functor from $\SC_1^{\ungr}$ (viewed as a 2-category with one object) to the ungraded version $\Sc(2,2)^{\ungr}$ of the 2-category $\Sc(2,2)^{*,*}$ from Section~\ref{sec:CatQGDef} above. Note that while $\Q$ coefficients are assumed in \cite{MSVSchur}, they are not used in the definition of this functor, which makes sense over $\Z$ and thus over $\F_2$. 

\begin{proposition}
Mackaay--Sto{\v{s}}i{\'c}--Vaz's functor respects the bigradings we define here on $\SC_1^{*,*}$ and $\Sc(2,2)^{*,*}$.
\end{proposition}

\begin{proof}
One can check that the bidegree of each generating 2-morphism of $\SC_1^{*,*}$ agrees with the bidegree of its image in $\Sc(2,2)^{*,*}$.
\end{proof}

From Theorem~\ref{thm:SkewHowe2Action} we get the following corollary.
\begin{corollary}\label{cor:SoergelFunctor}
The above constructions give a functor of bicategories
\[
\SC_1^{*,*} \to 2\Rep(\U^-)^{*,*}.
\]
\end{corollary}

We can extend this functor to the larger domain $(\SC'_1)^{*,*}$. For the 2-morphism \includegraphics[scale=0.4]{soergelDot1.eps} of $(\SC'_1)^{*,*}$, we define an endomorphism of ${^{\vee}}\id_{\A_{1,1}} = \id_{\A_{1,1}}$ to be the dual of multiplication by $U_1$ (which is itself just multiplication by $U_1$), a 2-endomorphism of the identity 1-morphism on $(\A_{1,1},F,\tau)$. Similarly, to the 2-morphism \includegraphics[scale=0.4]{soergelDot2.eps}, we associate multiplication by $U_2$, which is also a 2-endomorphism of the identity 1-morphism. These associations preserve bidegree.

\begin{theorem}
The relations in $(\SC'_1)^{*,*}$ hold for the maps defined above.
\end{theorem}

\begin{proof}
We need to check the relations \eqref{it:SCPrimeRel1}--\eqref{it:SCPrimeRel3} in the definition of $(\SC'_1)^{*,*}$.
\begin{itemize}
\item For relation \eqref{it:SCPrimeRel1}, the 2-morphism \includegraphics[scale=0.4]{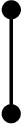} of $(\SC'_1)^{*,*}$ gets sent to the 2-morphism \includegraphics[scale=0.4]{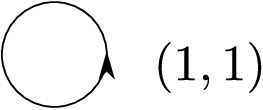} of $\Sc(2,2)^*$. In turn, this 2-morphism gets sent to the dual of $\varepsilon' \eta$, and the matrix of $\varepsilon' \eta$ is
\[
\kbordermatrix{
& XA' & YA' & XB' & YB' \\
\Ib_{\ou} & U_1 & 0 & 1 \otimes \lambda & 0 \\
\Ib_{\uo} & 0 & 0 & 0 & 1
}
\kbordermatrix{
& \Ib_{\ou} & \Ib_{\uo} \\
XA' & 1 & 0 \\
YA' & 0 & 1 \otimes \lambda \\
XB' & 0 & 0 \\
YB' & 0 & U_2
}
= \kbordermatrix{
& \Ib_{\ou} & \Ib_{\uo} \\
\Ib_{\ou} & U_1 & 0 \\
\Ib_{\uo} & 0 & U_2
}
\]
in the middle weight space and
\[
\kbordermatrix{
& Z (w_1 C') & Z C' \\
\Ib_{\oo} & U_1 & 1
}
\kbordermatrix{
& \Ib_{\oo} \\
Z (w_1 C') & 1 \\
Z C' & U_2
}
=
\kbordermatrix{
& \Ib_{\oo} \\
\Ib_{\oo} & U_1 + U_2
}
\]
in the lower weight space.

\item For relation \eqref{it:SCPrimeRel2}, the sum of  \includegraphics[scale=0.4]{soergelDot1.eps} and \includegraphics[scale=0.4]{soergelDot2.eps} gets sent to multiplication by $U_1 + U_2$ which is central in $\A_{1,1}$ (note that $U_2 \lambda = \lambda U_1 = 0$ in the middle summand of $\A_{1,1}$).
\item For relation \eqref{it:SCPrimeRel3}, the product of \includegraphics[scale=0.4]{soergelDot1.eps} and \includegraphics[scale=0.4]{soergelDot2.eps} gets sent to multiplication by $U_1 U_2$ which is central in $\A_{1,1}$.
\end{itemize}
\end{proof}

\subsection{Braid cobordisms}\label{sec:BraidCob}

Let $\BrCob(2)$ denote the monoidal category of two-strand braid cobordisms; objects are diagrams for two-strand braids, and morphisms are generated (over $\F_2$ for us) by movies modulo relations from movie moves. See Elias--Krasner \cite{EliasKrasner} for details. We define a bigraded version of $\BrCob(2)$, denoted $\BrCob(2)^{*,*}$, by assigning the following bidegrees to two-strand movie generators, indexed as in \cite[Section 3]{EliasKrasner}:

\begin{itemize}
\item Birth and death generator 1 (birth of a negative crossing): $\deg^q = -1$, $\deg^h = 0$
\item Birth and death generator 2 (death of a negative crossing): $\deg^q = 1$, $\deg^h = 0$
\item Birth and death generator 3 (birth of a positive crossing): $\deg^q = 1$, $\deg^h = 0$
\item Birth and death generator 4 (death of a positive crossing): $\deg^q = -1$, $\deg^h = 0$
\item Reidemeister 2 generator 1a: $\deg^q = 0$, $\deg^h = 0$
\item Reidemeister 2 generator 1b: $\deg^q = 0$, $\deg^h = 0$
\item Reidemeister 2 generator 2a: $\deg^q = 0$, $\deg^h = 0$
\item Reidemeister 2 generator 2b: $\deg^q = 0$, $\deg^h = 0$.
\end{itemize}

The constructions of \cite{EliasKrasner} give, in this case and ignoring gradings, a monoidal functor from $\BrCob(2)$ to the homotopy category of ungraded complexes in $\SC_1^{\ungr}$. Taking our bigradings into account (and reversing the role of positive and negative crossings in \cite{EliasKrasner} to match our conventions), we define a monoidal functor from $\BrCob^{*,*}$ to the homotopy category of bigraded complexes in $\SC_1^{*,*}$ (where differentials preserve $\deg^q$ and increase $\deg^h$ by one). On objects, we send
\[
\includegraphics[scale=0.4]{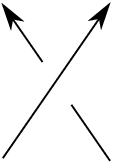} \qquad \qquad \mapsto \qquad \qquad
\xymatrix{
\Bigg| [-1] \ar@/^1.5pc/[rr]^{\includegraphics[scale=0.4]{lp_down.eps}} & \oplus & q \bullet
}
\]
and
\[
\includegraphics[scale=0.4]{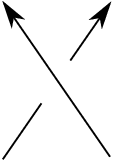} \qquad \qquad \mapsto \qquad \qquad
\xymatrix{
q^{-1} \bullet \ar@/^1.5pc/[rr]^{\includegraphics[scale=0.4]{lp_up.eps}} & \oplus & \Bigg| [-1]
}
\]
lifting the ungraded-complexes version of Elias--Krasner's functor after interchanging positive and negative crossings (we follow the notation of \cite{EliasKrasner} where $\Bigg|$ and $\bullet$ are the objects of $\SC_1$ corresponding to the nonnegative integers $1$ and $0$ respectively; we indicate shifts in $\deg^q$ by powers of $q$ and use $[1]$ for a downward shift by one in $\deg^h$). 

On morphisms, the functor is defined by requiring that it lift the ungraded Elias--Krasner functor after interchanging positive and negative crossings. Specifically:

\begin{itemize}
\item For the birth of a negative crossing, we associate Elias--Krasner's chain map for the birth of a positive crossing and vice-versa.
\item For the death of a negative crossing, we associate Elias--Krasner's chain map for the death of a positive crossing and vice-versa.
\item For the Reidemeister 2 generator 1a, we associate Elias--Krasner's chain map for the Reidemeister 2 generator 2a, and vice-versa.
\item For the Reidemeister 2 generator 1b, we associate Elias--Krasner's chain map for the Reidemeister 2 generator 2b, and vice-versa.
\end{itemize}

\begin{proposition}
The above monoidal functor respects the bigradings on $\BrCob(2)^{*,*}$ and the homotopy category of bigraded complexes in $\SC_1^{*,*}$.
\end{proposition}

\begin{proof}
One can check that the bidegree of each generating morphism of $\BrCob(2)^{*,*}$ is the same as the bidegree of its image in the homotopy category of complexes in $\SC_1^{*,*}$.

\end{proof}

As discussed in Remark~\ref{rem:Strong2Mor}, we do not build enough data into 2-morphisms for the mapping cone on a 2-morphism to carry 1-morphism structure, so we do not have a functor from the homotopy category of bigraded complexes in $\SC_1^{*,*}$ (viewed as a bicategory with one object) into $2\Rep(\U^-)^{*,*}$. However, if we let $\ADBimod^{*,*}$ be the $\Z^2$-graded bicategory whose
\begin{itemize}
\item objects are dg categories, 
\item 1-morphisms are finitely generated right bounded AD bimodules, 
\item 2-morphisms are homotopy classes of closed AD bimodule morphisms (of arbitrary bidegree),
\end{itemize}
then the functor of bicategories $\SC_1^{*,*} \to 2\Rep(\U^-)^{*,*} \xrightarrow{\forget} \ADBimod^{*,*}$ extends naturally to a functor from the homotopy category of complexes in $\SC_1^{*,*}$ (1-morphisms are complexes of 1-morphisms in $\SC_1^{*,*}$ etc.) into $\ADBimod^{*,*}$.

\begin{corollary}\label{cor:BrCob}
The above constructions give a functor of bicategories
\[
\BrCob(2)^{*,*} \to \ADBimod^{*,*}.
\]
\end{corollary}

\begin{remark}\label{rem:UpgradingBrCobFunctor}
The above functor sends a positive crossing to the mapping cone of $\eta$ and a negative crossing to the mapping cone of $\varepsilon'$; the map $\eta$ was given the structure of a strong 2-morphism in Proposition~\ref{prop:FirstEpsIs2MorSquare2} (see Remark~\ref{rem:EtaStrong}) while $\varepsilon'$ is a strong 2-morphism with $h=0$. Thus, by Remark~\ref{rem:Strong2Mor}, we can in fact define 1-morphism structure on the AD bimodules (and their dual DA bimodules) we associate to positive and negative crossings. Using this 1-morphism structure, one could check that the AD bimodule maps we assign to the generating braid cobordisms are 2-morphisms, giving us a functor
\[
\BrCob(2)^{*,*} \to 2\Rep(\U^-)^{*,*}.
\]
\end{remark}

\section{Bimodules for positive and negative crossings}\label{sec:PosNegCrossings}

In this section, we will give a matrix description of (the duals of) the positive and negative crossing bimodules arising from Sections \ref{sec:SkewHowe2Action}, \ref{sec:SoergelBraidCob} and relate them to bimodules coming from Heegaard diagrams.

\subsection{Mapping cone bimodule for positive crossing}\label{sec:MappingConePos}

We write $P' = P'_{\upp} \oplus P'_{\midd} \oplus P'_{\low}$ for the 1-morphism of 2-representations of $\U^-$ associated to a positive crossing by Sections \ref{sec:SkewHowe2Action}, \ref{sec:SoergelBraidCob}. Below we will discuss a simplified version $P''$ of $P'$.

\subsubsection{Upper weight space, positive crossing}

We have $P'_{\upp} = P''_{\upp} = q^{-1} \F_2$ (in homological degree zero) as a bimodule over the upper summand $(\A_{1,1})_{2\varepsilon_1} \cong \F_2$ of $\A_{1,1}$.

\subsubsection{Middle weight space, positive crossing}

The DA bimodule
\[
P'_{\midd} \coloneqq \xymatrix{
q^{-1} \mathbb{I}_{(\A_{1,1})_{\varepsilon_1 + \varepsilon_2}} \ar@/^1.5pc/[rr]^{\eta} & \oplus &  X_{\midd}[1]
}
\] 
over the middle summand $(\A_{1,1})_{\varepsilon_1 + \varepsilon_2}$ of $\A_{1,1}$, where $\eta$ is defined in Section~\ref{sec:BimodMapsFor2Mors}, has primary matrix
\[
\kbordermatrix{
& \ou & \uo \\
\ou & XA' \quad \Ib_{\ou} & XB'\\
\uo & YA' & YB' \quad \Ib_{\uo}
}
\]
and secondary matrix
\[
\resizebox{\textwidth}{!}{
\kbordermatrix{
& XA' & YA' & XB' & YB' & \Ib_{\ou} & \Ib_{\uo} \\
XA' & U_1^{k+1} \otimes U_1^{k+1} & \lambda & \begin{matrix} U_1^k \otimes (\lambda,U_1^{k+1}) \\+ U_1^k \otimes (U_2^{k+1},\lambda)\end{matrix} & 0 & 1 & 0 \\
YA' & 0 & U_2^{k+1} \otimes U_1^{k+1} & 0 & \begin{matrix} U_2^k \otimes (U_2^{k+1},\lambda) \\+ U_2^k \otimes (\lambda,U_1^{k+1})\end{matrix} & 0 & 1 \otimes \lambda \\
XB' & 0 & 0 & U_1^{k+1} \otimes U_2^{k+1} & \lambda & 0 & 0 \\
YB' & 0 & 0 & 0 & U_2^{k+1} \otimes U_2^{k+1} & 0 & U_2 \\
\Ib_{\ou} & 0 & 0 & 0 & 0 & U_1^{k+1} \otimes U_1^{k+1} & \lambda \otimes \lambda \\
\Ib_{\uo} & 0 & 0 & 0 & 0 & 0 & U_2^{k+1} \otimes U_2^{k+1}
}.
}
\]
Simplifying $P'_{\midd}$ using Procedure~\ref{sec:PrelimSimplifying}, we get a DA bimodule $P''_{\midd}$ with primary matrix
\[
\kbordermatrix{
& \ou & \uo \\
\ou & & XB'\\
\uo & YA' & YB' \quad \Ib_{\uo}
}
\]
and secondary matrix
\[
\kbordermatrix{
& YA' & XB' & YB' & \Ib_{\uo} \\
YA' & U_2^{k+1} \otimes U_1^{k+1} & 0 & \begin{matrix} U_2^k \otimes (U_2^{k+1},\lambda) \\+ U_2^k \otimes (\lambda,U_1^{k+1})\end{matrix} & 1 \otimes \lambda \\
XB' & 0 & U_1^{k+1} \otimes U_2^{k+1} & \lambda & 0 \\
YB' & 0 & 0 & U_2^{k+1} \otimes U_2^{k+1} & U_2 \\
\Ib_{\uo} & 0 & 0 & 0 & U_2^{k+1} \otimes U_2^{k+1}
}.
\]
The generators have degrees
\begin{itemize}
\item $\deg^q(YA') = 0$, $\deg^h(YA') = 0$,
\item $\deg^q(XB') = 0$, $\deg^h(XB') = 0$,
\item $\deg^q(YB') = 1$, $\deg^h(YB') = -1$,
\item $\deg^q(\Ib_{\uo}) = -1$, $\deg^h(\Ib_{\uo}) = 0$.
\end{itemize}

\subsubsection{Lower weight space, positive crossing}

The dg bimodule
\[
P'_{\low} \coloneqq \xymatrix{
q^{-1} \mathbb{I}_{(\A_{1,1})_{2\varepsilon_2}} \ar@/^1.5pc/[rr]^{\eta} & \oplus & X_{\low}[1]
}
\] 
over the lower summand $(\A_{1,1})_{2\varepsilon_2}$ of $\A_{1,1}$, where $\eta$ is defined in Section~\ref{sec:BimodMapsFor2Mors}, has primary matrix
\[
\kbordermatrix{
& \oo \\
\oo & Z(w_1 C') \quad ZC' \quad \Ib_{\oo}
}.
\]
The differential has matrix
\[
\kbordermatrix{
& Z(w_1 C') & ZC' & \Ib_{\oo} \\
Z(w_1 C')  & 0 & 0 & 1 \\
ZC' & 0 & 0 & U_2 \\
\Ib_{\oo} & 0 & 0 & 0
},
\]
the right action of $U_1$ has matrix
\[
\kbordermatrix{
& Z (w_1 C') & ZC' & \Ib_{\oo} \\
Z(w_1 C')  & U_1 + U_2 & 1 & 0 \\
ZC' & U_1 U_2 & 0 & 0 \\
\Ib_{\oo} & 0 & 0 & U_1
},
\]
and the right action of $U_2$ has matrix
\[
\kbordermatrix{
& Z (w_1 C') & ZC' & \Ib_{\oo} \\
Z(w_1 C')  & 0 & 1 & 0 \\
ZC' & U_1 U_2 & U_1 + U_2 & 0 \\
\Ib_{\oo} & 0 & 0 & U_2
}.
\]

\begin{proposition}\label{prop:SimplifyingPosCrossingLower}
The dg bimodule $P'_{\low}$ is quasi-isomorphic (or homotopy equivalent as DA bimodules) to the ordinary bimodule $P''_{\low}$ with primary matrix
\[
\kbordermatrix{
& \oo \\
\oo & ZC'
}
\]
and secondary matrix
\[
\kbordermatrix{
& ZC' \\
ZC' & U_1^l U_2^k \otimes U_1^k U_2^l
}.
\]
\end{proposition}

\begin{proof}
One can check that the map with matrix
\[
\kbordermatrix{
& Z(w_1 C') & ZC' & \Ib_{\oo} \\
Z C' & U_2 & 1 & 0
}
\]
is a homomorphism of dg bimodules. This map sends the generator $ZC'$ of the homology of the mapping cone (free on this one generator as a left module over $\F_2[U_1,U_2]$) to the single generator of the bimodule in the statement of the proposition (also free as a left module over $\F_2[U_1,U_2]$ on its one generator). Thus, the map is a quasi-isomorphism.

\end{proof}

The generator $ZC'$ of $P''_{\low}$ has degrees $\deg^q(Z C') = 1$ and $\deg^h(Z C') = -1$.

\subsection{Decategorification of the bimodule for a positive crossing}

Let $P'' = P''_{\upp} \oplus P''_{\midd} \oplus P''_{\low}$, a DA bimodule over $(\A_{1,1},\A_{1,1})$.

\begin{proposition}
The DA bimodule $P''$ categorifies the map from $K_0(\A_{1,1})$ to $K_0(\A_{1,1})$ with matrix
\[
\kbordermatrix{
& {[P_{\uu}]} & {[P_{\ou}]} & {[P_{\uo}]} & {[P_{\oo}]} \\
{[P_{\uu}]} & q^{-1} & 0 & 0 & 0 \\
{[P_{\ou}]} & 0 & 0 & 1 & 0 \\
{[P_{\uo}]} & 0 & 1 & q^{-1} - q & 0\\
{[P_{\oo}]} & 0 & 0 & 0 &  -q
}. 
\]
Equivalently, the AD bimodule ${^{\vee}}P''$ categorifies the map from $G_0(\A_{1,1})$ to $G_0(\A_{1,1})$ with matrix
\[
\kbordermatrix{
& {[S_{\uu}]} & {[S_{\ou}]} & {[S_{\uo}]} & {[S_{\oo}]} \\
{[S_{\uu}]} & q & 0 & 0 & 0 \\
{[S_{\ou}]} & 0 & 0 & 1 & 0 \\
{[S_{\uo}]} & 0 & 1 & q - q^{-1} & 0\\
{[S_{\oo}]} & 0 & 0 & 0 & -q^{-1}
}. 
\]
\end{proposition}

This latter map can be identified with the braiding acting on $V^{\otimes 2}$ as in Appendix~\ref{sec:SingularNonsingular}.

\subsection{Mapping-cone bimodule for negative crossing}\label{sec:MappingConeNeg}

Similarly to the previous section, we write $N' = N'_{\upp} \oplus N'_{\midd} \oplus N'_{\low}$ for the 1-morphism of 2-representations of $\U^-$ associated to a negative crossing by Sections \ref{sec:SkewHowe2Action}, \ref{sec:SoergelBraidCob}, and we will discuss a simplified version $N''$ of $N'$.

\subsubsection{Upper weight space, negative crossing}

We have $N'_{\upp} = N''_{\upp} = q^{1} \F_2$ (in homological degree zero) as a bimodule over the upper summand $(\A_{1,1})_{2\varepsilon_1} \cong \F_2$ of $\A_{1,1}$.

\subsubsection{Middle weight space, negative crossing}

The DA bimodule
\[
N'_{\midd} \coloneqq \xymatrix{
X_{\midd}[1] \ar@/^1.5pc/[rr]^{\varepsilon'} & \oplus & q \mathbb{I}_{(\A_{1,1})_{\varepsilon_1 + \varepsilon_2}}
}
\]
over the middle summand $(\A_{1,1})_{\varepsilon_1 + \varepsilon_2}$ of $\A_{1,1}$, where $\varepsilon'$ is defined in Section~\ref{sec:BimodMapsFor2Mors}, has primary matrix
\[
\kbordermatrix{
& \ou & \uo \\
\ou & \Ib_{\ou} \quad XA' & XB' \\
\uo & YA' & \Ib_{\uo} \quad YB'
}
\]
and secondary matrix
\[
\resizebox{\textwidth}{!}{
\kbordermatrix{
& \Ib_{\ou} & \Ib_{\uo} & XA' & YA' & XB' & YB' \\
\Ib_{\ou} & U_1^{k+1} \otimes U_1^{k+1} & \lambda \otimes \lambda & U_1 & 0 & 1 \otimes \lambda & 0 \\
\Ib_{\uo} & 0 & U_2^{k+1} \otimes U_2^{k+1} & 0 & 0 & 0 & 1 \\
XA' & 0 & 0 & U_1^{k+1} \otimes U_1^{k+1} & \lambda & \begin{matrix} U_1^k \otimes (\lambda,U_1^{k+1}) \\+ U_1^k \otimes (U_2^{k+1},\lambda)\end{matrix} & 0 \\
YA' & 0 & 0 & 0 & U_2^{k+1} \otimes U_1^{k+1} & 0 & \begin{matrix} U_2^k \otimes (U_2^{k+1},\lambda) \\+ U_2^k \otimes (\lambda, U_1^{k+1})\end{matrix} \\
XB' & 0 & 0 & 0 & 0 & U_1^{k+1} \otimes U_2^{k+1} & \lambda \\
YB' & 0 & 0 & 0 & 0 & 0 & U_2^{k+1} \otimes U_2^{k+1}
}.
}
\]

Simplifying $N'_{\midd}$, we get a DA bimodule $N''_{\midd}$ with primary matrix
\[
\kbordermatrix{
& \ou & \uo \\
\ou & \Ib_{\ou} \quad XA' & XB' \\
\uo & YA' & 
}
\]
and secondary matrix  
\[
\kbordermatrix{
& \Ib_{\ou} & XA' & YA' & XB' \\
\Ib_{\ou} & U_1^{k+1} \otimes U_1^{k+1} & U_1 & 0 & 1 \otimes \lambda \\
XA' & 0 & U_1^{k+1} \otimes U_1^{k+1} & \lambda & \begin{matrix} U_1^k \otimes (\lambda,U_1^{k+1}) \\+ U_1^k \otimes (U_2^{k+1},\lambda)\end{matrix} \\
YA' & 0 & 0 & U_2^{k+1} \otimes U_1^{k+1} & 0 \\
XB' & 0 & 0 & 0 & U_1^{k+1} \otimes U_2^{k+1} \\
}.
\]

The generators have degrees
\begin{itemize}
\item $\deg^q(\Ib_{\ou}) = 1$, $\deg^h(\Ib_{\ou}) = 0$.
\item $\deg^q(XA') = -1$, $\deg^h(XA') = 1$,
\item $\deg^q(YA') = 0$, $\deg^h(YA') = 0$,
\item $\deg^q(XB') = 0$, $\deg^h(XB') = 0$.
\end{itemize}

\subsubsection{Lower weight space, negative crossing}

The DA bimodule
\[
N'_{\low} \coloneqq \xymatrix{
X_{\low}[1] \ar@/^1.5pc/[rr]^{\varepsilon'} & \oplus & q\mathbb{I}_{(\A_{1,1})_{2\varepsilon_2}}
}
\]
over the lower summand $(\A_{1,1})_{2\varepsilon_2}$ of $\A_{1,1}$, where $\varepsilon'$ is defined in Section~\ref{sec:BimodMapsFor2Mors}, has primary matrix
\[
\kbordermatrix{
& \oo \\
\oo & \Ib_{\oo} \quad Z (w_1 C') \quad ZC' \\
}.
\]
The differential has matrix
\[
\kbordermatrix{
& \Ib_{\oo} & Z (w_1 C') &  ZC' \\
\Ib_{\oo} & 0 & U_1 & 1 \\
Z (w_1 C') & 0 & 0 & 0 \\
ZC' & 0 & 0 & 0
},
\]
the right action of $U_1$ has matrix
\[
\kbordermatrix{
& \Ib_{\oo} & Z (w_1 C') & ZC' \\
\Ib_{\oo} & U_1  & 0 & 0 \\
Z (w_1 C') & 0 & U_1 + U_2 & 1 \\
ZC' & 0 & U_1 U_2 & 0
},
\]
and the right action of $U_2$ has matrix
\[
\kbordermatrix{
& \Ib_{\oo} & Z (w_1 C') & ZC' \\
\Ib_{\oo} & U_2 & 0 & 0 \\
Z (w_1 C') & 0 & 0 & 1 \\
ZC' & 0 & U_1 U_2 & U_1 + U_2
}.
\]

\begin{proposition}
The dg bimodule $N'_{\low}$ is quasi-isomorphic (or homotopy equivalent as DA bimodules) to the ordinary bimodule $N''_{\low}$ with primary matrix
\[
\kbordermatrix{
& \oo \\
\oo & Z(w_1 C')
}
\]
and secondary matrix
\[
\kbordermatrix{
& Z(w_1 C') \\
Z(w_1 C') & U_1^l U_2^k \otimes U_1^k U_2^l
}.
\]
\end{proposition}

\begin{proof}
As in Proposition~\ref{prop:SimplifyingPosCrossingLower}, one can check that the map with matrix
\[
\kbordermatrix{
& Z(w_1 C') \\
\Ib_{\oo} & 0 \\
Z(w_1 C') & 1 \\
ZC' & U_1
}
\]
is a quasi-isomorphism of dg bimodules.
\end{proof}

The generator of $N''_{\low}$ has degrees $\deg^q(Z(w_1 C')) = -1$ and $\deg^h(Z (w_1 C')) = 1$.

\subsection{Decategorification of the bimodule for a negative crossing}

Let $N'' = N''_{\upp} \oplus N''_{\midd} \oplus N''_{\low}$, a DA bimodule over $(\A_{1,1},\A_{1,1})$.

\begin{proposition}
The DA bimodule $N''$ categorifies the map from $K_0(\A_{1,1})$ to $K_0(\A_{1,1})$ with matrix
\[
\kbordermatrix{
& {[P_{\uu}]} & {[P_{\ou}]} & {[P_{\uo}]} & {[P_{\oo}]} \\
{[P_{\uu}]} & q & 0 & 0 & 0 \\
{[P_{\ou}]} & 0 & q-q^{-1} & 1 & 0 \\
{[P_{\uo}]} & 0 & 1 & 0 & 0 \\
{[P_{\oo}]} & 0 & 0 & 0 & -q^{-1}
}. 
\]
Equivalently, the DA bimodule ${^{\vee}}N''$ categorifies the map from $G_0(\A_{1,1})$ to $G_0(\A_{1,1})$ with matrix
\[
\kbordermatrix{
& {[S_{\uu}]} & {[S_{\ou}]} & {[S_{\uo}]} & {[S_{\oo}]} \\
{[S_{\uu}]} & q^{-1} & 0 & 0 & 0 \\
{[S_{\ou}]} & 0 & q^{-1}-q & 1 & 0 \\
{[S_{\uo}]} & 0 & 1 & 0 & 0 \\
{[S_{\oo}]} & 0 & 0 & 0 & -q
}. 
\]
\end{proposition}

This latter map can be identified with the inverse of the braiding acting on $V^{\otimes 2}$ as in Appendix~\ref{sec:SingularNonsingular}.

\subsection{Bimodules from Heegaard diagrams}\label{sec:PositiveCrossingHD}

\begin{figure}
\includegraphics[scale=0.4]{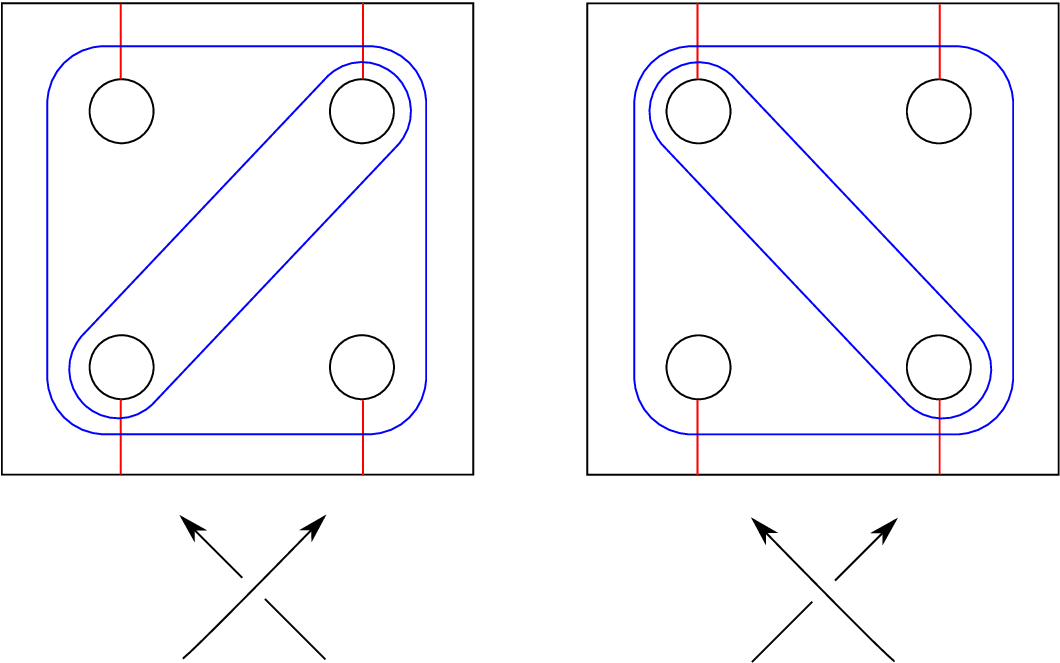}
\caption{Heegaard diagrams for positive and negative crossings.}
\label{fig:NonsingularCrossing}
\end{figure}

We now consider slightly different DA bimodules, motivated directly by holomorphic disk counts in the Heegaard diagrams of Figure~\ref{fig:NonsingularCrossing}.

\begin{definition}
Let $P$ be the DA bimodule over $(\A_{1,1},\A_{1,1})$ with primary matrix
\[
\kbordermatrix{
& \uu & \ou & \uo & \oo \\
\uu & I & & & \\
\ou & & & K & \\
\uo & & J & L \, M \\
\oo & & & & N
}
\]
(treating all three of its summands together). We set
\begin{itemize}
\item $\deg^q(I) = -1$, $\deg^h(I) = 0$,
\item $\deg^q(J) = 0$, $\deg^h(J) = 0$,
\item $\deg^q(K) = 0$, $\deg^h(K) = 0$,
\item $\deg^q(L) = 1$, $\deg^h(L) = -1$,
\item $\deg^q(M) = -1$, $\deg^h(M) = 0$,
\item $\deg^q(N) = 1$, $\deg^h(N) = -1$.
\end{itemize}
The secondary matrix of $P$ is given as follows.
\begin{itemize}
\item The secondary matrix in the top summand (with generator $I$) is zero (which, by convention, corresponds to the identity bimodule over $\F_2$).

\item The secondary matrix in the middle summand is
\[
\kbordermatrix{
& J & K & L & M \\
J & U_2^{k+1} \otimes U_1^{k+1} & 0 & U_2^k \otimes (\lambda, U_1^{k+1}) & 1 \otimes \lambda \\
K & 0 & U_1^{k+1} \otimes U_2^{k+1} & \lambda & 0 \\
L & 0 & 0 & 0 & U_2 \\
M & 0 & 0 & 0 & 0
}.
\]

\item The secondary matrix in the bottom summand is
\[
\kbordermatrix{
& N \\
N & U_1^l U_2^k \otimes U_1^k U_2^l
}.
\]
\end{itemize}
\end{definition}

\begin{figure}
\includegraphics[scale=0.6]{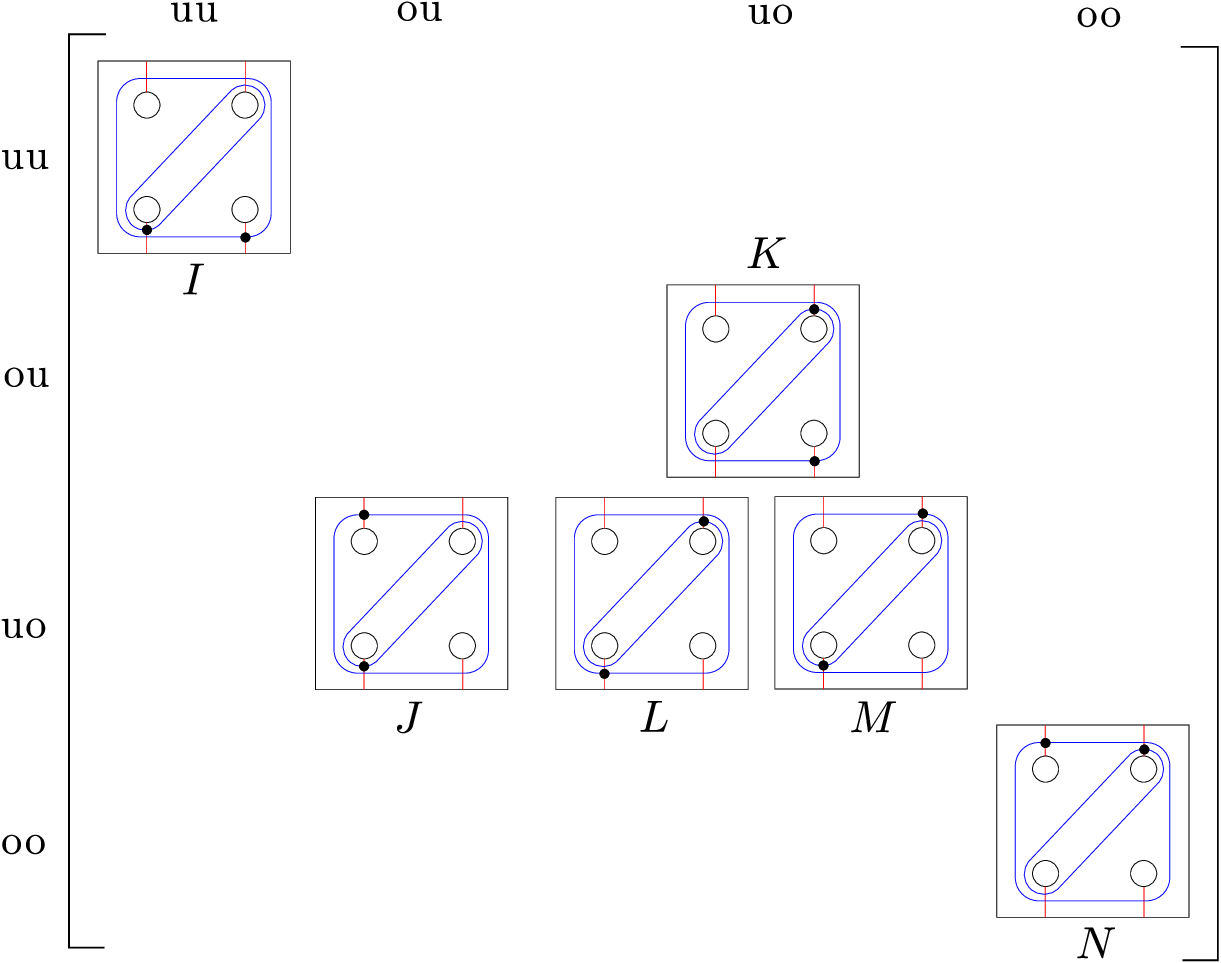}
\caption{Generators of the positive crossing bimodule $P$ in terms of intersection points in the Heegaard diagram of Figure~\ref{fig:NonsingularCrossing}.}
\label{fig:PositiveCrossingGens}
\end{figure}

\begin{figure}
\includegraphics[scale=0.6]{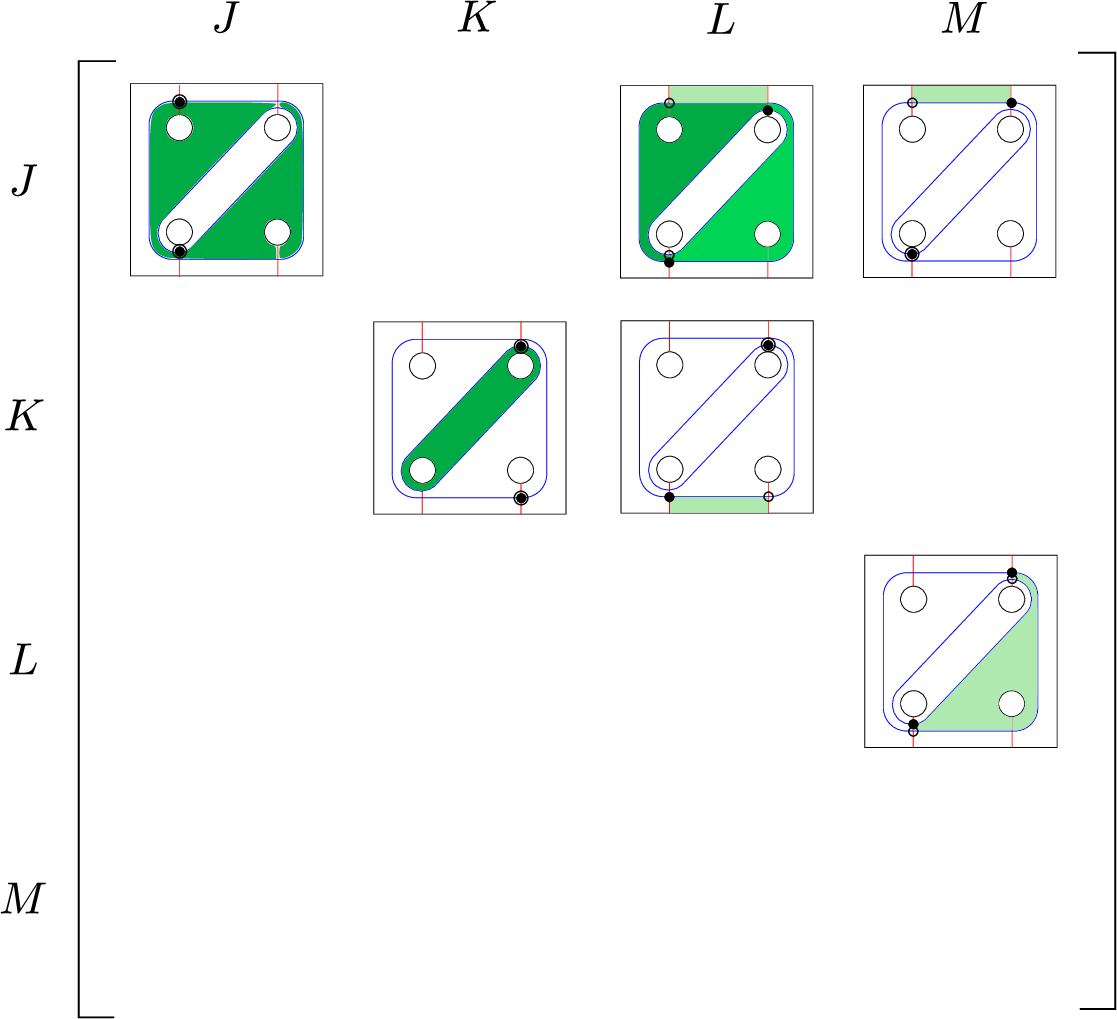}
\caption{Domains giving rise to secondary matrix entries for the positive crossing bimodule $P$ (middle summand).}
\label{fig:PositiveCrossingDomains}
\end{figure}

Figure~\ref{fig:PositiveCrossingGens} shows the generators of $P$ in terms of sets of intersection points in the Heegaard diagram of Figure~\ref{fig:NonsingularCrossing}. Figure~\ref{fig:PositiveCrossingDomains} shows the domains in this Heegaard diagram that give rise to nonzero entries in the middle summand of the secondary matrix of $P$ (the other two summands are simpler and are not shown).

\begin{proposition}
The DA bimodule $P$ for a positive crossing is isomorphic to the simplified mapping cone bimodule $P''$ from Section~\ref{sec:MappingConePos}.
\end{proposition}
 
\begin{proof}
The top and bottom summands of the two bimodules in question are identical under the identification of $I$ with $\Ib_{\uu}$ and $N$ with $ZC'$ (note that in the mapping cone bimodule, the degree of $\Ib_{\uu}$ is $q^{-1} h^0$ and the degree of $Z C'$ is $q^{1} h^{-1}$). In the middle weight space, an isomorphism from $P$ to the mapping cone bimodule is given by the matrix
\[
\kbordermatrix{
& J & K & L & M \\
YA' & 1 & 0 & 0 & 0 \\
XB' & 0 & 1 & 0 & 0 \\
YB' & 0 & 0 & 1 & 0 \\
\Ib_{\uo} & 0 & 0 & U_2^k \otimes U_2^{k+1} & 1 
}.
\]
One can check that this matrix represents a closed grading-preserving morphism of DA bimodules that is invertible with inverse 
\[
\kbordermatrix{
& YA' & XB' & YB' & \Ib_{\uo} \\
J & 1 & 0 & 0 & 0 \\
K & 0 & 1 & 0 & 0 \\
L & 0 & 0 & 1 & 0 \\
M & 0 & 0 & U_2^k \otimes U_2^{k+1} & 1 
};
\]
recall that the degrees of $YA'$, $XB'$, $YB'$, and $\Ib_{\uo}$ are $q^0 h^0$, $q^0 h^0$, $q^{1} h^{-1}$, and $q^{-1} h^0$ respectively.
\end{proof}

\begin{definition}
Let $N$ be the DA bimodule over $(\A_{1,1},\A_{1,1})$ with primary matrix
\[
\kbordermatrix{
& \uu & \ou & \uo & \oo \\
\uu & I' & & & \\
\ou & & J' \, K'& M' & \\
\uo & & L' & &\\
\oo & & & & N'
}
\]
(treating all three of its summands together). We set
\begin{itemize}
\item $\deg^q(I') = 1$, $\deg^h(I') = 0$,
\item $\deg^q(J') = 1$, $\deg^h(J') = 0$,
\item $\deg^q(K') = -1$, $\deg^h(K') = 1$,
\item $\deg^q(L') = 0$, $\deg^h(L') = 0$,
\item $\deg^q(M') = 0$, $\deg^h(M') = 0$,
\item $\deg^q(N') = -1$, $\deg^h(N') = 1$.
\end{itemize}
The secondary matrix of $N$ is given as follows.
\begin{itemize}
\item The secondary matrix in the top summand (with generator $I'$) is zero.

\item The secondary matrix in the middle summand is
\[
\kbordermatrix{
& J' & K' & L' & M' \\
J' & 0 & U_1 & 0 & 1 \otimes \lambda \\
K' & 0 & 0 & \lambda & U_1^k \otimes (U_2^{k+1},\lambda) \\
L' & 0 & 0 & U_2^{k+1} \otimes U_1^{k+1} & 0 \\
M' & 0 & 0 & 0 & U_1^{k+1} \otimes U_2^{k+1}
}.
\]

\item The secondary matrix in the bottom summand is
\[
\kbordermatrix{
& N' \\
N' & U_1^l U_2^k \otimes U_1^k U_2^l
}.
\]
\end{itemize}
\end{definition}

\begin{figure}
\includegraphics[scale=0.6]{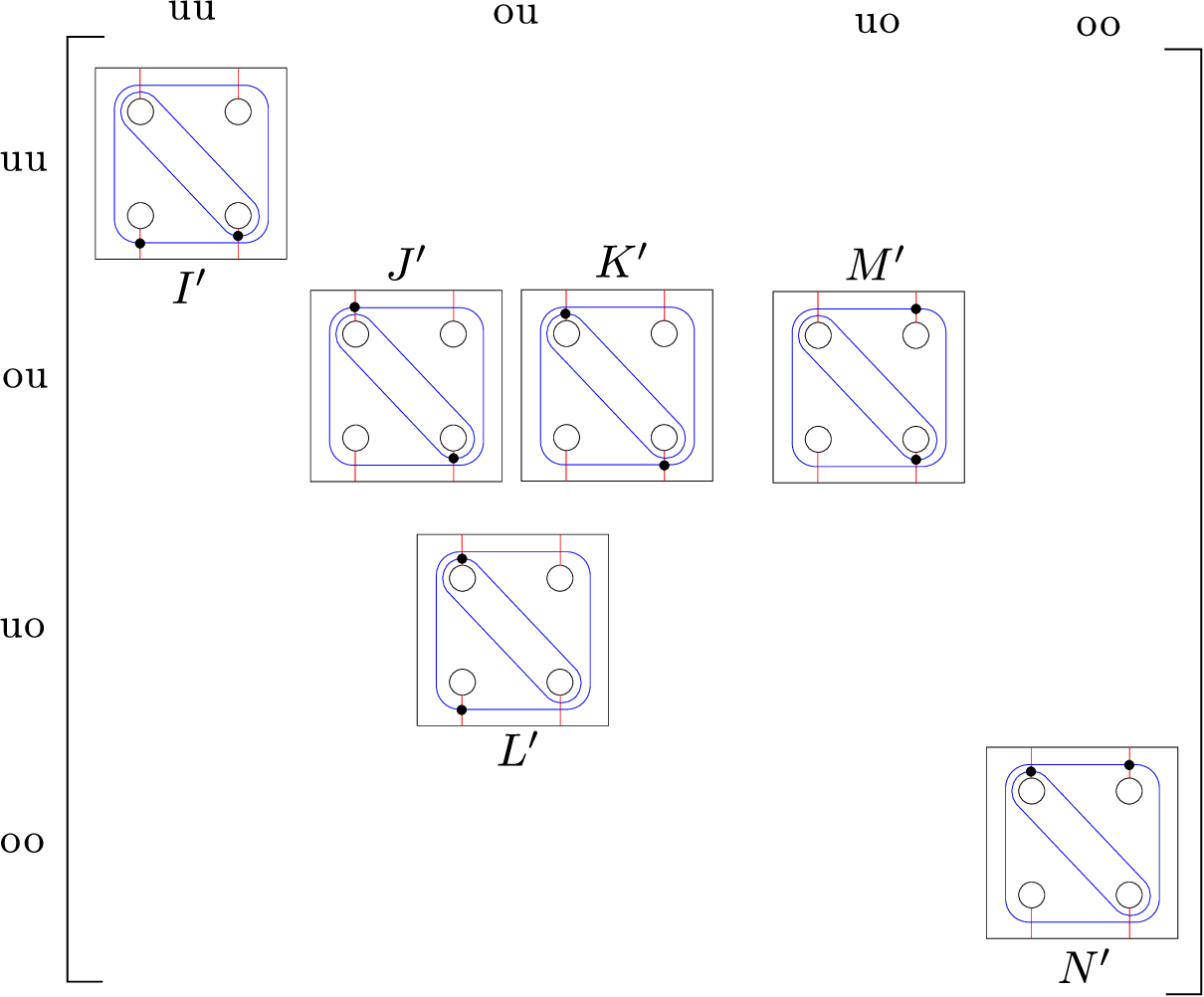}
\caption{Generators of the negative crossing bimodule $N$ in terms of intersection points in the right Heegaard diagram of Figure~\ref{fig:NonsingularCrossing}.}
\label{fig:NegativeCrossingGens}
\end{figure}

\begin{figure}
\includegraphics[scale=0.6]{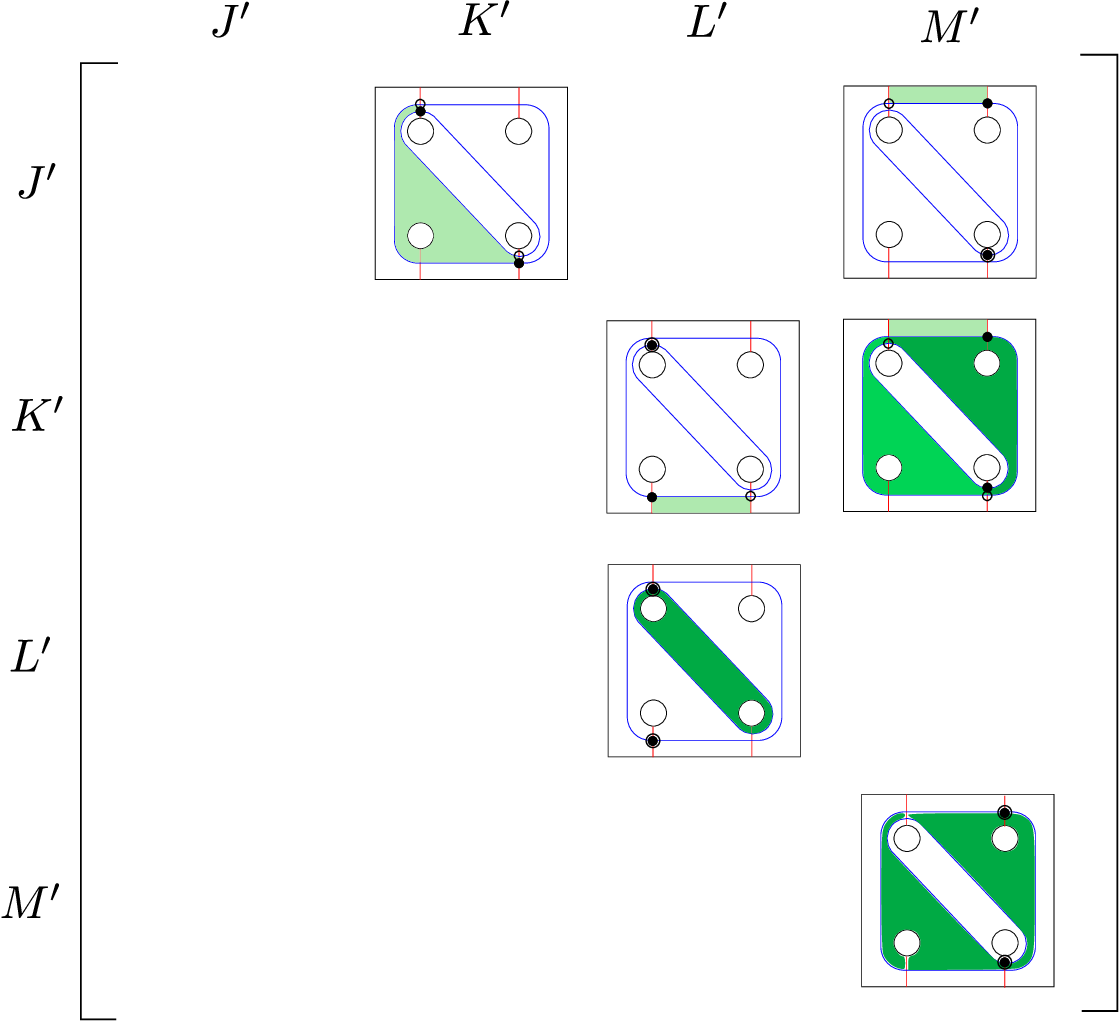}
\caption{Domains giving rise to secondary matrix entries for the negative crossing bimodule $N$ (middle summand).}
\label{fig:NegativeCrossingDomains}
\end{figure}

Figure~\ref{fig:NegativeCrossingGens} shows the generators of $N$ in terms of sets of intersection points in the right Heegaard diagram of Figure~\ref{fig:NonsingularCrossing}. Figure~\ref{fig:NegativeCrossingDomains} shows the domains in this Heegaard diagram that give rise to nonzero entries in the middle summand of the secondary matrix of $N$ (the other two summands are simpler and are not shown).

\begin{proposition}
The DA bimodule $N$ for a negative crossing is isomorphic to the simplified mapping cone bimodule $N''$ from Section~\ref{sec:MappingConeNeg}.
\end{proposition}

\begin{proof}
The top and bottom summands of the two bimodules in question are identical under the identification of $I'$ with $\Ib_{\uu}$ and $N'$ with $Z (w_1 C')$ (note that in the mapping cone bimodule, the degree of $\Ib_{\uu}$ is $q^{1} h^0$ and the degree of $Z (w_1 C')$ is $q^{-1} h^{1}$). In the middle weight space, an isomorphism from $N$ to the mapping cone bimodule is given by the matrix
\[
\kbordermatrix{
& J' & K' & L' & M' \\
\Ib_{\ou} & 1 & 0 & 0 & 0 \\
XA' & U_1^k \otimes U_1^{k+1} & 1 & 0 & 0 \\
YA' & 0 & 0 & 1 & 0 \\
XB' & 0 & 0 & 0 & 1 
}.
\]
One can check that this matrix represents a closed grading-preserving morphism of DA bimodules that is invertible with inverse
\[
\kbordermatrix{
& \Ib_{\ou} & XA' & YA' & XB' \\
J' & 1 & 0 & 0 & 0 \\
K' & U_1^k \otimes U_1^{k+1} & 1 & 0 & 0 \\
L' & 0 & 0 & 1 & 0 \\
M' & 0 & 0 & 0 & 1 
};
\]
recall that the degrees of $\Ib_{\ou}$, $XA'$, $YA'$, $XB'$ are $q^{1}h^0$, $q^{-1} h^{1}$, $q^0 h^0$, and $q^0 h^0$ respectively.
\end{proof}

\section{Change-of-basis bimodules}\label{sec:ChangeOfBasis}

\subsection{Ozsv{\'a}th--Szab{\'o} canonical basis algebras}

Here we define the Ozsv{\'a}th--Szab{\'o} algebra $\A_{1,1}^{\can}$.

\begin{figure}
\includegraphics[scale=0.4]{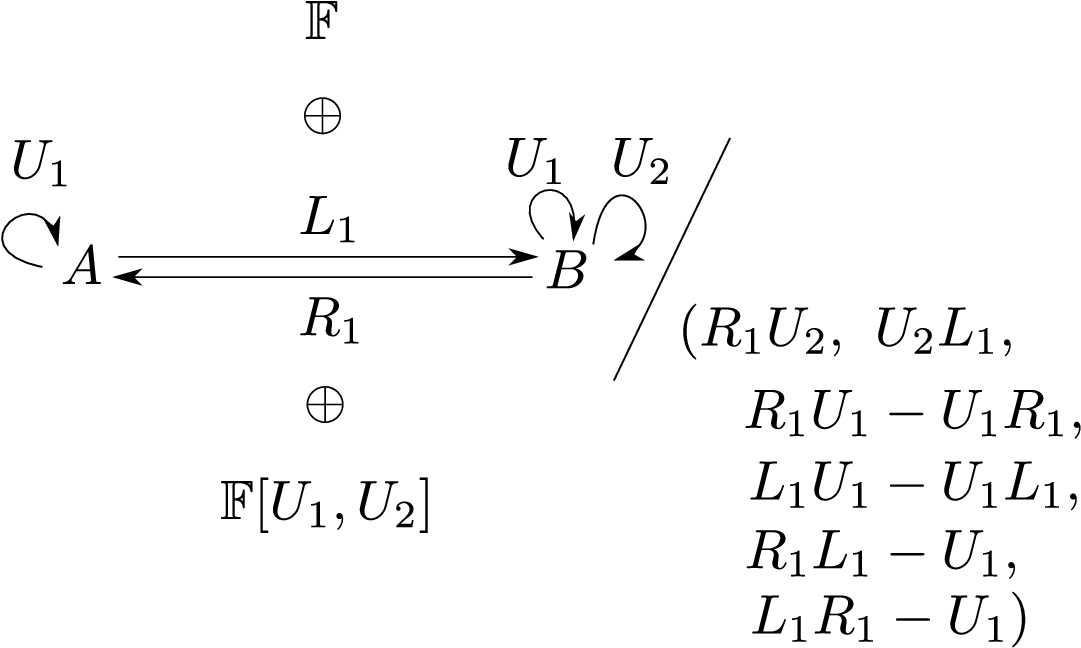}
\caption{The dg category $\A^{\can}_{1,1}$.}
\label{fig:A11CanQuiver}
\end{figure}

\begin{definition}
The dg category $\A^{\can}_{1,1}$ is shown in Figure~\ref{fig:A11CanQuiver}; arrows point from source to target. The differential on $\A^{\can}_{1,1}$ is zero; the gradings are given by
\begin{itemize}
\item $\deg^q(R_1) = -1$, $\deg^h(R_1) = 1$,
\item $\deg^q(L_1) = -1$, $\deg^h(L_1) = 1$,
\item $\deg^q(U_i) = -2$, $\deg^h(U_i) = 2$ for $i \in \{1,2\}$.
\end{itemize}
\end{definition}

\subsection{Bimodules}
We would like to relate the positive and negative crossing bimodules $P$ and $N$ of Section~\ref{sec:PositiveCrossingHD} to the bimodules for positive and negative crossings defined by Ozsv{\'a}th--Szab{\'o} in \cite{OSzNew}. To do so, we will first define categorified change-of-basis bimodules between the algebras $\A_{1,1}$ and $\A_{1,1}^{\can}$. 

\begin{figure}
\includegraphics[scale=0.4]{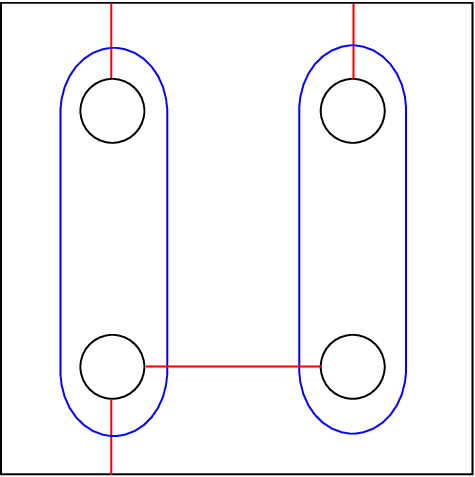}
\caption{Heegaard diagram for changing canonical basis to standard basis}
\label{fig:StandardToCanonical}
\end{figure}

\begin{figure}
\includegraphics[scale=0.4]{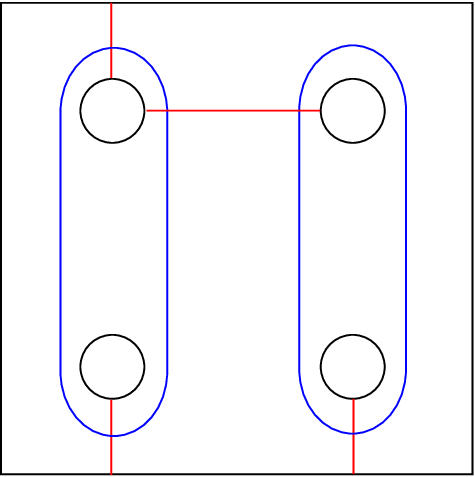}
\caption{Heegaard diagram for changing standard basis to canonical basis}
\label{fig:CanonicalToStandard}
\end{figure}

\begin{definition}\label{def:FirstCOBBimodule}
The change-of-basis DA bimodule over $(\A_{1,1}^{\can},\A_{1,1})$ has primary matrix
\[
\kbordermatrix{
& \uu & \ou & \uo & \oo \\
\varnothing & \kappa_1 & & & \\
A & & \kappa_2 & \kappa_3 & \\
B & & & \kappa_4 & \\
AB & & & & \kappa_5 \\
}; 
\]
we set
\begin{itemize}
\item $\deg^q(\kappa_1) = 0$, $\deg^h(\kappa_1) = 0$,
\item $\deg^q(\kappa_2) = 0$, $\deg^h(\kappa_2) = 0$,
\item $\deg^q(\kappa_3) = -1$, $\deg^h(\kappa_3) = 0$,
\item $\deg^q(\kappa_4) = 0$, $\deg^h(\kappa_4) = 0$,
\item $\deg^q(\kappa_5) = 0$, $\deg^h(\kappa_5) = 0$.
\end{itemize}
This change of basis bimodule has secondary matrix
\[
\kbordermatrix{
& \kappa_2 & \kappa_3 & \kappa_4 \\
\kappa_2 & U_1^{k+1} \otimes U_1^{k+1} & 1 \otimes \lambda & L_1 U_1^k \otimes (\lambda, U_1^{k+1}) \\
\kappa_3 & 0 & 0 & 0 \\
\kappa_4 & 0 & R_1 & U_2^{k+1} \otimes U_2^{k+1}
}
\]
in the middle summand and secondary matrix
\[
\kbordermatrix{
& \kappa_5 \\
\kappa_5 & U_1^k U_2^l \otimes U_1^k U_2^l
}
\]
in the bottom summand.
\end{definition}

\begin{figure}
\includegraphics[scale=0.6]{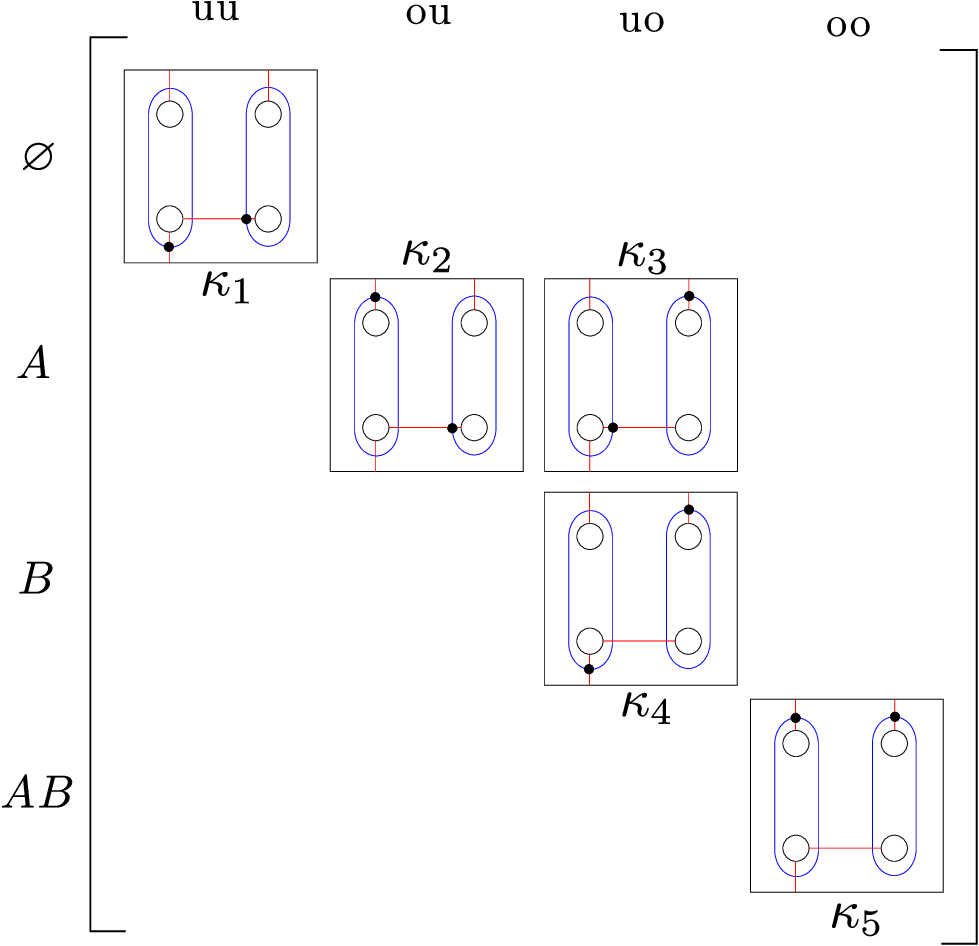}
\caption{Generators of one change of basis bimodule in terms of intersection points.}
\label{fig:KappaGens}
\end{figure}

\begin{figure}
\includegraphics[scale=0.6]{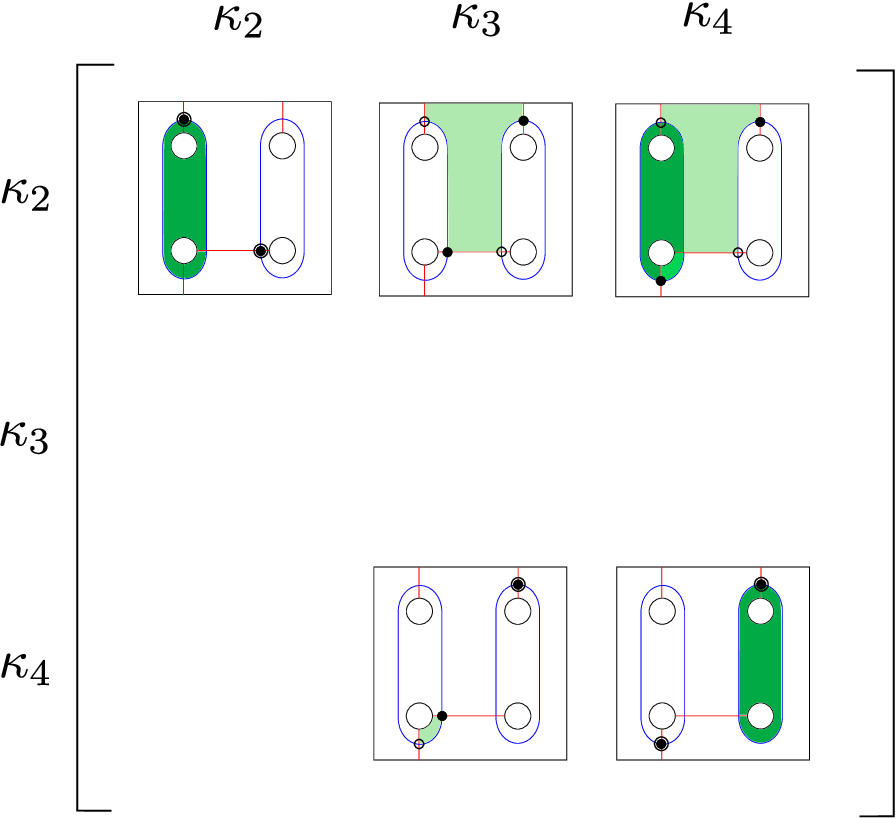}
\caption{Domains for one change of basis bimodule (middle summand).}
\label{fig:KappaDomains}
\end{figure}

\begin{definition}\label{def:SecondCOBBimodule}
The change-of-basis DA bimodule over $(\A_{1,1},\A_{1,1}^{\can})$ has primary matrix
\[
\kbordermatrix{
& \varnothing & A & B & AB\\
\uu & \lambda_1 & & & \\
\ou & & \lambda_2 & \lambda_3 & \\
\uo & & & \lambda_4 & \\
\oo & & & & \lambda_5
};
\]
we set
\begin{itemize}
\item $\deg^q(\lambda_1) = 0$, $\deg^h(\lambda_1) = 0$,
\item $\deg^q(\lambda_2) = 0$, $\deg^h(\lambda_2) = 0$,
\item $\deg^q(\lambda_3) = -1$, $\deg^h(\lambda_3) = 1$,
\item $\deg^q(\lambda_4) = 0$, $\deg^h(\lambda_4) = 0$,
\item $\deg^q(\lambda_5) = 0$, $\deg^h(\lambda_5) = 0$.
\end{itemize}
This change of basis bimodule has secondary matrix
\[
\kbordermatrix{
& \lambda_2 & \lambda_3 & \lambda_4 \\
\lambda_2 & U_1^{k+1} \otimes U_1^{k+1} & U_1^{k+1} \otimes L_1 U_1^k & 0 \\
\lambda_3 &  U_1^k \otimes R_1 U_1^k & U_1^{k+1} \otimes U_1^{k+1} & \lambda \\
\lambda_4 & 0 & 0 & U_2^{k+1} \otimes U_2^{k+1}
}
\]
in the middle summand, and secondary matrix
\[
\kbordermatrix{
& \lambda_5 \\
\lambda_5 & U_1^k U_2^l \otimes U_1^k U_2^l
}
\]
in the bottom summand.
\end{definition}

\begin{figure}
\includegraphics[scale=0.6]{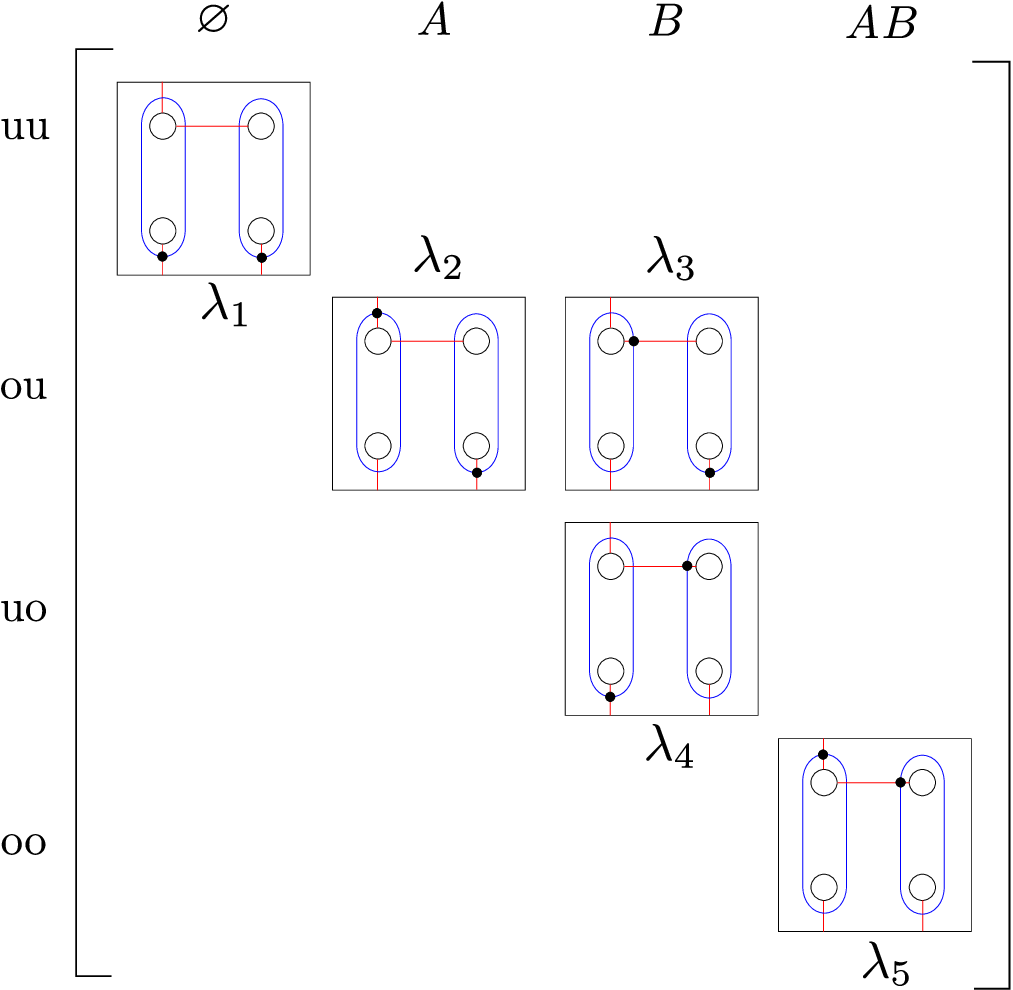}
\caption{Generators of the other change of basis bimodule in terms of intersection points.}
\label{fig:LambdaGens}
\end{figure}

\begin{figure}
\includegraphics[scale=0.6]{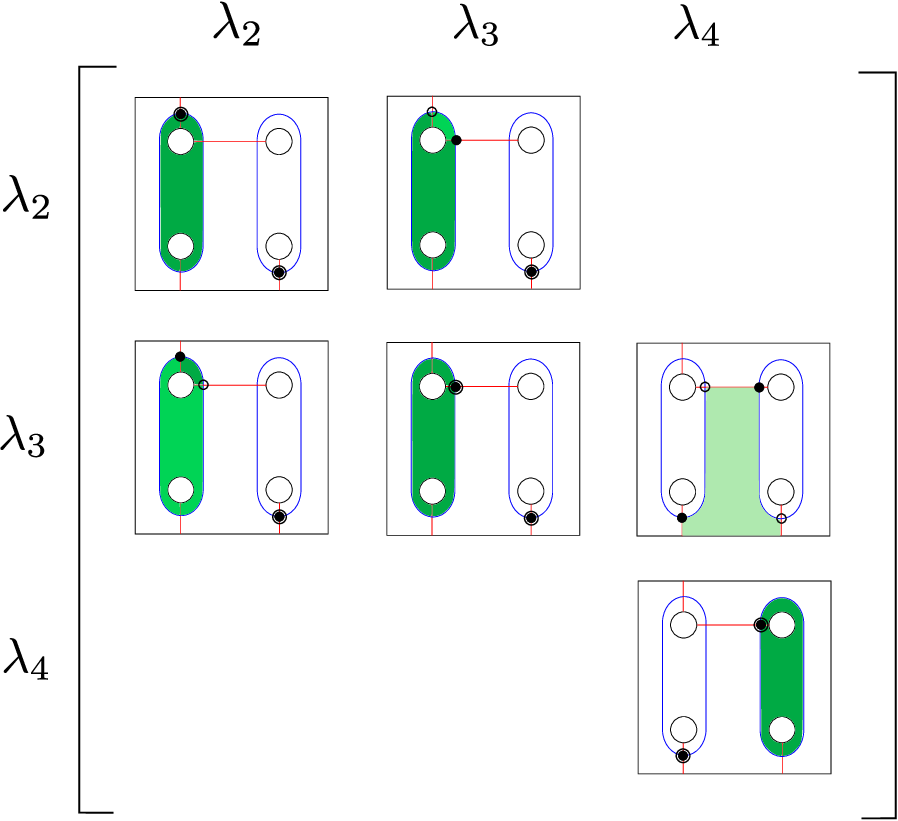}
\caption{Domains for the other change of basis bimodule (middle summand).}
\label{fig:LambdaDomains}
\end{figure}

One can check that these bimodules are well-defined and that their box tensor product either way is homotopy equivalent to the appropriate identity bimodule.

\begin{proposition}
The change-of-basis bimodule from Definition~\ref{def:FirstCOBBimodule} categorifies the map from $K_0(\A_{1,1})$ to $K_0(\A_{1,1}^{\can})$ with matrix
\[
\kbordermatrix{
& {[P_{\uu}]} & {[P_{\ou}]} & {[P_{\uo}]} & {[P_{\oo}]} \\
{[P_{\varnothing}]} & 1 & 0 & 0 & 0 \\
{[P_{A}]} & 0 & 1 & q^{-1} & 0 \\
{[P_{B}]} & 0 & 0 & 1 & 0 \\
{[P_{AB}]} & 0 & 0 & 0 & 1
};
\]
equivalently, its left dual categorifies the map from $G_0(\A_{1,1})$ to $G_0(\A_{1,1}^{\can})$ with matrix
\[
\kbordermatrix{
& {[S_{\varnothing}]} & {[S_{A}]} & {[S_{B}]} & {[S_{AB}]} \\
{[S_{\uu}]} & 1 & 0 & 0 & 0 \\
{[S_{\ou}]} & 0 & 1 & 0 & 0 \\
{[S_{\uo}]} & 0 & q & 1 & 0 \\
{[S_{\oo}]} & 0 & 0 & 0 & 1
}.
\]
Under our identifications, this latter map agrees with the change of basis matrix from the canonical basis (introduced in Section~\ref{sec:NonstandardBases}) to the standard basis of $V^{\otimes 2}$.
\end{proposition}

\begin{proposition}
The change-of-basis bimodule from Definition~\ref{def:SecondCOBBimodule} categorifies the map from $K_0(\A_{1,1}^{\can})$ to $K_0(\A_{1,1})$ with matrix
\[
\kbordermatrix{
& {[P_{\varnothing}]} & {[P_{A}]} & {[P_{B}]} & {[P_{AB}]} \\
{[P_{\uu}]} & 1 & 0 & 0 & 0 \\
{[P_{\ou}]} & 0 & 1 & -q^{-1} & 0 \\
{[P_{\uo}]} & 0 & 0 & 1 & 0 \\
{[P_{\oo}]} & 0 & 0 & 0 & 1
};
\]
equivalently, its left dual categorifies the map from $G_0(\A_{1,1})$ to $G_0(\A_{1,1}^{\can})$ with matrix
\[
\kbordermatrix{
& {[S_{\uu}]} & {[S_{\ou}]} & {[S_{\uo}]} & {[S_{\oo}]} \\
{[S_{\varnothing}]} & 1 & 0 & 0 & 0 \\
{[S_{A}]} & 0 & 1 & 0 & 0 \\
{[S_{B}]} & 0 & -q & 1 & 0 \\
{[S_{AB}]} & 0 & 0 & 0 & 1
}.
\]
Under our identifications, this latter map agrees with the change of basis matrix from the standard basis to the canonical basis of $V^{\otimes 2}$.
\end{proposition}

\section{Relationship with \texorpdfstring{Ozsvath-Szabo's}{Ozsv{\'a}th--Szab{\'o}'s} Kauffman-states functor}\label{sec:OSzRelationship}

\subsection{Positive crossing}

We now review Ozsv{\'a}th--Szab{\'o}'s bimodule for a positive crossing between two strands, with both strands oriented upwards, in our notation. 

\begin{proposition}
Ozsv{\'a}th--Szab{\'o}'s DA bimodule $P_{\OSz}$ over $\A_{1,1}^{\can}$ for a positive crossing has primary matrix
\[
\kbordermatrix{ & \varnothing & A & B & AB \\
\varnothing & S_{\varnothing} & & & \\
A & & S_A & & \\
B & & W & N_B & \\
AB & & & & N_{AB}
}.
\]
In the top weight space, the secondary matrix is zero and the bimodule is $\F_2$ as a bimodule over itself (generated by $S_{\varnothing}$). The secondary matrix in the middle weight space is
\[
\kbordermatrix{
& S_A & W & N_B \\
S_A & 0 & 0 & L_1 U_1^k \otimes (U_2^{k+1}, L_1) \\
W & R_1 & U_2^{k+1} \otimes U_1^{k+1} & U_2^k \otimes L_1 U_1^k \\
N_B & 0 & U_2^{k+1} \otimes R_1 U_1^k & \begin{matrix} U_2^{k+1} \otimes U_1^{k+1} \\+ U_1^{k+1} \otimes U_2^{k+1} \end{matrix}
}
\]
and the secondary matrix in the bottom weight space is
\[
\kbordermatrix{
& N_{AB} \\
N_{AB} & U_1^l U_2^k \otimes U_1^k U_2^l
}.
\]
\end{proposition}
The degrees of the generators depend only on their type $\in \{N,E,S,W\}$ and are given by
\begin{itemize}
\item $\deg^q(N) = 1$, $\deg^h(N) = -1$,
\item $\deg^q(E) = 0$, $\deg^h(E) = 0$,
\item $\deg^q(S) = -1$, $\deg^h(S) = 0$,
\item $\deg^q(W) = 0$, $\deg^h(W) = 0$.
\end{itemize}
(for the bimodule under consideration, there are no generators of type $E$.)

\begin{remark}
Actually, Ozsv{\'a}th--Szab{\'o} assign this bimodule to a negative crossing; to match the conventions used here, one should exchange positive and negative crossings in \cite{OSzNew}.
\end{remark}

\begin{remark} In \cite{OSzNew}, Ozsv{\'a}th--Szab{\'o} use dg algebras with extra $C_i$ generators when strands are oriented upwards. These generators help upward-oriented strands interact with downward-oriented strands, so that bimodules for maximum and minimum points can be defined. In \cite{OSzNewer} and \cite{OSzHolo}, these generators no longer appear except in descriptions of Koszul duals of Ozsv{\'a}th--Szab{\'o}'s algebras; the interaction between upward-pointing and downward-pointing strands is mediated differently. In any case, the orientations of the strands impact the grading of the algebras, and we use the grading that corresponds (in Ozsv{\'a}th--Szab{\'o}'s conventions) to upward-pointing strands. Since we do not consider duals or downward-pointing strands in this paper, we do not need to consider the modifications (e.g. curvature in \cite{OSzHolo}) that Ozsv{\'a}th--Szab{\'o} use when dealing with mixed orientations.
\end{remark}

\begin{figure}
\includegraphics[scale=0.6]{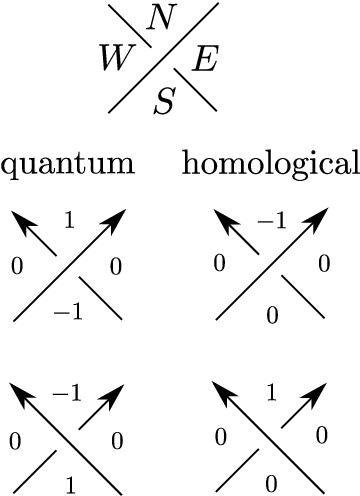}
\caption{Quantum and homological degrees of generators of type $N$, $E$, $S$, and $W$.}
\label{fig:AlexMaslov}
\end{figure}

\begin{remark}
Generators of Ozsv{\'a}th--Szab{\'o}'s bimodule generators correspond to certain ``partial Kauffman states,'' a local version of the Kauffman states defined in \cite{FKT}. At each crossing in a tangle diagram, a partial Kauffman state has a type $\in \{N,E,S,W\}$ depending on which of the four corners adjacent to the crossing is chosen for the state. One can compare Figure~\ref{fig:AlexMaslov} to \cite[Figure 2]{OSzNew}; our quantum degrees are $-2$ times Ozsv{\'a}th--Szab{\'o}'s Alexander degrees (the minus sign is due to the reversal of positive and negative crossings while the $2$ is due to the usual relation $q^2 = t$ between the parameter $t$ of the Alexander polynomial $\Delta_K(t)$ and the quantum parameter $q$), and our homological degrees agree with Ozsv{\'a}th--Szab{\'o}'s. Note that the agreement of homological degrees results from a product of two minus signs, one due to the reversal of positive and negative crossings with respect to \cite{OSzNew} and the other because we use $+1$ differentials while Ozsv{\'a}th--Szab{\'o} use $-1$ differentials.

\end{remark}

\begin{figure}
\includegraphics[scale=0.4]{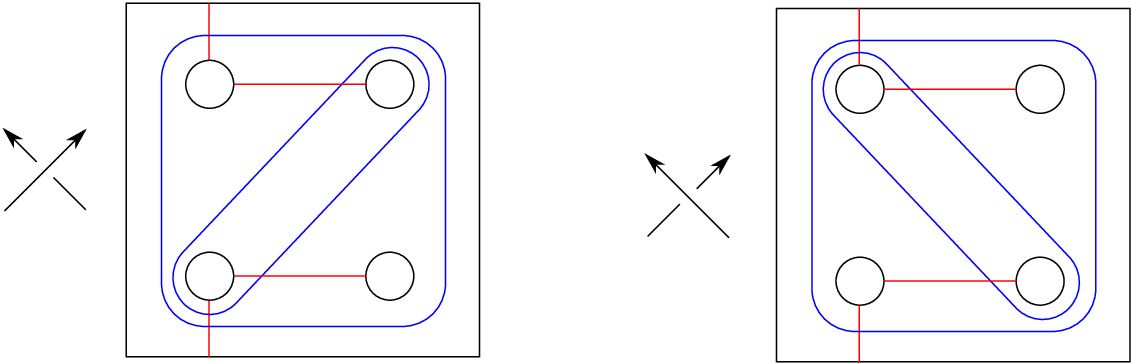}
\caption{Heegaard diagram for Ozsv{\'a}th--Szab{\'o}'s bimodules over $\A_{1,1}^{\can}$ for positive and negative crossings (in our conventions).}
\label{fig:OSzDiagramRightside}
\end{figure}

It follows from \cite{OSzHolo} that the bimodule $P_{\OSz}$ can be obtained by counting holomorphic disks in the Heegaard diagram on the left of Figure~\ref{fig:OSzDiagramRightside}; see Figures \ref{fig:OSzNegativeGens}, \ref{fig:OSzNegativeDomains}.

\begin{figure}
\includegraphics[scale=0.6]{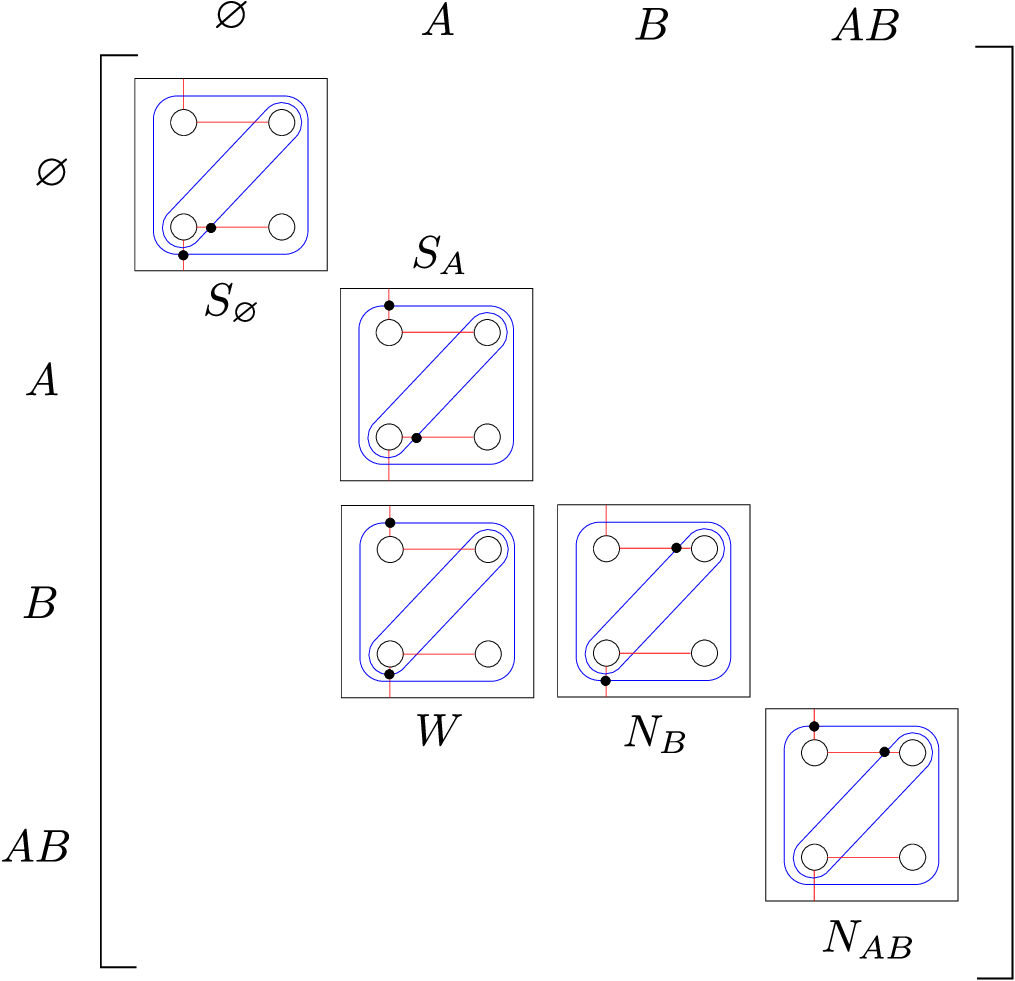}
\caption{Generators of the Ozsv{\'a}th--Szab{\'o} positive crossing bimodule $P_{\OSz}$ in terms of intersection points in the positive-crossing diagram of Figure~\ref{fig:OSzDiagramRightside}.}
\label{fig:OSzNegativeGens}
\end{figure}

\begin{figure}
\includegraphics[scale=0.6]{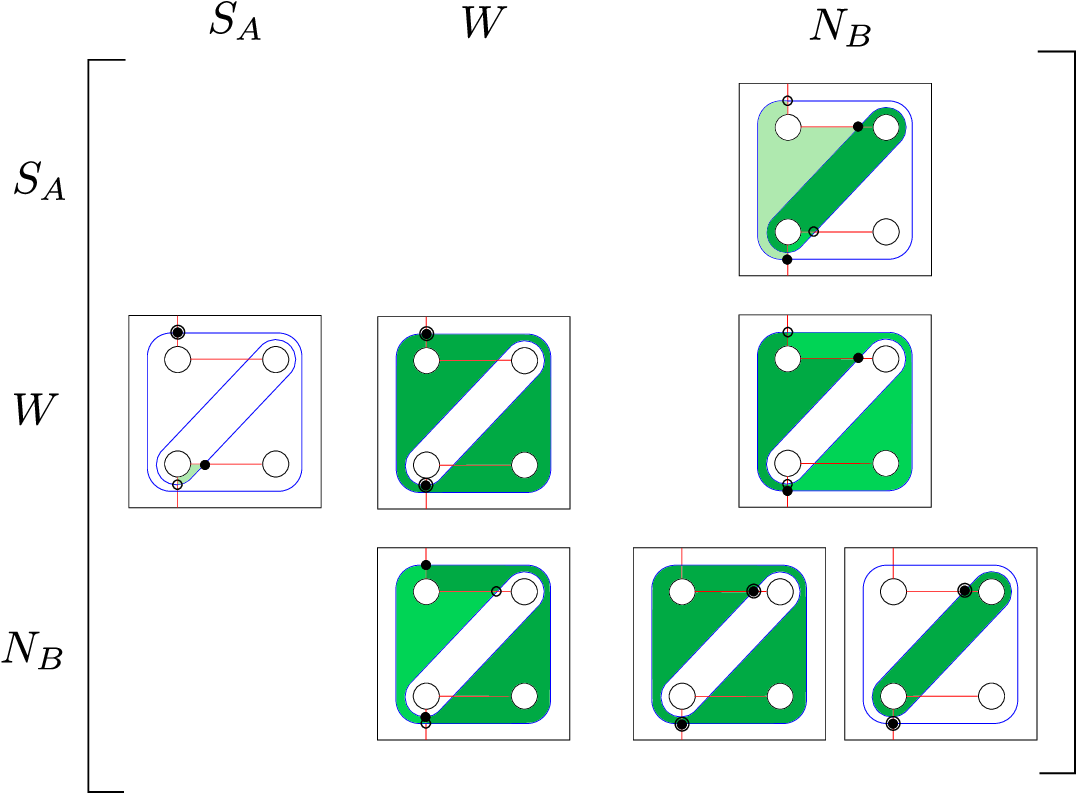}
\caption{Domains giving rise to secondary matrix entries for the Ozsv{\'a}th--Szab{\'o} positive crossing bimodule $P_{\OSz}$ (middle summand).}
\label{fig:OSzNegativeDomains}
\end{figure}

To relate $P_{\OSz}$ with our bimodule $P$, we first change basis on the left of $P_{\OSz}$ to get a DA bimodule over $(\A_{1,1},\A_{1,1}^{\can})$. In other words, we take the box tensor product of $P_{\OSz}$ with the DA bimodule over $(\A_{1,1},\A_{1,1}^{\can})$ from Definition~\ref{def:SecondCOBBimodule}. The result has primary matrix
\[
\kbordermatrix{
& \varnothing & A & B & AB \\
\uu & \lambda_{1} S_{\varnothing} & & & \\
\ou & & \lambda_2 S_A \quad \lambda_3 W & \lambda_3 N_B & \\
\uo & & \lambda_4 W & \lambda_4 N_B & \\
\oo & & & & \lambda_5 N_{AB}
},
\]
secondary matrix
\[
\kbordermatrix{
&\lambda_2 S_A & \lambda_3 W & \lambda_4 W & \lambda_3 N_B & \lambda_4 N_B \\
\lambda_2 S_A & 0 & 0 & 0 & U_1^{k+1} \otimes (U_2^{k+1}, L_1) & 0 \\
\lambda_3 W & 1 & 0 & \lambda & 1 \otimes L_1 & 0 \\
\lambda_4 W & 0 & 0 & U_2^{k+1} \otimes U_1^{k+1} & 0 & U_2^{k} \otimes L_1 U_1^k \\
\lambda_3 N_B & 0 & 0 & 0 & U_1^{k+1} \otimes U_2^{k+1} & \lambda \\
\lambda_4 N_B & 0 & 0 & U_2^{k+1} \otimes R_1 U_1^k & 0 & U_2^{k+1} \otimes U_1^{k+1}
}
\]
in the middle summand, and secondary matrix
\[
\kbordermatrix{
& \lambda_5 N_{AB} \\
\lambda_5 N_{AB} & U_1^l U_2^k \otimes U_1^k U_2^l
}
\]
in the lower summand. Simplifying the above bimodule, we get primary matrix
\[
\kbordermatrix{
& \varnothing & A & B & AB \\
\uu & \lambda_{1} S_{\varnothing} & & & \\
\ou & & & \lambda_3 N_B & \\
\uo & & \lambda_4 W & \lambda_4 N_B & \\
\oo & & & & \lambda_5 N_{AB}
}
\]
and secondary matrix
\[
\kbordermatrix{
& \lambda_4 W & \lambda_3 N_B & \lambda_4 N_B \\
\lambda_4 W &  U_2^{k+1} \otimes U_1^{k+1} & 0 & U_2^k \otimes L_1 U_1^k \\
\lambda_3 N_B & 0 & U_1^{k+1} \otimes U_2^{k+1} & \lambda \\
\lambda_4 N_B & U_2^{k+1} \otimes R_1 U_1^k & 0 & U_2^{k+1} \otimes U_1^{k+1}
}
\]
in the middle summand.

Now we change basis on the right as well, by taking a further box tensor product with the DA bimodule over $(\A_{1,1}^{\can},\A_{1,1})$ from Definition~\ref{def:FirstCOBBimodule}. The result has primary matrix
\[
\kbordermatrix{
& \uu & \ou & \uo & \oo \\
\uu & \lambda_{1} S_{\varnothing} \kappa_1 & & & \\
\ou & & & \lambda_3 N_B \kappa_4 & \\
\uo & & \lambda_4 W \kappa_2 & \lambda_4 N_B \kappa_4 \quad \lambda_4 W \kappa_3  & \\
\oo & & & & \lambda_5 N_{AB} \kappa_5
},
\]
secondary matrix
\[
\kbordermatrix{
& \lambda_4 W \kappa_2 & \lambda_3 N_B \kappa_4 & \lambda_4 N_B \kappa_4 & \lambda_4 W \kappa_3\\
\lambda_4 W \kappa_2 & U_2^{k+1} \otimes U_1^{k+1} & 0 & U_2^k \otimes (\lambda, U_1^{k+1}) & 1 \otimes \lambda \\
\lambda_3 N_B \kappa_4 & 0 & U_1^{k+1} \otimes U_2^{k+1} & \lambda & 0 \\
\lambda_4 N_B \kappa_4 & 0 & 0 & 0 & U_2 \\
\lambda_4 W \kappa_3 & 0 & 0 & 0 & 0 \\
}
\]
in the middle summand, and secondary matrix
\[
\kbordermatrix{
& \lambda_5 N_{AB} \kappa_5 \\
\lambda_5 N_{AB} \kappa_5 & U_1^l U_2^k \otimes U_1^k U_2^l
}
\]
in the lower summand. This bimodule agrees with the positive-crossing bimodule from Section~\ref{sec:PositiveCrossingHD} under the grading-preserving identification
\begin{align*}
& \lambda_1 S_{\varnothing} \kappa_1 \leftrightarrow I \\
& \lambda_4 W \kappa_2 \leftrightarrow J \\
& \lambda_3 N_B \kappa_4 \leftrightarrow K \\
& \lambda_4 N_B \kappa_4 \leftrightarrow L \\
& \lambda_4 W \kappa_3 \leftrightarrow M \\
& \lambda_5 N_{AB} \kappa_5 \leftrightarrow N.
\end{align*}

\subsection{Negative crossing}

\begin{proposition}
Ozsv{\'a}th--Szab{\'o}'s DA bimodule $N_{\OSz}$ over $\A_{1,1}^{\can}$ for a negative crossing (in Ozsv{\'a}th--Szab{\'o}'s papers, a positive crossing as remarked above) has primary matrix
\[
\kbordermatrix{ & \varnothing & A & B & AB \\
\varnothing & S'_{\varnothing} & & & \\
A & & S'_A & & \\
B & & W' & N'_B & \\
AB & & & & N'_{AB}
}.
\]
In the top weight space, the secondary matrix is zero and the bimodule is $\F_2$ as a bimodule over itself (generated by $S'_{\varnothing}$). The secondary matrix in the middle weight space is
\[
\kbordermatrix{
& S'_A & W' & N'_B \\
S_A' & 0 & L_1 & 0 \\
W' & 0 & U_2^{k+1} \otimes U_1^{k+1} & U_2^{k+1} \otimes L_1 U_1^k \\
N'_B & R_1 U_1^k \otimes (R_1, U_2^{k+1}) & U_2^k \otimes R_1 U_1^k & \begin{matrix} U_2^{k+1} \otimes U_1^{k+1} \\+ U_1^{k+1} \otimes U_2^{k+1} \end{matrix}
}
\]
and the secondary matrix in the bottom weight space is
\[
\kbordermatrix{
& N'_{AB} \\
N'_{AB} & U_1^l U_2^k \otimes U_1^k U_2^l
}.
\]
\end{proposition}
The degrees of the generators depend only on their type $\in \{N',E',S',W'\}$ and are given by
\begin{itemize}
\item $\deg^q(N') = -1$, $\deg^h(N') = 1$,
\item $\deg^q(E') = 0$, $\deg^h(E') = 0$,
\item $\deg^q(S') = 1$, $\deg^h(S') = 0$,
\item $\deg^q(W') = 0$, $\deg^h(W') = 0$
\end{itemize}
as described in the bottom row of Figure~\ref{fig:AlexMaslov} (again, there are no generators of type $E'$.) By \cite{OSzHolo}, the bimodule $N_{\OSz}$ can be obtained by counting holomorphic disks in the negative-crossing Heegaard diagram of Figure~\ref{fig:OSzDiagramRightside}.

\begin{figure}
\includegraphics[scale=0.6]{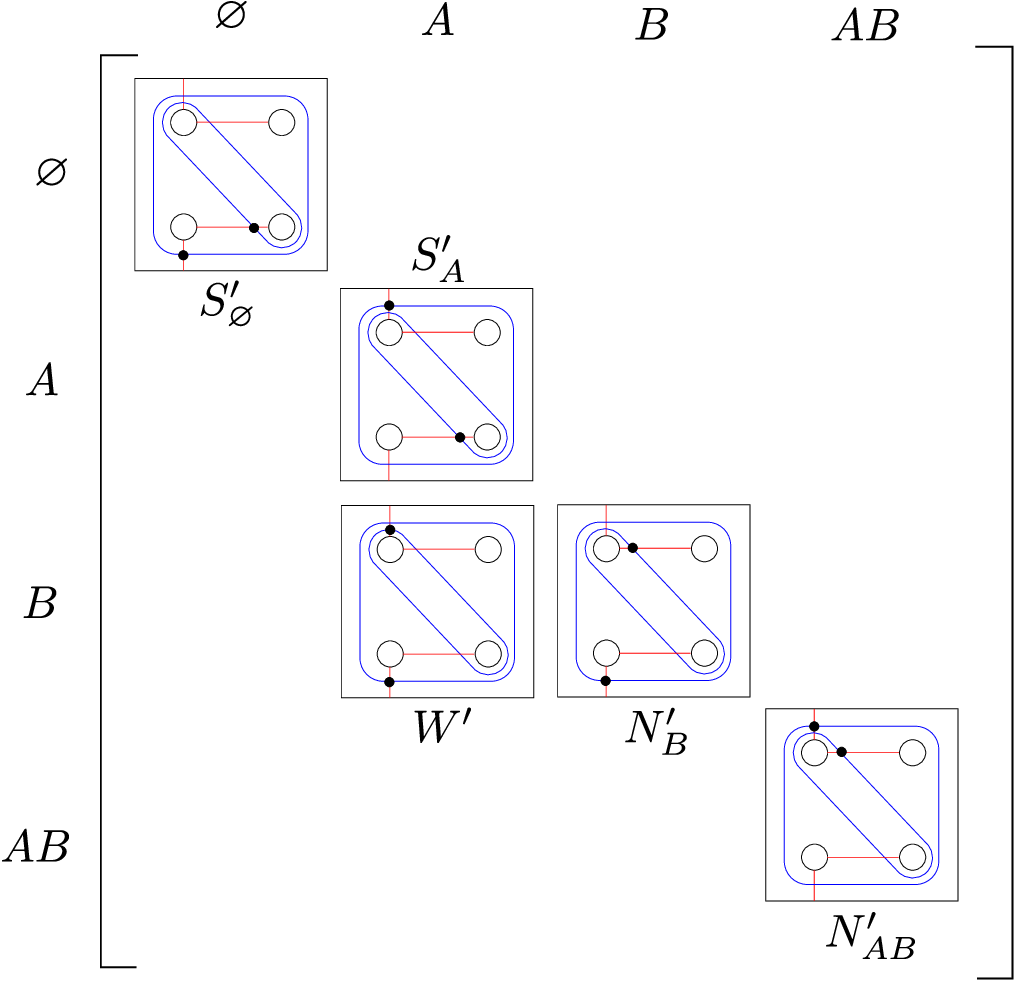}
\caption{Generators of the Ozsv{\'a}th--Szab{\'o} negative crossing bimodule $N_{\OSz}$ in terms of intersection points in the negative-crossing diagram of Figure~\ref{fig:OSzDiagramRightside}.}
\label{fig:OSzPositiveGens}
\end{figure}

\begin{figure}
\includegraphics[scale=0.6]{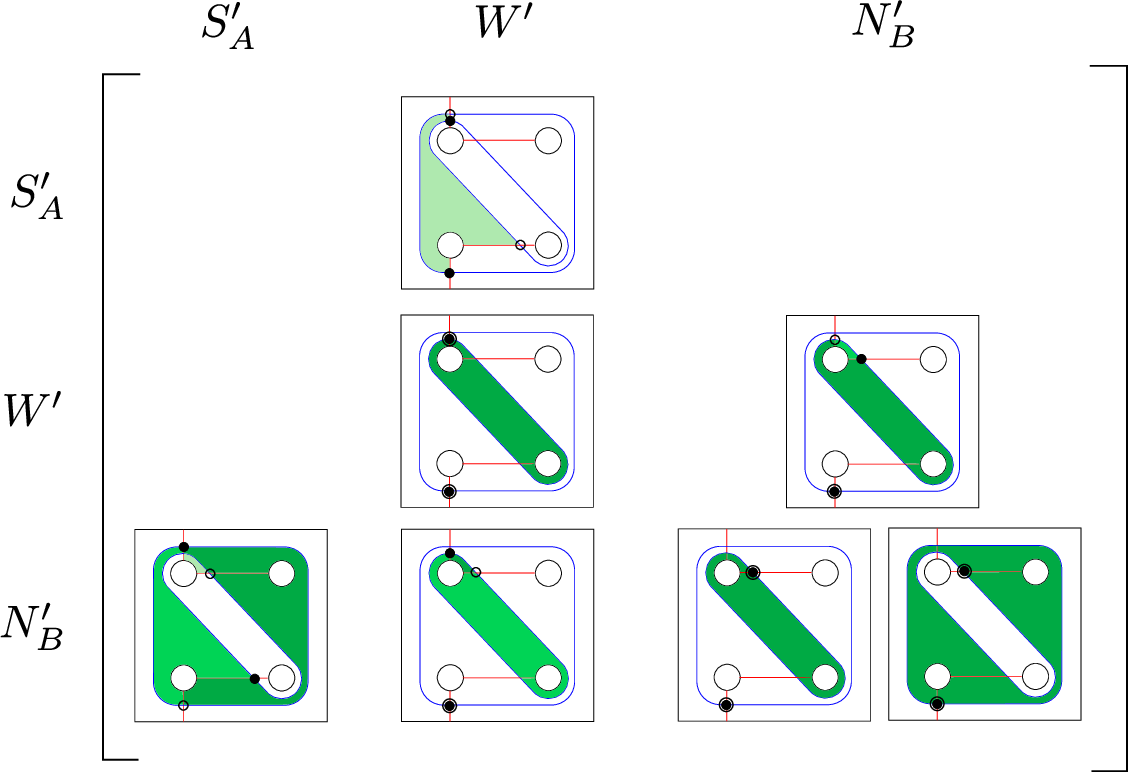}
\caption{Domains giving rise to secondary matrix entries for the Ozsv{\'a}th--Szab{\'o} negative crossing bimodule $N_{\OSz}$ (middle summand).}
\label{fig:OSzPositiveDomains}
\end{figure}

To relate $N_{\OSz}$ with our bimodule $N$, we first change basis on the right of $N_{\OSz}$ to get a DA bimodule over $(\A_{1,1}^{\can},\A_{1,1})$. In other words, we take the box tensor product of $N_{\OSz}$ with the DA bimodule over $(\A_{1,1}^{\can},\A_{1,1})$ from Definition~\ref{def:FirstCOBBimodule}. The result has primary matrix
\[
\kbordermatrix{
& \uu & \ou & \uo & \oo \\
\varnothing & S'_{\varnothing} \kappa_1 & & & \\
A & & S'_A \kappa_2 & S'_A \kappa_3 & \\
B & & W' \kappa_2 & W' \kappa_3 \quad N'_B \kappa_4 & \\
AB & & & & N'_{AB} \kappa_5
},
\]
secondary matrix
\[
\kbordermatrix{
&S'_A \kappa_2 & W' \kappa_2 & S'_A \kappa_3 & W' \kappa_3 & N'_B \kappa_4 \\
S'_A \kappa_2 & 0 & L_1 & 1 \otimes \lambda & 0 & 0 \\
W' \kappa_2 & 0 & U_2^{k+1} \otimes U_1^{k+1} & 0 & 1 \otimes \lambda & U_2^{k+1} \otimes (\lambda, U_1^{k+1}) \\
S'_A \kappa_3 & 0 & 0 & 0 & L_1 & 0 \\
W' \kappa_3 & 0 & 0 & 0 & 0 & 0 \\
N'_B \kappa_4 & 0 & 0 & R_1 U_1^k \otimes U_2^{k+1} & 1 & U_1^{k+1} \otimes U_2^{k+1}
}
\]
in the middle summand, and secondary matrix
\[
\kbordermatrix{
& N'_{AB} \kappa_5 \\
N'_{AB} \kappa_5 & U_1^l U_2^k \otimes U_1^k U_2^l
}
\]
in the lower summand. Simplifying the above bimodule, we get primary matrix
\[
\kbordermatrix{
& \uu & \ou & \uo & \oo \\
\varnothing & S'_{\varnothing} \kappa_1 & & & \\
A & & S'_A \kappa_2 & S'_A \kappa_3 & \\
B & & W' \kappa_2 & & \\
AB & & & & N'_{AB} \kappa_5
},
\]
and secondary matrix
\[
\kbordermatrix{
&S'_A \kappa_2 & W' \kappa_2 & S'_A \kappa_3  \\
S'_A \kappa_2 & 0 & L_1 & 1 \otimes \lambda \\
W' \kappa_2 & 0 & U_2^{k+1} \otimes U_1^{k+1} & R_1 U_1^k \otimes (U_2^{k+1}, \lambda) \\
S'_A \kappa_3 & 0 & 0 & U_1^{k+1} \otimes U_2^{k+1}
}
\]
in the middle summand.

Now we change basis on the left as well, by taking a further box tensor product with the DA bimodule over $(\A_{1,1},\A_{1,1}^{\can})$ from Definition~\ref{def:SecondCOBBimodule}. The result has primary matrix
\[
\kbordermatrix{
& \uu & \ou & \uo & \oo \\
\uu & \lambda_{1} S'_{\varnothing} \kappa_1 & & & \\
\ou & & \lambda_2 S'_A \kappa_2 \quad \lambda_3 W' \kappa_2 & \lambda_2 S'_A \kappa_3 & \\
\uo & & \lambda_4 W' \kappa_2 & & \\
\oo & & & & \lambda_5 N'_{AB} \kappa_5
},
\]
secondary matrix
\[
\kbordermatrix{
& \lambda_2 S'_A \kappa_2 & \lambda_3 W' \kappa_2 & \lambda_4 W' \kappa_2 & \lambda_2 S'_A \kappa_3 \\
\lambda_2 S'_A \kappa_2 & 0 & U_1 & 0 & 1 \otimes \lambda \\
\lambda_3 W' \kappa_2 & 0 & 0 & \lambda & U_1^k \otimes (U_2^{k+1},\lambda) \\
\lambda_4 W' \kappa_2 & 0 & 0 & U_2^{k+1} \otimes U_1^{k+1} & 0 \\
\lambda_2 S'_A \kappa_3 & 0 & 0 & 0 & U_1^{k+1} \otimes U_2^{k+1} \\
}
\]
in the middle summand, and secondary matrix
\[
\kbordermatrix{
& \lambda_5 N'_{AB} \kappa_5 \\
\lambda_5 N'_{AB} \kappa_5 & U_1^l U_2^k \otimes U_1^k U_2^l
}
\]
in the lower summand. This bimodule agrees with the negative-crossing bimodule from Section~\ref{sec:PositiveCrossingHD} under the grading-preserving identification
\begin{align*}
& \lambda_1 S'_{\varnothing} \kappa_1 \leftrightarrow I' \\
& \lambda_2 S'_A \kappa_2 \leftrightarrow J' \\
& \lambda_3 W' \kappa_2 \leftrightarrow K' \\
& \lambda_4 W' \kappa_2 \leftrightarrow L' \\
& \lambda_2 S'_A \kappa_3 \leftrightarrow M' \\
& \lambda_5 N'_{AB} \kappa_5 \leftrightarrow N'.
\end{align*}

\appendix

\section{Representations of \texorpdfstring{$\U_q(\gl(1|1))$}{Uq(gl(1|1))}}\label{sec:Uqgl11Review}

\subsection{Definition and fundamental representations}

All Lie algebras and superalgebras are assumed to have ground field $\C$ (for the ordinary versions) or $\C_q \coloneqq \C(q)$ (for the quantum versions).

\begin{definition}
The Hopf superalgebra $U_q(\gl(1|1))$ is generated as a superalgebra over $\C_q$ by two even generators $q^{H_1}, q^{H_2}$ (with inverses $q^{-H_1}, q^{-H_2}$) and two odd generators $E,F$ with relations
\begin{itemize}
\item $q^{H_1}E = qEq^{H_1}$,
\item $q^{H_2}E = q^{-1}Eq^{H_2}$,
\item $q^{H_1}F = q^{-1}Fq^{H_1}$,
\item $q^{H_2}F = qFq^{H_2}$,
\item $q^{H_1}q^{H_2} = q^{H_2}q^{H_1}$,
\item $E^2 = F^2 = 0$,
\item $EF + FE = \frac{q^{H_1+H_2} - q^{-H_1-H_2}}{q - q^{-1}}$.
\end{itemize}
The comultiplication is given by 
\begin{itemize}
\item $\Delta(E) = E \otimes q^{-H_1-H_2} + 1 \otimes E$, 
\item $\Delta(F) = F \otimes 1 + q^{H_1+H_2} \otimes F$, 
\item $\Delta(q^{H_i}) = q^{H_i} \otimes q^{H_i}$.
\end{itemize}
The counit is given by $\varepsilon(E) = \varepsilon(F) = 0$ and $\varepsilon(q^{H_i}) = 1$. The antipode is given by $S(E) = -Eq^{H_1+H_2}$, $S(F) = -q^{-H_1-H_2}F$, and $S(q^{H_i}) = q^{-H_i}$. 

\end{definition}

\begin{remark}
We follow the conventions of Queffelec--Sartori \cite{QS}, which agree (as far as we are aware) with those in \cite{SartoriAlexander,SartoriCat}.

\end{remark}

If we replace $\gl(1|1)$ by its negative half $\gl(1|1)^-$, we get a simpler superalgebra 
\[
U_q(\gl(1|1)^-) \coloneqq \C_q[F]/(F^2)
\]
without a coproduct. The subalgebra of $U_q(\gl(1|1))$ generated by $F$, $q^{H_1}$, and $q^{H_2}$ is a Hopf subalgebra; we can also work with a related $\C_q$-linear super category that we will call $\Udot_q(\gl(1|1)^-)$ (the dot indicates an idempotented form, and we take idempotents corresponding to the two-dimensional lattice of $\gl(1|1)$ weights). Concretely, $\Udot_q(\gl(1|1)^-)$ has objects $\omega = \omega_1 \varepsilon_1 + \omega_2 \varepsilon_2 \in \Z^2$; morphisms are generated by $F_{\omega}\co \omega \to \omega - \alpha$ for $\omega \in \Z^2$ (all of which are odd) modulo the relations $F_{\omega - \alpha} F_{\omega} = 0$. If $W_1$ and $W_2$ are modules over $\Udot_q(\gl(1|1)^-)$, we set 
\[
(W_1 \otimes W_2)_{\omega} = \bigoplus_{\omega' + \omega'' = \omega} (W_1)_{\omega'} \otimes (W_2)_{\omega''}.
\]
The action of $F_{\omega}$ on the summand $(W_1)_{\omega'} \otimes (W_2)_{\omega''}$ of $(W_1 \otimes W_2)_{\omega}$ is defined to be the action of $F_{\omega'} \otimes \id_{\omega''} + q^{\omega'_1 + \omega'_2} \id_{\omega'} \otimes F_{\omega''}$ with the usual super sign rules.

We can identify modules over $\Udot_q(\gl(1|1)^-)$ with modules over the subalgebra of $U_q(\gl(1|1))$ generated by $F, q^{H_1}$, and $q^{H_2}$ admitting a decomposition such that $q^{H_i}$ acts as an integral power of $q$ on each summand. In particular, all modules over $U_q(\gl(1|1))$ considered below give modules over $\Udot_q(\gl(1|1)^-)$.

\begin{definition}
The defining representation $V$ of $U_q(\gl(1|1))$ has one even basis element $\ket{0}$ and one odd basis element $\ket{1}$ as as super vector space over $\C_q$. The action of $E$ has matrix
\[
\kbordermatrix{
& \ket{0} & \ket{1} \\
\ket{0} & 0 & 1 \\
\ket{1} & 0 & 0 \\
},
\]
the action of $F$ has matrix
\[
\kbordermatrix{
& \ket{0} & \ket{1} \\
\ket{0} & 0 & 0 \\
\ket{1} & 1 & 0 \\
},
\]
the action of $q^{H_1}$ has matrix
\[
\kbordermatrix{
& \ket{0} & \ket{1} \\
\ket{0} & q & 0 \\
\ket{1} & 0 & 1 \\
},
\]
and the action of $q^{H_2}$ has matrix
\[
\kbordermatrix{
& \ket{0} & \ket{1} \\
\ket{0} & 1 & 0 \\
\ket{1} & 0 & q \\
}
\]
\end{definition}

\begin{definition}
The category $\Rep(U_q(\gl(1|1)))$ of finite-dimensional representations of $U_q(\gl(1|1))$ is ribbon (see e.g. \cite{SartoriAlexander}) and thus braided, so the constructions of \cite{BZ} give us braided symmetric and exterior algebras on objects of $\Rep(U_q(\gl(1|1)))$. In particular, if $W$ is an object of $\Rep(U_q(\gl(1|1)))$ and $K \geq 0$, then we have another object $\wedge_q^K W$ of $\Rep(U_q(\gl(1|1)))$.

\end{definition}

\begin{example}[(3.15) of \cite{QS}]
For the vector representation $V$ of $U_q(\gl(1|1))$, the super quantum exterior algebra $\wedge^{*}_q V$ is the quotient of the tensor (super)algebra $T^*(V)$ by the ideal generated by the elements $\ket{0} \otimes \ket{0}$ and $\ket{0} \otimes \ket{1} + q^{-1} \ket{1} \otimes \ket{0}$ of $T^2 (V) = V \otimes V$. 

\end{example}

We can take $\ket{0} \wedge \ket{1}^{K-1}$ and $\ket{1}^K$ as a basis for $\wedge_q^K V$.

\begin{proposition}
The action of $E$ on $\wedge^K_q V$ has matrix
\[
\kbordermatrix{
& \ket{0} \wedge \ket{1}^{K-1} &\ket{1}^K \\
\ket{0} \wedge \ket{1}^{K-1} & 0 & [K]_q \\
\ket{1}^K & 0 & 0
}
\]
where $[K]_q$ is the quantum integer $\frac{q^K - q^{-K}}{q-q^{-1}}$. The action of $F$ on $\wedge^K_q V$ has matrix
\[
\kbordermatrix{
& \ket{0} \wedge \ket{1}^{K-1} & \ket{1}^K \\
\ket{0} \wedge \ket{1}^{K-1} & 0 & 0 \\
\ket{1}^K & 1 & 0
}.
\]
The action of $q^{H_1}$ on $\wedge^K_q V$ has matrix
\[
\kbordermatrix{
& \ket{0} \wedge \ket{1}^{k-1} & \ket{1}^k \\
\ket{0} \wedge \ket{1}^{k-1} & q & 0 \\
\ket{1}^k & 0 & 1
}.
\]
The action of $q^{H_2}$ on $\wedge^K_q V$ has matrix
\[
\kbordermatrix{
& \ket{0} \wedge \ket{1}^{K-1} & \ket{1}^K \\
\ket{0} \wedge \ket{1}^{K-1} & q^{K-1} & 0 \\
\ket{1}^K & 0 & q^K
}.
\]
\end{proposition}

\begin{remark}
We have an isomorphism between $\wedge^K_q V$ and the irreducible representation $L(\lambda)$ where $\lambda = \varepsilon_1 + (K-1)\varepsilon_2$; an explicit formula for this latter representation is given in \cite[(3.4)]{SartoriAlexander}. The isomorphism identifies $\ket{0} \wedge \ket{1}^{K-1}$ with $v_0^{\lambda}$ and $\ket{1}^K$ with $[K]_q v_1^{\lambda}$ (this factor of $[K]_q$ is responsible for the slightly different appearance of Sartori's formulas).
\end{remark}

\subsection{Canonical bases}\label{sec:NonstandardBases}

Note that as $\C_q$-vector spaces, we can identify $V^{\otimes n}$ with the exterior algebra on a space formally generated by elements $e_1,\ldots,e_n$. Specifically, for $0 \leq k \leq n$, we identify $e_{i_{n-k}} \wedge \cdots \wedge e_{i_1}$ with the standard basis vector of $V^{\otimes n}$ having $\ket{0}$ in positions $i_1,\ldots,i_{n-k}$ and having $\ket{1}$ in the other positions.

\begin{remark}
Due to changes of convention (apparently related to the difference between $V$ and $V^*$), this identification of $V^{\otimes n}$ with an exterior algebra differs from the ones considered in \cite{ManionDecat, LaudaManion}.
\end{remark}

For $2 \leq i \leq n$, we let $\ell_i = qe_{i-1} + e_i$; let $\ell_1 = e_1$.

\begin{definition}
The canonical basis for $V^{\otimes n}$ is
\[
\{ \ell_{i_{n-k}} \wedge \cdots \wedge \ell_{i_{1}} : 0 \leq k \leq n, 1 \leq i_1 < \cdots  < i_{n-k} \leq n \}.
\]
\end{definition}

\begin{example}
The change-of-basis matrix from the canonical basis to the standard basis for $V \otimes V$ has matrix
\[
\kbordermatrix{
& \ell_2 \wedge \ell_1 & \ell_2 & \ell_1 & 1 \\
\ket{00} & 1 & 0 & 0 & 0 \\
\ket{10} & 0 & 1 & 0 & 0 \\
\ket{01} & 0 & q & 1 & 0 \\
\ket{11} & 0 & 0 & 0 & 1
}.
\]
Its inverse has matrix
\[
\kbordermatrix{
& \ket{00} & \ket{10} & \ket{01} & \ket{11} \\
\ell_2 \wedge \ell_1 & 1 & 0 & 0 & 0 \\
\ell_2 & 0 & 1 & 0 & 0 \\
\ell_1 & 0 & -q & 1 & 0 \\
1 & 0 & 0 & 0 & 1
}.
\]
\end{example}

\subsection{Maps from skew Howe duality}\label{sec:AppendixSkewHoweMaps}

Let $\Udot_q(\gl(2))$ denote the idempotented form of $U_q(\gl(2))$. Objects of $\Udot_q(\gl(2))$ are written as $1_{\lambda}$ for $\lambda \in \Z^2$; morphisms are generated by $1_{\lambda + \alpha} E_{\gl(2)} 1_{\lambda}$ from $1_{\lambda}$ to $1_{\lambda + \alpha}$ and $1_{\lambda - \alpha} F_{\gl(2)} 1_{\lambda}$ from $1_{\lambda}$ to $1_{\lambda - \alpha}$ for $\lambda \in \Z^2$, where $\alpha = (1,-1)$. 

In this section, write $\C_q^{1|1}$ instead of $V$ to avoid confusion with representation $\C_q^2$ of $\Udot(\gl(2))$. We will analyze $\wedge^K_q(\C_q^{1|1} \otimes \C_q^2)$ for $K \geq 1$ as a representation of $\Udot_q(\gl(2))$. We start by reviewing the usual names for basis elements of $\C_q^{1|1} \otimes \C_q^2$. Following \cite{QS}, define:
\begin{itemize}
\item $z_{11} = \ket{0}_{1|1} \otimes \ket{0}_{2}$
\item $z_{21} = \ket{1}_{1|1} \otimes \ket{0}_{2}$
\item $z_{12} = \ket{0}_{1|1} \otimes \ket{1}_{2}$
\item $z_{22} = \ket{1}_{1|1} \otimes \ket{1}_{2}$.
\end{itemize}
The elements $z_{11}$ and $z_{12}$ are even, while $z_{21}$ and $z_{22}$ are odd. The quantum exterior algebra $\wedge^{\bullet}_q(\C_q^{1|1} \otimes \C_q^2)$ can be described by imposing the relations
\begin{itemize}
\item $z_{1i} \wedge_q z_{1i} = 0$ for $i = 1,2$,
\item $z_{1i} \wedge_q z_{2i} = -q^{-1} z_{2i} \wedge_q z_{1i}$ for $i = 1,2$,
\item $z_{11} \wedge_q z_{12} = -q^{-1} z_{12} \wedge_q z_{11}$,
\item $z_{11} \wedge_q z_{22} = -z_{22} \wedge_q z_{11}$,
\item $z_{21} \wedge_q z_{12} = - z_{12} \wedge_q z_{21} + (q^{-1} - q) z_{11} \wedge z_{22}$,
\item $z_{21} \wedge z_{22} = q z_{22} \wedge z_{21}$
\end{itemize}
on the tensor algebra of $\C_q^{1|1} \otimes \C_q^2$ (see \cite[relations (4.2)]{QS}).

If $K_1 + K_2 = K$, a basis for the summand $\wedge^{K_1}_q \C_q^{1|1} \otimes \wedge^{K_2}_q \C_q^{1|1}$ of $\wedge^K_q(\C_q^{1|1} \otimes \C_q^2)$ is given by those basis elements of $\wedge^K_q(\C_q^{1|1} \otimes \C_q^2)$ with $K_1$ instances of $z_{i1}$ and $K_2$ instances of $z_{i2}$. For $K_1, K_2 > 0$ there are four such basis elements, namely
\begin{align*}
&\bigg\{ z_{11} \wedge_q (z_{21})^{\wedge_q (K_1 - 1)} \wedge_q z_{12} \wedge_q (z_{22})^{\wedge_q (K_2 - 1)}, \\
& (z_{21})^{\wedge_q K_1} \wedge_q z_{12} \wedge_q (z_{22})^{\wedge_q (K_2 - 1)}, \\
& z_{11} \wedge_q (z_{21})^{\wedge_q (K_1 - 1)} \wedge_q (z_{22})^{\wedge_q K_2}, \\
& (z_{21})^{\wedge_q K_1} \wedge_q (z_{22})^{\wedge_q K_2} \bigg\}.
\end{align*}

If $K_1$ or $K_2$ equals zero, then there are only two such basis elements (those with nonnegative exponents in the wedge products).  When $K_1, K_2 > 0$, we will write these basis elements as $\omega_i$ for $1 \leq i \leq 4$; when $(K_1,K_2) = (K,0)$ or $(0,K)$, we write the basis elements as $\{\omega_1, \omega_2\}$.

Via its coproduct, $\Udot_q(\gl(2))$ acts on the tensor algebra of $\C_q^{1|1} \otimes \C_q^2$ since it acts on $\C_q^2$. The action descends to an action on $\wedge^K(\C_q^{1|1} \otimes \C_q^2)$. We describe this action explicitly below.

First, if $K = 1$, the morphism $1_{0,1} F_{\gl(2)} 1_{1,0}$ of $\Udot_q(\gl(2))$ gives us the morphism from $\wedge^1_q \C_q^{1|1} \otimes \wedge^0_q \C_q^{1|1}$ to $\wedge^0_q \C_q^{1|1} \otimes \wedge^1_q \C_q^{1|1}$ with matrix
\[
\kbordermatrix{
& z_{11} & z_{21} \\
z_{12} & 1 & 0 \\
z_{22} & 0 & 1
}.
\]
The morphism $1_{1,0} E_{\gl(2)} 1_{0,1}$ of $\Udot_q(\gl(2))$ gives us the morphism from $\wedge^0_q \C_q^{1|1} \otimes \wedge^1_q \C_q^{1|1}$ to $\wedge^1_q \C_q^{1|1} \otimes \wedge^0_q \C_q^{1|1}$ with matrix
\[
\kbordermatrix{
& z_{12} & z_{22} \\
z_{11} & 1 & 0 \\
z_{21} & 0 & 1
}.
\]

Now assume $K \geq 2$. The morphism $1_{K-1,1} F_{\gl(2)} 1_{K,0}$ of $\Udot_q(\gl(2))$ gives us the morphism from $\wedge^K_q \C_q^{1|1} \otimes \wedge^0_q \C_q^{1|1}$ to $\wedge^{K-1}_q \C_q^{1|1} \otimes \wedge^1_q \C_q^{1|1}$ with matrix
\[
\kbordermatrix{
& \omega_1 & \omega_2 \\
\omega_1 & 0 & 0 \\
\omega_2 & (-1)^{K-1} & 0 \\
\omega_3 & q^{-1} [K-1] & 0 \\
\omega_4 & 0 & [K]
}.
\]
For $K_1, K_2 > 0$ and $K_1 + K_2 = K$, the morphism $1_{K_1 - 1, K_2 + 1} F_{\gl(2)} 1_{K_1,K_2}$ of $\Udot_q(\gl(2))$ gives us the morphism from $\wedge^{K_1}_q \C_q^{1|1} \otimes \wedge^{K_2}_q \C_q^{1|1}$ to $\wedge^{K_1 - 1}_q \C_q^{1|1} \otimes \wedge^{K_2 + 1}_q \C_q^{1|1}$ with matrix
\[
\kbordermatrix{
& \omega_1 & \omega_2 & \omega_3 & \omega_4 \\
\omega_1 & -[K_1-1] & 0 & 0 & 0 \\
\omega_2 & 0 & -q[K_1] & (-1)^{K_1 - 1} & 0 \\
\omega_3 & 0 & 0 & q^{-1} [K_1 - 1] & 0 \\
\omega_4 & 0 & 0 & 0 & [K_1]
}.
\]
The morphism $1_{0,K} F_{\gl(2)} 1_{1,K-1}$ of $\Udot_q(\gl(2))$ gives us the morphism from $\wedge^{1}_q \C_q^{1|1} \otimes \wedge^{K-1}_q \C_q^{1|1}$ to $\wedge^{0}_q \C_q^{1|1} \otimes \wedge^{K}_q \C_q^{1|1}$ with matrix
\[
\kbordermatrix{
& \omega_1 & \omega_2 & \omega_3 & \omega_4 \\
\omega_1 & 0 & -q & 1 & 0 \\
\omega_2 & 0 & 0 & 0 & 1
}.
\]
The morphism $1_{K,0} E_{\gl(2)} 1_{K-1,1}$ of $\Udot_q(\gl(2))$ gives us the morphism from $\wedge^{K-1}_q \C_q^{1|1} \otimes \wedge^{1}_q \C_q^{1|1}$ to $\wedge^{K}_q \C_q^{1|1} \otimes \wedge^{0}_q \C_q^{1|1}$ with matrix
\[
\kbordermatrix{
& \omega_1 & \omega_2 & \omega_3 & \omega_4 \\
\omega_1 & 0 & (-1)^{K-1} q^{K-1} & 1 & 0 \\
\omega_2 & 0 & 0 & 0 & 1
}.
\]
For $K_1, K_2 > 0$ and $K_1 + K_2 = K$, the morphism $1_{K_1 + 1, K_2 - 1} E_{\gl(2)} 1_{K_1,K_2}$ of $\Udot_q(\gl(2))$ gives us the morphism from $\wedge^{K_1}_q \C_q^{1|1} \otimes \wedge^{K_2}_q \C_q^{1|1}$ to $\wedge^{K_1 + 1}_q \C_q^{1|1} \otimes \wedge^{K_2 - 1}_q \C_q^{1|1}$ with matrix
\[
\kbordermatrix{
& \omega_1 & \omega_2 & \omega_3 & \omega_4 \\
\omega_1 & -[K_2 - 1] & 0 & 0 & 0 \\
\omega_2 & 0 & -[K_2 - 1] & 0 & 0 \\
\omega_3 & 0 & (-1)^{K_1} q^{1 + K_1 - K_2} & [K_2] & 0 \\
\omega_4 & 0 & 0 & 0 & [K_2]
}.
\]
Finally, the morphism $1_{1,K-1} E_{\gl(2)} 1_{0,K}$ of $\Udot_q(\gl(2))$ gives us the morphism from $\wedge^0_q \C_q^{1|1} \otimes \wedge^{K}_q \C_q^{1|1}$ to $\wedge^{1}_q \C_q^{1|1} \otimes \wedge^{K-1}_q \C_q^{1|1}$ with matrix
\[
\kbordermatrix{
& \omega_1 & \omega_2 \\
\omega_1 & 0 & 0 \\
\omega_2 & -[K-1] & 0 \\
\omega_3 & q^{1-K} & 0 \\
\omega_4 & 0 & [K]
}.
\]

\subsection{Singular and nonsingular crossings}\label{sec:SingularNonsingular}

In this section we restrict to $K = 2$. The morphism $1_{1,1} F_{\gl(2)} 1_{2,0} E_{\gl(2)} 1_{1,1}$ gives us the morphism from $\C_q^{1|1} \otimes \C_q^{1|1}$ with matrix
\[
\kbordermatrix{
& \omega_1 & \omega_2 \\
\ket{00} & 0 & 0 \\
\ket{10} & -1 & 0 \\
\ket{01} & q^{-1} & 0 \\
\ket{11} & 0 & q+q^{-1}
}
\kbordermatrix{
& \ket{00} & \ket{10} & \ket{01} & \ket{11} \\
\omega_1 & 0 & -q & 1 & 0 \\
\omega_2 & 0 & 0 & 0 & 1
}
\]
which equals
\[
\kbordermatrix{
& \ket{00} & \ket{10} & \ket{01} & \ket{11} \\
\ket{00} & 0 & 0 & 0 & 0 \\
\ket{10} & 0 & q & -1 & 0 \\
\ket{01} & 0 & -1 & q^{-1} & 0 \\
\ket{11} & 0 & 0 & 0 & q+q^{-1}
}.
\]

The braiding on $\C_q^{1|1} \otimes \C_q^{1|1}$ is defined to be $q(\id) - X$, with inverse $q^{-1}(\id) - X$. Explicitly, the braiding has matrix
\[
\kbordermatrix{
& \ket{00} & \ket{10} & \ket{01} & \ket{11} \\
\ket{00} & q & 0 & 0 & 0 \\
\ket{10} & 0 & 0 & 1 & 0 \\
\ket{01} & 0 & 1 & q-q^{-1} & 0 \\
\ket{11} & 0 & 0 & 0 & -q^{-1}
}
\]
and its inverse has matrix
\[
\kbordermatrix{
& \ket{00} & \ket{10} & \ket{01} & \ket{11} \\
\ket{00} & q^{-1} & 0 & 0 & 0 \\
\ket{10} & 0 & q^{-1} - q & 1 & 0 \\
\ket{01} & 0 & 1 & 0 & 0 \\
\ket{11} & 0 & 0 & 0 & -q
}.
\]

\begin{remark}
We use the conventions of \cite[Section 4.2]{SartoriCat}; our matrix for the inverse braiding is the same as Sartori's matrix for $\check{H}$, which Sartori identifies with the inverse braiding. In \cite[equation (5.13)]{QS} the braiding and its inverse are interchanged.
\end{remark}

\bibliographystyle{fouralpha}
\bibliography{bib_clean}

\end{document}